\documentclass[reqno]{amsart}
\usepackage{mathrsfs}
\usepackage{graphicx}
\usepackage{amsmath}
\usepackage{amscd}
\usepackage{cite}
\usepackage{hyperref}
\usepackage{epsfig}
\usepackage{subfigure}
\usepackage{caption}
\usepackage{tikz}
  \usepackage{float}
  \usepackage{amssymb}
\usepackage{amsfonts}
\usepackage[all,cmtip]{xy}
\usepackage{appendix}
\usepackage{bm}
\usepackage{geometry}
\newtheorem{Theorem}{Theorem}[section]

\newtheorem{Proposition}[Theorem]{Proposition}
\newtheorem{Lemma}[Theorem]{Lemma}
\newtheorem{Corollary}[Theorem]{Corollary}
\newtheorem{Conjecture}[Theorem]{Conjecture}
\newtheorem{Remark}[Theorem]{Remark}

\newcommand{\bC}{{\mathbb C}}

\newcommand{\bN}{{\mathbb N}}
\newcommand{\bP}{{\mathbb P}}

\newcommand{\bZ}{{\mathbb Z}}

\newcommand{\cM}{{\mathcal M}}

\newcommand{\cP}{{\mathcal P}}

\newcommand{\half}{\frac{1}{2}}

\newcommand{\Mbar}{\overline{\cM}}

\newcommand{\tW}{{\widetilde W}}

\newcommand{\tJ}{{\widetilde J}}

\newcommand{\be}{\begin{equation}}
\newcommand{\ee}{\end{equation}}
\newcommand{\bea}{\begin{eqnarray}}
\newcommand{\ben}{\begin{eqnarray*}}
\newcommand{\een}{\end{eqnarray*}}
\newcommand{\eea}{\end{eqnarray}}

\DeclareMathOperator{\vdim}{vdim}

\usepackage{color}

\definecolor{yellow}{rgb}{1,1,0}
\definecolor{orange}{rgb}{1,.7,0}
\definecolor{red}{rgb}{1,0,0}
\definecolor{green}{rgb}{0,1,1}
\definecolor{white}{rgb}{1,1,1}

\definecolor{A}{rgb}{.75,1,.75}

\begin{document}

\newtheorem{myDef}{Definition}
\newtheorem{thm}{Theorem}
\newtheorem{eqn}{equation}

\title[Leading Gopakumar--Vafa invariants and leading Betti numbers]
{On leading Gopakumar--Vafa invariants of local $\mathbb{P}^2$\\ and leading Betti numbers of moduli spaces\\ of one-dimensional sheaves on $\mathbb{P}^2$}

\author{Zhiyuan Wang}
\address{Zhiyuan Wang, School of Mathematics and Statistics,
	Huazhong University of Science and Technology,
	Wuhan, China}
\email{wangzy23@hust.edu.cn}

\author{Chenglang Yang}
\address{Chenglang Yang, Institute for Math and AI, Wuhan University, Wuhan, China}
\email{yangcl@whu.edu.cn}

\begin{abstract}

In this paper we use the topological vertex formalism to derive a structural formula
for the leading terms of the generating series of Gromov--Witten invariants of local $\mathbb P^2$.
Then we give two applications of this formula.
First we use this result to derive an explicit formula for the leading terms of the transformed Gopakumar--Vafa invariants of local $\mathbb P^2$.
This is an extension of a formula of Guo--Zhou \cite{gz} from degree $2d-5$ to $3d-1$,
and shares some similar features with the famous G\"ottsche--Yau--Zaslow formula.
And as the second application, we prove Guo--Wu--Moreira's Conjecture \cite{gw} 
about the leading Betti numbers of the moduli space
of one-dimensional Gieseker semistable sheaves on $\mathbb P^2$.

\end{abstract}

\maketitle

\tableofcontents

\section{Introduction}

\subsection{Backgrounds}

The Gromov--Witten invariants are the
virtual counting of holomorphic maps from nodal Riemann surfaces to a target space.
For a toric Calabi--Yau threefold $X$,
the all genera (open) Gromov--Witten invariants of $X$ can be studied using the topological vertex formalism.
Roughly speaking,
the topological vertex is the generating series of all genera open Gromov--Witten invariants of $\bC^3$,
and the all genera open Gromov--Witten theory of a general toric Calabi--Yau threefold
can be obtained from the topological vertex by certain gluing rules.
This formalism was first developed by string theorists Aganagic, Klemm, Mari\~{n}o, Vafa \cite{akmv}
in their study of the duality between open string theory and large $N$ Chern--Simons theory.
See \cite{lllz} and \cite{moop} for mathematical theories of the topological vertex formalism,
and see e.g. \cite{adkmv, dz, wyz} for the relation between the topological vertex formalism
and the multi-component KP integrability of Gromov--Witten invariants of toric Calabi--Yau threefolds.

A particularly interesting case is when $X = K_S$ is the total space of the canonical line bundle over
a toric Fano surface $S$.
In this case $X = K_S$ is a noncompact toric Calabi--Yau threefold,
and the Gromov--Witten invariants of $X$ are called the local Gromov--Witten invariants.
The definition and computation of such invariants can be localized to the surface $S$,
see \cite{ckyz, kz}.
And see \cite{wyz2} for the study of the multi-component KP integrability of local Gromov--Witten invariants
using the topological vertex formalism.

For a projective algebraic variety $X$,
the Gromov--Witten invariants of $X$ are known to be rational numbers
since they are mathematically defined by certain intersection numbers on the moduli spaces,
and these moduli spaces are in general Deligne-Mumford stacks.
In \cite{gv} Gopakumar and Vafa conjectured that these rational numbers can be recasted into
integer invariants in the following way:
\begin{equation*}
\sum_{\beta \not= 0}\sum_{g\geq 0}
K_\beta^g  \lambda^{2g-2} t^\beta =
\sum_{\beta \not= 0}\sum_{g\geq 0} n_\beta^g
\sum_{k>0} \frac{1}{k} \big( 2\sin(\frac{k\lambda}{2})\big)^{2g-2}
t^{k\beta},
\end{equation*}
where $K_\beta^g$ are the Gromov--Witten invariants,
and the (conjectural) integer invariants $n_\beta^g (X)$ are called the Gopakumar--Vafa BPS invariants of $X$.
Similar structures also appear in the case of local Gromov--Witten theory.
For a local toric Calabi--Yau threefold $X = K_S$,
one can similarly define the local Gopakumar--Vafa invariants which are conjectured to be integers.
In \cite{pe} Peng proved the integrality of these invariants for a local toric Calabi--Yau threefold,
and in \cite{koni} Konishi gave a proof of the integrality for all toric Calabi--Yau threefolds.
See e.g. \cite{kkv, koni2, zhou5, amv, bp, hiv, ip, gwz,gz25} for more about the Gopakumar--Vafa invariants.

In \cite{gz},
Guo and Zhou have used the topological vertex formalism
to study the leading terms of the Gromov--Witten invariants and transformed Gopakumar--Vafa invariants of local $\bP^2$,
i.e., the total space of the line bundle $\mathcal O(-3)$ over $\bP^2$.
The transformed Gopakumar--Vafa invariants $N_d^g (K_{\bP^2})$ of local $\bP^2$
are related to the Gopakumar--Vafa invariants $n_d^g (K_{\bP^2})$ by:
\be
\sum_{g=0}^{g(d)}
N_d^g(K_{\bP^2})\cdot (q^g +q^{g-2}+\cdots +q^{-g})
= \sum_{g=0}^{g(d)} n_d^g(K_{\bP^2}) \cdot(-1)^g (q^\half -q^{-\half})^{2g},
\ee
where $g(d) = \half (d-1)(d-2)$.
They proved the following explicit formula for the leading transformed Gopakumar--Vafa invariants
(see \cite[Theorem 6.1]{gz}):
\begin{align*}
	(-1)^{d-1} \sum_{\delta\geq0}
	&N^{g(d)-\delta}_d(K_{\bP^2}) q^{\delta}
	=\frac{[1]_{q}^2[2]_{q}}{[\infty]_{q}!^3}
	\bigg(\binom{d+2}{2}-3\sum_{i\geq1}\frac{q^i}{(1-q^i)^2}\bigg)\\
	&-\frac{3\cdot q^{d-1}[2]_{q}[3]_{q}}{[\infty]_{q}!^3}
	\bigg(\binom{d+1}{2}-3\sum_{i\geq1}\frac{q^i}{(1-q^i)^2}
	-3\frac{q^3}{1-q^3}\bigg) \Big|_{q^{\leq 2d-5}},
\end{align*}
where $[a]_{q}= 1-q^{a}$ and
$[\infty]_{q}! =\prod_{a=1}^\infty [a]_{q}$.
Guo--Zhou's formula shares some similar features with the famous G\"ottsche--Yau--Zaslow formula \cite{go, yz},
and also confirms and generalizes some observations in \cite{kkv, zhou5}.

Recently,
Guo--Zhou's formula for the leading Gopakumar--Vafa invariants was used by Guo--Wu \cite{gw}
to study the Betti numbers of moduli spaces of sheaves on $\mathbb{P}^2$,
and motive them, together with Moreira,
to propose a more general conjectural formula for the Betti numbers.
Here we briefly recall their conjecture and some backgrounds.
Let $S$ be a del Pezzo surface, 
and let $M_{\beta, \chi}$ be the moduli space of pure one dimensional Gieseker semistable sheaves over $S$
with support class $\beta$ and Euler characteristic $\chi$.
It is known that the Poincar\'e polynomial of the intersection cohomology of the moduli space $M_{\beta,\chi}$
is independent of $\chi$, see \cite{ms, yuan}.
Then one may simply take $\chi = 1$,
and denote by $M_\beta$ the corresponding moduli space.
Now let $P_\beta(y)$ be the Poincar\'e polynomial of $M_\beta$,
and denote $\omega = -K_S \cdot \beta$.
It is conjectured by Choi--van Garrel--Katz--Takahashi \cite{cgkt} that
\begin{equation*}
P_\beta (y) / \sum_{j=0}^{\omega-1} y^j
\end{equation*}
is a palindromic polynomial in $y$.
In \cite{gw}, Guo and Wu gave a proof of the case $S = \bP^2$ by showing that
Choi--van Garrel--Katz--Takahashi's conjecture can be implied by Bousseau's Conjecture \cite{bou20},
and in the literature it is known that Bousseau's Conjecture holds in the case of $\bP^2$
\cite{ms,yuan, bou}.

Moreover,
Guo--Wu carried out some explicit calculations of the Betti numbers of the moduli space in the case of $\bP^2$
using the relation to Gopakumar--Vafa invariants.
Denote by $H$ the hyperplane class on $\mathbb{P}^2$,
and let $M_{d}$ be
the moduli space of pure one dimensional Gieseker semi-stable sheaves with support class $dH$
and Euler characteristic $\chi=1$.
Following the notations in \cite{gw},
consider:
\begin{align}\label{eqn:def hatOmega main}
	\widehat{\Omega}_d(y)
	=\frac{1}{\sum_{i=0}^{3d-1} y^i} \cdot \sum_{j=0}^{d^2+1} b_{2j}(M_{d}) y^j,
\end{align}
where $g(d)=(d-1)(d-2)/2$ and $b_{2j}(M_{d})$ is the $2j$-th Betti number of $M_{d}$.
They derived a closed formula for the leading terms of $\widehat{\Omega}_d(y)$
of degree no more than $2d-11$, 
which is an extension of a formula of Yuan \cite{yuan, yuan2} up to degree $d$.
They then observed that their formula should hold for degree up to $2d-5$,
see \cite[Theorem 1.6; Remark 1.8]{gw}.
Together with Moreira,
they conjectured that:
\begin{Conjecture}[Conjecture 1.9 in \cite{gw}]
	\label{thm:main leading betti 3d-10}
	For $d\geq4$, one has:
	\begin{align}
		\widehat{\Omega}_d(y)
		=\frac{[1]_y^2}{[\infty]_y!^3}\bigg(1-3y^{d-1}\frac{[3]_y}{[1]_y^2}
		+3y^{2d-4}\frac{[3]_y(-2+2y+4y^2+2y^3+y^4+2y^5)}{[1]_y^2 [2]_y}\bigg)
		\Big|_{y^{\leq 3d-10}}.
	\end{align}
\end{Conjecture}

\subsection{Description of main results}

In this paper,
we derive a structural formula for the generating series of
all genera (closed) Gromov--Witten invariants of local $\bP^2$.
As applications,
we first extend Guo--Zhou's $2d$-level formula for the leading Gopakumar--Vafa invariants
to a more general $3d$-level formula,
and then use this formula to prove the conjecture of Guo--Wu--Moreira.
In what follows we briefly describe the main results of this paper.

Denote by $K^g_d(K_{\mathbb{P}^2})$ the genus $g$ and degree $d$ Gromov--Witten invariant of local $\bP^2$,
and let $I(d)$ be the generating series of possibly disconnected degree $d$ Gromov--Witten invariants:
\begin{align}\label{eqn:def I(d) main}
	1+\sum_{d=1}^\infty I(d) t^d
	=\exp\Big(\sum_{g=0}^\infty \lambda^{2g-2}
	\sum_{d=1}^\infty K^g_d (K_{\bP^2}) \cdot t^d\Big).
\end{align}
Here $q=e^{i\lambda}$, and $I(d) \in \bC[q,q^{-1}]]$ is known to be a formal Laurent series in a formal variable $q$ of degree:
\begin{equation*}
\deg I(d) = \half d(d-3).
\end{equation*}

Following the notations in \cite{gz},
for two Laurent series
$f(q),g(q) \in \bC[q,q^{-1}]]$
we denote:
\begin{equation*}
f(q) = g(q) \big|_{q^{\geq m}}
\end{equation*}
if the coefficient of $q^n$ in $f(q)$ equals the coefficient of $q^n$ in $g(q)$
for every $n\geq m$.
Then our first main result is the following structural formula for the leading terms of $I(d)$:
\begin{Theorem}
[=Theorem \ref{thm-main1-leadI}]
\label{thm:main GW}
For $d\geq 4$,
	the leading terms of $I(d)$ are given by:
	\be
\label{eq-intromain1-GW}
	\begin{split}
		I(d)
		= \frac{ (-1)^d \cdot q^{ \half d(d-3)} }{\prod_{i=1}^\infty (1-q^{-i})^3}
		\Big(  J_0(d) - q^{-d-1} I(1) J_1(d)
		+q^{-2d-1} I(2) J_2(d)
		\Big)
		\Big|_{q^{\geq \half d(d-3) -3d+1}},
	\end{split}
	\ee
	where $J_0,J_1,J_2$ are given by
\be
\label{eq-intro-Jidef}
		J_a(d) = \binom{d+2-a}{2} +3a- 3\sum_{i\geq 1}\frac{q^{-i}}{(1-q^{-i})^2},
		\qquad\qquad a=0,1,2,
\ee
	and $I(1), I(2)$ are exactly the special cases $d=1,2$ of $I(d)$ whose explicit formulas are:
	\begin{equation*}
		I(1) = - \frac{3q^{-1}}{(1-q^{-1})^2},
		\qquad\qquad
		I(2) =\frac{3q^{-1}(2+q^{-1}+q^{-3}+2q^{-4})}{(1-q^{-1})^2(1-q^{-2})^2}.
	\end{equation*}
\end{Theorem}

This formula is an extension of Guo--Zhou \cite[Theorem 6.4]{gz}.
Here the summation
\begin{equation*}
\sum_{i\geq 1}\frac{q^{-i}}{(1-q^{-i})^2} =
\sum_{n=1}^\infty \Big( \sum_{d|n} d \Big) \cdot q^{-n}
\end{equation*}
is a shift of the second Eisenstein series written in the variable $q^{-1}$.
From our representation \eqref{eq-intromain1-GW} one may naturally expect
that a similar formula holds in the region $q^{\geq \frac{1}{2}d(d-3)-k(d-k+3) +1}$ for a general $k$.
Here the truncation $\frac{1}{2}d(d-3)-k(d-k+3)+1$ is suggested by
a concrete degree estimate of the two-legged topological vertices,
see Proposition \ref{prop-WWWdeg-ineq}.
More precisely,
it is natural to conjecture that:
\be
\begin{split}
I(d)
= \frac{ (-1)^d \cdot q^{ \half d(d-3)} }{\prod_{i=1}^\infty (1-q^{-i})^3}
\Big( \sum_{a=0}^{k-1} (-1)^a q^{-a(d-a+3) -\half a(a-3)}\cdot I(a)J_a(d) \Big)
\Big|_{q^{\geq \half (d^2-3d) -k(d-k+3)+1}},
\end{split}
\ee
where $I(0)=1$, and $J_a(d)$ for a general $a\in \bN$ is still given by \eqref{eq-intro-Jidef}.
The above theorem confirms this conjectural formula for $k\leq 3$.
However,
in order to check this formula for a larger $k$ or find a proof of the general case
one may need a more subtle degree estimate and more complicated computations
of explicit formulas for the $2$-legged topological vertices $W_{\mu,\nu}$,
and this is beyond the scope of the present paper.
It is also natural to expect that the leading terms of Gromov--Witten invariants
of a general (local) toric Calabi--Yau threefold can also be packaged in such a simple fashion,
and we hope to address this problem in future works.

Next,
we use the above theorem to derive a formula for the leading terms of the transformed Gopakumar--Vafa invariants
of local $\bP^2$,
and this result is an extension of \cite[Theorem 6.1]{gz}.
We prove the following:
\begin{Theorem}
[=Theorem \ref{thm-mainTGV-leading}]
\label{thm:main GV}
	For $d\geq 16$,
	the leading terms of the transformed Gopakumar--Vafa invariants of local $\bP^2$ are given by:
	\be
	\label{eq-LeaddingGVsum-thm main}
	\begin{split}
		\sum_{g = -3d+1}^0  N_d^{g+g(d)}(K_{\bP^2}) q^{g}
		= \frac{(-1)^{d-1} }{([\infty]_{q^{-1}}!)^3}
		\Big( \tJ_0  + q^{-d+1} \tJ_1  + q^{-2d+4} \tJ_2 + q^{-3d+9} \tJ_3
		\Big)\Big|_{q^{\geq -3d +1}},
	\end{split}
	\ee
	where $g(d)=\frac{1}{2}(d-1)(d-2)$ is the top degree of the transformed Gopakumar--Vafa invariants of local $\bP^2$,
and the functions $\tJ_i$ are given by:
	\begin{equation*}
		\begin{split}
			\tJ_0 =& [1]_{q^{-1}}^2 [2]_{q^{-1}} \Big(
			\binom{d+2}{2} - 3\sum_{i\geq 1}\frac{q^{-i}}{(1-q^{-i})^2} \Big),\\
			\tJ_1 =& -3 \cdot  [2]_{q^{-1}} [3]_{q^{-1}}
			\Big( \binom{d+1}{2} -3\sum_{i\geq 1} \frac{q^{-i}}{(1-q^{-i})^2} \Big)
			+ 9q^{-3} \cdot [2]_{q^{-1}}, \\
			\tJ_2 =& 3 \big(-2 +4q^{-2}+8q^{-3}+7q^{-4}+5q^{-5}+3q^{-6}+2q^{-7} \big) \cdot [1]_{q^{-1}}
			\Big( \binom{d}{2} -3\sum_{i\geq 1} \frac{q^{-i}}{(1-q^{-i})^2} \Big)  \\
			& - \frac{9q^{-4} (3+9 q^{-1}+8q^{-2}+6q^{-3}+6q^{-4}+4q^{-5} )}{1+q^{-1}},
\end{split}
	\end{equation*}
and
\begin{equation*}
		\begin{split}
			\tJ_3 =& -\frac{ 10+7q^{-1}-28q^{-2} -60q^{-3} -26q^{-4} +116q^{-5} +178q^{-6} +62q^{-7} -152q^{-8}}
			{[2]_{q^{-1}} [3]_{q^{-1}}^2} \\
			& \cdot \Big( \binom{d-1}{2} -3\sum_{i\geq 1} \frac{q^{-i}}{(1-q^{-i})^2} \Big)
			-\frac{54q^{-5} (1-q^{-1}-3q^{-2}+q^{-3})}{[1]_{q^{-1}}^2 [2]_{q^{-1}}}.
		\end{split}
	\end{equation*}
Here we use the notations:
\begin{equation*}
[a]_{q^{-1}}= 1-q^{-a},\qquad\qquad
[\infty]_{q^{-1}}! =\prod_{a=1}^\infty [a]_{q^{-1}}.
\end{equation*}
\end{Theorem}

As pointed out by Guo--Wu in \cite{gw},
the information of leading Gopakumar--Vafa invariants can be used to 
compute the leading Betti numbers of the moduli spaces of sheaves on $\mathbb{P}^2$.
As an application of Theorem \ref{thm:main GV},
we prove:
\begin{Theorem}
	[=Theorem \ref{eq-Confirmgwconj}]
\label{thm-intro-GWMConj}
Guo--Wu--Moreira's \cite[Conjecture 1.9]{gw}
(see Conjecture \ref{thm:main leading betti 3d-10}) holds.
\end{Theorem}

\subsection{Plan of the paper}

The main body of this paper consists of three parts.
We arrange the rest of the paper in the following way.

In the first part of this paper,
we recall some preliminary facts and the topological vertex formalism 
in \S \ref{sec-pre} and \S \ref{sec-TVformalism} respectively,
and then study the leading terms of the generating series $I(d)$ of Gromov--Witten invariants
of local $\bP^2$ (Theorem \ref{thm:main GW}) in \S \ref{sec-Leading-GW}.

In the second part of this paper,
we introduce two applications of the formula for leading terms of $I(d)$
-- we prove Theorem \ref{thm:main GV} and Theorem \ref{thm-intro-GWMConj} in
\S \ref{sec-LeadingTGV} and \S \ref{sec-betti} respectively.

And finally in \S \ref{app-eval-schur}, \S \ref{app-WdCompute}, \S \ref{app-secC-delta12}
we give the detailed proofs of two theorems.
These two proofs are lengthy and consists of a lot of subtle computations,
and thus we put them at the end of this paper to make the paper easier to read.

\section{Preliminaries}
\label{sec-pre}

In this section we recall the definition and some basic properties
of the Schur functions,
Gromov--Witten invariants and Gopakumar--Vafa invariants,
and the moduli spaces of one-dimensional Gieseker semistable sheaves on $\mathbb{P}^2$.

\subsection{Partitions and  Schur functions}
\label{subsec-Schur}

In this subsection we recall the
definition of Schur functions and their evaluations.
See the books \cite{sta, mac} for references.

Let $\lambda=(\lambda_1,\lambda_2,\cdots,\lambda_{l})$ be a partition of an integer
where $\lambda_1\geq \cdots\geq \lambda_{l}> 0$.
Denote:
\begin{equation*}
|\lambda| = \lambda_1+ \lambda_2 + \cdots +\lambda_l.
\end{equation*}
The number $l = l(\lambda)$ is called the length of this partition.
Here we allow the empty partition $(\emptyset)$ of $0$,
and we will denote by $\mathcal{P}$ the set of all partitions.
The transpose $\lambda^t$ is defined to be the partition whose Young diagram
is obtained by flipping the Young diagram associated with $\lambda$ along the diagonal.
Or equivalently,
$\lambda^t=(\lambda_1^t,\cdots,\lambda_m^t)$ is the partition defined by
$m=\lambda_1$
and $\lambda_j^t = |\{i| \lambda_i\geq j\}|$.
For a partition $\lambda = (\lambda_1,\lambda_2,\cdots,\lambda_l)$,
we will use the following notations:
\be
\label{eq-def-kappamu}
n_\lambda = \sum_{i=1}^l (i-1)\lambda_i, \qquad\qquad
\kappa_\lambda = \sum_{i=1}^{l} \lambda_i(\lambda_i -2i +1).
\ee
We use the notations $n_{(\emptyset)} = 0$ and $\kappa_{(\emptyset)} =0$.
The two numbers $\kappa_\lambda$ and $n_\lambda$ are related by:
\begin{equation*}
\kappa_\lambda = 2(n_{\lambda^t} -n_\lambda).
\end{equation*}

The Frobenius notation of a partition $\lambda$ is defined to be:
\begin{equation*}
\lambda= (m_1,m_2,\cdots,m_k | n_1,n_2,\cdots,n_k),
\end{equation*}
where $k$ is the number of boxes on the diagonal of the corresponding Young diagram,
and
\begin{equation*}
m_i = \lambda_i - i,
\qquad
\qquad
n_i=\lambda_i^t - i,
\end{equation*}
for $i=1,2,\cdots,k$.
For a partition $\lambda = (m_1,\cdots, m_k|n_1,\cdots, n_k)$,
the number $\kappa_\lambda$ can also be represented as:
\be
\label{eq-def-kappa2}
\kappa_\lambda
=\sum_{i=1}^k (m_i+\half)^2 - \sum_{i=1}^k (n_i+\half)^2.
\ee

For a partition $\lambda=(m|n) = (m+1,1^n)$ whose Young diagram is a hook,
the Schur function $ s_{(m|n)}$ is defined by:
\begin{equation*}
s_{(m|n)}= h_{m+1}e_n - h_{m+2}e_{n-1} + \cdots
+ (-1)^n h_{m+n+1},
\end{equation*}
where $h_n$ and $e_n$ are the
complete symmetric function and elementary symmetric function of degree $n$ respectively.
And for a general partition $\lambda=(m_1,\cdots,m_k|n_1,\cdots,n_k)$,
the Schur function $s_\lambda$ is defined by:
\begin{equation*}
s_\lambda  = \det (s_{(m_i|n_j)} )_{1\leq i,j\leq k}.
\end{equation*}
The Schur function indexed by the empty partition is defined to be $s_{(\emptyset)} = 1$.
Denote by $\Lambda$ the infinite-dimensional linear space of all symmetric functions,
then Schur functions $\{s_\lambda\}_{\lambda\in\mathcal{P}}$ form a basis of  $\Lambda$.
One can introduce a natural inner product $(\cdot,\cdot)$ on $\Lambda$ by
taking $\{s_\lambda\}_{\lambda\in\mathcal{P}}$ to be an orthonormal basis,
i.e., taking $( s_\lambda, s_\mu) = \delta_{\lambda,\mu}$.

Now denote by $p_n\in \Lambda$ the Newton symmetric function of degree $n$,
and $\{p_\lambda\}$ gives another basis for $\Lambda$,
where $p_\lambda = p_{\lambda_1}\cdots p_{\lambda_l}$ for a partition $\lambda = (\lambda_1,\cdots,\lambda_l)$.
The above two bases for $\Lambda$ are related by (see e.g. \cite{mac}):
\be
\label{eq-newton-schur}
s_\mu = \sum_{\lambda} \frac{\chi^\mu (C_\lambda)}{z_\lambda} p_\lambda,
\ee
where $\chi^\mu$ is the irreducible character of the symmetric group $S_{|\mu|}$ indexed by $\mu$
and $C_\lambda \subset S_{|\mu|}$ is the conjugacy class whose cycle type is $\lambda$.
Sometimes we regard $s_\mu$ as a function $s_\mu (\mathbf{t})$ in the variables $\mathbf{t} = (t_1,t_2,\cdots)$
via \eqref{eq-newton-schur} where $t_n = \frac{p_n}{n}$,
and call it the Schur polynomial indexed by $\mu$.
If one defines a degree by setting the degree of $t_n$ to be $n$,
then $s_\mu (\mathbf{t})$ is a weighted homogeneous polynomial of degree $|\mu|$ in
the variables $t_1,t_2,\cdots,t_{|\mu|}$.

Now we recall the definition of the skew Schur functions.
Let $\lambda,\mu$ be two partitions,
and then the skew Schur function $s_{\lambda/\mu} \in \Lambda$ is determined by:
\begin{equation*}
(s_{\lambda/\mu},s_\nu ) = (s_\lambda, s_\mu s_\nu),
\qquad \forall \text{ partition $\nu$},
\end{equation*}
where $(\cdot,\cdot)$ is the inner product defined above.
Or equivalently,
\begin{equation*}
s_{\lambda/\mu} = \sum_\nu c_{\mu\nu}^\lambda s_\nu
\end{equation*}
where $\{c_{\mu\nu}^\lambda\}$ are the Littlewood-Richardson coefficients
defined by
$s_\mu s_\nu = \sum_{\lambda} c_{\mu\nu}^\lambda s_\lambda.$

In the rest of this subsection,
we recall some results of the evaluations of symmetric functions.
The evaluation of Schur functions can be represented in terms of semistandard Young tableau,
see e.g. \cite[\S 7.10]{sta}.
Let $\mu$ be a partition of integer,
then a semistandard Young tableaux (SSYT) of shape $\mu$ is an assignment of a positive integer
to each box of the Young diagram associated with $\mu$,
such that the numbers in every row are non-decreasing from left to right,
and the numbers in every column are strictly increasing from top to bottom.
Assume that the positive integer $i$ appears $k_i$ times in a SSYT $T$,
then we denote $x^T = \prod_{i\geq 1} x_i^{k_i}$.
One has:
\begin{align*}
	s_\mu (x_1,x_2,x_3,\cdots) = \sum_T x^T,
\end{align*}
where the summation is over all SSYT of shape $\mu$.
For example,
a SSYT of shape $\mu = (2,1)$ must be of the following form:
\begin{equation*}
\begin{tikzpicture}[scale=0.45]
\draw (0,0) -- (2,0);
\draw (0,-1) -- (2,-1);
\draw (0,-2) -- (1,-2);
\draw (1,0) -- (1,-2);
\draw (0,0) -- (0,-2);
\draw (2,0) -- (2,-1);
\node [align=center,align=center] at (0.5,-0.5) {$i$};
\node [align=center,align=center] at (1.5,-0.5) {$j$};
\node [align=center,align=center] at (0.5,-1.5) {$k$};
\end{tikzpicture}
\end{equation*}
where $i,j,k$ are positive integers satisfying $j\geq i$ and $k >i$.
Thus one may obtain:
\begin{equation*}
s_{(2,1)}(x_1,x_2,x_3,\cdots) = \sum_{j>i\geq 1} (x_i^2 x_j + x_ix_j^2)
+ 2 \sum_{k>j>i\geq 1} x_i x_j x_k.
\end{equation*}

For our purpose,
we are particularly interested in the evaluation of Schur functions at
the following sequence:
\be
\label{eq-defqrho}
q^\rho = (q^{-\half}, q^{-\frac{3}{2}}, q^{-\frac{5}{2}}, \cdots).
\ee
We also denote by $\rho$ the sequence:
\be
\label{eq-defrho}
\rho = (-\half, -\frac{3}{2}, -\frac{5}{2}, \cdots).
\ee
In the literature it is known that
(see e.g. \cite[p.45, Exmaple 2]{mac}):
\be
\label{eq-eva-smurho}
s_\mu (q^\rho) = q^{ -\half |\mu| - n_\mu} \cdot
\prod_{x\in \mu} \frac{1}{1 - q^{- h(x)}},
\ee
where $x\in \mu$ means $x$ runs over boxes in the Young diagram associated with $\mu$,
and $n_\mu$ is defined by \eqref{eq-def-kappamu}.
For a box $x$ in the Young diagram,
we denote by $h(x)$ the hook length of $x$:
\begin{equation*}
h(x) = (\mu_i - i) + (\mu_j^t -j) +1,
\end{equation*}
if $x$ is in the $i$-th row and $j$-th column of the Young diagram.
By taking $\mu$ to be $(m)$ and $(1^m)$ one obtains the following identities respectively:
\be
\label{eq-hneneva}
h_m(q^\rho) = q^{-\half m} \cdot \prod_{i=1}^m \frac{1}{1-q^{-i}},
\qquad\qquad
e_m(q^\rho) = q^{-\half m^2} \cdot \prod_{i=1}^m \frac{1}{1-q^{-i}}.
\ee

\subsection{Gromov--Witten invariants and Gopakumar--Vafa invariants}

In this subsection we briefly recall the Gromov--Witten invariants and Gopakumar--Vafa invariants.

Let $X$ be a projective algebraic variety and $\beta \in H_2(X,\bZ)$,
and let $\Mbar_{g,n}(X,\beta)$ be the moduli space of stable maps from $n$-pointed curves of genus $g$ to $X$
with image in $\beta$:
\begin{equation*}
\Mbar_{g,n}(X,\beta) =
\{\text{stable map } f:(C;p_1,\cdots,p_n)\to X | f_*[C] = \beta \}/\sim.
\end{equation*}
The virtual dimension of this moduli space is
\begin{equation*}
\vdim \Mbar_{g,n}(X,\beta) = (1-g)(\dim X-3) +
\int_\beta c_1(TX) +n.
\end{equation*}
If $X$ is a Calabi--Yau threefold,
one has $\vdim \Mbar_{g,n}(X,\beta) = n$.
The $0$-pointed Gromov--Witten invariants of $X$ are defined by:
\begin{equation*}
K_\beta^g (X) = \int_{[\Mbar_{g,0}(X,\beta)]^{\text{vir}}} 1,
\end{equation*}
where $[\Mbar_{g,0}(X,\beta)]^{\text{vir}}$ is the virtual fundamental class on the moduli space.

Gromov--Witten invariants are known to be rational numbers.
It is conjectured that these rational numbers can be recasted into
integer invariants in the following way \cite{gv}:
\begin{equation*}
\sum_{\beta \not= 0}\sum_{g\geq 0}
K_\beta^g (X) \lambda^{2g-2} t^\beta =
\sum_{\beta \not= 0}\sum_{g\geq 0} n_\beta^g (X)
\sum_{k>0} \frac{1}{k} \big( 2\sin(\frac{k\lambda}{2})\big)^{2g-2}
t^{k\beta},
\end{equation*}
where $\lambda$ and $t$ are two formal variables,
and the (conjectural) integers $n_\beta^g(X)$ are called Gopakumar--Vafa BPS invariants.

A powerful tool to compute the Gromov--Witten invariants is the virtual localization technique \cite{ko, gp}.
Let $S$ be a toric Fano surface,
then the total space $X = K_S$ of the canonical line bundle over $S$ is a noncompact toric Calabi--Yau threefold.
For such a noncompact Calabi--Yau threefold $K_S$,
the definition and computation of their Gromov--Witten invariants $K_\beta^g (K_S)$
(now known as local Gromov--Witten invariants)
can be localized to the surface $S$,
see \cite{ckyz, kz}.
Similar to the compact case,
one can convert these invariants into some integer invariants:
\begin{equation*}
\sum_{\beta \not= 0}\sum_{g\geq 0}
K_\beta^g (K_S) \lambda^{2g-2} t^\beta =
\sum_{\beta \not= 0}\sum_{g\geq 0} n_\beta^g (K_S)
\sum_{k>0} \frac{1}{k} \big( 2\sin(\frac{k\lambda}{2})\big)^{2g-2}
t^{k\beta},
\end{equation*}
and the integer invariants $n_\beta^g (K_S)$ are called the local Gopakumar--Vafa invariants.
See \cite{pe} for a proof of the integrality for a local toric Calabi--Yau threefold,
and see \cite{koni} for a proof of the integrality for all toric Calabi--Yau threefolds.
Denote by $F^{K_S}$ the left-hand side of the above formula,
and take $q = e^{\sqrt{-1} \lambda}$,
then one has:
\be
\label{eq-def-GVlocal}
F^{K_S} =
\sum_{\beta \not= 0}\sum_{g\geq 0} n_\beta^g (K_S)
\sum_{k>0} \frac{(-1)^{g-1}}{k} \big( q^{\frac{k}{2}} - q^{-\frac{k}{2}}\big)^{2g-2}
t^{k\beta}.
\ee
This $F^{K_S}$ is called the free energy of the local Calabi--Yau threefold $K_S$.

\subsection{The moduli spaces of one-dimensional Gieseker semistable sheaves on $\mathbb{P}^2$}
In this subsection,
we review the moduli spaces of one-dimensional Gieseker semistable sheaves on $\mathbb{P}^2$.

Let $H$ be the line class on $\mathbb{P}^2$ and $d$ be a positive integer.
Denote by $M_{d,\chi}$ the moduli space
which parameterizes the Gieseker semistable pure one-dimensional sheaves $\mathcal{F}$ on $\mathbb{P}^2$ with $c_1(\mathcal{F})=dH$ and Euler characteristic $\chi$.
The stability condition is taken with respect to $\mathcal{O}_{\mathbb{P}^2}(1)$,
and the Hilbert polynomial is $dm+\chi$.
We recommend \cite{bou22,bou,gw,lp,mark,ms,yuan} for more details and properties of the moduli space.
If $\text{gcd}(d,\chi)=1$,
then the moduli space $M_{d,\chi}$ is a smooth irreducible projective variety \cite{lp}.
Moreover,
it is conjectured by Bousseau \cite{bou22} and proved by Maulik--Shen \cite{ms} that the intersection Betti numbers of $M_{d,\chi}$ are independent of $\chi$ (see also \cite{yuan}).
In this paper,
we are only interested in the Betti numbers of the moduli space,
therefore we set $M_d=M_{d,1}$ and work with its ordinary cohomology.

We consider the Fitting support morphism
\begin{align*}
	\pi: M_d \rightarrow |dH| \cong \mathbb{P}^{d(d+3)/2},
\end{align*}
which sends each sheaf $\mathcal{F}$ to its support $\text{Supp}(\mathcal{F})$.
Note that for a smooth curve $C\in|d H|$ with inclusion $i:C\hookrightarrow\mathbb{P}^2$,
every sheaf in the fiber over $C$ is the pushforward $i_*L$ of a line bundle $L$ with $\deg L=g(d)=(d-1)(d-2)/2$ on $C$,
thus the fiber $\pi^{-1}(C)$ is isomorphic to $\text{Pic}^{g(d)}(C)$.
Then the moduli space $M_d$ is a compactification of the relative Picard variety over smooth-curve locus (see Section 2 of \cite{ms} for more details).
In particular,
the moduli space $M_d$ is of dimension $d^2+1$.
Moreover, its odd Betti numbers vanish \cite{mark}.
Then the Poincar\'e polynomial of $M_{d}$ is given by
\begin{align*}
	P_d(y)
	=\sum_{j=0}^{d^2+1} b_{2j}(M_{d}) y^j,
	\qquad b_{2j}(M_{d})=\dim H^{2j}(M_d,\mathbb{Q}).
\end{align*}
The Poincar\'e duality gives $P_d(y)=y^{d^2+1}P_d(y^{-1})$.
Following the notation of Guo and Wu in \cite{gw},
we use
\begin{align*}
	\Omega_d(y^{1/2})
	=y^{-(d^2+1)/2}\cdot P_d(y),
	\qquad \widehat{\Omega}_d(y)
	=\frac{1}{\sum_{i=0}^{3d-1} y^i} \cdot P_d(y).
\end{align*}
Guo--Wu prove that $\widehat{\Omega}_d(y)$ is a palindromic polynomial of $y$ and its degree is $2g(d)=(d-1)(d-2)$ (see Theorem 1.3 in \cite{gw}).
This result proves a special case of a conjecture proposed by Choi--van Garrel--Katz--Takahashi \cite{cgkt}.
Note that under this notation, the variable $y$ records half the cohomological degree, so a formula through $y^m$ determines
Betti numbers of the moduli space $M_d$ through degree $2m$.

\section{The Topological Vertex Formalism}
\label{sec-TVformalism}

In this section we recall the topological vertex formalism for the
Gromov--Witten theory of toric Calabi--Yau threefolds,
especially the Gromov--Witten invariants of local $\bP^2$.
We also give some explicit formulas for certain $2$-legged topological vertices.

\subsection{The topological vertex}

First we recall the definition and combinatorial expressions
of the topological vertex.

Roughly speaking,
the open Gromov--Witten invariants of a toric Calabi--Yau threefold $X$
count maps from bordered Riemann surfaces to $X$
with boundary mapped into some specific Lagrangian submanifolds (D-branes).
The topological vertex is the generating series of open Gromov--Witten invariants of $\bC^3$
with three special D-branes.
The notion of the topological vertex was introduced by Aganagic {\em et al} \cite{akmv}
using a string theoretical approach.
A mathematical theory for the topological vertex was developed in the setting of
the open Gromov--Witten theory of $\mathbb{C}^3$ by Li {\em et al} in \cite{lllz},
and in that work the localization technique was used in
the study of the local relative Gromov--Witten theory of toric Calabi--Yau threefolds.
The combinatorial expression for the topological vertex constructed in \cite{lllz}
looks slightly different from the original version proposed in \cite{akmv,adkmv},
and the equivalence of these two expressions was first proved in \cite{moop}  using
GW/DT correspondence,
and then by a pure combinatorial approach in \cite{nt}.

Now we recall the combinatorial expressions of the topological vertex.
Let $\rho$ and $q^\rho$ be the sequences \eqref{eq-defrho} and \eqref{eq-defqrho} respectively.
For a partition $\mu = (\mu_1,\mu_2,\cdots,\mu_l)$ we denote:
\begin{align*}
	q^{\mu+\rho} =
	( q^{\mu_1 - \half}, q^{\mu_2-\frac{3}{2}}, \cdots, q^{\mu_l - l+\half}.
	q^{ -l-\half}, q^{-l-\frac{3}{2}},\cdots ).
\end{align*}
Let $\mu^1,\mu^2,\mu^3$ be three partitions,
and the topological vertex $W_{\mu^1,\mu^2,\mu^3} (q)$ labeled by $\mu^1$, $\mu^2$, $\mu^3$
is defined by:
\be
\label{eq-TVdef-1}
W_{\mu^1,\mu^2,\mu^3}(q) = q^{\kappa_{\mu^2}/2 + \kappa_{\mu^3}/2}
\sum_{\sigma^1,\sigma^3} \Big( \sum_\eta
c_{\eta \sigma^1}^{\mu^1} c_{\eta (\sigma^3)^t}^{(\mu^3)^t} \Big)
\frac{W_{(\mu^2)^t \sigma^1}(q)  W_{\mu^2 (\sigma^3)^t}(q)}{W_{\mu^2}(q)},
\ee
where $\{c^\lambda_{\mu\nu}\}$ are the Littlewood-Richardson coefficients,
and the one-legged and two-legged topological vertices are defined by
\be
\label{eq-W12inSchur}
W_{\mu}(q) = s_\mu (q^\rho),
\qquad\qquad
W_{\mu,\nu} (q) = s_\mu (q^\rho) \cdot s_\nu (q^{\mu + \rho})
\ee
respectively.
The one-legged and two-legged topological vertices can be related to $W_{\mu^1,\mu^2,\mu^3}(q)$ more straightforwardly by:
\be
\label{eq-2l3lrelation}
W_\mu(q) = W_{\mu,(\emptyset),(\emptyset)}(q),
\qquad\qquad
W_{\mu,\nu}(q) =
q^{ \half \kappa_\mu}  \cdot W_{\mu^t , (\emptyset), \nu}(q).
\ee
And the two-legged topological vertex are known to satisfy the symmetry condition:
\be
\label{eq-symm-Wmunu}
W_{\mu,\nu}(q) = W_{\nu,\mu}(q),
\qquad
\forall \mu,\nu\in \cP.
\ee

The topological vertex $W_{\mu^1,\mu^2,\mu^3}(q)$ can also be represented in terms of
skew-Schur functions (see e.g. \cite[Prop. 4.4]{zhou4}):
\be
\label{eqn:def W as skew}
W_{\mu^1,\mu^2,\mu^3}(q)
= (-1)^{|\mu^2|} q^{\kappa_{\mu^3}/2} s_{(\mu^2)^t} (q^{-\rho})
\sum_\eta s_{\mu^1 /\eta} (q^{(\mu^2)^t + \rho}) s_{(\mu^3)^t /\eta} (q^{\mu^2 + \rho}),
\ee
where $q^{-\rho}$ is the sequence $(q^{\half},q^{\frac{3}{2}},q^{\frac{5}{2}},\cdots)$.

\subsection{Gluing of topological vertices and Gromov--Witten invariants of local $\bP^2$}
\label{sec-glueKP2}

In this subsection,
we review the gluing rule for the topological vertex,
see \cite{adkmv}.

The simplest Calabi--Yau threefold is $\bC^3$,
and the toric diagram of $\bC^3$ is simply a single trivalent vertex.
It serves as the basic building block of general toric Calabi--Yau threefolds.
For an arbitrary toric Calabi--Yau threefold $X$,
the toric diagram of $X$ is a trivalent planar graph which can be constructed by gluing such blocks,
and geometrically $X$ can be constructed accordingly by gluing some copies of $\bC^3$  together.
The open Gromov--Witten invariants of $X$ can be obtained by certain gluing procedure
of the topological vertices \cite{akmv}.
In such a procedure,
one needs to multiply the topological vertex by some extra factors along the edges
to get the framed topological vertex.
This is a sort of Feynman rules as in quantum field theory.

In this paper we focus on the case of the total space of the canonical line bundle over $\bP^2$,
i.e., the local $\bP^2$.
Geometrically local $\bP^2$ is obtained by gluing three copies of $\bC^3$ together,
and the toric diagram in this case is:
\begin{equation*}
\begin{tikzpicture}[scale=0.75]
\draw [thick] (0,0) -- (0,2);
\draw [thick] (0,0) -- (2,0);
\draw [thick] (2,0) -- (0,2);
\draw [thick] (0,0) -- (-0.9,-0.9);
\draw [thick] (0,2) -- (-0.6,3.2);
\draw [thick] (2,0) -- (3.2,-0.6);
\node [align=center,align=center] at (-1.2,-1.05) {$\lambda^1$};
\node [align=center,align=center] at (0.25,-0.35) {$\mu^1$};
\node [align=center,align=center] at (1.6,-0.35) {$\mu^{1,t}$};
\node [align=center,align=center] at (3.5,-0.9) {$\lambda^2$};
\node [align=center,align=center] at (-1,3.2) {$\lambda^3$};
\node [align=center,align=center] at (1.85,0.8) {$\mu^2$};
\node [align=center,align=center] at (0.75,2) {$\mu^{2,t}$};
\node [align=center,align=center] at (-0.35,1.5) {$\mu^3$};
\node [align=center,align=center] at (-0.45,0.4) {$\mu^{3,t}$};
\end{tikzpicture}
\end{equation*}
where the three partitions $\lambda^1,\lambda^2,\lambda^3$
encode how the boundaries of the Riemann surface are mapped to the three specific Lagrangian submanifolds
of local $\bP^2$.
The partition function of the open Gromov--Witten invariants of local $\bP^2$ is:
\begin{equation*}
Z^{K_{\bP^2}}_{\text{open}} (\bm t^1,\bm t^2,\bm t^3) =
\sum_{\lambda^1,\lambda^2,\lambda^3\in \cP}
Z_{\lambda^1,\lambda^2,\lambda^3}^{K_{\bP^2}}(q)
 \cdot s_{\lambda^1}(\bm t^1)s_{\lambda^2}(\bm t^2)s_{\lambda^3}(\bm t^3),
\end{equation*}
where
\begin{equation*}
\begin{split}
&Z_{\lambda^1,\lambda^2,\lambda^3}^{K_{\bP^2}}(q) \\= &
 \sum_{\mu^1,\mu^2, \mu^3 \in \cP}
(-1)^{\sum_{i=1}^3  |\mu^i|}
q^{ \sum_{i=1}^3  \kappa_{\mu^i}}
\prod_{i=1}^3 Q_i^{|\mu^i|}
W_{\mu^{3,t},\lambda^1,\mu^1}(q) W_{\mu^{1,t},\lambda^2,\mu^2}(q)
W_{\mu^{2,t},\lambda^3,\mu^3}(q),
\end{split}
\end{equation*}
and $Q_i$ are some K\"ahler parameters.
Here $s_{\lambda^i}(\bm t^i)$ is the Schur polynomial in a sequence of
time variables $\bm t^i = (t^i_1,t^i_2,t^i_3,\cdots)$.
See \cite{akmv} for details,
and see also \cite{ek, gz, iq, pe, wyz2} for the study of
open Gromov--Witten invariants of local toric Calabi--Yau geometries.

From the above partition function $Z_{\lambda^1,\lambda^2,\lambda^3}^{K_{\bP^2}}(q)$
one can obtain the closed Gromov--Witten invariants of local $\bP^2$
by simply taking the partitions $\lambda^1$, $\lambda^2$, $\lambda^3$ to be the empty partition.
By simplicity we will take $Q_1 = Q_2 = Q_3=t$ where $t$ is a formal variable.
Then applying \eqref{eq-2l3lrelation} one can represent the results
in terms of two-legged topological vertices
(see e.g. \cite{amv, iq, zhou3, gz}):
\be
\label{eq-partitionfn-kp2}
Z^{K_{\bP^2}}_{\text{closed}} =
\sum_{\mu^1,\mu^2, \mu^3 \in \cP}
(-t)^{\sum_{i=1}^3  |\mu^i|}
q^{ \half \sum_{i=1}^3  \kappa_{\mu^i}}
W_{\mu^{3},\mu^1}(q) W_{\mu^{1}\mu^2}(q)
W_{\mu^{2},\mu^3}(q).
\ee

\subsection{Explicit formula for $2$-legged topological vertex with $\Delta \leq 2$}

In this subsection we recall Guo--Zhou's explicit formula for
the $2$-legged topological vertex $W_{\mu,\nu}(q)$ with $\Delta \leq 1$,
and also describe the results for $\Delta =2$.
Here for two partitions $\mu,\nu$ the notation $\Delta$ means:
\begin{equation*}
\Delta = |\mu|+|\nu|-\mu_1-\nu_1
= \sum_{i=2}^{l(\mu)} \mu_i + \sum_{j=2}^{l(\nu)} \nu_j.
\end{equation*}

For a positive integer $k$,
denote:
\be
\label{eq-notation-[k]q}
[k]_{q^{-1}} = 1-q^{-k},
\qquad
[k]_{q^{-1}}! = \prod_{j=1}^k [j]_{q^{-1}},
\ee
and we also use the notations $[0]_{q^{-1}} = [0]_{q^{-1}}! =1$.
First consider the case $\Delta = 0$,
i.e., the case that $l(\mu) = l(\nu) =1$.
By applying the evaluation formulas \eqref{eq-hneneva},
Guo--Zhou derived the following explicit formula for these two-legged topological vertex
(see \cite[(183), (184)]{gz}):
\be
\label{eq-gz-183184}
W_{(m),(n)} =
q^{mn-\half (m+n)} \cdot \tW_{(m),(n)}(q)
\ee
for $m,n\in \bN$,
where
\be
\label{eq-notationtW}
\tW_{(m),(n)}(q) = \sum_{k=0}^m \frac{q^{-k(n+1)}}{[n]_{q^{-1}}! [k]_{q^{-1}}!}
= \sum_{k=0}^n \frac{q^{-k(m+1)}}{[m]_{q^{-1}}! [k]_{q^{-1}}!}
\ee
is a Laurent series in $q$ of degree $0$.
By the symmetry property \eqref{eq-symm-Wmunu} is clear that:
\begin{equation*}
\tW_{(m),(n)} = \tW_{(n),(m)}.
\end{equation*}

Similarly for $\Delta=1$,
Guo--Zhou derived the following (see \cite[Lemma 6.7]{gz}):
\be
\label{eq-gzlem6.7}
W_{(m),(n-1,1)} = q^{mn-\frac{3}{2}m -\half n-1} \cdot
\Big( \frac{\tW_{(m),(n-1)}}{[1]_{q^{-1}}} -
\frac{q^{-(m+1)(n-1)}}{[m]_{q^{-1}}! [n]_{q^{-1}}!}
\Big).
\ee
For our purpose,
we will also need the explicit formulas in the case $\Delta=2$.
The results are:
\begin{Theorem}
\label{thm-Delta2-expl}
The two-legged topological vertices $W_{\mu,\nu}(q)$ with $\Delta=2$ are given by
the following explicit formulas:
\begin{itemize}
\item[1)]
For $m\geq 4$ and $n\geq 0$ we have:
\begin{equation*}
W_{(m-2,2),(n)} (q) =
\frac{q^{mn-\half m-\frac{5}{2} n -2}}{[2]_{q^{-1}}!} \cdot \tW_{(m-1),(n)}(q)
- \frac{q^{-\frac{3}{2}m +\half n +1}}{[m-1]_{q^{-1}}!} \cdot \tW_{(2),(n)}(q).
\end{equation*}

\item[2)]
For $m\geq 3$ and $n\geq 0$ we have:
\begin{equation*}
\begin{split}
W_{(m-2,1,1),(n)}(q) =&
\frac{[m-1]_{q^{-1}} [m-2]_{q^{-1}}}{[2]_{q^{-1}}!}
\cdot \frac{q^{-\half m -\half n -3}}{[m]_{q^{-1}}! \cdot [n]_{q^{-1}}!} \\
&+ \frac{q^{mn-\half m - \frac{5}{2}n -3}}{[2]_{q^{-1}}!} \cdot \tW_{(m-3),(n-1)}(q).
\end{split}
\end{equation*}
Here for $n=0$ we use the notation $\tW_{(m-3),(-1)}=0$.

\item[3)]
For $m\geq 2$ and $n\geq 2$ we have:
\begin{equation*}
\begin{split}
W_{(m-1,1),(n-1,1)}(q) =&
q^{-\half m-\half n-2}\cdot
\frac{[m-1]_{q^{-1}} \cdot [n-1]_{q^{-1}}}{ [m]_{q^{-1}}! \cdot [n]_{q^{-1}}! \cdot [1]_{q^{-1}}^2 } \\
& + q^{mn-\frac{3}{2}m-\frac{3}{2}n}\cdot \frac{1-q^{-1}+q^{-2}}{[1]_{q^{-1}}^2}
\cdot \tW_{(m-2),(n-2)}(q).
\end{split}
\end{equation*}
\end{itemize}
\end{Theorem}

The proof of this theorem consists of a lot of concrete combinatorial computations,
and we put the detailed proof in \S \ref{app-eval-schur}
to make the main body of this paper easier to read.

\section{Leading Terms of Gromov--Witten Invariants of Local $\bP^2$}
\label{sec-Leading-GW}

In this section we give a structural formula for the leading terms of
the generating series of possibly disconnected Gromov--Witten invariants of local $\bP^2$.
More precisely, let
\begin{equation*}
\begin{split}
Z^{K_{\bP^2}}_{\text{closed}} = 1+ \sum_{d=1}^\infty I(d)t^d
\end{split}
\end{equation*}
where $Z^{K_{\bP^2}}_{\text{closed}}$ is the partition function of local $\bP^2$,
and we will give an explicit formula for the leading terms:
\begin{equation*}
I(d) \big|_{q^{ \geq \half d(d-3) -3d+1}}.
\end{equation*}
Here $I(d)$ is a Laurent series in $q$ of degree $\half d(d-3)$.

Following the notations in \cite{gz},
for a Laurent polynomial
\begin{equation*}
f(q) = \sum_{i=-\infty}^N a_i q^i \in \bC[q,q^{-1}]]
\end{equation*}
we denote:
\begin{equation*}
f(q)|_{q^{\geq n}} = \sum_{i=n}^N a_i q^i.
\end{equation*}
And for two Laurent polynomials $f(q),g(q)$ satisfying $f(q)|_{q^{\geq n}} = g(q)|_{q^{\geq n}}$,
we denote:
\begin{equation*}
f(q) = g(q) \big|_{q^{\geq n}}.
\end{equation*}

\subsection{Degree estimates of the $2$-legged topological vertex}

In \S \ref{sec-glueKP2} we already see that
the $2$-legged topological vertex $W_{\mu,\nu}(q)$ is the building block of the
(closed) Gromov-Witten invariants of local $\bP^2$.
Now in this subsection we compute the degree of $W_{\mu,\nu}(q)$.

We will follow the notations in \cite{gz}.
Let $q$ be a formal variable,
then for a nonzero series $f(q) \in \bC [q,q^{-1}]]$
we denote by $\deg f(q)$ the degree of $f(q)$ as a Laurent series in $q$.
Then it is easy to see that
\begin{equation*}
\deg \frac{f(q)}{g(q)} = \deg f(q) - \deg g(q)
\end{equation*}
for nonzero $f,g\in \bC [q,q^{-1}]]$.
We also use the following notation:
\begin{equation*}
\deg \frac{q^\half  f(q)}{g(q)} = \half + \deg f(q) - \deg g(q).
\end{equation*}

Now let $\mu,\nu\ \in \cP$ be two partitions,
and denote:
\be
\mu\cdot \nu = \sum_{i\geq 1} \mu_i\nu_i
\ee
where we use the convention $\mu_j = 0$ for $j> l(\mu)$.
Then we have:
\begin{Proposition}
The degree of the two-legged topological vertex $W_{\mu,\nu}(q)$ is:
\be
\label{eq-prop-degW}
\deg W_{\mu,\nu}(q) = \mu\cdot \nu
- n_\mu - n_\nu -\half(|\mu|+|\nu|).
\ee
\end{Proposition}
\begin{proof}

Recall that $W_{\mu,\nu} (q) = s_\mu (q^\rho) \cdot s_\nu (q^{\mu+\rho})$.
By \eqref{eq-eva-smurho} the degree of the first factor is:
\begin{equation*}
\deg s_\mu(q^\rho) = -\half |\mu| -n_\mu.
\end{equation*}
Now consider the second factor
$s_\nu (q^{\mu+\rho})$.
Denote $x_i (q) = q^{\mu_i-i+\half}$,
and one has:
\begin{equation*}
s_\nu (q^{\mu+\rho}) =
 s_\nu (x_1(q),x_2(q),x_3(q),\cdots)
=\sum_T x^T,
\end{equation*}
where $T$ runs over all  possible semistandard Young tableau of shape $\nu$ (see \S \ref{subsec-Schur}).
Here it is clear that $\deg x_1 > \deg x_2 > \deg x_3 >\cdots$,
and thus the SSYT $T$ whose weight $x^T$ is of the highest degree
is exactly the tableau whose entries in the $i$-th row are all $i$.
Thus:
\begin{equation*}
\deg s_\nu (q^{\mu+\rho}) = \deg \prod_{i\geq 1} x_i(q)^{v_i}
= \sum_{i\geq 1} v_i (\mu_i - i +\half)
= \mu\cdot \nu -n_\nu -\half |\nu|.
\end{equation*}
This completes the proof.
\end{proof}

The following are some examples for $l(\mu)$ and $l(\nu)$ small:
\begin{equation*}
\begin{split}
& \deg W_{(m),(n)}(q) = mn -\half (m+n),\\
& \deg W_{(m-1,1),(n)}(q) =  mn -\half m -\frac{3}{2} n -1, \\
& \deg W_{(m-2,2),(n)}(q) = mn -\half m -\frac{5}{2} n -2, \\
& \deg W_{(m-2,1^2),(n)}(q) = mn -\half m -\frac{5}{2} n -3, \\
& \deg W_{(m-1,1),(n-1,1)}(q) = mn -\frac{3}{2} m -\frac{3}{2} n. \\
\end{split}
\end{equation*}
By the above proposition we can prove the following:
\begin{Proposition}
\label{prop-degWineq}
Let $m,n$ be two positive integers.
Let $\mu = (\mu_1,\cdots,\mu_{l(\mu)})$ be a partition of $m$,
and $\nu = (\nu_1,\cdots,\nu_{l(\nu)})$ be a partition of $n$.
Denote:
\begin{equation*}
\delta_\mu = m-\mu_1 =\sum_{i=2}^{l(\mu)} \mu_i,
\qquad\qquad
\delta_\nu = n-\nu_1 =\sum_{i=2}^{l(\nu)} \nu_i,
\end{equation*}
then one has:
\be
\label{eq-prop-degWineq}
\deg W_{(m),(n)}(q) - \deg W_{\mu,\nu}(q) \geq
\delta_\mu \cdot (\nu_1 +1) +  \delta_\nu \cdot (\mu_1 +1),
\ee
and the equality holds if and only if $l(\mu) \leq 2$ and $l(\nu) \leq 2$.
\end{Proposition}
\begin{proof}
By \eqref{eq-prop-degW} one has:
\begin{equation*}
\begin{split}
\deg W_{(m),(n)}(q) - \deg W_{\mu,\nu}
= & \big(mn-\frac{m+n}{2}\big) -
\big(  \sum_i \mu_i \nu_i - n_\mu-n_\nu -\frac{m+n}{2}\big) \\
= & \sum_{i\not= j}\mu_i\nu_j +n_\mu +n_\nu.
\end{split}
\end{equation*}
Notice that
\begin{equation*}
\begin{split}
\sum_{i\not= j}\mu_i\nu_j \geq
\mu_1 \cdot \sum_{j\geq 2}\nu_j + \nu_1 \cdot \sum_{j\geq 2}\mu_j
= \mu_1 \cdot \delta_\nu + \nu_1 \cdot \delta_\mu,
\end{split}
\end{equation*}
and the equality holds if and only if $l(\mu) \leq 1$,
or $l(\nu) \leq 1$,
or $l(\mu) = l(\nu) =2$.

Moreover,
from the definition \eqref{eq-def-kappamu} of $n_\nu$ one easily sees that
\begin{equation*}
n_\mu \geq n-\mu_1 =\delta_\mu,
\qquad\qquad
n_\nu \geq n-\nu_1 =\delta_\nu,
\end{equation*}
and these two equalities hold if and only if $l(\mu) \leq 2$ and $l(\nu) \leq 2$.
\end{proof}

By taking $\mu = (m)$,
one obtains the following generalization of Guo--Zhou \cite[Lemma 4.3]{gz}:
\begin{Corollary}
Let $m,n$ be two positive integers,
and let $\nu$ be a partition of $n$ with $l(\nu)\geq 2$.
Denote $\delta_\nu = n- \nu_1 = \sum_{i=2}^{l(\nu)} \nu_i$,
then one has:
\be
\label{eq-deggap-Wm}
\deg W_{(m),(n)}(q) - \deg W_{(m), \nu} (q) \geq  \delta_\nu \cdot (m+1),
\ee
and the equality holds if and only if $l(\nu) =2$,
i.e.,
$\nu$ is of the form $(\nu_1, n-\nu_1)$.
\end{Corollary}

Similarly,
by taking $\mu = (m-1,1)$ in \eqref{eq-prop-degWineq} one obtains:
\begin{Corollary}
Let $m,n$ be two positive integers with $m,n\geq 2$,
and let $\nu$ be a partition of $n$ with $l(\nu)\geq 2$.
Then one has:
\be
\label{eq-deggap-Wm-11}
\deg W_{(m),(n)}(q) - \deg W_{(m-1,1),\nu} \geq
\delta_\nu \cdot (m-1)+n+1
\ee
where $\delta_\nu = n - \nu_1 =\sum_{i=2}^{l(\nu)} \nu_i$,
and the equality holds if and only if $\nu = (\nu_1, n-\nu_1)$.
\end{Corollary}

\subsection{Partition function of local $\bP^2$ and degree of $I(d)$}

In this subsection we consider the degree of each $I(d)$
in the partition function of local $\bP^2$.
We follow the notations in \cite{gz}.

Denote by $H \in H_2 (\bP^2;\bZ)$ the hyperplane class of $\bP^2$,
and then $H_2 (\bP^2 ;\bZ) = \bZ H$.
For an integer $d \in \bZ$,
denote by $K_g^d$ the $0$-pointed Gromov--Witten invariant of $K_{\bP^2}$
of genus $g$ and class $dH \in H_2 (\bP^2;\bZ)$.
Then the free energy of local $\bP^2$ is the generating series:
\begin{equation*}
F^{K_{\bP^2}} = \sum_{g\geq 0} \lambda^{2g-2} \sum_{d>0} K_d^g \cdot t^d.
\end{equation*}
This free energy is related to the partition function \eqref{eq-partitionfn-kp2} by:
\begin{equation*}
Z^{K_{\bP^2}}_{\text{closed}} = \exp ( F^{K_{\bP^2}}).
\end{equation*}
Expand $Z^{K_{\bP^2}}_{\text{closed}}$ and $F^{K_{\bP^2}}$ as formal power series in $t$,
and denote the coefficients of $t^d$ by $I(d)$ and $F(d)$ respectively:
\be
\begin{split}
Z^{K_{\bP^2}}_{\text{closed}} = 1+ \sum_{d=1}^\infty I(d)t^d,
\qquad\qquad
F^{K_{\bP^2}} =  \sum_{d=1}^\infty F(d)t^d.
\end{split}
\ee
Then $I(d)$ and $F(d)$ are related by:
\be
\label{eq-IFexpansion}
\begin{split}
& I(d) = \sum_{k\geq 1}
\sum_{d_1+\cdots +d_k = d} \frac{1}{k!}  \cdot F(d_1)F(d_2) \cdots F(d_k),\\
& F(d) = \sum_{k\geq 1}
\sum_{d_1+\cdots +d_k = d} \frac{(-1)^{k-1}}{k} \cdot I(d_1)I(d_2) \cdots I(d_k).
\end{split}
\ee

In what follows,
we represent $I(d)$ in terms of two-legged topological vertices
and estimate the degrees of terms in $I(d)$.
By \eqref{eq-partitionfn-kp2} we have:
\be
\label{eq-defId}
I(d) = \sum_{\sum_{i=1}^3 |\mu^i| =d}
(-1)^d q^{ \half \sum_{i=1}^3  \kappa_{\mu^i}}
W_{\mu^{3},\mu^1}(q) W_{\mu^{1},\mu^2}(q) W_{\mu^{2},\mu^3}(q).
\ee
As observed in Guo--Zhou \cite[\S 5.2]{gz},
$I(d)$ should be decomposed into:
\be
\label{eq-Id-twotermdecomp}
\begin{split}
I(d) = & \sum_{d_1+d_2+d_3 =d}
(-1)^d q^{ \half \sum_{i=1}^3  \kappa_{(d_i)}}
W_{(d_3),(d_1)}(q) W_{(d_1),(d_2)}(q) W_{(d_2),(d_3)}(q) \\
& + \sum_{\substack{ \sum |\mu^i| =d \\ \exists i, \text{ s.t. } l(\mu^i)>1 }}
(-1)^d q^{ \half \sum_{i=1}^3  \kappa_{\mu^i}}
W_{\mu^{3},\mu^1}(q) W_{\mu^{1},\mu^2}(q) W_{\mu^{2},\mu^3}(q),
\end{split}
\ee
where the degree of the first term is:
\begin{align*}
	\begin{split}
		&\deg \Big( q^{ \half \sum_{i=1}^3  \kappa_{(d_i)}}
		W_{(d_3),(d_1)}(q) W_{(d_1),(d_2)}(q) W_{(d_2),(d_3)}(q) \Big) \\
		=& \half \sum_{i=1}^3 \kappa_{(d_i)} + d_3d_1-\frac{d_3+d_1}{2}
		+ d_1d_2-\frac{d_1+d_2}{2} + d_2d_3-\frac{d_2+d_3}{2} \\
		=& \half (d^2-3d),
	\end{split}
\end{align*}
and the degree of the second term is $\half (d^2-3d) -d-2$,
see \cite[Lemma 5.1]{gz}.
Here one has:
\be
\label{eq-degId}
\deg I(d) = \half d(d-3).
\ee
Following the notations in \cite{gz},
we denote:
\be
\label{eq-def-Wdq}
W_d(q) = q^{-\half (d^2-3d)}\sum_{d_1+d_2+d_3 =d}
q^{ \half \sum_{i=1}^3  \kappa_{(d_i)}}
W_{(d_3),(d_1)}(q) W_{(d_1),(d_2)}(q) W_{(d_2),(d_3)}(q),
\ee
and then $W_d(q)$ is of degree $0$.

\subsection{Expicit formulas for $I(1)$, $I(2)$, $I(3)$}
\label{sec-app-I123}

In this subsection we write down the explicit formulas for $I(1)$, $I(2)$, $I(3)$.
These formulas serve as concrete examples of $I(d)$,
and will be useful in \S \ref{eq-compGVleading-expl}.

First by the definitions \eqref{eq-defId} one knows that:
\begin{equation*}
I(1) = -3 (W_{(1),(\emptyset)})^2 \cdot W_{(\emptyset),(\emptyset)},
\end{equation*}
Moreover,
by \eqref{eq-W12inSchur} one has:
\begin{equation*}
W_{(\emptyset),(\emptyset)}(q)=1,
\qquad\qquad
W_{(1),(\emptyset)}(q) = \frac{q^{-\half}}{[1]_{q^{-1}}},
\end{equation*}
and thus:
\be
I(1) = - \frac{3q^{-1}}{[1]_{q^{-1}}^2}.
\ee

As for $d=2$,
by definition one has:
\be
\label{eq-defI2sum}
I(2) = 3q\cdot (W_{(2),(\emptyset)})^2 W_{(\emptyset),(\emptyset)}
+ 3q^{-1}\cdot (W_{(1^2),(\emptyset)})^2 W_{(\emptyset),(\emptyset)}
+ 3 (W_{(1),(\emptyset)})^2 W_{(1),(1)}.
\ee
By \eqref{eq-eva-smurho}
one has:
\begin{equation*}
W_{(2),(\emptyset)} = s_{(2)}(q^\rho) = \frac{q^{-1}}{[2]_{q^{-1}}!},
\qquad\qquad
W_{(1,1),(\emptyset)} = s_{(1,1)}(q^\rho) = \frac{q^{-2}}{[2]_{q^{-1}}!}.
\end{equation*}
And moreover,
by \eqref{eq-gz-183184} one has:
\begin{equation*}
\begin{split}
W_{(1),(1)} =\frac{1}{[1]_{q^{-1}}} \big(1+\frac{q^{-2}}{[1]_{q^{-1}}}\big).
\end{split}
\end{equation*}
Plug these into \eqref{eq-defI2sum},
and we obtain:
\be
\begin{split}
I(2)
=& \frac{3q^{-1}(2+q^{-1}+q^{-3}+2q^{-4})}{([2]_{q^{-1}}!)^2}.
\end{split}
\ee

Now consider $d=3$.
By definition one has:
\be
\label{eq-defI3sum}
\begin{split}
I(3) =& -3q^3 \cdot (W_{(3),(\emptyset)})^2 W_{(\emptyset),(\emptyset)}
-3  \cdot (W_{(2,1),(\emptyset)})^2 W_{(\emptyset),(\emptyset)}
-3q^{-3} \cdot (W_{(1^3),(\emptyset)})^2 W_{(\emptyset),(\emptyset)} \\
& -6q W_{(2),(1)}W_{(2),(\emptyset)}W_{(1),(\emptyset)}
-6q^{-1} W_{(1^2),(1)}W_{(1^2),(\emptyset)}W_{(1),(\emptyset)}
- (W_{(1),(1)})^3.
\end{split}
\ee
By \eqref{eq-eva-smurho} one knows that:
\begin{equation*}
\begin{split}
W_{(3),(\emptyset)} =  \frac{q^{-\frac{3}{2}}}{[3]_{q^{-1}}!},
\qquad
W_{(1^3),(\emptyset)} =  \frac{q^{-\frac{9}{2}}}{[3]_{q^{-1}}!},
\qquad
W_{(2,1),(\emptyset)} =  \frac{q^{-\frac{5}{2}}}{[1]_{q^{-1}}^2 [3]_{q^{-1}}}.
\end{split}
\end{equation*}
Moreover,
by \eqref{eq-gz-183184} one has:
\begin{equation*}
\begin{split}
W_{(1),(1)} =\frac{1}{[1]_{q^{-1}}} \big(1+\frac{q^{-2}}{[1]_{q^{-1}}}\big),
\qquad\qquad
W_{(2),(1)} = \frac{ q^{\half}}{[2]_{q^{-1}}!} \big(1+\frac{q^{-3}}{[1]_{q^{-1}}} \big),
\end{split}
\end{equation*}
and by \eqref{eq-gzlem6.7} one has:
\begin{equation*}
\begin{split}
W_{(1^2),(1)} = q^{-\frac{3}{2}} \Big(
\frac{1}{[1]_{q^{-1}}^2} \big( 1+\frac{q^{-2}}{[1]_{q^{-1}}} \big)
- \frac{q^{-2}}{[1]_{q^{-1}}! \cdot [2]_{q^{-1}}!}
\Big) =
\frac{q^{-\frac{3}{2}} (1-q^{-2}+q^{-3})}{[1]_{q^{-1}}^2 [2]_{q^{-1}}}.
\end{split}
\end{equation*}
Plug these into \eqref{eq-defI3sum} and we obtain:
\be
\begin{split}
I(3) = -\frac{1}{([3]_{q^{-1}}!)^2} \big(&
10 + 7q^{-1} +8q^{-2} +3q^{-3} +10q^{-4} +26q^{-5} +34q^{-6} \\
& +26q^{-7}
+10q^{-8} +3q^{-9} + 8q^{-10} +7q^{-11} +10q^{-12} \big).
\end{split}
\ee

\subsection{Degree estimates of summands in $I(d)$}
\label{sec-degreeest-Id}

For our purpose we need a further detailed estimate
of the degrees of the summands in $I(d)$.
First by Proposition \ref{prop-degWineq} we easily see that:
\begin{Lemma}
\label{eq-lem-degWprod3}
Let $\lambda,\mu,\nu$ be three partitions of integers.
Denote $m=|\lambda|$, $n=|\mu|$, $k =|\nu|$,
and:
\begin{equation*}
\delta_\lambda = m-\lambda_1,
\qquad
\delta_\mu = n-\mu_1,
\qquad
\delta_\nu = k-\nu_1,
\end{equation*}
then one has:
\begin{equation*}
\begin{split}
& \deg \Big( W_{(k),(m)}(q) W_{(m),(n)}(q) W_{(n),(k)}(q) \Big) -
\deg \Big( W_{\nu,\lambda}(q) W_{\lambda,\mu}(q) W_{\mu,\nu}(q) \Big) \\
\geq & (\delta_\lambda + \delta_\mu) (\nu_1+1)
+ (\delta_\lambda + \delta_\nu) (\mu_1+1)
+ (\delta_\mu + \delta_\nu) (\lambda_1+1),
\end{split}
\end{equation*}
and the equality holds if and only if $\lambda,\mu,\nu$ are all of length $1$ or $2$.
\end{Lemma}

And for the factors $q^{ \sum  \kappa_{\mu^i}/2}$,
we need the following:
\begin{Lemma}
\label{lem-estkappa}
Let $\mu$ be a partition of a positive integer $N$.
Denote $\delta = N-\mu_1$,
then:
\begin{align*}
	\kappa_\mu \leq N^2 - (2\delta+1)\cdot N +2(\delta^2-\delta),
\end{align*}
and the equality holds if and only if $l(\mu)\leq 2$.
As a consequence,
one has:
\begin{align*}
	\kappa_{(N)} - \kappa_\mu \geq 2\delta\cdot N -2(\delta^2-\delta)
	= 2\delta \cdot (\mu_1 +1).
\end{align*}
\begin{proof}
In the case $l(\mu)=1$,
by the definition of $\kappa_\mu$ one directly checks that:
\begin{equation*}
\kappa_{(N)} = N^2-N.
\end{equation*}

Now assume $\mu=(m_1,\cdots,m_k|n_1,\cdots,n_k)$ in the Frobenius notation.
If $k=1$,
one has:
\begin{equation*}
\kappa_\mu = (\mu_1-\half)^2 -(N-\mu_1+\half)^2 =
(N-\delta -\half)^2 -(\delta +\half)^2,
\end{equation*}
and the conclusion clearly holds.
Next we assume $k\geq 2$, and in this case the numbers $m_2>\cdots >m_k\geq 0$ and $n_1>\cdots >n_k\geq 0$ satisfy:
\begin{equation*}
k+ \sum_{i=2}^k m_i + \sum_{i=1}^k n_i = N-m_1 = N- (N-\delta -1)= \delta+1.
\end{equation*}
From the equivalent description \eqref{eq-def-kappa2} of $\kappa_\mu$,
one easily sees that $\kappa_\mu$ attains its maximum only when $k=2$ and:
\begin{equation*}
m_2 = \delta-2,
\qquad
n_1 =1,
\qquad
n_2 = 0.
\end{equation*}
In this case one can directly check that
$\kappa_{(N-\delta-1,\delta-2|1,0)} = N^2 - (2\delta+1)\cdot N +2(\delta^2-\delta)$.
\end{proof}
\end{Lemma}

Now by this lemma and Lemma \ref{eq-lem-degWprod3},
we have the following degree estimate for the terms in the decomposition \eqref{eq-Id-twotermdecomp}:
\begin{Proposition}
\label{prop-WWWdeg-ineq}
Let $\mu^1,\mu^2,\mu^3$ be three partitions.
Let $d_i = |\mu^i|$, and $d=d_1 + d_2 +d_3$,
and we denote:
\begin{equation*}
\begin{split}
& \delta_i = d_i- \mu^i_1 = \mu^i_2+\cdots +\mu^i_{l(\mu^i)},
\qquad i=1,2,3; \\
& \Delta = \delta_1+\delta_2 +\delta_3 =d- \mu^1_1 -\mu^2_1 -\mu^3_1,
\end{split}
\end{equation*}
then one has:
\begin{equation*}
\begin{split}
& \deg \big(
q^{ \half \sum_{i=1}^3  \kappa_{(d_i)}}
\prod_{i=1}^3 W_{(d_i),(d_{i+1})}(q) \big)
- \deg \big( q^{ \half \sum_{i=1}^3  \kappa_{\mu^i}}
\prod_{i=1}^3 W_{\mu^{i},\mu^{i+1}}(q) \big) \geq \Delta (d-\Delta+3),
\end{split}
\end{equation*}
and the equality holds if and only if $l(\mu^i) \leq 2$ for every $i$.
Here we use the conventions $\mu^4=\mu^1$ and $d_4 =d_1$.
\end{Proposition}
\begin{proof}
Denote by L.H.S. the left-hand side of the above inequality.
Then by Lemma \ref{eq-lem-degWprod3} and Lemma \ref{lem-estkappa} one has:
\begin{equation*}
\begin{split}
\text{L.H.S.} \geq &  \half\cdot \sum_{i=1}^3 2\delta_i (\mu^i_1 +1) +
\sum_{i=1}^3 (\delta_i+\delta_{i+1}) (\mu^{i+2}_1 +1) \\
=& (\delta_1+\delta_2+\delta_3) (\mu^1_1 + \mu^2_1 + \mu^3_1 +3) \\
=& \Delta\cdot (d-\Delta +3),
\end{split}
\end{equation*}
where we use the conventions $\mu^{i+3}=\mu^i$ and $\delta_{i+3} =\delta_i$ for $i=1,2$.
\end{proof}

As a corollary we have the following:
\begin{Corollary}
\label{cor-k=1234-degWWW}
Fix a number $k\in \{1,2,3,4\}$.
Let $\mu^1,\mu^2,\mu^3$ be three partitions,
and let $d_i$, $d$, $\Delta$ be as in Proposition \ref{prop-WWWdeg-ineq}.
If $\Delta \geq k$,
then one has:
\be
\label{eq-k=1234-degWWW}
\begin{split}
& \deg \big(
q^{ \half \sum_{i=1}^3  \kappa_{(d_i)}}
\prod_{i=1}^3 W_{(d_i),(d_{i+1})}(q) \big)
- \deg \big( q^{ \half \sum_{i=1}^3  \kappa_{\mu^i}}
\prod_{i=1}^3 W_{\mu^{i},\mu^{i+1}}(q) \big) \geq k(d+3-k),
\end{split}
\ee
and the equality holds if and only if $\Delta = k$
and $l(\mu^i)\leq 2$ for every $i$.
\end{Corollary}
\begin{proof}
The case $k=1$ is Guo--Zhou \cite[Lemma 5.1]{gz},
and the case $k=2$ follows easily from Guo--Zhou \cite[Lemma 4.3]{gz}.
Here we prove all the cases $k=1,2,3,4$ using a unified approach.

Assume $\Delta>0$,
then by the definition of $d$ and $\Delta$ it is clear that $0<\Delta \leq d-1$.
The quadratic function $f(\Delta) = \Delta (d-\Delta +3)$
is increasing when $1 \leq \Delta \leq \half (d+3)$,
and the graph of this function is symmetric with respect to the line $\Delta = \half (d+3)$.
Therefore $\Delta\geq k$ implies $f(\Delta)\geq f(k)$ when $k=1,2,3,4$,
and then the inequality \eqref{eq-k=1234-degWWW} follows from Proposition \ref{prop-WWWdeg-ineq}.
Moreover,
when $k=1,2,3$ one easily sees that $\Delta > k$ implies $\Delta >f(k)$,
and thus the equality holds if and only if $\Delta = k$ in these cases.

As for the case $k=4$,
one has $f(\Delta) > f(4) = f(d-1)$ for $4<\Delta<d-1$.
If $\Delta = d-1$,
by the definition of $d$ and $\Delta$ one knows that
two of $\mu^1,\mu^2,\mu^2$ are the empty partition and the remaining one is $(1^d)$ (where $d\geq 5$),
and thus the inequality in Proposition \ref{prop-WWWdeg-ineq} is strict in this case.
Therefore for $k=4$ the equality in \eqref{eq-k=1234-degWWW} holds
only when $k=4$.
\end{proof}

In what follows,
for a nonnegative number $k (\leq d)$ we denote:
\be
I^{\leq k}(d) =
 \sum_{\sum_{i=1}^3 |\mu^i| =d, \text{ } \sum \delta_i \leq k}
(-1)^d q^{ \half \sum_{i=1}^3  \kappa_{\mu^i}}
W_{\mu^{3},\mu^1}(q) W_{\mu^{1},\mu^2}(q) W_{\mu^{2},\mu^3}(q),
\ee
where $\delta_i = |\mu^i| - \mu^i_1$ as above.
In particular,
one has:
\begin{equation*}
I^{\leq 0}(d) =
(-1)^d q^{\half d(d-3)} W_d(q),
\qquad\qquad
I^{\leq d}(d) = I(d).
\end{equation*}
Then by Corollary \ref{cor-k=1234-degWWW} one can see that
for a fixed $k\in \{1,2,3,4\}$,
the following relation holds:
\be
\deg I(d) -\deg \Big(I(d)-I^{\leq k-1}(d)\Big) \geq k(d+3-k).
\ee
One may rewrite these inequalities more explicitly:
\be
\label{eq-IdIleqd1234}
\begin{split}
&\deg I(d) -\deg \Big( I(d) - I^{\leq 0} (d) \Big) \geq d+2, \\
&\deg I(d) -\deg \Big(I(d)-I^{\leq 1}(d)\Big) \geq 2d+2, \\
&\deg I(d) -\deg \Big(I(d)-I^{\leq 2}(d)\Big) \geq 3d, \\
&\deg I(d) -\deg \Big(I(d)-I^{\leq 3}(d)\Big) \geq 4d-4. \\
\end{split}
\ee
And as a straightforward consequence,
one has:
\begin{Proposition}
\label{prop-Ileqkdegree}
Fix a number $k\in \{1,2,3,4\}$,
and then for $d\geq k$ one has:
\be
I(d) = I^{\leq k-1} (d) \big|_{q^{\geq \half d(d-3) -k(d+3-k) +1}}.
\ee
\end{Proposition}

In particular,
for $k=1$ one has
(\cite[Lemma 5.2]{gz}):
\begin{equation*}
I(d) = I^{\leq 0}(d) = (-1)^d q^{\half d(d-3)} W_d(q)
\big|_{q^{\geq \half d(d-3) -d-1}}.
\end{equation*}
The following explicit formula for $W_d(q)$ is known in \cite[\S 5.2]{gz}:
\be
\label{eq-Wd-sumdef}
W_d(q) = \sum_{d_1 + d_2+d_3 =d}
\sum_{k_1 = 0}^{d_1} \sum_{k_2 = 0}^{d_2} \sum_{k_3 = 0}^{d_3}
\frac{q^{-k_2(d_1 +1)} q^{-k_3(d_2 +1)} q^{-k_1 (d_3+1)}}{\prod_{i=1}^3 ([d_i]_{q^{-1}}![k_i]_{q^{-1}}!) }.
\ee
Then one sees that:
\be
\begin{split}
I(d) = &(-1)^d q^{\half d(d-3)}\cdot \\
& \sum_{d_1 + d_2+d_3 =d}
\sum_{k_1 = 0}^{d_1} \sum_{k_2 = 0}^{d_2} \sum_{k_3 = 0}^{d_3}
\frac{q^{-k_2(d_1 +1)} q^{-k_3(d_2 +1)} q^{-k_1 (d_3+1)}}{\prod_{i=1}^3 ([d_i]_{q^{-1}}![k_i]_{q^{-1}}!) }
\Big|_{q^{\geq \half d(d-3) -d-1}}.
\end{split}
\ee

\subsection{Explicit formula for leading terms of $I(d)$}
\label{sec-leadingId}

In this subsection we give an explicit formula
for the leading terms
\begin{equation*}
I(d) \big|_{q^{\geq \half d(d-3) -3d+1}}
\end{equation*}
for $d\geq 4$.
This is a generalization of Guo--Zhou's \cite[Theorem 6.4]{gz}.
Our main result is:

\begin{Theorem}
\label{thm-main1-leadI}
For $d\geq 4$,
the leading terms of $I(d)$ are given by:
\be
\label{thm-mainIleading}
\begin{split}
I(d)
= \frac{ 3 (-1)^d q^{ \half d(d-3)} }{ ([\infty]_{q^{-1}}!)^3}
\Big(&  \frac{1}{3} J_0(d) + \frac{q^{-d-2}}{[1]_{q^{-1}}^2} J_1(d) \\
&+\frac{q^{-2d-2}( 2+q^{-1} +q^{-3}+2q^{-4} )}{([2]_{q^{-1}}!)^2} J_2(d)
\Big)
\Big|_{q^{\geq \half (d^2-3d) -3d+1}},
\end{split}
\ee
where:
\be
\label{eq-defa0a1a2}
\begin{split}
&J_0(d) =  \binom{d+2}{2} - 3\sum_{i\geq 1}\frac{q^{-i}}{(1-q^{-i})^2}, \\
&J_1(d) = \binom{d+1}{2}+3 - 3\sum_{i\geq 1}\frac{q^{-i}}{(1-q^{-i})^2}, \\
&J_2(d) = \binom{d}{2} +6- 3\sum_{i\geq 1}\frac{q^{-i}}{(1-q^{-i})^2},
\end{split}
\ee
and we use the notation:
\be
\label{eq-definftyfactorial}
[\infty]_{q^{-1}}! =\prod_{i=1}^\infty [i]_{q^{-1}} = \prod_{i=1}^\infty (1-q^{-i}).
\ee
\end{Theorem}

This theorem is proved by explicit calculations of Schur functions and hypergeometric series.
The detailed proof is lengthy and
we put it in \S \ref{app-WdCompute} -- \S \ref{app-secC-delta12} to make the main body easier to read.
In \S \ref{app-WdCompute} we compute the leading terms of $W_d$
(which is the part $\Delta=0$ in $I(d)$, and see Proposition \ref{prop-WWWdeg-ineq} for the notation $\Delta$),
and in \S \ref{app-secC-delta12} we compute the leading terms of
the parts $\Delta=1$ and $\Delta=2$ in $I(d)$
using Guo--Zhou's formula \eqref{eq-gzlem6.7} and Theorem \ref{thm-Delta2-expl} respectively.
Here we only need to consider the three cases $\Delta = 0,1,2$ because by Proposition \ref{prop-Ileqkdegree} we have:
\begin{equation*}
I(d) = I^{\leq 2} (d) \big|_{q^{\geq \half d(d-3) -3d +1}}.
\end{equation*}

It is worthwhile mentioning that
in \S \ref{sec-app-I123} we have already written down
explicitly the formulas for $I(1)$ and $I(2)$:
\begin{equation*}
I(1) = - \frac{3q^{-1}}{[1]_{q^{-1}}^2},
\qquad\qquad
I(2) =\frac{3q^{-1}(2+q^{-1}+q^{-3}+2q^{-4})}{([2]_{q^{-1}}!)^2},
\end{equation*}
and thus one can rewrite the above theorem as follows:
\begin{Theorem}
\label{thm-main1-leadI-2}
For $d\geq 4$,
the leading terms of $I(d)$ are given by:
\be
\begin{split}
I(d)
= \frac{ (-1)^d \cdot q^{ \half d(d-3)} }{ ([\infty]_{q^{-1}}!)^3}
\Big(  J_0(d) - q^{-d-1} I(1) J_1(d)
+q^{-2d-1} I(2) J_2(d)
\Big)
\Big|_{q^{\geq \half (d^2-3d) -3d+1}},
\end{split}
\ee
where $J_0,J_1,J_2$ are given by \eqref{eq-defa0a1a2}.
\end{Theorem}

\subsection{Degree estimates of the free energy of local $\bP^2$}

Now we consider the degrees of the terms $F(d)$ in the free energy.

First recall that $F(d)$ and $I(d)$ are related by:
\be
\label{eq-relFI}
F(d) = I(d)
+ \sum_{k\geq 2} \sum_{d_1+\cdots +d_k = d} \frac{(-1)^{k-1}}{k} I(d_1)I(d_2) \cdots I(d_k).
\ee
and the degree of $I(d)$ is:
\begin{equation*}
\deg I(d) = \half (d^2 - 3d)
= \half (d-1)(d-2)-1.
\end{equation*}
Let $k\geq 2$,
and for fixed $d_1,\cdots,d_k$ with $d_1+\cdots +d_k =d$ one has:
\be
\deg I(d_1)\cdots I(d_k) =
\sum_{i=1}^k \half (d_i^2-3d_i)
= (\half\sum_{i=1}^k d_i^2 ) -\frac{3}{2}d.
\ee
Notice that for $k\geq 2$,
the sum of squares $\sum_{i=1}^k d_i^2$ attains the maximum only when
one of $d_1,\cdots,d_k$ is $(d-k+1)$ and all others are $1$,
thus one has:
\be
\deg I(d_1) \cdots I(d_k)
\leq \half\big(k-1+(d-k+1)^2 \big) -\frac{3}{2} d,
\ee
and thus we obtain:
\be
\label{eq-degreeIdkgap}
\deg I(d) - \deg I(d_1) \cdots I(d_k)
\geq (k-1)(d-\frac{k}{2}).
\ee
Then by applying \eqref{eq-degreeIdkgap} to \eqref{eq-relFI} we obtain:
\begin{equation*}
F(d) = I(d) +\sum_{j=2}^{k} \sum_{d_1+\cdots +d_j = d}
\frac{(-1)^{j-1}}{j}I(d_1)\cdots I(d_j)
\Big|_{q^{\geq \half (d^2-3d) -k(d-\frac{k+1}{2} ) +1 }}.
\end{equation*}
To summarize, we have:
\begin{Lemma}
\label{lem-F&Iqgeqkd}
For a fixed positive number $k$,
one has:
\be
F(d) = I(d) +\sum_{j=2}^{k} \sum_{d_1+\cdots +d_j = d}
\frac{(-1)^{j-1}}{j}I(d_1)\cdots I(d_j)
\Big|_{q^{\geq g(d) - kd + \half k(k+1)}},
\ee
where we use the notation:
\be
\label{eq-notation-gd}
g(d) = \half(d-1)(d-2) = \half d(d-3)+1.
\ee
\end{Lemma}
In particular,
for $k = 1$ one has (\cite[Proposition 5.3]{gz}):
\be
\label{eq-gz-F&Igeqd}
F(d) = I(d) \big|_{q^{\geq g(d) - d+1}}.
\ee
And for $k=2,3,4$ we have:
\be
\label{eq-gz-F&Igeqkd}
\begin{split}
F(d) = \Big(& I(d) -\half \sum_{d_1+d_2 =d}I(d_1)I(d_2) \Big)
\Big|_{q^{\geq g(d) - 2d+3}}, \\
F(d) =  \Big(& I(d) -\half \sum_{d_1+d_2 =d}I(d_1)I(d_2)
+ \frac{1}{3}\sum_{d_1+d_2 +d_3=d}I(d_1)I(d_2)I(d_3) \Big)
\Big|_{q^{\geq g(d) - 3d+6}}, \\
F(d) =  \Big(& I(d) -\half \sum_{d_1+d_2 =d}I(d_1)I(d_2)
+ \frac{1}{3} \sum_{d_1+d_2 +d_3=d}I(d_1)I(d_2)I(d_3) \\
& -\frac{1}{4} \sum_{d_1+d_2 +d_3 +d_4 =d}I(d_1)I(d_2)I(d_3)I(d_4) \Big)
\Big|_{q^{\geq g(d) - 4d+10}}, \\
\end{split}
\ee
respectively.

\begin{Remark}
One can use the above relations and the leading terms of $I(d)$ to compute the
leading terms of $F(d)$.
Here we omit this discussion since the computation is exactly the same as
the computation of leading terms of the transformed Gopakumar--Vafa invariants
in next section.
In fact, by Proposition \ref{prop-GVleadingderiveI} we know that the leading terms
\begin{equation*}
F(d) \big|_{ q^{\geq  \half d(d-3) -3d+1}}
\end{equation*}
and the leading terms of the transformed Gopakumar--Vafa invariants
computed in \S \ref{eq-compGVleading-expl} differ only by an overall factor $-q (1-q^{-1})^2(1-q^{-2})$.
\end{Remark}

%\subsection{Explicit formula for leading terms of $F(d)$}
%\label{sec-leadingFd}

%Now we use the results of the above two subsections
%to derive an explicit formula for the leading terms
%of the generating series $F(d)$ of (connected) Gromov-Witten invariants of local $\bP^2$.

\section{Leading Terms of Transformed Gopakumar--Vafa Invariants of Local $\bP^2$}
\label{sec-LeadingTGV}

In this section,
we briefly review the Gopakumar--Vafa invariants and transformed Gopakumar--Vafa invariants
of local $\bP^2$,
and then compute the leading terms of the transformed Gopakumar--Vafa invariants
of local $\bP^2$ using the leading terms of $I(d)$.

\subsection{Gopakumar--Vafa invariants of local $\bP^2$}

In this subsection we recall the relation between the Gopakumar--Vafa invariants and Gromov--Witten invariants
of local $\bP^2$.

The Gopakumar--Vafa invariants $n_\beta^g(K_S)$ of a local Calabi-Yau threefold $K_S$
are given by \eqref{eq-def-GVlocal}.
Thus for local $\bP^2$ we have:
\begin{equation*}
\sum_{d>0} F(d)t^d =
\sum_{d>0}\sum_{g\geq 0} n_d^g (K_{\bP^2})
\sum_{k>0} \frac{(-1)^{g-1}}{k} \big( q^{\frac{k}{2}} - q^{-\frac{k}{2}}\big)^{2g-2}
t^{kd},
\end{equation*}
and then:
\be
\label{eq-FdinGV-1}
F(d) = \sum_{g\geq 0} \sum_{k|d}
n_{d/k}^g (K_{\bP^2})\cdot \frac{(-1)^{g-1}}{k} \big( q^{\frac{k}{2}} - q^{-\frac{k}{2}}\big)^{2g-2}.
\ee
It has been observed in \cite{gv, amv} and proved in \cite{pe} that:
\begin{align*}
	n_{d}^g(K_{\bP^2}) = 0,\qquad\qquad \forall g>g(d)
\end{align*}
where $g(d)$ is \eqref{eq-notation-gd},
and thus the right-hand side of \eqref{eq-FdinGV-1} is a finite sum:
\be
\label{eq-FdinGV-2}
F(d) = \sum_{k|d} \sum_{g= 0}^{g(d/k)}
n_{d/k}^g (K_{\bP^2})\cdot \frac{(-1)^{g-1}}{k} \big( q^{\frac{k}{2}} - q^{-\frac{k}{2}}\big)^{2g-2}.
\ee
Following \cite{gz},
denote:
\begin{equation*}
F^k(d) = \sum_{g= 0}^{g(d/k)}
n_{d/k}^g (K_{\bP^2})\cdot \frac{(-1)^{g-1}}{k} \big( q^{\frac{k}{2}} - q^{-\frac{k}{2}}\big)^{2g-2},
\end{equation*}
for $k|d$,
and then:
\begin{equation*}
F(d) = \sum_{k|d} F^k(d).
\end{equation*}
Notice that one has:
\begin{equation*}
\deg F^k(d) \leq \frac{k}{2} \cdot \big(2g(\frac{d}{k})-2\big)
= \frac{d}{2} (\frac{d}{k} - 3).
\end{equation*}
By comparing this degree at $k=2$ with $\deg F(d) = \half d(d-3)$,
one obtains (\cite[(96)]{gz}):
\be
\label{eq-fdf1dqgeqd}
F(d) = F^1(d) \big|_{q^{\geq g(d)-d+1}}.
\ee
And similarly,
one can see that:
\be
\label{eq-fdf1dqgeqkd}
\begin{split}
& F(d) = F^1(d) \big|_{q^{\geq g(d)-2d+3}},
\qquad \forall d\geq 5;\\
& F(d) = F^1(d) \big|_{q^{\geq g(d)-3d+6}},
\qquad \forall d\geq 9;\\
& F(d) = F^1(d) \big|_{q^{\geq g(d)-4d+10}},
\qquad \forall d\geq 13,
\end{split}
\ee
where by definition
\begin{equation*}
F^1(d) = \sum_{g= 0}^{g(d)}
n_{d}^g (K_{\bP^2})\cdot (-1)^{g-1} \big( q^{\frac{1}{2}} - q^{-\frac{1}{2}}\big)^{2g-2}.
\end{equation*}

By comparing \eqref{eq-fdf1dqgeqd} with \eqref{eq-gz-F&Igeqd},
one obtains \cite[(99)]{gz}:
\be
\sum_{g=g(d)-d+2}^{g(d)}
n_{d}^g (K_{\bP^2}) (-1)^{g-1} \big( q^{\frac{1}{2}} - q^{-\frac{1}{2}}\big)^{2g-2}
= I(d) \big|_{q^{\geq g(d)-d+1}}.
\ee
And by \eqref{eq-fdf1dqgeqkd} and Lemma \ref{lem-F&Iqgeqkd} one finds that
for fix $k\in \{1,2,3,4\}$,
\be
\label{eq-gknIqgeq}
\begin{split}
\sum_{g= g^k (d)+1}^{g(d)}
n_{d}^g (K_{\bP^2}) (-1)^{g-1} \big( q^{\frac{1}{2}} - q^{-\frac{1}{2}}\big)^{2g-2}
=& \sum_{j=1}^k \sum_{d_1+\cdots +d_j =d} \frac{(-1)^{j-1}}{j} I(d_1)\cdots I(d_j)
\Big|_{q^{\geq g^k (d) }}
\end{split}
\ee
holds for every $d\geq 4k-3$,
where we use the notation:
\be
g^k(d) = g(d) -kd +\half k(k+1),
\qquad\qquad 1\leq k\leq 4.
\ee
The following is \cite[Lemma 5.5]{gz}:
\begin{Lemma}
\label{lem-gz-seriesmultres}
Let $f(q)$ and $g(q)$ be two Laurent polynomials satisfying $f(q) = g(q)|_{q^{\geq m}}$,
then:
\begin{itemize}
\item[1)]
$qf(q) = qg(q)|_{q^{\geq m+1}}.$
\item[2)]
$(1-q^{-j}) f(q) = (1-q^{-j}) g(q) |_{q^{\geq m}}$ for $j>0$.
\end{itemize}
\end{Lemma}

Moreover,
one can further prove that:
\begin{Proposition}
\label{prop-degLaurentprod}
Let $f(q)$ and $g(q)$ be two Laurent polynomials satisfying $f(q) = g(q)|_{q^{\geq m}}$,
and let $\varphi(q)$ be a Laurent polynomial of degree $k$.
Then:
\begin{equation*}
f(q)\varphi(q) = g(q)\varphi(q) \big|_{q^{\geq m+k}}.
\end{equation*}
\end{Proposition}
\begin{proof}
Denote:
\begin{equation*}
f(q)\big|_{q^{\geq m}} = g(q) \big|_{q^{\geq m}} = \sum_{i=m}^M a_i q^i,
\qquad\qquad
\varphi(q) = \sum_{i=-\infty}^k b_k q^i.
\end{equation*}
Then for every $j\geq k+m$,
it is clear that the coefficients of $q^{j}$ in $f(q)\varphi(q)$ and $g(q)\varphi(q)$ are both
$\sum_{i=m}^M a_ib_{j-i}$.
\end{proof}

Now multiply both sides of \eqref{eq-gknIqgeq} by $(q^\half - q^{-\half})^2 = q(1-q^{-1})^2$
and then use the proposition,
and in this way we may obtain the following:
\begin{Proposition}
Fix $k\in \{1,2,3,4\}$,
then for $\forall d\geq 4k-3$ one has:
\be
\label{eq-gknIqgeq-2g}
\begin{split}
& \sum_{g= g^k (d)+1}^{g(d)}
n_{d}^g (K_{\bP^2}) (-1)^{g-1} \big( q^{\frac{1}{2}} - q^{-\frac{1}{2}}\big)^{2g} \\
=& \Big( q(1-q^{-1})^2\sum_{j=1}^k \sum_{d_1+\cdots +d_j =d} \frac{(-1)^{j-1}}{j} I(d_1)\cdots I(d_j) \Big)
\Big|_{q^{\geq g^k (d) +1 }}.
\end{split}
\ee
\end{Proposition}

\subsection{Transformed Gopakumar--Vafa invariants}

Now consider the transformed Gopakumar--Vafa invariants $N_d^g (K_{\bP^2})$ of local $\bP^2$,
see \cite{gv},
and see also \cite{zhou5, gz}.
In this subsection we recall the definition of these invariants,
and use the leading terms of $I(d)$ to represent
the leading terms of generating series of $N_d^g (K_{\bP^2})$.

The transformed Gopakumar--Vafa invariants $N_d^g (K_{\bP^2})$ are related to the Gopakumar--Vafa invariants $n_d^g (K_{\bP^2})$ by:
\be
\label{eq-TGV&GV}
\sum_{g=0}^{g(d)}
N_d^g(K_{\bP^2})\cdot (q^g +q^{g-2}+\cdots +q^{-g})
= \sum_{g=0}^{g(d)} n_d^g(K_{\bP^2}) \cdot(-1)^g (q^\half -q^{-\half})^{2g}.
\ee
Here the two sequences
\begin{equation*}
\begin{split}
& S_g = (q^\half - q^{-\half})^{2g}, \qquad g\geq 0; \\
& R_g = q^g + q^{g-2}+\cdots +q^{-g+2} +q^{-g}, \qquad g\geq 0,
\end{split}
\end{equation*}
form two bases of the space of Laurent polynomials $p(q) \in \bC[p,p^{-1}]$ satisfying $p(q) = p(q^{-1})$.
By multiplying two sides of \eqref{eq-TGV&GV} by $(1-q^{-2})$,
one obtains (see \cite[\S 5.5]{gz}):
\begin{equation*}
\sum_{g=0}^{g(d)} N_d^g(K_{\bP^2})\cdot (q^g - q^{-g-2})
= (1-q^{-2}) \sum_{g=0}^{g(d)} n_d^g(K_{\bP^2}) \cdot(-1)^g (q^\half -q^{-\half})^{2g}.
\end{equation*}
Then:
\begin{equation*}
\begin{split}
\sum_{g= g^k(d) +1}^{g(d)} N_d^g(K_{\bP^2})\cdot  q^g
=& \sum_{g=0}^{g(d)} N_d^g(K_{\bP^2})\cdot (q^g - q^{-g-2}) \Big|_{q^{\geq g^k(d) +1}}\\
=& (1-q^{-2}) \sum_{g=0}^{g(d)} n_d^g(K_{\bP^2}) \cdot(-1)^g (q^\half -q^{-\half})^{2g}
\Big|_{q^{\geq g^k(d)+1 } }.
\end{split}
\end{equation*}
and then by Lemma \ref{lem-gz-seriesmultres} one obtains:
\begin{equation*}
\begin{split}
\sum_{g= g^k(d) +1}^{g(d)} N_d^g(K_{\bP^2})\cdot  q^g
=& (1-q^{-2}) \sum_{g=g^k(d) +1}^{g(d)} n_d^g(K_{\bP^2}) \cdot(-1)^g (q^\half -q^{-\half})^{2g}
\Big|_{q^{\geq g^k(d)+1 } }.
\end{split}
\end{equation*}
And then by Lemma \ref{lem-gz-seriesmultres} and \eqref{eq-gknIqgeq-2g},
one has:
\begin{Proposition}
\label{prop-GVleadingderiveI}
Fix $k\in \{1,2,3,4\}$,
then for $\forall d\geq 4k-3$ one has:
\be
\label{eq-prop-GVleadingderive}
\begin{split}
& \sum_{g= g^k(d) +1}^{g(d)} N_d^g(K_{\bP^2})q^g
= \Big( -q [1]_{q^{-1}}^2 [2]_{q^{-1}}\sum_{j=1}^k \sum_{d_1+\cdots +d_j =d} \frac{(-1)^{j-1}}{j} I(d_1)\cdots I(d_j) \Big)
\Big|_{q^{\geq g^k (d) +1 }}.
\end{split}
\ee
Here we use the notation $[j]_{q^{-1}} = (1-q^{-j})$.
\end{Proposition}

Here one can further reduce the number of terms in the right-hand side.
In fact,
consider $k=2$ and $d_1+d_2=d$,
and one has:
\begin{equation*}
\deg I(d) -\deg I(d_1)I(d_2)
=\half d^2 - \half d_1^2 -\half d_2^2
= d_1d_2.
\end{equation*}
If $d_1\geq 4$ and $d_2\geq 4$,
one has:
\begin{equation*}
\deg I(d) -\deg I(d_1)I(d_2)
=\half d^2 - \half d_1^2 -\half d_2^2
\geq 4d-16.
\end{equation*}
When $d\geq 16$ one always has:
\begin{equation*}
	\half d(d-3) - (4d-16) < \half d(d-3) -3d+1,
\end{equation*}
and thus for $d\geq 16$ we only need to consider the cases with $d_i < 4$ for at least one $i$
when computing the leading terms in the summation
\begin{equation*}
\sum_{d_1+d_2 =d} \frac{(-1)}{2} I(d_1) I(d_2).
\end{equation*}

And similarly for $k=3,4$,
it is not hard to see that
when restricting to the leading terms we want,
one only need to consider the cases with $d_i +d_j < 4$ for two distinct indices $i,j$ in the summation
\begin{equation*}
\sum_{d_1 +d_2 +d_3 =d} \frac{(-1)^2}{3} I(d_1) I(d_2) I(d_3)
\end{equation*}
for $d\geq 13$;
and we only need to consider the cases with $d_i +d_j +d_k< 4$ for three distinct indices $i,j,k$ in the summation
\begin{equation*}
\sum_{d_1 +d_2 +d_3 +d_4 =d} \frac{(-1)^3}{4} I(d_1) I(d_2) I(d_3) I(d_4)
\end{equation*}
for $d\geq 11$.
Therefore we have:
\begin{Corollary}
For $d\geq 16$,
the leading terms of the transformed Gopakumar--Vafa invariants of local $\bP^2$ are given by:
\be
\label{eq-corGVinI}
\begin{split}
\sum_{g= g(d)-3d+1}^{g(d)} N_d^g(K_{\bP^2})\cdot  q^g =& \bigg(
-q [1]_{q^{-1}}^2 [2]_{q^{-1}} \Big( I(d) -\sum_{i=1}^3 I(i)I(d-i)
+I(1)^2I(d-2) \\
&  + 2I(1)I(2) I(d-3) -I(1)^3 I(d-3) \Big) \bigg)
\bigg|_{q^{\geq \half d(d-3) -3d+2 }}.
\end{split}
\ee
\end{Corollary}
Notice that for $i= 1,2,3$ one always has $\deg I(i) =\half i(i-3) \leq 0$,
and thus in order to compute the restriction $q^{\geq g(d) -3d+1}$,
one can simply replace the terms $I(d)$ and $I(d-i)$ in the right-hand side
by $I(d)|_{q^{\geq \half d(d-3) -3d+1}}$ and $I(d-i)|_{q^{\geq \half d(d-3) -3d+1}}$
respectively by Proposition \ref{prop-degLaurentprod}.

\subsection{Explicit calculations of the leading terms}
\label{eq-compGVleading-expl}

Now we derive explicit formulas for the leading terms of the transformed Gopakumar--Vafa invariants.

First we rewrite \eqref{eq-corGVinI} as follows:
\be
\label{eq-LeaddingGVsum-1}
\begin{split}
\sum_{g= g(d)-3d+1}^{g(d)} N_d^g(K_{\bP^2}) q^g =& \bigg(
-q [1]_{q^{-1}}^2 [2]_{q^{-1}} \Big( I(d) - I(1)I(d-1)
+\big(I(1)^2 -I(2)\big)I(d-2) \\
&  - \big( I(3) - 2I(1)I(2) + I(1)^3 \big) I(d-3) \Big)
\bigg)\bigg|_{q^{\geq \half d(d-3) -3d+2 }}.
\end{split}
\ee
Then we analyze the leading terms of each summand in the right-hand side
using the results in \S \ref{sec-leadingId}.
The leading terms of $I(d)$ has been given in Theorem \ref{thm-main1-leadI-2}.

As for $I(1)I(d-1)$,
by \eqref{eq-degId} one has $\deg I(1) =-1$,
and Theorem \ref{thm-main1-leadI-2} gives:
\begin{equation*}
\begin{split}
I(d-1)
= \frac{ (-1)^{d-1} \cdot q^{ \half (d-1)(d-4)} }{ ([\infty]_{q^{-1}}!)^3}
\Big(  &J_0(d-1) - q^{-d} I(1) J_1(d-1) \\
&+q^{-2d+ 1} I(2) J_2(d-1)
\Big)
\Big|_{q^{\geq \half (d^2-3d) -4d+6}},
\end{split}
\end{equation*}
and thus for $d\geq 4$ one has:
\be
\begin{split}
I(1) I(d-1)
= \frac{ (-1)^{d-1} \cdot q^{ \half (d-1)(d-4)} }{ ([\infty]_{q^{-1}}!)^3}
\Big(  & I(1)J_0(d-1) - q^{-d} I(1)^2 J_1(d-1) \\
&+q^{-2d+ 1} I(1)I(2) J_2(d-1)
\Big)
\Big|_{q^{\geq \half (d^2-3d) -4d+6}}.
\end{split}
\ee
Here we have also used Proposition \ref{prop-degLaurentprod}.
Similarly one has:
\begin{equation*}
\deg \big( I(1)^2 -I(2) \big) =-1,
\end{equation*}
and by Theorem \ref{thm-main1-leadI-2} one has:
\begin{equation*}
\begin{split}
I(d-2)
= \frac{ (-1)^d \cdot q^{ \half (d-2)(d-5)} }{ ([\infty]_{q^{-1}}!)^3}
&\Big(  J_0(d-2) - q^{-d+1} I(1) J_1(d-2) \\
& +q^{-2d+3} I(2) J_2(d-2)
\Big)
\Big|_{q^{\geq \half (d^2-3d) -5d+12}}.
\end{split}
\end{equation*}
Notice that
\begin{equation*}
\deg \big( q^{ \half (d-2)(d-5)} \cdot q^{-2d+3} I(2) J_2(d-2) \big)
 < \half d(d-3)-3d+1,
\qquad \forall d\geq 7,
\end{equation*}
and then for $d\geq 7$ one has:
\be
\begin{split}
&\big( I(1)^2 -I(2) \big) \cdot I(d-2) \\
=& \bigg( \frac{ (-1)^{d}  q^{ \half (d-2)(d-5)} }{ ([\infty]_{q^{-1}}!)^3}
\big( I(1)^2 -I(2) \big) \Big( J_0(d-2) - q^{-d+1} I(1) J_1(d-2)
\Big) \bigg) \bigg|_{q^{\geq \half d(d-3) -3d+1}}.
\end{split}
\ee
Finally,
one has:
\begin{equation*}
\deg\big( I(3) - 2I(1)I(2) + I(1)^3 \big) =0,
\end{equation*}
and
\begin{equation*}
\deg I(d-3) = \half (d-3)(d-6) = \half d(d-3) -3d+9.
\end{equation*}
Moreover,
by Theorem \ref{thm-main1-leadI-2} we know that for $d\geq 7$ one has:
\begin{equation*}
\begin{split}
I(d-3)
= \frac{ (-1)^{d-3} \cdot q^{ \half (d-3)(d-6)} }{ ([\infty]_{q^{-1}}!)^3}
& \Big(  J_0(d-3) - q^{-d+2} I(1) J_1(d-3) \\
& +q^{-2d+5} I(2) J_2(d-3)
\Big)
\Big|_{q^{\geq \half (d^2-3d) -6d+19}}.
\end{split}
\end{equation*}
And for $d\geq 10$ one always has:
\be
\begin{split}
&\big( I(3) - 2I(1)I(2) + I(1)^3 \big) \cdot I(d-3) \\
=& \Big( \frac{ (-1)^{d-3} \cdot q^{ \half (d-3)(d-6)} }{ ([\infty]_{q^{-1}}!)^3}
\big( I(3) - 2I(1)I(2) + I(1)^3 \big) \cdot J_0(d-3)
\Big) \Big|_{q^{\geq \half d(d-3) -3d+1}}
\end{split}
\ee
since
\begin{equation*}
\deg \big( q^{ \half (d-3)(d-6)} \cdot q^{-d+2} I(1) J_1(d-3) \big)
 < \half d(d-3)-3d+1,
\qquad \forall d\geq 10.
\end{equation*}
%Here using the results of Appendix \ref{sec-app-I123} one has:
%\be
%\begin{split}
%I(3) - 2I(1)I(2) + I(1)^3= & -\frac{1}{([3]_{q^{-1}}!)^2}
%\big( 10 + 7q^{-1} -28q^{-2} -60q^{-3} -26q^{-4} +98q^{-5} \\
%& +160 q^{-6} +98q^{-7} -26q^{-8} -60q^{-9} -28q^{-10} +7q^{-11} +10q^{-12}
%\big).
%\end{split}
%\ee

Now plug these into \eqref{eq-LeaddingGVsum-1},
and we obtain:
\be
\label{eq-LeaddingGVsum-2}
\begin{split}
& \sum_{g \geq \half d(d-3)-3d+2} N_d^g(K_{\bP^2})q^g \\
=& \frac{(-1)^{d-1} q^{\half d(d-3) +1} \cdot [1]_{q^{-1}}^2 [2]_{q^{-1}}}{([\infty]_{q^{-1}}!)^3}
\bigg( J_0(d)
+ q^{-d+2} \Big( -q^{-3}I(1)J_1(d) + I(1)J_0(d-1) \Big) \\
& + q^{-2d+5} \Big( q^{-6}I(2)J_2(d) -q^{-3} I(1)^2 J_1(d-1) +\big(I(1)^2-I(2)\big) J_0(d-2)\Big) \\
& + q^{-3d+9} \Big( q^{-6}I(1)I(2)J_2(d-1) -q^{-3} I(1)\big( I(1)^2 -I(2)\big)J_1(d-2) \\
& + \big(I(3)-2I(1)I(2) +I(1)^3 \big) J_0(d-3) \Big)
\bigg)\bigg|_{q^{\geq \half d(d-3) -3d +2}}.
\end{split}
\ee
The explicit formulas for $I(1)$, $I(2)$, $I(3)$ have been computed in \S \ref{sec-app-I123}.
Now we plug these expressions together with \eqref{eq-defa0a1a2} into \eqref{eq-LeaddingGVsum-2}
(and delete some terms whose degree is obviously lower than $\half d(d-3) -3d+2$),
and in this way we finally obtain:
\begin{Theorem}
\label{thm-mainTGV-leading}
For $d\geq 16$,
the leading terms of the transformed Gopakumar--Vafa invariants are given by:
\be
\label{eq-LeaddingGVsum-thm}
\begin{split}
&\sum_{g \geq \half d(d-3)-3d+2} N_d^g(K_{\bP^2})q^g  \\
= &\frac{(-1)^{d-1} q^{\half d(d-3) +1} }{([\infty]_{q^{-1}}!)^3}
 \Big( \tJ_0  + q^{-d+1} \tJ_1  + q^{-2d+4} \tJ_2 + q^{-3d+9} \tJ_3
\Big)\Big|_{q^{\geq \half d(d-3) -3d +2}},
\end{split}
\ee
where:
\begin{equation*}
\begin{split}
\tJ_0 =& [1]_{q^{-1}}^2 [2]_{q^{-1}} \Big(
\binom{d+2}{2} - 3\sum_{i\geq 1}\frac{q^{-i}}{(1-q^{-i})^2} \Big),\\
\tJ_1 =& -3 \cdot  [2]_{q^{-1}} [3]_{q^{-1}}
\Big( \binom{d+1}{2} -3\sum_{i\geq 1} \frac{q^{-i}}{(1-q^{-i})^2} \Big)
+ 9q^{-3} \cdot [2]_{q^{-1}}, \\
\tJ_2 =& 3 \big(-2 +4q^{-2}+8q^{-3}+7q^{-4}+5q^{-5}+3q^{-6}+2q^{-7} \big) \cdot [1]_{q^{-1}}
\Big( \binom{d}{2} -3\sum_{i\geq 1} \frac{q^{-i}}{(1-q^{-i})^2} \Big)  \\
& - \frac{9q^{-4} (3+9 q^{-1}+8q^{-2}+6q^{-3}+6q^{-4}+4q^{-5} )}{1+q^{-1}},
\end{split}
\end{equation*}
and
\begin{equation*}
\begin{split}
\tJ_3 =& -\frac{ 10+7q^{-1}-28q^{-2} -60q^{-3} -26q^{-4} +116q^{-5} +178q^{-6} +62q^{-7} -152q^{-8}}
{[2]_{q^{-1}} [3]_{q^{-1}}^2} \\
& \cdot \Big( \binom{d-1}{2} -3\sum_{i\geq 1} \frac{q^{-i}}{(1-q^{-i})^2} \Big)
-\frac{54q^{-5} (1-q^{-1}-3q^{-2}+q^{-3})}{[1]_{q^{-1}}^2 [2]_{q^{-1}}}.
\end{split}
\end{equation*}
\end{Theorem}

\section{Leading Betti numbers of moduli space $M_d$}
\label{sec-betti}
In this section,
we prove Theorem \ref{thm-intro-GWMConj}.
We will first review several necessary backgrounds,
and then use the formula for leading Gopakumar--Vafa invariants of $K_{\mathbb{P}^2}$ derived in last section
to prove the formula for leading Betti numbers of semistable one-dimensional sheaves on $\mathbb{P}^2$.

\subsection{Relative and local Gromov--Witten invariants}
First we briefly review the concepts of the maximal contact Gromov--Witten invariants of the pair $(\mathbb{P}^2, E)$
and the local invariants of the elliptic curve $E$.
They can be used to compute the Betti numbers of the moduli space $M_d$, see \cite{bou,bfgw,gw,ms} for details.

Fix a smooth cubic curve $E\subset \mathbb{P}^2$,
and let $\overline{\mathcal{M}}_g(\mathbb{P}^2/E,d)$ be the moduli stack of connected relative stable maps of genus $g$ and degree $d$ to the pair $(\mathbb{P}^2, E)$,
with a single relative marking of contact order $3d$ along $E$.
The virtual dimension of $\overline{\mathcal{M}}_g(\mathbb{P}^2/E,d)$ is $g$.
Then the corresponding maximal contact Gromov--Witten invariants of the pair $(\mathbb{P}^2, E)$ are defined by
\begin{align*}
	N^{\mathbb{P}^2/E}_{g,d}
	=(-1)^g  \int_{[\overline{\mathcal{M}}_g(\mathbb{P}^2/E,d)]^{\text{vir}}} \lambda_g
	\in\mathbb{Q},
\end{align*}
where $\lambda_g$ is the top Chern class of the Hodge bundle.

For  the local invariants of the elliptic curve,
consider the rank two vector bundle
\begin{align*}
	N=N_{E/\mathbb{P}^2}\oplus N^{\vee}_{E/\mathbb{P}^2}
	=\mathcal{O}_{E}(3)\oplus\mathcal{O}_E(-3)
\end{align*}
over $E$.
There is an anti-diagonal $\mathbb{C}^*$-action on $N$ of weights $1$ and $-1$,
and we set $t$ to be the equivariant parameter.
For a tuple of positive integers ${\bm d}=(d_1\dots,d_n)$,
the local invariants of the elliptic curve $E$ are defined by
\begin{align*}
	N^{E,tw}_{g,d_E}({\bm d})
	=\int_{[\overline{\mathcal{M}}_{g,n}(E,d_E)]^{vir}} \bigg(\prod_{i=1}^n\frac{t\ \text{ev}_i^* \omega}{t- d_i \psi_i}\bigg)
	e_{\mathbb{C}^*}(-R^\bullet\pi_*f^* N),
\end{align*}
where $\overline{\mathcal{M}}_{g,n}(E,d_E)$ is the moduli stack of connected genus $g$, $n$-marking stable maps of degree $d_E$ to $E$,
$\psi_i$ is the first Chern class of the $i$-th cotangent line bundle,
$\omega$ is the point class of $E$,
and $\pi$ and $f$ are the universal curve and universal map
respectively.
Note that here the integral is homogeneous of degree zero in the equivariant parameter $t$,
so the local invariant $N^{E,tw}_{g,d_E}({\bm d})$ is independent of $t$.
Let ${\bm \beta}=(\beta_1,\dots,\beta_n)$ record the positive plane degrees,
then the generating function $F^{E,tw}_{d_E,{\bm \beta}}(\hbar)$ of local invariants of $E$ is defined by
\begin{align}\label{eqn:def F^E}
	F^{E,tw}_{d_E,{\bm \beta}}(\hbar)
	=\sum_{g=0}^\infty
	(-1)^n\prod_{i=1}^n (3\beta_i)^2
	\cdot N^{E,tw}_{g,d_E}(3\beta_1,\dots,3\beta_n) \cdot \hbar^{2g-2+n}.
\end{align}

On the one hand,
the Betti numbers of the moduli space $M_d$ are related to the maximal contact Gromov--Witten invariants of $(\mathbb{P}^2,E)$ by the refined sheaves/Gromov--Witten correspondence as (see \cite{bou, ms}, and see also \cite{gw})
\begin{align}\label{eqn:betti and relative}
	\frac{(-1)^{d-1}}{3d}
	\sum_{g=0}^\infty N^{\mathbb{P}^2/E}_{g,d}
	\hbar^{2g-1}
	= \sqrt{-1} \cdot \sum_{k|d, k>0}	\frac{(-1)^{d^2/k^2+1}}{k^2}
	\frac{\Omega_{d/k}(y^{k/2})}{y^{k/2}-y^{-k/2}}
\end{align}
where $y=e^{i\hbar}$.
On the other hand,
the maximal contact Gromov--Witten invariants of $(\mathbb{P}^2,E)$ can be related to the Gromov--Witten invariants of $K_{\mathbb{P}^2}$ and the local invariants of $E$ by the all genus local/relative correspondence as (see \cite{bfgw})
\begin{align}\label{eqn:relative GV E}
	K^g_d
	=\frac{(-1)^{d-1}}{3d}N^{\mathbb{P}^2/E}_{g,d}
	+\sum_{n=0}^\infty
	\sum_{g=g_0+\cdots g_n \atop  {d=3d_E+\beta_1+\dots+\beta_n \atop d_E\geq0, \beta_j>0} }
	\frac{(-1)^{|{\bm \beta}|} N^{E,\text{tw}}_{g_0,d_E}(3\beta_1,\dots,3\beta_n)}
	{|\text{Aut}((\beta_1,g_1),\dots,(\beta_n,g_n))|}
	\prod_{j=1}^n (3\beta_j)N^{\mathbb{P}^2/E}_{g_j,\beta_j},
\end{align}
where $|{\bm \beta}|=\sum_{i=1}^n \beta_i$ and
$|\text{Aut}((\beta_1,g_1),\dots,(\beta_n,g_n))|=\{\sigma\in S_n|(\beta_{\sigma(j)},g_{\sigma(j)})=(\beta_j,g_j) \text{\ for\ all\ }j\}$.
The above two relations \eqref{eqn:betti and relative} and \eqref{eqn:relative GV E} can be used to study the properties of the Betti numbers of moduli space $M_d$.

In the rest of this section,
we will use the following notation to package the Gopakumar--Vafa invariants of $K_{\mathbb{P}^2}$.
Consider the generating function
\begin{align*}
	\mathcal{F}_d(y)
	=\sum_{g\geq0} n_{g,d}(-1)^g(y^{1/2}-y^{-1/2})^{2g},
	\qquad A_d(y)
	=(-1)^{d-1} y^{g(d)} \mathcal{F}_d(y).
\end{align*}
Note that we have $\mathcal{F}_d(y)=\mathcal{F}_d(y^{-1})$ and
\begin{align*}
	\mathcal{F}_d(y)\in y^{-g(d)}  \mathbb{Z}[y] ,
\end{align*}
and then $A_d(y)\in\mathbb{Z}[y]$ is a polynomial of degree at most $2g(d)$.
The connected local Gromov--Witten series $\mathcal{W}_\beta(y)$ is defined to be
\begin{align}\label{eqn:W as F}
	\begin{split}
		\mathcal{W}_\beta(y)
		=\sum_{g=0}^\infty K^g_{\beta} \hbar^{2g-2}
		=-\sum_{\alpha|\beta, \alpha>0} \frac{\mathcal{F}_{\beta/\alpha}(y^\alpha)}{\alpha(y^{\alpha/2}-y^{-\alpha/2})^2},
	\end{split}
\end{align}
where $y=e^{i\hbar}$.
Since $\mathcal{F}_d(y)$ and $\mathcal{W}_d(y)$ are invariant under the substituting $y\leftrightarrow y^{-1}$,
the results derived in previous sections can be used here by substituting $q^{-1}=y$.

\subsection{Guo--Wu's explicit formula for the Betti numbers of $M_d$}
In \cite{gw},
Guo and Wu made a important progress on giving explicit formulas for the Betti numbers of moduli space $M_d$.
They use the local/relative correspondence to establish a formula which connects the Betti numbers, Gopakumar--Vafa invariants, relative Gromov--Witten invariants and local invariants of $E$ all together.
In the case of local $\mathbb{P}^2$,
equations (10) and (11) in \cite{gw} give
\begin{align}\label{eqn:10 in GW}
	\begin{split}
	&\sum_{k|d}
	\frac{1}{k(y^{k/2}-y^{-k/2})^2}
	\bigg(\frac{\frac{3d}{k}(-1)^{\frac{d^2}{k^2}}\Omega_{d/k}(y^{k/2})}{\frac{y^{3d/2}-y^{-3d/2}}{y^{k/2}-y^{-k/2}}} +  \mathcal{F}_{d/k}(y^k)\bigg)\\
	=&\sum_{s\geq0} \sum_{d_E, \beta_1,\dots,\beta_s>0 \atop 3d_E+|{\bm \beta}|=d}
	\frac{1}{|Aut({\bm \beta})|} G^E_{d_E,{\bm \beta}} \prod_{i=1}^s \mathcal{W}^{K_{\mathbb{P}^2}}_{\beta_i},
	\end{split}
\end{align}
where $G^E_{d_E,{\bm \beta}}$ is a Laurent polynomial given by a sum of rooted trees using the generating function $F^{E,tw}_{d_E,{\bm \beta}}(\hbar)$ of the local invariants of $E$ defined in equation \eqref{eqn:def F^E}.
In this paper,
the following two facts are essentially required to compute the Betti numbers of the moduli space $M_d$.
The first one is the following closed formula for $G^E_{1,{\bm \beta}}(y)$ (see equation (13) in \cite{gw}):
\begin{align}\label{eqn:G d=1 formula}
	G^E_{1,{\bm \beta}}(y)
	=\prod_{i=1}^s (y^{-3\beta_i/2}-y^{3\beta_i/2})^2.
\end{align}
The second one is the following degree bound for the $d_E\geq2$ case (see Lemma 4.4 in \cite{gw}):
\begin{align}
	G^E_{d_E,{\bm \beta}}(y)
	\in y^{-9g(d_E)-3(d_E-1)|{\bm \beta}|} \mathbb{Q}[y].
\end{align}

\subsection{Several degree estimates}
In this subsection,
we study the degree bounds of terms appearing in Guo--Wu's formula \eqref{eqn:10 in GW}.
For convenience,
we introduce the following two notations:
\begin{align}\label{eqn:Def BL}
	B_{L;k,d}(y) =
	y^{g(d)}\frac{(y^{1/2}-y^{-1/2})^2}{k(y^{k/2}-y^{-k/2})^2}
	\Bigg(\frac{\frac{3d}{k}(-1)^{\frac{d^2}{k^2}+1}\Omega_{d/k}(y^{k/2})}{\frac{y^{3d/2}-y^{-3d/2}}{y^{k/2}-y^{-k/2}}} -  \mathcal{F}_{d/k}(y^k)\Bigg)
\end{align}
and
\begin{align}\label{eqn:Def BR}
	B_{R;d_E,{\bm \beta},{\bm \alpha}}(y)
	=(-1)^{s+1} y^{g(d)} (y^{1/2}-y^{-1/2})^2
	\frac{G^E_{d_E,{\bm \beta}}(y)}{|Aut({\bm \beta})|} \prod_{i=1}^s \frac{\mathcal{F}_{\beta_i/\alpha_i}(y^{\alpha_i})}{\alpha_i(y^{\alpha_i/2}-y^{-\alpha_i/2})^2},
\end{align}
where ${\bm \beta}=(\beta_1,\dots,\beta_s)$ and ${\bm \alpha}=(\alpha_1,\dots,\alpha_s)$ satisfying $\alpha_i|\beta_i$ for $i=1,\dots,s$.
Here the subscripts $L$ and $R$ stand for left and right respectively.
Then by using equation \eqref{eqn:W as F},
Guo--Wu's formula \eqref{eqn:10 in GW} can be rewritten as:
\begin{align}\label{eqn:main GW}
	\sum_{k|d}
	B_{L;k,d}(y)
	=\sum_{s\geq0} \sum_{d_E, \beta_i,\alpha_i>0, \atop {\alpha_i|\beta_i, 1\leq i\leq s \atop 3d_E+|{\bm \beta}|=d}}
	B_{R;d_E,{\bm \beta},{\bm \alpha}}(y).
\end{align}
In the following,
we expand all functions at $y=0$.
We have the following degree bounds of functions appearing in the above formula.
\begin{Lemma}\label{lem:degree bound}
	For given $k,d$ such that $k>1$ and $k|d$, we have
	\begin{align}\label{eqn:LHS k>1 deg}
		B_{L;k,d}(y)
		\in y^\frac{d^2(k-1)}{2k} \mathbb{Q}[y].
	\end{align}
	For given $s$ and $\beta_1,\alpha_1,\dots,\beta_s,\alpha_s>0$ such that $\alpha_i|\beta_i$ and $3d_E+|{\bm \beta}|=d$,
	we denote
	$$\delta({\bm \beta}, {\bm \alpha}) =\frac{1}{2}\sum_{i=1}^s \frac{\beta_i^2(\alpha_i-1)}{\alpha_i}.$$
	Then for $d_E=1$, we have
	\begin{align}\label{eqn:RHS dE=1 deg}
		B_{R;1,{\bm \beta},{\bm \alpha}}(y)
		\in y^{\sum_{i<j}\beta_i\beta_j+\delta({\bm \beta}, {\bm \alpha})}
		\mathbb{Q}[y],
	\end{align}
	whereas for $d_E\geq2$,
	\begin{align}\label{eqn:RHS dE>1 deg}
		B_{R;d_E,{\bm \beta},{\bm \alpha}}(y)
		\in y^{3d-9+\sum_{i<j}\beta_i\beta_j+\delta({\bm \beta}, {\bm \alpha})}	\mathbb{Q}[y],
	\end{align}
	where $d=3d_E+\sum_{i=1}^s\beta_i$.
\end{Lemma}
\begin{proof}
	Equation \eqref{eqn:LHS k>1 deg} has essentially been proved by Guo--Wu \cite{gw}.
	In their proof of Proposition 5.1,
	they proved that when $k>1$,
	\begin{align}
		\frac{(y^{1/2}-y^{-1/2})^2}{(y^{k/2}-y^{-k/2})^2}
		\Bigg(\frac{\frac{3d}{k}(-1)^{\frac{d^2}{k^2}+1}\Omega_{d/k}(y^{k/2})}{\frac{y^{3d/2}-y^{-3d/2}}{y^{k/2}-y^{-k/2}}} -  \mathcal{F}_{d/k}(y^k)\Bigg)
		\in y^{-\frac{d^2}{2k}+\frac{3d}{2}-1}\mathbb{Q}[y].
	\end{align}
	Notice that $-\frac{d^2}{2k}+\frac{3d}{2}-1+g(d)=\frac{d^2(k-1)}{2k}$.
	
	Next,
	we prove equation \eqref{eqn:RHS dE=1 deg}.
	We use the explicit formula \eqref{eqn:G d=1 formula} for $G^E_{1,{\bm \beta}}(y)$ to obtain
	\begin{align}\label{eqn:BRd=1 formula}
		\begin{split}
		B_{R;1,{\bm \beta},{\bm \alpha}}(y)
		=&\frac{(-1)^{1+\sum_i(\beta_i/\alpha_i)} y^{g(d)-1-3(d-3)+|{\bm \alpha}|-\sum_{i=1}^s\alpha_i g(\beta_i/\alpha_i)} (1-y)^2}{|\text{Aut}({\bm \beta})|}\\
		&\cdot\prod_{i=1}^s \frac{(1-y^{3\beta_i})^2}{\alpha_i(1-y^{\alpha_i})^2} A_{\beta_i/\alpha_i}(y^{\alpha_i})
		\end{split}
	\end{align}
	Consequently,
	the lowest degree term in $B_{R;1,(\beta_1,\dots,\beta_s),(\alpha_1,\dots,\alpha_s)}(y)$ is exactly
	\begin{align*}
		\frac{(-1)^{1+\sum_i(\beta_i/\alpha_i)} y^{g(d)-1-3(d-3)+|{\bm \alpha}|-\sum_{i=1}^s\alpha_i g(\beta_i/\alpha_i)}}{|\text{Aut}({\bm \beta})|} \prod_{i=1}^s \frac{A_{\beta_i/\alpha_i}(0)}{\alpha_i},
	\end{align*}
	which is of degree
	\begin{align*}
		g(d)-1-3(d-3)+|{\bm \alpha}|-\sum_{i=1}^s\alpha_i g(\beta_i/\alpha_i)
		=&\frac{1}{2}(d-3)^2-\frac{1}{2}\sum_{i=1}^s\frac{\beta_i^2}{\alpha_i}
		=\sum_{i<j}\beta_i\beta_j
		+\delta({\bm \beta}, {\bm \alpha}).
	\end{align*}
	
	To prove the third equation \eqref{eqn:RHS dE>1 deg},
	we first recall that,
	in Lemma 4.4 of \cite{gw},
	Guo and Wu have proved the following:
	\begin{align}
		G^E_{d_E,{\bm \beta}}
		\in y^{-9g(d_E)-3(d_E-1)\sum_{i=1}^s\beta_i} \mathbb{Q}[y]
	\end{align}
	for $d_E>1$.
	Thus from the definition \eqref{eqn:Def BR} of $B_{R;d_E,{\bm \beta},{\bm \alpha}}(y)$,
	its lowest $y$-degree is at least
	\begin{align*}
		g(d)-1-9g(d_E)-3(d_E-1)\sum_{i=1}^s\beta_i
		+\sum_{i=1}^s\big(\alpha_i-\alpha_i g(\beta_i/\alpha_i)\big)
		=3d-9+\sum_{i<j}\beta_i\beta_j
		+\delta({\bm \beta}, {\bm \alpha}).
	\end{align*}
	This completes the proof of equation \eqref{eqn:RHS dE>1 deg}.
\end{proof}

\subsection{The formula for $\widehat \Omega_d(y)$ up to degree $3d-10$}
In this subsection,
we finish the proof of Conjecture \ref{thm:main leading betti 3d-10},
which provides a closed formula for $\widehat{\Omega}_d(y)$ up to degree $3d-10$.

Based on the degree bounds in Lemma \ref{lem:degree bound},
we first show that only one term in the left-hand side and eleven terms in the right-hand side of equation \eqref{eqn:main GW}  make non-trivial contributes when restricting to $y^{\leq3d-10}$.
More precisely,
we have:
\begin{Lemma}\label{lem:up to 3d-10}
	For $d\geq16$,
	we have
	\begin{align}\label{eqn:up to 3d-10}
		B_{L;1,d}(y)
		=\sum_{\big((\beta_i), (\alpha_i)\big)}
		B_{R;1,{\bm \beta},{\bm \alpha}}(y)
		|_{y^{\leq 3d-10}},
	\end{align}
	where the sum runs over the following eleven cases:
	\begin{align*}
		\big((\beta_i), (\alpha_i)\big)=&\big((d-3), (1)\big);\  \big((1,d-4), (1,1)\big); \ \big((2,d-5), (1,1)\big); \ \big((2,d-5), (2,1)\big); \\
		&\big((3,d-6), (1,1)\big);\ \big((3,d-6), (3,1)\big);\  \big((1,1,d-5), (1,1,1)\big); \ \big((1,2,d-6), (1,1,1)\big);\\ &\big((1,2,d-6), (1,2,1)\big); \ \big((1,1,1,d-6), (1,1,1,1)\big);\ \big((4,d-7),(1,1)\big).
	\end{align*}
	Moreover,
	the last case $\big((4,d-7),(1,1)\big)$ vanishes after truncation $y^{\leq 3d-10}$
	when $d\geq19$.
\end{Lemma}
\begin{proof}
	We start from equation \eqref{eqn:main GW} and analysis the degree of each term.
	We first focus on the terms in the left-hand side of equation \eqref{eqn:main GW}.
	For $k>1$,
	from equation \eqref{eqn:LHS k>1 deg} and
	\begin{align*}
		\frac{d^2(k-1)}{2k}\geq\frac{d^2}{4}>3d-10,
	\end{align*}
	we have
	\begin{align}
		B_{L;k,d}(y)
		\in y^{3d-9} \mathbb{Q}[y]
	\end{align}
	for $k>1$.
	Thus,
	only the $k=1$ case can contribute.
	
	Next, we focus on the terms in the right-hand side of equation \eqref{eqn:main GW}.
	For $d_E\geq2$,
	from equation \eqref{eqn:RHS dE>1 deg} and $3d-9>3d-10$,
	we have
	\begin{align*}
		B_{R;d_E,{\bm \beta},{\bm \alpha}}(y)
		\in y^{3d-9}\mathbb{Q}[y]
	\end{align*}
	when $d_E\geq2$.
	For $d_E=1$,
	from equation \eqref{eqn:RHS dE=1 deg},
	the lowest degree of $B_{R;1,{\bm \beta},{\bm \alpha}}(y)$  is $\sum_{i<j}\beta_i\beta_j+\delta({\bm \beta}, {\bm \alpha})$.
	We have
	\begin{align*}
		\sum_{i<j}\beta_i\beta_j
		=\frac{(d-3)^2-\sum_i \beta_i^2}{2}
		\geq\frac{(d-3)^2-(d-3)\max_i\{\beta_i\}}{2}.
	\end{align*}
	If $\max_i\{\beta_i\}<(d-3)/2$,
	then $\sum_{i<j}\beta_i\beta_j>(d-3)^2/4>3d-10$ since $d\geq16$.
	This kind of term cannot contribute.
	If there is one $\beta_{i_0}\geq (d-3)/2$,
	then we set $\beta_{i_0}=d-3-a$ and
	\begin{align*}
		\sum_{i<j}\beta_i\beta_j
		\geq a\big(d-3-a\big)
		=ad-3a-a^2,
	\end{align*}
	which is greater than $3d-10$ unless $a=0,1,2,3,4$.
	Moreover,
	in these cases if $\alpha_{i_0}\neq1$,
	then
	\begin{align*}
		\delta({\bm \beta}, {\bm \alpha})=\frac{1}{2}\sum_{i=1}^s \frac{\beta_i^2(\alpha_i-1)}{\alpha_i}
		\geq \frac{\beta_{i_0}^2(\alpha_{i_0}-1)}{2\alpha_{i_0}}
		\geq \frac{(d-3-a)^2}{4}
	\end{align*}
	and the lowest degree is $\sum_{i<j}\beta_i\beta_j+\delta({\bm \beta}, {\bm \alpha})\geq a(d-3-a)+\frac{(d-3-a)^2}{4}> 3d-10$.
	This means only the term satisfying $\alpha_{i_0}=1$ can contribute.
	In particular,
	if $a=4$,
	then the inequality $ad-3a-a^2\leq 3d-10$ holds only when $d=16, 17, 18$.
	One can directly verify that only the term satisfying $\big(({\bm \beta},{\bm \alpha})\big)=\big((4,d-7),(1,1)\big)$ can contribute.
	
	As a consequence,
	after restricting to the $y^{\leq 3d-10}$ part in both sides of equation \eqref{eqn:main GW},
	the $({\bm \beta},{\bm \alpha})$ of non-trivial terms are precisely the cases listed in the statement of this lemma.
\end{proof}

\begin{Theorem}
	\label{eq-Confirmgwconj}
	For $d\geq4$, we have
	\begin{align}\label{eqn:omega  up to 3d-10}
		\widehat{\Omega}_d(y)
		=\frac{[1]_y^2}{[\infty]_y!^3}\bigg(1-3y^{d-1}\frac{[3]_y}{[1]_y^2}+
		3y^{2d-4}\frac{[3]_y(-2+2y+4y^2+2y^3+y^4+2y^5)}{[1]_y^2 [2]_y}\bigg)
		\Big|_{y^{\leq 3d-10}}.
	\end{align}
\end{Theorem}
\begin{proof}
	We start from the equation \eqref{eqn:up to 3d-10}.
	When $d\geq16$,
	from the definition equation \eqref{eqn:Def BL} of $B_{L;k,d}(y)$,
	the left-hand side of equation \eqref{eqn:up to 3d-10} is given by
	\begin{align*}
		B_{L;1,d}(y)
		=(-1)^{d^2+1}3d \hat{\Omega}_d(y)
		-(-1)^{d-1}A_d(y).
	\end{align*}
	To compute the right-hand side of equation \eqref{eqn:up to 3d-10},
	we need to compute all those eleven terms,
	which will be accomplished in Lemma \ref{lem:first one} to Lemma \ref{lem:last one a=4}.
	As a consequence,
	we obtain
	\begin{align}\label{eqn:omega as A}
		\begin{split}
		3d\widehat{\Omega}_d(y)
		=&A_d(y)-A_{d-3}(y)
		-3y^{d-4}(1+y+y^2)^2A_{d-4}(y)
		-y^{2d-10}P(y)A_{d-5}(y)\\
		&-y^{3d-18}Q(y)A_{d-6}(y)
		-y^{4d-28}\frac{(15+12y+48y^2)A_{d-7}(y)}{(1-y)^2}
		\ \Big|_{y^{\leq 3d-10}},
		\end{split}
	\end{align}
	where the functions $P(y)$ and $Q(y)$ are given by
	\begin{align*}
		P(y)=&3(1+y+y^2)^2 (2+y+4y^2+7y^3+4y^4+y^5+2y^6),\\
		Q(y)=&10+27y+72y^2+154y^3+306y^4+531y^5+848y^6+1197y^7+1512y^8,
	\end{align*}
	and in Lemma \ref{lem:last one a=4},
	only the $y^{\leq2}$ part of $A_4(y)$ can contribute since we assume $d\geq16$ here.
	Now,
	we can use the leading Gopakumar--Vafa formula in Theorem \ref{thm:main GV} to compute the $\widehat{\Omega}_d(y)|_{y^{\leq 3d-10}}$.
	Recall that from equation \eqref{eq-LeaddingGVsum-thm main},
	after setting $q^{-1}=y$,
	the $A_d(y)$ is given by
	\begin{align}\label{eqn:A_d formula}
		A_d(y)=&\frac{1}{[2]_y ([\infty]_{y}!)^3}
		\Big( \tilde{J}_0(y)  + y^{d-1} \tilde{J}_1(y)   + y^{2d-4} \tilde{J}_2(y)  + y^{3d-9} \tilde{J}_3(y)
		\Big)\Big|_{y^{\leq 3d-1}}.
	\end{align}
	Note that the above formula holds only for $d\geq16$.
	So in order to substitute it to compute the $\widehat{\Omega}_d(y)$ using equation \eqref{eqn:omega as A},
	we first need to restrict us to the region $d\geq19$ because of the appearance of the term $A_{d-3}(y)$.
	
	Now in the case $d\geq19$,
	we use equations \eqref{eqn:omega as A} and \eqref{eqn:A_d formula} to compute $\widehat{\Omega}_d(y)$.
	After factoring out the common denominator $\frac{1}{[\infty]_y!^3}$,
	we collect the terms according to their initial powers as $y^0, y^{d-4}, y^{2d-10}, y^{3d-18}$.
	We use the notation:
	\begin{align}
		L_n(y)=\binom{n+2}{2} -3\sum_{i=1}^\infty \frac{y^i}{(1-y^i)^2}.
	\end{align}
	As a consequence,
	the coefficient of $y^0$ comes from $A_d$ and $A_{d-3}$,
	which is equal to
	\begin{align*}
		(1-y)^2\big(L_d(y)-L_{d-3}(y)\big)
		=3d[1]_y^2.
	\end{align*}
	The coefficient of $y^{d-4}$ comes from $A_d, A_{d-3}$ and $A_{d-4}$,
	which is equal to
	\begin{align*}
		3(1-y^3)\big(-3y^3+L_{d-4}(y)-y^3L_{d-1}(y)\big)-3(1+y+y^2)^2(1-y)^2L_{d-4}(y)
		=-9dy^3[3]_y |_{y^{\leq 2d-6}}.
	\end{align*}
	The coefficient of $y^{2d-10}$ comes from $A_d, A_{d-3}, A_{d-4}$ and $A_{d-5}$,
	which is equal to
	\begin{align*}
		&\frac{3[3]_y(-2+2y+4y^2+2y^3+y^4+2y^5)}{[2]_y}
		\big(y^6L_{d-2}(y)-L_{d-5}(y)\big)\\
		&+9y^4\frac{[1]_y[6]_y(3+9y+8y^2+6y^3+6y^4+4y^5)}
		{[2]_y^2}
		+9y\frac{[3]_y^2}{[1]_y^2}\big([3]_yL_{d-5}(y)-3y^3\big)
		-[1]_y^2P(y)L_{d-5}(y)\\
		=&9dy^6\frac{[3]_y(-2+2y+4y^2+2y^3+y^4+2y^5)}{[2]_y} |_{y^{\leq d}}.
	\end{align*}
	The coefficient of $y^{3d-18}$ comes from $A_{d-3}, A_{d-4}, A_{d-5}$ and $A_{d-6}$,
	which is equal to
	\begin{align*}
		&\frac{10+7y-28y^2-60y^3-26y^4+116y^5+178y^6+62y^7-152y^8}
		{[2]_y^2[3]_y^2} L_{d-6}\\
		&+54y^5\frac{(1-y-3y^2+y^3)}{[1]_y^2[2]_y^2}
		-9y^2\frac{[3]_y^2(-2+4y^2+8y^3+7y^4+5y^5+3y^6+2y^7)}
		{[1]_y[2]_y} L_{d-6}\\
		&+27y^6\frac{[3]_y^2(3+9y+8y^2+6y^3+6y^4+4y^5)}
		{[1]_y[2]_y^2}
		+3y^2P(y)\big([3]_yL_{d-6}-3y^3\big)
		-[1]_y^2Q(y)L_{d-6}
		=0 |_{y^{\leq 8}}.
	\end{align*}
	The coefficient of $y^{4d-28}$ cannot contribute since $d\geq19$ as we remark in Lemma \ref{lem:up to 3d-10}.
	
	Next,
	we deal with the finitely many degrees $4\leq d\leq 18.$
	For the case of $d=17,18$,
	only the substitution for the term $A_{d-3}(y)$ and the coefficient of $y^{4d-28}$ require corrections.
	Actually,
	on the one hand,
	by direct computations	the corrections from $-A_{d-3}(y)$ are
	\begin{align}\label{eqn:correction A_{d-3}}
		990y^{40}+3717y^{41} |_{y^{\leq 3\times 17-10}},
		\quad \text{and}\quad 1170y^{44} |_{y^{\leq 3\times 18-10}}
	\end{align}
	for $d=17$ and $18$, respectively.
	On the other hand,
	the corrections from the coefficient of $y^{4d-28}$ can be directly computed by Lemma \ref{lem:last one a=4},
	which are
	\begin{align*}
		-y^{4d-28}\frac{(15+12y) A_{d-7}(y)}{(1-y)^2}|_{y^{\leq 3d-10}}
		=-990y^{40}-3717y^{41}|_{y^{\leq 3\times 17-10}},
		\quad \text{and}\quad -1170y^{44} |_{y^{\leq 3\times 18-10}}
	\end{align*}
	for $d=17$ and $18$, respectively.
	The above corrections cancel with the corrections \eqref{eqn:correction A_{d-3}} from $-A_{d-3}(y)$.
	Thus equation \eqref{eqn:omega  up to 3d-10} holds for $d=17$ and $18$.
	For the case of $4\leq d\leq16$,
	Guo and Wu provide a complete list of $\widehat{\Omega}_d(y)$ in their website, see \S 6 in \cite{gw}.
	Then the equation \eqref{eqn:omega  up to 3d-10} for these cases can be directly checked using their data.
\end{proof}

Below,
we compute the contributions from right-hand side of equation \eqref{eqn:up to 3d-10} explicitly,
which have been used in the computations of the previous theorem.
\begin{Lemma}\label{lem:first one}
	For $d\geq7$,
	we have
	\begin{align}\label{eqn:first one}
		\begin{split}
			&B_{R;1,(d-3),(1)}(y)
			+B_{R;1,(1,d-4),(1,1)}(y)\\
			=&(-1)^d\Big(A_{d-3}(y)
			+3y^{d-4}(1+y+y^2)^2A_{d-4}(y)\Big) \Big|_{y^{\leq 3d-10}}.
		\end{split}
	\end{align}
\end{Lemma}
\begin{proof}
	From the definition \eqref{eqn:Def BR} of $B_{R;d_E,(\beta_1,\dots,\beta_s),(\alpha_1,\dots,\alpha_s)}(y)$
	and the explicit formula \eqref{eqn:G d=1 formula} for $G^E_{1,{\bm \beta}}(y)$,
	we have
	\begin{align*}
		B_{R;1,(d-3),(1)}(y)
		=&(-1)^{d}(1-y^{3d-9})^2A_{d-3}(y),\\
		B_{R;1,(1,d-4),(1,1)}(y)=&
		(-1)^d 3y^{d-4} \frac{[3]_y^2}{[1]_y^2}(1-y^{3d-12})^2A_{d-4}(y),
	\end{align*}
	where in the second equality,
	we have used $A_1(y)=3$.
	Since we are only interested in the $y^{\leq 3d-10}$ part,
	both of the factors
	\begin{align*}
		(1-y^{3d-9})^2=1-2y^{3d-9}+y^{6d-18}
		\text{\ \ and\ \ } (1-y^{3d-12})^2=1-2y^{3d-12}+y^{6d-24}
	\end{align*}
	could be replaced by $1$ when $d\geq7$.
	Thus we obtain equation \eqref{eqn:first one}.
\end{proof}

\begin{Lemma}
	For $d\geq8$,
	we have
	\begin{align}\label{eqn:second one}
		\begin{split}
			&B_{R;1,(2,d-5),(1,1)}(y)
			+B_{R;1,(2,d-5),(2,1)}(y)
			+B_{R;1,(1,1,d-5),(1,1,1)}(y)\\
			=&(-1)^d y^{2d-10} P(y) A_{d-5}(y) |_{y^{\leq 3d-10}},
		\end{split}
	\end{align}
	where we recall that $P(y)=3(1+y+y^2)^2 (2+y+4y^2+7y^3+4y^4+y^5+2y^6)$.
\end{Lemma}
\begin{proof}
	The functions $B_{R;1,(2,d-5),(1,1)}(y), B_{R;1,(2,d-5),(2,1)}(y)$ and $B_{R;1,(1,1,d-5),(1,1,1)}(y)$ can be explicitly computed by
	\begin{align*}
		B_{R;1,(2,d-5),(1,1)}(y)
		=&(-1)^d 6y^{2d-10} \frac{[6]_y^2}{[1]_y^2}(1-y^{3d-15})^2A_{d-5}(y),\\
		B_{R;1,(2,d-5),(2,1)}(y)
		=&(-1)^{d+1} \frac{3}{2}y^{2d-9} \frac{[6]_y^2}{[2]_y^2}(1-y^{3d-15})^2A_{d-5}(y),\\
		B_{R;1,(1,1,d-5),(1,1,1)}(y)
		=&(-1)^d \frac{9}{2}y^{2d-9} \frac{[3]_y^4}{[1]_y^4}(1-y^{3d-15})^2A_{d-5}(y),
	\end{align*}
	where we have used $A_1(y)=3$ and $A_2(y)=6$.
	When restricting to the $y^{\leq 3d-10}$ part and assuming $d\geq8$,
	we have
	\begin{align*}
		B_{R;1,(2,d-5),(1,1)}(y)
		=&(-1)^d 6y^{2d-10} (1+y+y^2+y^3+y^4+y^5)^2 A_{d-5}(y) |_{y^{\leq 3d-10}},\\
		B_{R;1,(2,d-5),(2,1)}(y)
		=&(-1)^{d+1} \frac{3}{2}y^{2d-9} (1+y^2+y^4)^2 A_{d-5}(y) |_{y^{\leq 3d-10}},\\
		B_{R;1,(1,1,d-5),(1,1,1)}(y)
		=&(-1)^d \frac{9}{2}y^{2d-9} (1+y+y^2)^4 A_{d-5}(y) |_{y^{\leq 3d-10}},
	\end{align*}
	which give the equation \eqref{eqn:second one}.
\end{proof}

\begin{Lemma}\label{lem:last one}
	For $d\geq10$,
	we have
	\begin{align}\label{eqn:last one}
		\begin{split}
			&B_{R;1,(3,d-6),(1,1)}(y)
			+B_{R;1,(3,d-6),(3,1)}(y)
			+B_{R;1,(1,2,d-6),(1,1,1)}(y)\\
			&\ \ +B_{R;1,(1,2,d-6),(1,2,1)}(y)
			+B_{R;1,(1,1,1,d-6),(1,1,1,1)}(y)
			=(-1)^d y^{3d-18}Q(y)  A_{d-6}(y) |_{y^{\leq 3d-10}},
		\end{split}
	\end{align}
	where $Q(y)=10+27y+72y^2+154y^3+306y^4+531y^5
	+848y^6+1197y^7+1512y^8$.
\end{Lemma}
\begin{proof}
	From the equations \eqref{eqn:G d=1 formula} and \eqref{eqn:Def BR},
	the first function $B_{R;1,(3,d-6),(1,1)}(y)$ is given by
	\begin{align}\label{eqn:B_{R;1,(3,d-6),(1,1)}(y)}
		B_{R;1,(3,d-6),(1,1)}(y)
		=&(-1)^d y^{3d-18} \frac{[9]_y^2}{[1]_y^2}(10+7y+10y^2)(1-y^{3d-18})^2A_{d-6}(y),
	\end{align}
	where we have used $A_3(y)=10+7y+10y^2$.
	The next four terms can be computed similarly as
	\begin{align}\label{eqn:4B}
		\begin{split}
			B_{R;1,(3,d-6),(3,1)}(y)
			=&(-1)^d y^{3d-15} \frac{[9]_y^2}{[3]_y^2}(1-y^{3d-18})^2A_{d-6}(y),\\
			B_{R;1,(1,2,d-6),(1,1,1)}(y)
			=&(-1)^d 18y^{3d-16} \frac{[3]_y^2[6]_y^2}{[1]_y^4}(1-y^{3d-18})^2A_{d-6}(y),\\
			B_{R;1,(1,2,d-6),(1,2,1)}(y)
			=&(-1)^{d+1} \frac{9}{2}y^{3d-15} \frac{[3]_y^2[6]_y^2}{[1]_y^2[2]_y^2}(1-y^{3d-18})^2A_{d-6}(y),\\
			B_{R;1,(1,1,1,d-6),(1,1,1,1)}(y)
			=&(-1)^d \frac{9}{2}y^{3d-15} \frac{[3]_y^6}{[1]_y^6}(1-y^{3d-18})^2A_{d-6}(y).
		\end{split}
	\end{align}
	As a consequence,
	the left-hand side of equation \eqref{eqn:last one} has the following form
	\begin{align}\label{eqn:lhs use Q}
		(-1)^d y^{3d-18} \tilde{Q}(y) (1-y^{3d-18})^2 A_{d-6}(y),
	\end{align}
	where $\tilde{Q}(y)$ is a polynomial in $y$ obtained from taking the sum of the above equations \eqref{eqn:B_{R;1,(3,d-6),(1,1)}(y)}, \eqref{eqn:4B},
	and more precisely,
	\begin{align*}
		\tilde{Q}(y)
		=&10+27y+72y^2+154y^3+306y^4+531y^5
		+848y^6+1197y^7+1512y^8+1630y^9\\
		&+1512y^{10}+1197y^{11}+848y^{12}+531y^{13}
		+306y^{14}+154y^{15}+72y^{16}+27y^{17}+10y^{18}.
	\end{align*}
	Note that we are only interested in the $y^{\leq 3d-10}$ part.
	On the one hand,
	the factor $(1-y^{3d-18})^2$ in equation \eqref{eqn:lhs use Q} could be replaced by $1$ when $d\geq10$.
	On the other hand,
	since the appearance of the factor $y^{3d-18}$,
	we could only keep the terms in the expansion of $\tilde{Q}(y)$ up to $y^8$.
	Thus,
	we obtain the equation \eqref{eqn:last one}.
\end{proof}

\begin{Lemma}\label{lem:last one a=4}
	For $d\geq12$,
	we have
	\begin{align}\label{eqn:last one a=4}
		\begin{split}
			&B_{R;1,(4,d-7),(1,1)}(y)
			=(-1)^d y^{4d-28}\frac{A_4(y)A_{d-7}(y)}{(1-y)^2}
			|_{y^{\leq 3d-10}},
		\end{split}
	\end{align}
	where $A_4(y)=15+12y+48y^2+42y^3+48y^4+12y^5+15y^6$.
\end{Lemma}
\begin{proof}
	By equations \eqref{eqn:G d=1 formula} and \eqref{eqn:Def BR},
	and $d-7>4$ when $d\geq12$,
	we obtain
	\begin{align*}
		B_{R;1,(4,d-7),(1,1)}(y)
		=(-1)^d y^{4d-28}
		\frac{(1-y^{12})^2(1-y^{3d-21})^2}{(1-y)^2} A_4(y)A_{d-7}(y).
	\end{align*}
	Then when restricting to the $y^{\leq 3d-10}$ part,
	the factor $(1-y^{12})^2(1-y^{3d-21})^2$ can be replaced by $1$ when $d\geq10$.
	This proves the lemma.
\end{proof}

\section{Proof of Theorem \ref{thm-Delta2-expl}}
\label{app-eval-schur}

In this section we prove Theorem \ref{thm-Delta2-expl},
i.e.,
we calculate the explicit formulas for the two-legged topological vertex
$W_{\mu,\nu}(q)$ with $\Delta =2$.

\subsection{Computation of $W_{(m-2,2),(n)}(q)$}

First we compute $W_{(m-2,2),(n)}(q)$ for  $m\geq 4$ and $n\geq 0$.
By the definition \eqref{eq-W12inSchur} we have:
\be
\label{eq-Wm-22-decomp}
W_{(m-2,2),(n)}(q) = s_{(m-2,2)}(q^\rho) s_{(n)}(q^{(m-2,2)+\rho}),
\ee
and by the definition of Schur functions and the evaluation formula \eqref{eq-hneneva} we have:
\be
\label{eq-Wm-22-sm-22}
\begin{split}
s_{(m-2,2)}(q^\rho) =& (h_{m-2}h_2 - h_{m-1}h_1) (q^\rho)
= q^{-\half m-2} \cdot \frac{[m-3]_{q^{-1}}} {[m-1]_{q^{-1}}!\cdot [2]_{q^{-1}}!},
\end{split}
\ee
and
\begin{equation*}
\begin{split}
s_{(n)}(q^{(m-2,2)+\rho}) = h_n (q^{(m-2,2)+\rho})
= h_n (q^{m-\frac{5}{2}}, q^{\half}, q^{-\frac{5}{2}}, q^{-\frac{7}{2}}, q^{-\frac{9}{2}},\cdots).
\end{split}
\end{equation*}
Let $t$ be a formal variable and consider the generating series of $h_n(q^{(m-2,2)+\rho})$:
\begin{equation*}
\begin{split}
\sum_{n\geq 0} h_n(q^{(m-2,2)+\rho}) t^n
=& \frac{1}{1-q^{m-\frac{5}{2}}t} \cdot \frac{1}{1-q^{\frac{1}{2}}t} \cdot
\prod_{j\geq 2} \frac{1}{1-q^{-j-\half}t} \\
=& \frac{1}{1-q^{m-\frac{5}{2}}t} \cdot \frac{1}{1-q^{\frac{1}{2}}t} \cdot
\sum_{k\geq 0} h_k (q^\rho) (q^{-2}t)^k \\
=& \Big(\sum_{i\geq 0} q^{(m-\frac{5}{2})i} t^i\Big)
\Big(\sum_{j\geq 0} q^{\half j} t^j\Big)
\Big(\sum_{k\geq 0} \frac{q^{-\half k}}{[k]_{q^{-1}}!}\cdot q^{-2k} t^k\Big).
\end{split}
\end{equation*}
By taking the coefficient of $t^n$,
we obtain:
\begin{equation*}
\begin{split}
h_n(q^{(m-2,2)+\rho})
=& \sum_{k=0}^n
\sum_{j=0}^{n-k} \frac{1}{[k]_{q^{-1}}!} \cdot
q^{(m-\frac{5}{2})(n-k-j) + \half j -\frac{5}{2}k} \\
=& q^{(m-\frac{5}{2})n} \sum_{k=0}^n \frac{q^{-mk}}{[k]_{q^{-1}}!} \cdot
\frac{1-q^{-(m-3) (n-k+1)}}{1-q^{-(m-3)}},
\end{split}
\end{equation*}
and thus:
\begin{equation*}
\begin{split}
h_n(q^{(m-2,2)+\rho})
=& \frac{q^{mn-\frac{5}{2}n}}{[m-3]_{q^{-1}}}\sum_{k=0}^n \frac{q^{-mk}}{[k]_{q^{-1}}!}
- \frac{q^{-m+\half n +3}}{[m-3]_{q^{-1}}} \sum_{k=0}^n \frac{q^{-3k}}{[k]_{q^{-1}}!}.
\end{split}
\end{equation*}
Using the notation \eqref{eq-notationtW},
we conclude that:
\be
\label{eq-Wm-22-hnm-22}
\begin{split}
h_n(q^{(m-2,2)+\rho})
=&  \frac{q^{mn-\frac{5}{2}n} \cdot [m-1]_{q^{-1}}!}{[m-3]_{q^{-1}}} \cdot \tW_{(m-1),(n)}(q)
-\frac{q^{-m+\half n +3} \cdot [2]_{q^{-1}}!}{[m-3]_{q^{-1}}} \cdot \tW_{(2),(n)}(q).
\end{split}
\ee
Now plug \eqref{eq-Wm-22-sm-22} and \eqref{eq-Wm-22-hnm-22} into \eqref{eq-Wm-22-decomp},
and we obtain Case 1) of Theorem \ref{thm-Delta2-expl}.

\subsection{Computation of $W_{(m-2,1,1),(n)}(q)$}

Now we compute the case $W_{(m-2,1,1),(n)}(q)$ where $m\geq 3$ and $n\geq 0$.
By \eqref{eq-W12inSchur} we have:
\be
\label{eq-Wm-211-decomp}
W_{(m-2,1,1),(n)}(q) = s_{(m-2,1,1)}(q^\rho) s_{(n)}(q^{(m-2,1,1)+\rho}),
\ee
where
\begin{equation*}
\begin{split}
s_{(m-2,1,1)}(q^\rho) =& (h_{m-2}e_2 -h_{m-1}e_1 + h_m ) (q^\rho).
\end{split}
\end{equation*}
Then by \eqref{eq-hneneva} one has:
\be
\label{eq-Wm-211-sm-211}
\begin{split}
s_{(m-2,1,1)}(q^\rho) =&
\frac{q^{-\half (m-2)}}{[m-2]_{q^{-1}}!}\cdot \frac{q^{-2}}{[2]_{q^{-1}}!}
- \frac{q^{-\half (m-1)}}{[m-1]_{q^{-1}}!}\cdot \frac{q^{-\half}}{[1]_{q^{-1}}!}
+ \frac{q^{-\half m}}{[m]_{q^{-1}}!} \\
=& \frac{q^{-\half m -3}}{[m]_{q^{-1}}! \cdot [2]_{q^{-1}}!} \cdot [m-1]_{q^{-1}}\cdot [m-2]_{q^{-1}}.
\end{split}
\ee

Then we compute the second factor $s_{(n)}(q^{(m-2,1,1)+\rho})$
in \eqref{eq-Wm-211-decomp}.
By definition one has:
\begin{equation*}
s_{(n)}(q^{(m-2,1,1)+\rho})
= h_n (q^{m-\frac{5}{2}}, q^{-\half}, q^{-\frac{3}{2}},
q^{-\frac{7}{2}},q^{-\frac{9}{2}}, q^{-\frac{11}{2}}, \cdots).
\end{equation*}
Consider the generating series:
\begin{equation*}
\begin{split}
\sum_{n\geq 0} h_n(q^{(m-2,1,1)+\rho}) t^n
=& \frac{1}{1-q^{m-\frac{5}{2}}t} \cdot (1-q^{-\frac{5}{2}} t) \cdot
\prod_{j\geq 0} \frac{1}{1-q^{-j-\half}t} \\
=& \frac{1}{1-q^{m-\frac{5}{2}}t} \cdot (1-q^{-\frac{5}{2}} t) \cdot
\sum_{k\geq 0} h_k(q^\rho) t^k \\
=& (1-q^{-\frac{5}{2}} t)
\Big(\sum_{i\geq 0} q^{(m-\frac{5}{2})i} t^i \Big)
\Big( \sum_{k\geq 0} \frac{q^{-\half k}}{[k]_{q^{-1}}!} \cdot t^k \Big).
\end{split}
\end{equation*}
By comparing the coefficients of $t^n$ one obtains:
\begin{equation*}
\begin{split}
h_n(q^{(m-2,1,1)+\rho}) =&
\sum_{k=0}^n q^{(m-\frac{5}{2})(n-k)} \frac{q^{-\half k}}{[k]_{q^{-1}}!}
-\sum_{k=0}^{n-1} q^{-\frac{5}{2}} q^{(m-\frac{5}{2})(n-k-1)} \frac{q^{-\half k}}{[k]_{q^{-1}}!} \\
=& \frac{q^{-\half n}}{[n]_{q^{-1}}!}
+ q^{(m-\frac{5}{2})n} \cdot [m]_{q^{-1}} \cdot
\sum_{k=0}^{n-1} \frac{q^{-(m-2) k}}{[k]_{q^{-1}}!}.
\end{split}
\end{equation*}
And using the notation \eqref{eq-notationtW} one has:
\be
\label{eq-Wm-211-hn211}
\begin{split}
h_n(q^{(m-2,1,1)+\rho}) = \frac{q^{-\half n}}{[n]_{q^{-1}}!}
+ q^{(m-\frac{5}{2})n} \cdot [m]_{q^{-1}} \cdot [m-3]_{q^{-1}}!
\cdot \tW_{(m-3),(n-1)}(q).
\end{split}
\ee
Now plug \eqref{eq-Wm-211-sm-211} and \eqref{eq-Wm-211-hn211} into \eqref{eq-Wm-211-decomp},
and we obtain Case 2) of Theorem \ref{thm-Delta2-expl}.

\subsection{Computation of $W_{(m-1,1),(n-1,1)}(q)$}

Finally we compute the $W_{(m-1,1),(n-1,1)}(q)$ where $m,n\geq 2$.
By \eqref{eq-W12inSchur} we have:
\be
\label{eq-Wm-1n-1-decomp}
W_{(m-1,1),(n-1,1)}(q) = s_{(m-1,1)}(q^\rho) s_{(n-1,1)}(q^{(m-1,1)+\rho}),
\ee
where
\begin{equation*}
\begin{split}
s_{(m-1,1)}(q^\rho) =& (h_{m-1}h_1 -h_m) (q^\rho)
= q^{-\half m}\Big( \frac{1}{[m-1]_{q^{-1}} !} \cdot \frac{1}{[1]_{q^{-1}} !}
-\frac{1}{[m]_{q^{-1}} !}\Big).
\end{split}
\end{equation*}
It is clear that:
\begin{equation*}
\begin{split}
\frac{1}{[m-1]_{q^{-1}} !} \cdot \frac{1}{[1]_{q^{-1}} !}
-\frac{1}{[m]_{q^{-1}} !}
=\frac{1}{[m]_{q^{-1}} !} \cdot \sum_{j=1}^{m-1}q^{-j}
= \frac{q^{-1}}{[m]_{q^{-1}}!} \cdot \frac{[m-1]_{q^{-1}}}{[1]_{q^{-1}}},
\end{split}
\end{equation*}
and thus:
\be
\label{sm-11qrho}
s_{(m-1,1)}(q^\rho) =
\frac{q^{-\half m -1}}{[m]_{q^{-1}}!} \cdot \frac{[m-1]_{q^{-1}}}{[1]_{q^{-1}}}.
\ee

Now we compute the second factor in \eqref{eq-Wm-1n-1-decomp}.
One has:
\be
\label{eq-sm-11qrho-1}
\begin{split}
s_{(n-1,1)}(q^{(m-1,1)+\rho}) =&
h_{n-1}(q^{(m-1,1)+\rho}) h_1(q^{(m-1,1)+\rho}) -h_n( q^{(m-1,1)+\rho} ).
\end{split}
\ee
Consider the generating series of $h_n( q^{(m-1,1)+\rho} )$:
\begin{equation*}
\begin{split}
\sum_{k\geq 0} h_k( q^{(m-1,1)+\rho} ) t^k
=& \frac{1}{1- q^{m-\frac{3}{2}} t} \cdot \frac{1}{1-q^{-\half} t}\cdot
\prod_{j\geq 2} \frac{1}{1-q^{-j-\half} t} \\
=& \frac{1- q^{-\frac{3}{2}} t}{1- q^{m-\frac{3}{2}} t} \cdot
\sum_{k\geq 0} h_k(q^\rho) t^k \\
=& (1- q^{-\frac{3}{2}} t)
\Big( \sum_{i\geq 0}q^{(m-\frac{3}{2})i} t^i \Big)
\Big( \sum_{k\geq 0} h_k(q^\rho) t^k \Big).
\end{split}
\end{equation*}
By comparing the coefficients of $t^n$ and using \eqref{eq-hneneva} one obtains:
\begin{equation*}
\begin{split}
h_n(q^{(m-1,1)+\rho}) =&
\sum_{k=0}^n q^{(m-\frac{3}{2})(n-k)} \cdot \frac{q^{-\half k}}{[k]_{q^{-1}}!}
-\sum_{k=0}^{n-1} q^{-\frac{3}{2}} q^{(m-\frac{3}{2})(n-k-1)}
\cdot \frac{q^{-\half k}}{[k]_{q^{-1}}!} \\
=& \frac{q^{-\half n}}{[n]_{q^{-1}}!}
+q^{mn-\frac{3}{2}n} \cdot [m]_{q^{-1}}\cdot
\sum_{k=0}^{n-1} \frac{q^{- k (m-1)}}{[k]_{q^{-1}}!}.
\end{split}
\end{equation*}
Using the notation \eqref{eq-notationtW},
one has:
\be
\begin{split}
h_n(q^{(m-1,1)+\rho}) = \frac{q^{-\half n}}{[n]_{q^{-1}}!}
+q^{mn-\frac{3}{2}n} \cdot \frac{[m]_{q^{-1}}!}{[m-1]_{q^{-1}}}\cdot
\tW_{(m-2),(n-1)}(q).
\end{split}
\ee
Then by \eqref{eq-sm-11qrho-1} one has:
\be
\label{eq-sm-11qrho-2}
\begin{split}
& s_{(n-1,1)}(q^{(m-1,1)+\rho}) \\
=&
\Big( \frac{q^{-\half (n-1)}}{[n-1]_{q^{-1}}!}
+ \frac{q^{(m-\frac{3}{2})(n-1)} [m]_{q^{-1}}!}{[m-1]_{q^{-1}}}
\tW_{(m-2),(n-2)}(q) \Big)
\Big( \frac{q^{-\half}}{[1]_{q^{-1}}} +q^{m-\frac{3}{2}} [m]_{q^{-1}} \Big) \\
& -\Big( \frac{q^{-\half n}}{[n]_{q^{-1}}!}
+q^{mn-\frac{3}{2}n} \cdot \frac{[m]_{q^{-1}}!}{[m-1]_{q^{-1}}}
\tW_{(m-2),(n-1)}(q) \Big).
\end{split}
\ee
From the definition \eqref{eq-notationtW} of $\tW_{(m),(n)}$ it is clear that:
\begin{equation*}
\tW_{(m-2),(n-1)}(q) =\tW_{(m-2),(n-2)}(q)
+ \frac{q^{-(n-1)(m-1)}}{[m-2]_{q^{-1}}!\cdot [n-1]_{q^{-1}}!}.
\end{equation*}
Now we plug this into \eqref{eq-sm-11qrho-2},
and then after some simplification we obtain:
\begin{equation*}
\begin{split}
s_{(n-1,1)}(q^{(m-1,1)+\rho})
=&
q^{-\half n} \Big( \frac{1}{[1]_{q^{-1}} \cdot [n-1]_{q^{-1}}!} - \frac{1}{[n]_{q^{-1}}!}\Big) \\
& + q^{mn-\frac{3}{2}n} \cdot \frac{[m]_{q^{-1}}!}{[m-1]_{q^{-1}}}
\Big( \frac{q^{-m+1}}{[1]_{q^{-1}}} + [m]_{q^{-1}} -1\Big) \tW_{(m-2),(n-2)}(q),
\end{split}
\end{equation*}
or equivalently,
\be
\label{eq-sm-11qrho-3}
\begin{split}
& s_{(n-1,1)}(q^{(m-1,1)+\rho}) \\
=& \frac{q^{-\half n-1}}{[n]_{q^{-1}}!} \cdot \frac{[n-1]_{q^{-1}}}{[1]_{q^{-1}}}
+ q^{mn - m -\frac{3}{2}n +1} \cdot \frac{(1-q^{-1}+q^{-2})\cdot [m]_{q^{-1}}!}{[m-1]_{q^{-1}}\cdot [1]_{q^{-1}}}
\cdot \tW_{(m-2),(n-2)}(q).
\end{split}
\ee
Now plug \eqref{sm-11qrho} and \eqref{eq-sm-11qrho-3} into \eqref{eq-Wm-1n-1-decomp},
and we obtain Case 3) of Theorem \ref{thm-Delta2-expl}.

\section{Proof of Theorem \ref{thm-main1-leadI} --- Part I: Leading Terms of $W_d(q)$}
\label{app-WdCompute}

In this section we compute the leading terms of $W_d(q)$,
where $W_d(q)$ is defined by \eqref{eq-def-Wdq} or equivalently by \eqref{eq-Wd-sumdef}.
See Theorem \ref{thm-Wdqleadingterm} for the final result.

\subsection{Pochhammer symbols and hypergeometric series}
\label{sec-appB1-gz}

First we recall the relation between Pochhammer symbols
\begin{equation*}
\begin{split}
& (a;q^{-1})_{\infty}= \prod_{i=0}^\infty (1-aq^{-i}), \qquad\qquad
 (a;q^{-1})_{n}= \prod_{i=0}^{n-1} (1-aq^{-i})
= \frac{(a;q^{-1})_\infty}{(aq^{-n};q^{-1})_\infty},
\end{split}
\end{equation*}
and hypergeometric series,
and Guo--Zhou's decomposition of $W_d(q)$.

It is well-known that
(see e.g. \cite{gr}):
\be
\label{eq-pochhammerseries}
\sum_{n\geq 0} \frac{z^n}{[n]_{q^{-1}}!} = \frac{1}{(z;q^{-1})_\infty}.
\ee
For a formal power series $f(x) = a_0 + a_1 x +a_2 x^2+ \cdots  \in \bC[[x]]$,
denote:
\begin{equation*}
T_m^x f(x) = a_0 +a_1 + \cdots +a_m \in \bC.
\end{equation*}
The following two identities follow from \eqref{eq-pochhammerseries}
(see Guo--Zhou \cite[Lemma 5.6]{gz}):
\be
\label{eq-gzlem-5.6}
\begin{split}
&\sum_{k=1}^m \frac{q^{-k(d+1)}}{[k]_{q^{-1}}!} =
T_m^x \Big( \frac{1}{(xq^{-d-1} ;q^{-1})_\infty} -1 \Big),\\
&\frac{1}{[d_2]_{q^{-1}}!}
\sum_{k=0}^{d_3} \frac{q^{-k(d_2 +1)}}{[k]_{q^{-1}}!} =
\frac{1}{[\infty]_{q^{-1}}!}\cdot
T_{d_3}^x \frac{(q^{-d_2 -1} ;q^{-1})_\infty}{(xq^{-d_2 -1} ;q^{-1})_\infty},
\end{split}
\ee
where $[\infty]_{q^{-1}}!$ is defined by \eqref{eq-definftyfactorial}.
By \eqref{eq-notationtW} one easily sees the symmetry:
\be
\label{eq-symmTmx}
T_{d_3}^x \frac{(q^{-d_2 -1} ;q^{-1})_\infty}{(xq^{-d_2 -1} ;q^{-1})_\infty}
= T_{d_2}^x \frac{(q^{-d_3 -1} ;q^{-1})_\infty}{(xq^{-d_3 -1} ;q^{-1})_\infty},
\qquad \forall d_2,d_3\geq 0.
\ee
Moreover,
one has (see \cite[Lemma 5.7]{gz}):
\be
\label{eq-gzlem-5.7}
\begin{split}
& \deg \Big( \frac{1}{[\infty]_{q^{-1}}!} - \frac{1}{[d]_{q^{-1}}!} \Big) = -(d+1),\\
& \deg  T_{m}^x \Big( \frac{(q^{-d -1} ;q^{-1})_\infty}{(xq^{-d -1} ;q^{-1})_\infty} -1\Big)
= -(m+1)(d+1).
\end{split}
\ee

Now we recall Guo--Zhou's decomposition of $W_d(q)$.
By \eqref{eq-gzlem-5.6} and \eqref{eq-Wd-sumdef} one can rewrite $W_d(q)$ as:
\be
\label{eq-Wddefexpand}
\begin{split}
W_d(q) = \sum_{d_1+d_2+d_3 =d} \prod_{i=1}^3
\Big( \frac{1}{[\infty]_{q^{-1}}!}\cdot
 T_{d_{i+1}}^x\Big( \frac{(q^{-d_i -1} ;q^{-1})_\infty}{(xq^{-d_i -1} ;q^{-1})_\infty} -1 \Big)
+ \frac{1}{[\infty]_{q^{-1}}!} \Big),
\end{split}
\ee
where we use the convention $d_4 =d_1$.
Following \cite{gz} we denote:
\be
\label{eq-defWd1Wd2Wd3}
\begin{split}
& W_d^1(q) = \sum_{d_1+d_2+d_3 =d} \frac{1}{([\infty]_{q^{-1}}!)^3}
= \frac{(d+1)(d+2)}{2} \cdot \frac{1}{([\infty]_{q^{-1}}!)^3}, \\
& W_d^2(q) = \sum_{d_1+d_2+d_3 =d} \frac{3}{([\infty]_{q^{-1}}!)^3}
 T_{d_2}^x\Big( \frac{(q^{-d_1 -1} ;q^{-1})_\infty}{(xq^{-d_1 -1} ;q^{-1})_\infty} -1 \Big),\\
& W_d^3(q) = \sum_{d_1+d_2+d_3 =d} \frac{3}{([\infty]_{q^{-1}}!)^3}
 T_{d_2}^x\Big( \frac{(q^{-d_1 -1} ;q^{-1})_\infty}{(xq^{-d_1 -1} ;q^{-1})_\infty} -1 \Big)
 T_{d_3}^x\Big( \frac{(q^{-d_2 -1} ;q^{-1})_\infty}{(xq^{-d_2 -1} ;q^{-1})_\infty} -1 \Big),
\end{split}
\ee
and
\be
\label{eq-defWd4}
\begin{split}
W_d^4(q) = \sum_{d_1+d_2+d_3 =d} &\frac{1}{([\infty]_{q^{-1}}!)^3}
\cdot
T_{d_1}^x\Big( \frac{(q^{-d_3 -1} ;q^{-1})_\infty}{(xq^{-d_3 -1} ;q^{-1})_\infty} -1 \Big) \\
& \cdot T_{d_2}^x\Big( \frac{(q^{-d_1 -1} ;q^{-1})_\infty}{(xq^{-d_1 -1} ;q^{-1})_\infty} -1 \Big)
T_{d_3}^x\Big( \frac{(q^{-d_2 -1} ;q^{-1})_\infty}{(xq^{-d_2 -1} ;q^{-1})_\infty} -1 \Big),
\end{split}
\ee
and then by expanding \eqref{eq-Wddefexpand} directly one can see that:
\begin{equation*}
W_d(q) = W_d^1(q) + W_d^2(q) + W_d^3(q) + W_d^4(q).
\end{equation*}
By \eqref{eq-gzlem-5.7} one may derive that
(see \cite[\S 5.6]{gz}):
\begin{equation*}
\deg W_d^3(q) = -(d+2),
\qquad\qquad
\deg W_d^4(q) = - (2d+3),
\end{equation*}
and thus:
\begin{equation*}
W_d(q) = W_d^1(q) + W_d^2 (q) \big|_{q^{\geq -d-1}}.
\end{equation*}
In the next subsections,
we will compute the leading terms $W_d^i(q)|_{q^{\geq -3d+1}}$ for every $i$.

\subsection{Leading terms of $W_d^2(q)$}

In this subsection we compute the leading terms of $W_d^2(q)$.
Denote:
\be
\label{eq-W2infty-def}
W_\infty^2(q) = \frac{3}{([\infty]_{q^{-1}}!)^3} \cdot
\sum_{d_1,d_2 \geq 0}
T_{d_2}^x\Big( \frac{(q^{-d_1 -1} ;q^{-1})_\infty}{(xq^{-d_1 -1} ;q^{-1})_\infty} -1 \Big),
\ee
and Guo--Zhou proved that (see \cite[Prop. 5.9]{gz}):
\be
\label{eq-W2infty=}
W_\infty^2 (q) =
-\frac{3}{([\infty]_{q^{-1}}!)^3} \cdot
\sum_{i\geq 1} \frac{q^{-i}}{(1-q^{-i})^2}.
\ee
Now denote:
\be
\label{eq-defVd2-Wd2}
V_d^2 (q) = W_\infty^2 (q) - W_d^2(q),
\ee
then by the definitions of $W_\infty^2 (q)$ and $W_d^2(q)$ it is clear that:
\be
V_d^2(q) = \frac{3}{([\infty]_{q^{-1}}!)^3} \cdot
\sum_{d_1 + d_2 \geq d+1}
T_{d_2}^x\Big( \frac{(q^{-d_1 -1} ;q^{-1})_\infty}{(xq^{-d_1 -1} ;q^{-1})_\infty} -1 \Big),
\ee
and thus by \eqref{eq-gzlem-5.7} one has:
\begin{equation*}
\deg T_{d_2}^x\Big( \frac{(q^{-d_1 -1} ;q^{-1})_\infty}{(xq^{-d_1 -1} ;q^{-1})_\infty} -1 \Big)
= -(d_1+1)(d_2+1),
\end{equation*}
and thus (see \cite[Prop 5.10]{gz}):
\begin{equation*}
\deg V_d^2(q) = \max_{d_1 + d_2 \geq d+1} \{-(d_1+1)(d_2+1)\} = -(d+2).
\end{equation*}
Moreover,
it is easy to see that:
\be
\label{eq-degreeest-V2}
\begin{split}
& -(d_1+1)(d_2+1)
\leq -2 (d+1),
\qquad \text{if $d_1,d_2\geq 1$}; \\
& -(d_1+1)(d_2+1)
\leq -3 d,
\qquad \text{if $d_1,d_2\geq 2$}; \\
& -(d_1+1)(d_2+1)
\leq -4 (d-1),
\qquad \text{if $d_1,d_2\geq 3$},
\end{split}
\ee
provided $d_1+d_2 \geq d+1$,
and thus:
\begin{Lemma}
Fix $k\in \{1,2,3,4\}$,
and then:
\begin{equation*}
\begin{split}
W_d^2 (q) =
\Big( W_\infty^2 (q) - 
\frac{3}{([\infty]_{q^{-1}}!)^3}
\sum_{\substack{d_1 + d_2 \geq d+1\\ d_1< k-1 \text{ or } d_2< k-1 }}
T_{d_2}^x\Big( \frac{(q^{-d_1 -1} ;q^{-1})_\infty}{(xq^{-d_1 -1} ;q^{-1})_\infty} -1 \Big) \Big)
\Big|_{q^{\geq -k(d+3-k) +1}}.
\end{split}
\end{equation*}
\end{Lemma}

In particular,
by taking $k=1$ one obtains:
\begin{equation*}
W_d^2(q) = W_\infty^2 (q) \big|_{q^{\geq -d-1}},
\end{equation*}
and then by \eqref{eq-W2infty=} one concludes that (see \cite[Prop. 5.10]{gz}):
\be
W_d^2 (q) =
-\frac{3}{([\infty]_{q^{-1}}!)^3} \cdot
\sum_{i\geq 1} \frac{q^{-i}}{(1-q^{-i})^2}
\Big|_{q^{\geq -d-1}}.
\ee
Here we want to compute the coefficients of $q^{-D}$ in $V_d^2(q)$ for $0\leq D\leq 3d-1$,
and for our purpose we need to consider the case $k=3$ since:
\begin{equation*}
\begin{split}
W_d^2 (q) =
\bigg( W_\infty^2 (q) -
\frac{3}{([\infty]_{q^{-1}}!)^3}
 \sum_{\substack{d_1 + d_2 \geq d+1\\ d_1=0,1 \text{ or } d_2=0,1 }}
T_{d_2}^x\Big( \frac{(q^{-d_1 -1} ;q^{-1})_\infty}{(xq^{-d_1 -1} ;q^{-1})_\infty} -1 \Big) \bigg)
\bigg|_{q^{\geq -k(d+3-k) +1}}.
\end{split}
\end{equation*}
First by induction one can easily obtain the following:
\begin{Lemma}
\label{lem-sumqoverdfactorial}
For every positive integer $n$ one has:
\begin{equation*}
\begin{split}
\sum_{n=0}^m \frac{q^{-n}}{[n]_{q^{-1}}!} = \frac{1}{[m]_{q^{-1}}!},\qquad\qquad
\sum_{n=0}^m \frac{q^{-2n}}{[n]_{q^{-1}}!} = \frac{1-q^{-1}+q^{-m-1}}{[m]_{q^{-1}}!},
\end{split}
\end{equation*}
\end{Lemma}

Then by \eqref{eq-pochhammerseries} one has:
\begin{equation*}
\begin{split}
T_{d_2}^x\Big( \frac{(q^{-d_1 -1} ;q^{-1})_\infty}{(xq^{-d_1 -1} ;q^{-1})_\infty} -1 \Big)
=& T_{d_2}^x\Big( \frac{[\infty]_{q^{-1}}!}{[d_1]_{q^{-1}}!} \cdot
\sum_{n\geq 0} \frac{x^n q^{-n(d_1+1)}}{[n]_{q^{-1}}!} -1 \Big) \\
=& \frac{[\infty]_{q^{-1}}!}{[d_1]_{q^{-1}}!} \cdot
\sum_{n= 0}^{d_2} \frac{q^{-n(d_1+1)}}{[n]_{q^{-1}}!} -1.
\end{split}
\end{equation*}
In particular,
for $d_1 = 0,1$ one has:
\be
\label{eq-computed1=01}
\begin{split}
&T_{d_2}^x \Big( \frac{(q^{-1} ;q^{-1})_\infty}{(xq^{-1} ;q^{-1})_\infty} -1 \Big)
= \frac{[\infty]_{q^{-1}}!}{[d_2]_{q^{-1}}!} -1; \\
&T_{d_2}^x \Big( \frac{(q^{-2} ;q^{-1})_\infty}{(xq^{-2} ;q^{-1})_\infty} -1 \Big)
= \frac{1-q^{-1}+q^{-d_2-1}}{[1]_{q^{-1}}} \cdot \frac{[\infty]_{q^{-1}}!}{[d_2]_{q^{-1}}!} -1,
\end{split}
\ee
respectively by the above lemma;
and for $d_2 = 0,1$ one has:
\be
\label{eq-computed2=01}
\begin{split}
& T_{0}^x\Big( \frac{(q^{-d_1 -1} ;q^{-1})_\infty}{(xq^{-d_1 -1} ;q^{-1})_\infty} -1 \Big)
= \frac{[\infty]_{q^{-1}}!}{[d_1]_{q^{-1}}!} -1; \\
& T_{1}^x\Big( \frac{(q^{-d_1 -1} ;q^{-1})_\infty}{(xq^{-d_1 -1} ;q^{-1})_\infty} -1 \Big)
= (1+ \frac{q^{-(d_1+1)}}{[1]_{q^{-1}}!})\cdot \frac{[\infty]_{q^{-1}}!}{[d_1]_{q^{-1}}!} -1,
\end{split}
\ee
respectively.
Here \eqref{eq-computed2=01} can also be obtained by applying the symmetry \eqref{eq-symmTmx}
to \eqref{eq-computed1=01}.

For $d\geq 2$,
the four cases $(d_1,d_2) = (0,d_2),(1,d_2),(d_1,0),(d_1,1)$ do not overlap when $d_1+d_2 \geq d+1$,
and thus by \eqref{eq-degreeest-V2} one has:
\be
\label{eq-Vd2compute}
\begin{split}
V_d^2 (q) = & \frac{3}{[\infty]_{q^{-1}}!^3} \Big(
\sum_{d_1 = d+1}^{\infty} \big( T_{0}^x \frac{(q^{-d_1 -1} ;q^{-1})_\infty}{(xq^{-d_1 -1} ;q^{-1})_\infty} -1\big)
+ \sum_{d_1 = d}^{\infty} \big(  T_{1}^x \frac{(q^{-d_1 -1} ;q^{-1})_\infty}{(xq^{-d_1 -1} ;q^{-1})_\infty} -1\big)\\
& + \sum_{d_2 = d+1}^{\infty} \big( T_{d_2}^x \frac{(q^{ -1} ;q^{-1})_\infty}{(xq^{ -1} ;q^{-1})_\infty} -1\big)
+ \sum_{d_2 = d}^{\infty} \big(T_{d_2}^x \frac{(q^{ -2} ;q^{-1})_\infty}{(xq^{-2} ;q^{-1})_\infty}-1 \big)
\Big) \Big|_{q^{\geq -3d+1}} \\
=& \frac{6}{[\infty]_{q^{-1}}!^3} \Big(
\sum_{d_1 = d+1}^\infty \big(\frac{[\infty]_{q^{-1}}!}{[d_1]_{q^{-1}}!} -1\big)
+ \sum_{d_1 =d}^\infty \big(
(1+ \frac{q^{-(d_1+1)}}{[1]_{q^{-1}}!}) \frac{[\infty]_{q^{-1}}!}{[d_1]_{q^{-1}}!} -1 \big)
\Big) \Big|_{q^{\geq -3d+1}}.
\end{split}
\ee
Here we have also used Proposition \ref{prop-degLaurentprod}
and the fact $\deg ([\infty]_{q^{-1}}!) =0$.
Moreover,
\begin{equation*}
\begin{split}
\sum_{d_1 = d+1}^{\infty} \Big( T_{0}^x \frac{(q^{-d_1 -1} ;q^{-1})_\infty}{(xq^{-d_1 -1} ;q^{-1})_\infty} -1\Big)
\Big|_{q^{\geq -3d+1}}
=& \sum_{d_1 = d+1}^{\infty} \Big( \frac{[\infty]_{q^{-1}}!}{[d_1]_{q^{-1}}!} -1 \Big)
\Big|_{q^{\geq -3d+1}} \\
=& \sum_{d_1 = d+1}^{\infty} \Big( \prod_{ i= 1}^\infty (1-q^{-d_1-i}) -1 \Big)
\Big|_{q^{\geq -3d+1}}.
\end{split}
\end{equation*}
By expanding $\frac{[\infty]_{q^{-1}}!}{[d_1]_{q^{-1}}!} = \prod_{ i= 1}^\infty (1-q^{-d_1-i})$ explicitly
it is easy to see that:
\be
\label{eq-comutedisumWd2}
\begin{split}
\Big(\frac{[\infty]_{q^{-1}}!}{[d_1]_{q^{-1}}!}-1\Big)
=& \Big(  - \sum_{i=d_1+1}^\infty q^{-i}
+\sum_{i=2d_1+3}^\infty \lfloor \frac{i-2d_1-1}{2} \rfloor \cdot q^{-i} \Big)
\Big|_{q^{\geq -3d+1}} \\
=&\Big( - \frac{q^{-d_1-1}}{1-q^{-1}} + \frac{q^{-2d_1-3}}{(1-q^{-1})(1-q^{-2})} \Big)
\Big|_{q^{\geq -3d+1}}
\end{split}
\ee
for $d_1\geq d+1$,
where $\lfloor \cdot \rfloor$ is the floor function.
And similarly,
for $d_1\geq d$ one has:
\begin{equation*}
\begin{split}
& \Big((1+ \frac{q^{-(d_1+1)}}{[1]_{q^{-1}}!}) \frac{[\infty]_{q^{-1}}!}{[d_1]_{q^{-1}}!} -1\Big) \\
=& \Big(  - \sum_{i=d_1+1}^\infty q^{-i}
+\sum_{i=2d_1+3}^\infty \lfloor \frac{i-2d_1-1}{2} \rfloor  q^{-i}
+\frac{1}{1-q^{-1}} \big( q^{-d_1-1} -\sum_{i=2d_1+2}^\infty q^{-i} \big)\Big)
\Big|_{q^{\geq -3d+1}} \\
=& \Big(  - \frac{q^{-2d_1-2}}{(1-q^{-1})^2} + \frac{q^{-2d_1-3}}{(1-q^{-1})(1-q^{-2})} \Big)
\Big|_{q^{\geq -3d+1}}.
\end{split}
\end{equation*}
Plug these into \eqref{eq-Vd2compute} and use Proposition \ref{prop-degLaurentprod},
and we finally obtain:
\begin{equation*}
\begin{split}
V_d^2 (q)
= &\frac{6}{([\infty]_{q^{-1}}!)^3} \Big(
-\frac{q^{-d-2}}{(1-q^{-1})^2} + \frac{q^{-2d-5}}{(1-q^{-1})(1-q^{-2})^2} \\
& \qquad\qquad
-\frac{q^{-2d-2}}{(1-q^{-1})^2 (1-q^{-2})} + \frac{q^{-2d-3}}{(1-q^{-1})(1-q^{-2})^2}
\Big) \Big|_{q^{\geq -3d+1}} \\
=&\frac{6}{([\infty]_{q^{-1}}!)^3} \Big(
\frac{q^{-2d-3}\cdot [4]_{q^{-1}}}{[1]_{q^{-1}} \cdot [2]_{q^{-1}}^3}
-\frac{q^{-d-2}}{[1]_{q^{-1}}^2} - \frac{q^{-2d-2}}{[1]_{q^{-1}}^2 \cdot [2]_{q^{-1}}}
\Big) \Big|_{q^{\geq -3d+1}}.
\end{split}
\end{equation*}
Now combine this with \eqref{eq-W2infty=} and \eqref{eq-defVd2-Wd2},
and we conclude that:

\begin{Proposition}
\label{prop-w2dleading}
For $d\geq 2$ one has:
\begin{equation*}
W_d^2(q) = \frac{3}{([\infty]_{q^{-1}}!)^3}\bigg(
- \sum_{i\geq 1} \frac{q^{-i}}{(1-q^{-i})^2} -2\Big(
\frac{q^{-2d-3} [4]_{q^{-1}}}{[1]_{q^{-1}}  [2]_{q^{-1}}^3}
-\frac{q^{-d-2}}{[1]_{q^{-1}}^2} - \frac{q^{-2d-2}}{[1]_{q^{-1}}^2  [2]_{q^{-1}}}
\Big)
\bigg)\Big|_{q^{\geq -3d+1}}.
\end{equation*}
\end{Proposition}

\subsection{Leading terms of $W_d^3(q)$ -- Case 1)}

Now we compute the coefficients of $q^{-D}$ with $0\leq D\leq 3d-1$ in $W_d^3(q)$.
First consider the degree of $W_d^3(q)$.
Assume $d_1+d_2+d_3 =d$,
and then by \eqref{eq-gzlem-5.7} one has:
\begin{equation*}
\begin{split}
\deg\Big(
T_{d_2}^x\big( \frac{(q^{-d_1 -1} ;q^{-1})_\infty}{(xq^{-d_1 -1} ;q^{-1})_\infty} -1 \big)
 T_{d_3}^x\big( \frac{(q^{-d_2 -1} ;q^{-1})_\infty}{(xq^{-d_2 -1} ;q^{-1})_\infty} -1 \big)
\Big)
= -(d_2+1)(d+2-d_2).
\end{split}
\end{equation*}
When $2\leq d_2 <d$ one has
$-(d_2+1)(d+2-d_2) \leq -3d$.
Thus to compute the coefficients of $q^{-D}$ with $0\leq D\leq 3d-1$,
we only need to consider three cases $d_2 =0,1$ and $d_2 = d$
in the definition \eqref{eq-defWd1Wd2Wd3} of $W_d^3(q)$.
We compute the first case $d_2 = 0$ in the rest subsection,
and compute the other two cases in the next two subsections.

For $d_2 = 0$,
by \eqref{eq-computed1=01} and \eqref{eq-computed2=01} one has:
\begin{equation*}
\begin{split}
&\sum_{d_1+d_3=d}T_{0}^x\Big( \frac{(q^{-d_1 -1} ;q^{-1})_\infty}{(xq^{-d_1 -1} ;q^{-1})_\infty} -1 \Big)
 T_{d_3}^x\Big( \frac{(q^{ -1} ;q^{-1})_\infty}{(xq^{ -1} ;q^{-1})_\infty} -1 \Big) \\
=& ([\infty]_{q^{-1}}!)^2 \cdot\sum_{d_1+d_3=d}
\Big( \frac{1}{[d_1]_{q^{-1}}!}-\frac{1}{[\infty]_{q^{-1}}!} \Big)
\Big( \frac{1}{[d_3]_{q^{-1}}!}-\frac{1}{[\infty]_{q^{-1}}!} \Big).
\end{split}
\end{equation*}
In the proof of \cite[Prop. 6.5]{gz},
Guo--Zhou have derived the following identity:
\be
\label{eq-W3dcase1-gzeq}
\begin{split}
&\sum_{d_1+d_3=d}
\Big( \frac{1}{[d_1]_{q^{-1}}!}-\frac{1}{[\infty]_{q^{-1}}!} \Big)
\Big( \frac{1}{[d_3]_{q^{-1}}!}-\frac{1}{[\infty]_{q^{-1}}!} \Big) \\
=& \sum_{d_1+d_3 =d}
\sum_{k=d_1+1}^\infty q^{-k}\big( \frac{1}{[k]_{q^{-1}}!} - \frac{1}{[\infty]_{q^{-1}}!} \big)
\sum_{l=d_3+1}^\infty q^{-l}\big( \frac{1}{[l]_{q^{-1}}!} - \frac{1}{[\infty]_{q^{-1}}!} \big) \\
& + \frac{2q^{-d-1}}{(1-q^{-1}) [\infty]_{q^{-1}}!} \sum_{d_1=0}^d \sum_{k=d_1+1}^\infty q^{d_1-k}\big( \frac{1}{[k]_{q^{-1}}!} - \frac{1}{[\infty]_{q^{-1}}!} \big)
+\frac{(d+1)q^{-d-2}}{(1-q^{-1})^2 ([\infty]_{q^{-1}}!)^2}.
\end{split}
\ee
In the following lemmas,
we compute the leading terms of the summations in the right-hand side.

\begin{Lemma}
\label{lem-W3dleading1}
The following identity holds:
\begin{equation*}
\begin{split}
& \frac{2q^{-d-1}}{(1-q^{-1}) [\infty]_{q^{-1}}!} \sum_{d_1=0}^d \sum_{k=d_1+1}^\infty q^{d_1-k}\big( \frac{1}{[k]_{q^{-1}}!} - \frac{1}{[\infty]_{q^{-1}}!} \big) \\
=&\Big(
-\frac{2q^{-d-2}}{[1]_{q^{-1}}^2 ([\infty]_{q^{-1}}!)^2} \sum_{i\geq 2} \frac{q^{-i}}{1-q^{-i}}
+ \frac{2q^{-2d-5}}{[1]_{q^{-1}}^3 [2]_{q^{-1}} ([\infty]_{q^{-1}}!)^2}
\Big) \Big|_{q^{\geq -3d-8}}.
\end{split}
\end{equation*}
(Here we use the truncation $q^{\geq -3d+8}$ instead of $q^{\geq -3d+1}$ for the convenience of the computations
in next sections.)
\end{Lemma}
\begin{proof}
First by \eqref{eq-pochhammerseries} and Lemma \ref{lem-sumqoverdfactorial} one easily sees that:
\be
\label{eq-sum1overkfac-infty}
\Big( \frac{1}{[k]_{q^{-1}}!}-\frac{1}{[\infty]_{q^{-1}}!} \Big)
=\sum_{n=0}^{k} \frac{q^{-n}}{[n]_{q^{-1}}!} -\sum_{n\geq 0} \frac{q^{-n}}{[n]_{q^{-1}}!}
=-\sum_{n=k+1}^\infty \frac{q^{-n}}{[n]_{q^{-1}}!},
\ee
and then one has:
\begin{equation*}
\begin{split}
& \sum_{d_1=0}^d \sum_{k=d_1+1}^\infty q^{d_1-k}\big( \frac{1}{[k]_{q^{-1}}!} - \frac{1}{[\infty]_{q^{-1}}!} \big) \\
=& \sum_{d_1=0}^\infty \sum_{k=d_1+1}^\infty q^{d_1-k}\big( \frac{1}{[k]_{q^{-1}}!} - \frac{1}{[\infty]_{q^{-1}}!} \big)
-\sum_{d_1 = d+1}^\infty \sum_{k=d_1+1}^\infty q^{d_1-k}\big( \frac{1}{[k]_{q^{-1}}!} - \frac{1}{[\infty]_{q^{-1}}!} \big)\\
=& \sum_{d_1=0}^\infty \sum_{k=d_1+1}^\infty q^{d_1-k}\big( \frac{1}{[k]_{q^{-1}}!} - \frac{1}{[\infty]_{q^{-1}}!} \big)
+ \sum_{d_1 = d+1}^\infty \sum_{k=d_1+1}^\infty \sum_{n=k+1}^\infty
\frac{q^{d_1-k-n}}{[n]_{q^{-1}}!}.
\end{split}
\end{equation*}
The first term of this expression has been computed by Guo--Zhou (see \cite[Lemma 6.6]{gz}):
\be
\label{eq-gzlem6.6}
\sum_{d_1=0}^\infty \sum_{k=d_1+1}^\infty q^{d_1-k}\big( \frac{1}{[k]_{q^{-1}}!} - \frac{1}{[\infty]_{q^{-1}}!} \big)
= -\frac{q^{-1}}{(1-q^{-1})[\infty]_{q^{-1}}!}\sum_{i\geq 2} \frac{q^{-i}}{1-q^{-i}},
\ee
and the second term can be rewritten as:
\begin{equation*}
\begin{split}
& \sum_{d_1 = d+1}^\infty \sum_{k=d_1+1}^\infty \sum_{n=k+1}^\infty
\frac{q^{d_1-k-n}}{[n]_{q^{-1}}!}
= \sum_{d_1 = d+1}^\infty \sum_{k=d_1+1}^\infty \sum_{n=k+1}^\infty \Big(
q^{d_1-k-n}\big( \frac{1}{[n]_{q^{-1}}!} - \frac{1}{[\infty]_{q^{-1}}!} \big)
+ \frac{q^{d_1-k-n}}{[\infty]_{q^{-1}}!}\Big).
\end{split}
\end{equation*}
Here by \eqref{eq-gzlem-5.7} one knows that
\begin{equation*}
\deg
q^{d_1-k-n}\big( \frac{1}{[n]_{q^{-1}}!} - \frac{1}{[\infty]_{q^{-1}}!} \big)
=d_1-k-2n-1 <-2d-7,
\end{equation*}
and thus when restricting to $q^{\geq -2d-7}$ one has:
\begin{equation*}
\begin{split}
\sum_{d_1 = d+1}^\infty \sum_{k=d_1+1}^\infty \sum_{n=k+1}^\infty
\frac{q^{d_1-k-n}}{[n]_{q^{-1}}!} =&
 \sum_{d_1 = d+1}^\infty \sum_{k=d_1+1}^\infty \sum_{n=k+1}^\infty
\frac{q^{d_1-k-n}}{[\infty]_{q^{-1}}!} \\
=& \frac{q^{-d-4}}{(1-q^{-1})^2(1-q^{-2})[\infty]_{q^{-1}}!}
\Big|_{q^{\geq -2d-7}},
\end{split}
\end{equation*}
and thus by Proposition \ref{prop-degLaurentprod} we conclude that:
\begin{equation*}
\begin{split}
& \frac{2q^{-d-1}}{(1-q^{-1}) [\infty]_{q^{-1}}!} \sum_{d_1=0}^d \sum_{k=d_1+1}^\infty q^{d_1-k}\big( \frac{1}{[k]_{q^{-1}}!} - \frac{1}{[\infty]_{q^{-1}}!} \big) \\
=& \frac{2q^{-d-1}}{(1-q^{-1}) [\infty]_{q^{-1}}!}\Big(
-\frac{q^{-1}}{(1-q^{-1})[\infty]_{q^{-1}}!}\sum_{i\geq 2} \frac{q^{-i}}{1-q^{-i}}
+ \frac{q^{-d-4}}{(1-q^{-1})^2(1-q^{-2})[\infty]_{q^{-1}}!} \Big) \\
=&\Big(
-\frac{2q^{-d-2}}{(1-q^{-1})^2 ([\infty]_{q^{-1}}!)^2} \sum_{i\geq 2} \frac{q^{-i}}{1-q^{-i}}
+ \frac{2q^{-2d-5}}{(1-q^{-1})^3(1-q^{-2})([\infty]_{q^{-1}}!)^2}
\Big) \Big|_{q^{\geq -3d-8}}.
\end{split}
\end{equation*}
Thus the conclusion holds.
\end{proof}

\begin{Lemma}
\label{lem-W3dleading2}
The following identity holds:
\begin{equation*}
\begin{split}
&\sum_{d_1+d_3 =d}
\sum_{k=d_1+1}^\infty q^{-k}\big( \frac{1}{[k]_{q^{-1}}!} - \frac{1}{[\infty]_{q^{-1}}!} \big)
\sum_{l=d_3+1}^\infty q^{-l}\big( \frac{1}{[l]_{q^{-1}}!} - \frac{1}{[\infty]_{q^{-1}}!} \big) \\
=& \frac{1}{([\infty]_{q^{-1}}!)^2} \Big(
\frac{2q^{-2d-8}}{[1]_{q^{-1}}^2 [2]_{q^{-1}}^3}
+ \frac{(d+1)q^{-2d-6}}{[1]_{q^{-1}}^2 [2]_{q^{-1}}^2}
-\frac{2 q^{-2d-6}}{[1]_{q^{-1}}^2 [2]_{q^{-1}}^2} \sum_{i\geq 2}\frac{q^{-i}}{1-q^{-i}}
\Big)|_{q^{\geq -3d-9}}.
\end{split}
\end{equation*}
\end{Lemma}
\begin{proof}
By \eqref{eq-sum1overkfac-infty} one can rewrite the left-hand side (L.H.S.) of the above identity as:
\begin{equation*}
\begin{split}
\text{L.H.S.}
=& \sum_{d_1+d_3 =d}
\sum_{k=d_1+1}^\infty \sum_{n>k} q^{-k-n} \cdot \frac{1}{[n]_{q^{-1}}!}
\sum_{l=d_3+1}^\infty \sum_{m>l} q^{-l-m} \cdot \frac{1}{[m]_{q^{-1}}!}.
\end{split}
\end{equation*}
Replacing $\frac{1}{[n]_{q^{-1}}!}$ and $\frac{1}{[m]_{q^{-1}}!}$ by
\begin{equation*}
(\frac{1}{[n]_{q^{-1}}!} - \frac{1}{[\infty]_{q^{-1}}!})+\frac{1}{[\infty]_{q^{-1}}!},
\qquad\qquad
(\frac{1}{[m]_{q^{-1}}!} - \frac{1}{[\infty]_{q^{-1}}!})+\frac{1}{[\infty]_{q^{-1}}!}
\end{equation*}
respectively, we obtain:
\begin{equation*}
\begin{split}
\text{L.H.S.}
= &  \sum_{d_1+d_3 = d}
\sum_{k=d_1+1}^\infty \sum_{n>k} q^{-k-n} (\frac{1}{[n]_{q^{-1}}!}- \frac{1}{[\infty]_{q^{-1}}!})
\sum_{l=d_3+1}^\infty \sum_{m>l} q^{-l-m} (\frac{1}{[m]_{q^{-1}}!}- \frac{1}{[\infty]_{q^{-1}}!}) \\
&+2 \sum_{d_1=0}^d \sum_{k=d_1+1}^\infty \sum_{n>k} q^{-k-n} (\frac{1}{[n]_{q^{-1}}!}- \frac{1}{[\infty]_{q^{-1}}!})
\cdot \sum_{l=d-d_1+1}^\infty\sum_{m>l} \frac{q^{-m-l}}{[\infty]_{q^{-1}}!}
\\
&+ \sum_{d_1+d_3 = d}\sum_{k=d_1+1}^\infty \sum_{l=d_3+1}^\infty \frac{q^{-2k-2l-2}}{(1-q^{-1})^2 ([\infty]_{q^{-1}}!)^2},
\end{split}
\end{equation*}
or more explicitly,
\begin{equation*}
\begin{split}
\text{L.H.S.}
= &  \sum_{d_1+d_3 = d}
\sum_{k=d_1+1}^\infty \sum_{n>k} q^{-k-n} (\frac{1}{[n]_{q^{-1}}!}- \frac{1}{[\infty]_{q^{-1}}!})
\sum_{l=d_3+1}^\infty \sum_{m>l} q^{-l-m} (\frac{1}{[m]_{q^{-1}}!}- \frac{1}{[\infty]_{q^{-1}}!}) \\
&+\frac{2 q^{-2d-3}}{[2]_{q^{-1}}! [\infty]_{q^{-1}}!} \sum_{d_1=0}^d \sum_{k=d_1+1}^\infty \sum_{n>k} q^{2d_1-k-n}
(\frac{1}{[n]_{q^{-1}}!}- \frac{1}{[\infty]_{q^{-1}}!})
+ \frac{(d+1)q^{-2d-6}}{([2]_{q^{-1}}!)^2  ([\infty]_{q^{-1}}!)^2}.
\end{split}
\end{equation*}
The first summand is of degree:
\begin{equation*}
\begin{split}
\deg\Big( q^{-k-n} (\frac{1}{[n]_{q^{-1}}!}- \frac{1}{[\infty]_{q^{-1}}!})
 q^{-l-m} (\frac{1}{[m]_{q^{-1}}!}- \frac{1}{[\infty]_{q^{-1}}!})\Big)
 =&-k-2n-l-2m-2.
\end{split}
\end{equation*}
Here $-k-2n-l-2m-2 < -3d-11$, and thus:
\begin{equation*}
\begin{split}
\text{L.H.S.}
=
&\Big(
\frac{2 q^{-2d-3}}{[2]_{q^{-1}}! [\infty]_{q^{-1}}!} \sum_{d_1=0}^d \sum_{k>d_1} \sum_{n>k} q^{2d_1-k-n}
(\frac{1}{[n]_{q^{-1}}!}- \frac{1}{[\infty]_{q^{-1}}!}) \\
& + \frac{(d+1)q^{-2d-6}}{ ([2]_{q^{-1}}![\infty]_{q^{-1}}!)^2}
\Big)\Big|_{q^{\geq -3d-11}}.
\end{split}
\end{equation*}
Moreover,
notice that
\begin{equation*}
\begin{split}
& \sum_{d_1=0}^d \sum_{k>d_1} \sum_{n>k} q^{2d_1-k-n}
(\frac{1}{[n]_{q^{-1}}!}- \frac{1}{[\infty]_{q^{-1}}!}) \\
=& \sum_{d_1=0}^\infty \sum_{k>d_1} \sum_{n>k} q^{2d_1-k-n}
(\frac{1}{[n]_{q^{-1}}!}- \frac{1}{[\infty]_{q^{-1}}!})
- \sum_{d_1>d} \sum_{k>d_1} \sum_{n>k} q^{2d_1-k-n}
(\frac{1}{[n]_{q^{-1}}!}- \frac{1}{[\infty]_{q^{-1}}!}),
\end{split}
\end{equation*}
and for $d_1>d$ one has:
\begin{equation*}
\deg q^{2d_1-k-n}
(\frac{1}{[n]_{q^{-1}}!}- \frac{1}{[\infty]_{q^{-1}}!}) =2d_1-k-2n -1 < -d-6,
\end{equation*}
and so:
\be
\deg \Big(
\frac{2 q^{-2d-3}}{[2]_{q^{-1}}! [\infty]_{q^{-1}}!} \cdot
\sum_{d_1>d} \sum_{k>d_1} \sum_{n>k} q^{2d_1-k-n}
(\frac{1}{[n]_{q^{-1}}!}- \frac{1}{[\infty]_{q^{-1}}!}) \Big)
< -3d-9.
\ee
Then we may conclude that:
\begin{equation*}
\begin{split}
\text{L.H.S.}
=
&\Big(
\frac{2 q^{-2d-3}}{[2]_{q^{-1}}! [\infty]_{q^{-1}}!}
\sum_{d_1=0}^\infty \sum_{k>d_1} \sum_{n>k} q^{2d_1-k-n}
(\frac{1}{[n]_{q^{-1}}!}- \frac{1}{[\infty]_{q^{-1}}!}) \\
& + \frac{(d+1)q^{-2d-6}}{ ([2]_{q^{-1}}! [\infty]_{q^{-1}}!)^2}
\Big)\Big|_{q^{\geq -3d-9}}.
\end{split}
\end{equation*}
Then the conclusion follows from the next lemma.
\end{proof}

\begin{Lemma}
\label{lem-eqlem-3sum}
The following identity holds:
\be
\label{eq-lem-3sum}
\begin{split}
& \sum_{d_1=0}^\infty \sum_{k>d_1} \sum_{n>k} q^{2d_1-k-n}
(\frac{1}{[n]_{q^{-1}}!}- \frac{1}{[\infty]_{q^{-1}}!}) \\
=& \frac{1}{[\infty]_{q^{-1}}!}\Big(
\frac{q^{-5}}{(1-q^{-1})(1-q^{-2})^2}
-\frac{q^{-3}}{(1-q^{-1})(1-q^{-2})} \sum_{i=2}^\infty \frac{q^{-i}}{1-q^{-i}}
\Big).
\end{split}
\ee
\end{Lemma}
\begin{proof}
The left-hand side (L.H.S.) of \eqref{eq-lem-3sum} can be rewritten as:
\begin{equation*}
\begin{split}
\text{L.H.S.}
=& \sum_{n=2}^\infty (\frac{1}{[n]_{q^{-1}}!}- \frac{1}{[\infty]_{q^{-1}}!})
\sum_{d_1 = 0}^{n-2} \sum_{k=d_1+1}^{n-1} q^{2d_1-k-n} \\
= & \frac{q^{-3}}{(1-q^{-1})(1-q^{-2})}
\sum_{n=2}^\infty (1-q^{-n} -q^{-n+1} +q^{-2n+1}) (\frac{1}{[n]_{q^{-1}}!}- \frac{1}{[\infty]_{q^{-1}}!})\\
=& \frac{q^{-3}}{(1-q^{-1})(1-q^{-2})}
\Big( \sum_{n=2}^\infty \frac{q^{-2n+1} -q^{-n+1}(1+q^{-1})}{[n]_{q^{-1}}!} \\
&\qquad\qquad\qquad
+\sum_{n=2}^\infty (\frac{1}{[n]_{q^{-1}}!}- \frac{1}{[\infty]_{q^{-1}}!})
+\frac{q^{-1} +2q^{-2}}{1-q^{-2}}\cdot\frac{1}{[\infty]_{q^{-1}}!} \Big).
\end{split}
\end{equation*}
By \eqref{eq-pochhammerseries} one has:
\begin{equation*}
\begin{split}
\sum_{n=2}^\infty \frac{q^{-2n+1}}{[n]_{q^{-1}}!} =
q \cdot \frac{1}{(q^{-2} ; q^{-1})_\infty} -q-\frac{q^{-1}}{[1]_{q^{-1}}!}
= \frac{q(1-q^{-1})}{[\infty]_{q^{-1}}!} -q -\frac{q^{-1}}{1-q^{-1}},
\end{split}
\end{equation*}
and similarly,
\begin{equation*}
\begin{split}
\sum_{n=2}^\infty \frac{q^{-n+1}}{[n]_{q^{-1}}!} =
q\cdot \frac{1}{(q^{-1};q^{-1})_\infty} -q -\frac{1}{[1]_{q^{-1}}!}
= \frac{q}{[\infty]_{q^{-1}}!} -q -\frac{1}{1-q^{-1}}.
\end{split}
\end{equation*}
Moreover,
Guo--Zhou have proved the following identity
(see \cite[Prop. 5.8]{gz}):
\be
\label{eq-gz-prop5.8}
\sum_{n=0}^\infty (\frac{1}{[n]_{q^{-1}}!}- \frac{1}{[\infty]_{q^{-1}}!})
= -\frac{1}{[\infty]_{q^{-1}}!} \cdot \sum_{i\geq 1} \frac{q^{-i}}{1-q^{-i}},
\ee
which implies
\begin{equation*}
\sum_{n=2}^\infty (\frac{1}{[n]_{q^{-1}}!}- \frac{1}{[\infty]_{q^{-1}}!})
= \frac{1}{[\infty]_{q^{-1}}!}
\Big( 2-  \sum_{i\geq 1} \frac{q^{-i}}{1-q^{-i}} \Big)
-1-\frac{1}{1-q^{-1}}.
\end{equation*}
Now we can conclude that:
\begin{equation*}
\begin{split}
& \sum_{d_1=0}^\infty \sum_{k>d_1} \sum_{n>k} q^{2d_1-k-n}
(\frac{1}{[n]_{q^{-1}}!}- \frac{1}{[\infty]_{q^{-1}}!}) \\
=& \frac{q^{-3}}{(1-q^{-1})(1-q^{-2})}
\Big( \frac{q(1-q^{-1})}{[\infty]_{q^{-1}}!} -q -\frac{q^{-1}}{1-q^{-1}}
-(1+q^{-1})( \frac{q}{[\infty]_{q^{-1}}!} -q -\frac{1}{1-q^{-1}}) \\
&\qquad\qquad\qquad +\frac{1}{[\infty]_{q^{-1}}!}
\Big( 2 -  \sum_{i\geq 1} \frac{q^{-i}}{1-q^{-i}} \Big)
- 1-\frac{1}{1-q^{-1}}
+\frac{q^{-1} +2q^{-2}}{1-q^{-2}}\cdot\frac{1}{[\infty]_{q^{-1}}!} \Big) \\
=& \frac{q^{-4} (2q^{-1} +1)}{(1-q^{-1})(1-q^{-2})^2 \cdot [\infty]_{q^{-1}}!}
-  \frac{q^{-3}}{(1-q^{-1})(1-q^{-2})} \cdot
\frac{1}{[\infty]_{q^{-1}}!} \cdot \sum_{i\geq 1} \frac{q^{-i}}{1-q^{-i}},
\end{split}
\end{equation*}
and this is equivalent to \eqref{eq-lem-3sum}.
\end{proof}

Now we plug Lemma \ref{lem-W3dleading1} and Lemma \ref{lem-W3dleading2} into
the identity \eqref{eq-W3dcase1-gzeq},
and in this way we obtain the conclusion in the case $d_2 = 0$ in $W_d^3(q)$:
\begin{Proposition}
\label{prop-w3dcase1}
The following identity holds:
\begin{equation*}
\begin{split}
&\sum_{d_1+d_3=d}T_{0}^x\Big( \frac{(q^{-d_1 -1} ;q^{-1})_\infty}{(xq^{-d_1 -1} ;q^{-1})_\infty} -1 \Big)
 T_{d_3}^x\Big( \frac{(q^{ -1} ;q^{-1})_\infty}{(xq^{ -1} ;q^{-1})_\infty} -1 \Big) \\
=& \bigg(
\frac{2q^{-2d-5}}{[1]_{q^{-1}}^3 [2]_{q^{-1}}} + \frac{2q^{-2d-8}}{[1]_{q^{-1}}^2 [2]_{q^{-1}}^3}
+ \frac{(d+1)q^{-d-2}}{[1]_{q^{-1}}^2} + \frac{(d+1)q^{-2d-6}}{([2]_{q^{-1}}!)^2} \\
& - \Big(\frac{2q^{-d-2}}{[1]_{q^{-1}}^2} + \frac{2q^{-2d-6}}{([2]_{q^{-1}}!)^2}\Big)
\sum_{i\geq 2}\frac{q^{-i}}{1-q^{-i}}
\bigg) \bigg|_{q^{\geq -3d-8}}.
\end{split}
\end{equation*}
\end{Proposition}

\subsection{Leading terms of $W_d^3(q)$ -- Case 2)}
\label{sec-app-w3dcase2}

Now consider the second case $d_2 = 1$ in $W_d^3(q)$.
By \eqref{eq-computed1=01} and \eqref{eq-computed2=01} one has:
\begin{equation*}
\begin{split}
&\sum_{d_1+d_3=d-1}T_{1}^x\Big( \frac{(q^{-d_1 -1} ;q^{-1})_\infty}{(xq^{-d_1 -1} ;q^{-1})_\infty} -1 \Big)
 T_{d_3}^x\Big( \frac{(q^{ -2} ;q^{-1})_\infty}{(xq^{ -2} ;q^{-1})_\infty} -1 \Big) \\
=& \sum_{d_1+d_3=d-1} \Big( (1+ \frac{q^{-(d_1+1)}}{[1]_{q^{-1}}!}) \frac{[\infty]_{q^{-1}}!}{[d_1]_{q^{-1}}!} -1 \Big)
\Big( (1+ \frac{q^{-(d_3+1)}}{[1]_{q^{-1}}!}) \frac{[\infty]_{q^{-1}}!}{[d_3]_{q^{-1}}!} -1 \Big).
\end{split}
\end{equation*}
Then by expanding the right-hand side one obtains:
\begin{equation*}
\begin{split}
\sum_{d_1+d_3=d-1}T_{1}^x\Big( \frac{(q^{-d_1 -1} ;q^{-1})_\infty}{(xq^{-d_1 -1} ;q^{-1})_\infty} -1 \Big)
 T_{d_3}^x\Big( \frac{(q^{ -2} ;q^{-1})_\infty}{(xq^{ -2} ;q^{-1})_\infty} -1 \Big)
= \text{(I)} + \text{(II)} + \text{(III)},
\end{split}
\end{equation*}
where we denote:
\be
\label{eq-W3d-IIIIII}
\begin{split}
& \text{(I)} = \sum_{d_1+d_3=d-1}  \Big( \frac{[\infty]_{q^{-1}}!}{[d_1]_{q^{-1}}!} -1 \Big)
\Big( \frac{[\infty]_{q^{-1}}!}{[d_3]_{q^{-1}}!} -1 \Big);\\
& \text{(II)} = 2\cdot \sum_{d_1 +d_3 = d-1} \frac{q^{-d_1-1}}{[1]_{q^{-1}}} \cdot
\frac{[\infty]_{q^{-1}}!}{[d_1]_{q^{-1}}!}
\Big( \frac{[\infty]_{q^{-1}}!}{[d_3]_{q^{-1}}!} -1 \Big); \\
& \text{(III)} = \sum_{d_1+d_3=d-1} \frac{q^{-d-1}}{[1]_{q^{-1}}^2} \cdot
\frac{[\infty]_{q^{-1}}!}{[d_1]_{q^{-1}}!}\cdot \frac{[\infty]_{q^{-1}}!}{[d_3]_{q^{-1}}!}.
\end{split}
\ee

First by the conclusion of Case 1) we already know that:
\be
\label{eq-W3d-Iresult}
\begin{split}
\text{(I)} = & \sum_{d_1+d_3=d-1}
T_{0}^x\Big( \frac{(q^{-d_1 -1} ;q^{-1})_\infty}{(xq^{-d_1 -1} ;q^{-1})_\infty} -1 \Big)
 T_{d_3}^x\Big( \frac{(q^{ -1} ;q^{-1})_\infty}{(xq^{ -1} ;q^{-1})_\infty} -1 \Big) \\
=& \bigg(
\frac{2q^{-2d-3}}{[1]_{q^{-1}}^3 [2]_{q^{-1}}} + \frac{2q^{-2d-6}}{[1]_{q^{-1}}^2 [2]_{q^{-1}}^3}
+ \frac{d\cdot q^{-d-1}}{[1]_{q^{-1}}^2} + \frac{d\cdot q^{-2d-4}}{([2]_{q^{-1}}!)^2} \\
& - \Big(\frac{2q^{-d-1}}{[1]_{q^{-1}}^2} + \frac{2q^{-2d-4}}{([2]_{q^{-1}}!)^2}\Big)
\sum_{i\geq 2}\frac{q^{-i}}{1-q^{-i}}
\bigg) \bigg|_{q^{\geq -3d-5}}.
\end{split}
\ee

Now consider $\text{(II)}$.
By \eqref{eq-sum1overkfac-infty} one has:
\begin{equation*}
\begin{split}
\text{(II)} = & \frac{2 ([\infty]_{q^{-1}}!)^2}{[1]_{q^{-1}}}
\sum_{d_1+d_3 = d-1}  \frac{q^{-d_1-1}}{[d_1]_{q^{-1}}!} \cdot
(\frac{1}{[d_3]_{q^{-1}}!} - \frac{1}{[\infty]_{q^{-1}}!}) \\
=& -\frac{2 ([\infty]_{q^{-1}}!)^2}{[1]_{q^{-1}}}
\sum_{d_1+d_3 = d-1}  \frac{q^{-d_1-1}}{[d_1]_{q^{-1}}!} \cdot
\sum_{n>d_3}
\frac{q^{-n}}{[n]_{q^{-1}}!}.
\end{split}
\end{equation*}
Similar to the computations in Case 1),
we decompose $\frac{1}{[d_1]_{q^{-1}}!}$ and $\frac{1}{[n]_{q^{-1}}!}$ as follows:
\begin{equation*}
\begin{split}
&\frac{1}{[d_1]_{q^{-1}}!} = (\frac{1}{[d_1]_{q^{-1}}!} - \frac{1}{[\infty]_{q^{-1}}!})
+ \frac{1}{[\infty]_{q^{-1}}!}, \qquad
\frac{1}{[n]_{q^{-1}}!} = (\frac{1}{[n
]_{q^{-1}}!} - \frac{1}{[\infty]_{q^{-1}}!})
+ \frac{1}{[\infty]_{q^{-1}}!},
\end{split}
\end{equation*}
and in this way we obtain:
\be
\label{eq-W3d-IIsum}
\begin{split}
\text{(II)}
=& -\frac{2 ([\infty]_{q^{-1}}!)^2}{[1]_{q^{-1}}}
\sum_{d_1+d_3 = d-1} \sum_{n>d_3} q^{-d_1-n-1}
\big(\frac{1}{[d_1]_{q^{-1}}!} - \frac{1}{[\infty]_{q^{-1}}!}\big)
\big(\frac{1}{[n]_{q^{-1}}!} - \frac{1}{[\infty]_{q^{-1}}!}\big) \\
&- \frac{2 \cdot [\infty]_{q^{-1}}! }{[1]_{q^{-1}}}
\sum_{d_1+d_3 = d-1} \sum_{n>d_3} q^{-d_1-n-1} \big(\frac{1}{[n]_{q^{-1}}!} - \frac{1}{[\infty]_{q^{-1}}!}\big)\\
& - \frac{2 \cdot [\infty]_{q^{-1}}! }{[1]_{q^{-1}}}
\sum_{d_1+d_3 = d-1} \sum_{n>d_3} q^{-d_1-n-1} \big(\frac{1}{[d_1]_{q^{-1}}!} - \frac{1}{[\infty]_{q^{-1}}!}\big)
-\frac{2}{[1]_{q^{-1}}}\cdot\frac{d\cdot q^{-d-1}}{1-q^{-1}}.
\end{split}
\ee
We compute the three summations in the right-hand side in the following three lemmas.

\begin{Lemma}
The following identity holds:
\begin{equation*}
\begin{split}
&\sum_{d_1+d_3 = d-1} \sum_{n>d_3} q^{-d_1-n-1}
\big(\frac{1}{[d_1]_{q^{-1}}!} - \frac{1}{[\infty]_{q^{-1}}!}\big)
\big(\frac{1}{[n]_{q^{-1}}!} - \frac{1}{[\infty]_{q^{-1}}!}\big) \\
=& \frac{1}{([\infty]_{q^{-1}}!)^2}\Big(
\frac{q^{-2d-5}}{([2]_{q^{-1}}!)^2}
+\frac{d\cdot q^{-2d-3}}{[1]_{q^{-1}}^2  [2]_{q^{-1}}}
- \frac{2 q^{-2d-3}}{[1]_{q^{-1}}^2  [2]_{q^{-1}}}\sum_{i= 2}^\infty \frac{q^{-i}}{1-q^{-i}}
\Big) \Big|_{q^{\geq -3d+1}}.
\end{split}
\end{equation*}
\end{Lemma}
\begin{proof}
By \eqref{eq-sum1overkfac-infty} we can rewrite the left-hand side (L.H.S.) as:
\begin{equation*}
\begin{split}
\text{L.H.S.} =&
\sum_{d_1+d_3 = d-1} \sum_{n>d_3} q^{-d_1-n-1} \sum_{k>d_1}\sum_{l>n}
\frac{q^{-k}}{[k]_{q^{-1}}!} \cdot \frac{q^{-l}}{[l]_{q^{-1}}!} \\
=& \sum_{d_1+d_3 = d-1} \sum_{n>d_3} \sum_{k>d_1}\sum_{l>n} q^{-d_1-n-k-l-1}\Big(
\big( \frac{1}{[k]_{q^{-1}}!} - \frac{1}{[\infty]_{q^{-1}}!} \big)
\big( \frac{1}{[l]_{q^{-1}}!} - \frac{1}{[\infty]_{q^{-1}}!} \big) \\
&+ \frac{1}{[\infty]_{q^{-1}}!} \big( \frac{1}{[k]_{q^{-1}}!} - \frac{1}{[\infty]_{q^{-1}}!} \big)
+ \frac{1}{[\infty]_{q^{-1}}!} \big( \frac{1}{[l]_{q^{-1}}!} - \frac{1}{[\infty]_{q^{-1}}!} \big)
+ \frac{1}{([\infty]_{q^{-1}}!)^2} \Big).
\end{split}
\end{equation*}
Here by \eqref{eq-gzlem-5.7} one has
\begin{equation*}
\deg q^{-d_1-n-k-l-1} \big( \frac{1}{[k]_{q^{-1}}!} - \frac{1}{[\infty]_{q^{-1}}!} \big)
\big( \frac{1}{[l]_{q^{-1}}!} - \frac{1}{[\infty]_{q^{-1}}!} \big)
<-3d,
\end{equation*}
and thus:
\be
\label{eq-pfW3dcase2-1}
\begin{split}
\text{L.H.S.}
=& \bigg( \sum_{d_1+d_3 = d-1} \sum_{n>d_3} \sum_{k>d_1}\sum_{l>n} \frac{q^{-d_1-n-k-l-1}}{[\infty]_{q^{-1}}!}\Big(
\big( \frac{1}{[k]_{q^{-1}}!} - \frac{1}{[\infty]_{q^{-1}}!} \big) \\
&+ \big( \frac{1}{[l]_{q^{-1}}!} - \frac{1}{[\infty]_{q^{-1}}!} \big) \Big)
+ \frac{d\cdot q^{-2d-3}}{[1]_{q^{-1}}^2 \cdot[2]_{q^{-1}} \cdot([\infty]_{q^{-1}}!)^2}
\bigg)\bigg|_{q^{\geq -3d+1}}.
\end{split}
\ee
Moreover,
one has:
\begin{equation*}
\begin{split}
& \sum_{d_1+d_3 = d-1} \sum_{n>d_3} \sum_{k>d_1}\sum_{l>n}
\frac{q^{-d_1-n-k-l-1}}{[\infty]_{q^{-1}}!}
\big( \frac{1}{[k]_{q^{-1}}!} - \frac{1}{[\infty]_{q^{-1}}!} \big) \\
=& \Big(\sum_{d_1=0}^\infty -\sum_{d_1=d}^\infty \Big)
\sum_{n>d-1-d_1} \sum_{k>d_1}\sum_{l>n}
\frac{q^{-d_1-n-k-l-1}}{[\infty]_{q^{-1}}!}
\big( \frac{1}{[k]_{q^{-1}}!} - \frac{1}{[\infty]_{q^{-1}}!} \big). \\
\end{split}
\end{equation*}
and when $d_1 \geq d$ one has:
\begin{equation*}
\deg \frac{q^{-d_1-n-k-l-1}}{[\infty]_{q^{-1}}!}
\big( \frac{1}{[k]_{q^{-1}}!} - \frac{1}{[\infty]_{q^{-1}}!} \big) <-3d,
\end{equation*}
therefore:
\begin{equation*}
\begin{split}
& \sum_{d_1+d_3 = d-1} \sum_{n>d_3} \sum_{k>d_1}\sum_{l>n}
\frac{q^{-d_1-n-k-l-1}}{[\infty]_{q^{-1}}!}
\big( \frac{1}{[k]_{q^{-1}}!} - \frac{1}{[\infty]_{q^{-1}}!} \big) \\
=& \sum_{d_1 = 0}^\infty
\sum_{n > d-1-d_1} \sum_{k>d_1}\sum_{l>n}
\frac{q^{-d_1-n-k-l-1}}{[\infty]_{q^{-1}}!}
\big( \frac{1}{[k]_{q^{-1}}!} - \frac{1}{[\infty]_{q^{-1}}!} \big) \\
= & \frac{q^{-2d-2}}{[2]_{q^{-1}}! \cdot [\infty]_{q^{-1}}!}\cdot
\sum_{d_1=0}^\infty \sum_{k>d_1} q^{d_1-k}
\big( \frac{1}{[k]_{q^{-1}}!} - \frac{1}{[\infty]_{q^{-1}}!} \big)
\Big|_{q^{\geq -3d+1}},
\end{split}
\end{equation*}
and by \eqref{eq-gzlem6.6} one obtains:
\be
\label{eq-pfW3dcase2-2}
\begin{split}
& \sum_{d_1+d_3 = d-1} \sum_{n>d_3} \sum_{k>d_1}\sum_{l>n}
\frac{q^{-d_1-n-k-l-1}}{[\infty]_{q^{-1}}!}
\big( \frac{1}{[k]_{q^{-1}}!} - \frac{1}{[\infty]_{q^{-1}}!} \big) \\
= & \Big( -\frac{q^{-2d-3}}{[1]_{q^{-1}}^2 \cdot [2]_{q^{-1}} \cdot ([\infty]_{q^{-1}}!)^2}\cdot
\sum_{i\geq 2} \frac{q^{-i}}{1-q^{-i}} \Big)
\Big|_{q^{\geq -3d+1}}.
\end{split}
\ee

Next we compute:
\begin{equation*}
\begin{split}
& \sum_{d_1+d_3 = d-1} \sum_{n>d_3} \sum_{k>d_1}\sum_{l>n}
\frac{q^{-d_1-n-k-l-1}}{[\infty]_{q^{-1}}!}
\big( \frac{1}{[l]_{q^{-1}}!} - \frac{1}{[\infty]_{q^{-1}}!} \big) \\
=& \sum_{d_3 = 0}^{d-1} \sum_{n>d_3} \sum_{k>d_1}\sum_{l>n}
\frac{q^{-d +d_3 -n-k-l}}{[\infty]_{q^{-1}}!}
\big( \frac{1}{[l]_{q^{-1}}!} - \frac{1}{[\infty]_{q^{-1}}!} \big) \\
=& \Big( \sum_{d_3=0}^\infty -\sum_{d_3 = d}^\infty \Big)
\sum_{n>d_3} \sum_{k>d-1-d_3}\sum_{l>n}
\frac{q^{-d +d_3 -n-k-l}}{[\infty]_{q^{-1}}!}
\big( \frac{1}{[l]_{q^{-1}}!} - \frac{1}{[\infty]_{q^{-1}}!} \big).
\end{split}
\end{equation*}
When $d_3 \geq d$ one has:
\begin{equation*}
\deg \frac{q^{-d +d_3 -n-k-l}}{[\infty]_{q^{-1}}!}
\big( \frac{1}{[l]_{q^{-1}}!} - \frac{1}{[\infty]_{q^{-1}}!} \big) <-3d,
\end{equation*}
and thus:
\begin{equation*}
\begin{split}
& \sum_{d_1+d_3 = d-1} \sum_{n>d_3} \sum_{k>d_1}\sum_{l>n}
\frac{q^{-d_1-n-k-l-1}}{[\infty]_{q^{-1}}!}
\big( \frac{1}{[l]_{q^{-1}}!} - \frac{1}{[\infty]_{q^{-1}}!} \big) \\
=& \sum_{d_3=0}^\infty
\sum_{n>d_3} \sum_{k>d-1-d_3} \sum_{l>n}
\frac{q^{-d +d_3 -n-k-l}}{[\infty]_{q^{-1}}!}
\big( \frac{1}{[l]_{q^{-1}}!} - \frac{1}{[\infty]_{q^{-1}}!} \big) \\
= &\frac{q^{-2d}}{[1]_{q^{-1}}\cdot [\infty]_{q^{-1}}!}
\sum_{d_3=0}^\infty  \sum_{n>d_3} \sum_{l>n}
q^{2d_3 - l -n} \big( \frac{1}{[l]_{q^{-1}}!} - \frac{1}{[\infty]_{q^{-1}}!} \big)
\Big|_{q^{\geq -3d+1}}.
\end{split}
\end{equation*}
And by Lemma \ref{lem-eqlem-3sum} and Proposition \ref{prop-degLaurentprod} one obtains:
\begin{equation*}
\begin{split}
& \sum_{d_1+d_3 = d-1} \sum_{n>d_3} \sum_{k>d_1}\sum_{l>n}
\frac{q^{-d_1-n-k-l-1}}{[\infty]_{q^{-1}}!}
\big( \frac{1}{[l]_{q^{-1}}!} - \frac{1}{[\infty]_{q^{-1}}!} \big) \\
= &\frac{1}{[1]_{q^{-1}} \cdot ([\infty]_{q^{-1}}!)^2}
\Big(
\frac{q^{-2d -5}}{[1]_{q^{-1}} \cdot [2]_{q^{-1}}^2}
-\frac{q^{-2d -3}}{[2]_{q^{-1}}!} \sum_{i=2}^\infty \frac{q^{-i}}{1-q^{-i}}
\Big)
\Big|_{q^{\geq -3d+1}}.
\end{split}
\end{equation*}
Now plug this identity and \eqref{eq-pfW3dcase2-2} into \eqref{eq-pfW3dcase2-1},
and we have proved the lemma.
\end{proof}

\begin{Lemma}
The following identity holds:
\begin{equation*}
\begin{split}
&\sum_{d_1+d_3 = d-1} \sum_{n>d_3} q^{-d_1-n-1} \big(\frac{1}{[n]_{q^{-1}}!} - \frac{1}{[\infty]_{q^{-1}}!}\big) \\
=& \frac{1}{[\infty]_{q^{-1}}!} \Big(
\frac{q^{-2d-3}}{[1]_{q^{-1}}^2 [2]_{q^{-1}} }
- \frac{q^{-d-1}}{[1]_{q^{-1}}} \sum_{i\geq 2}\frac{q^{-i}}{1-q^{-i}}
\Big)\Big|_{q^{\geq -3d+1}}.
\end{split}
\end{equation*}
\end{Lemma}
\begin{proof}
By \eqref{eq-sum1overkfac-infty}  the left-hand side (L.H.S.) of the above identity equals to:
\begin{equation*}
\begin{split}
\text{L.H.S.} =&
- \sum_{d_1+d_3 = d-1} \sum_{n>d_3} \sum_{k>n}
q^{-d_1-n-1} \cdot \frac{q^{-k}}{[k]_{q^{-1}}!} \\
=& - \sum_{d_1+d_3 = d-1} \sum_{n>d_3} \sum_{k>n}
q^{-d_1-n-k- 1} \big( \frac{1}{[k]_{q^{-1}}!} - \frac{1}{[\infty]_{q^{-1}}!} \big)
- \frac{q^{-d-3}-q^{-2d-3}}{[1]_{q^{-1}}^2 [2]_{q^{-1}}\cdot [\infty]_{q^{-1}}!} \\
=& \Big( \sum_{d_3=d}^\infty -\sum_{d_3=0}^\infty \Big)
\sum_{n>d_3} \sum_{k>n}
q^{-d + d_3 -n-k } \big( \frac{1}{[k]_{q^{-1}}!} - \frac{1}{[\infty]_{q^{-1}}!} \big)
-\frac{q^{-d-3}-q^{-2d-3}}{[1]_{q^{-1}}^2 [2]_{q^{-1}}\cdot [\infty]_{q^{-1}}!}.
\end{split}
\end{equation*}
When $d_3\geq d$ one has:
\begin{equation*}
\begin{split}
\deg q^{-d + d_3 -n-k } \big( \frac{1}{[k]_{q^{-1}}!} - \frac{1}{[\infty]_{q^{-1}}!} \big)
<-3d,
\end{split}
\end{equation*}
and thus:
\be
\label{eq-lempfw3xyz}
\begin{split}
\text{L.H.S.} =&
- \sum_{d_3=0}^\infty
\sum_{n>d_3} \sum_{k>n}
q^{-d + d_3 -n-k } \big( \frac{1}{[k]_{q^{-1}}!} - \frac{1}{[\infty]_{q^{-1}}!} \big)
 -\frac{q^{-d-3}-q^{-2d-3}}{[1]_{q^{-1}}^2 [2]_{q^{-1}} [\infty]_{q^{-1}}!} \\
=& \Big( - \sum_{d_3=0}^\infty
\sum_{n>d_3} \sum_{k>n}
 \frac{q^{-d + d_3 -n-k }}{[k]_{q^{-1}}!}
+\frac{q^{-2d-3}}{[1]_{q^{-1}}^2 [2]_{q^{-1}} [\infty]_{q^{-1}}!}
\Big) \Big|_{q^{\geq -3d+1}}
\end{split}
\ee
Moreover,
one has:
\be
\label{eq-lempfw3ooo}
\begin{split}
\sum_{d_3=0}^\infty
\sum_{n>d_3} \sum_{k>n}
\frac{q^{-d + d_3 -n-k }}{[k]_{q^{-1}}!}
=& \sum_{k=2}^\infty \frac{1}{[k]_{q^{-1}}!}
\sum_{n=1}^{k-1} \sum_{d_3 = 0}^{n-1} q^{-d + d_3 -n-k } \\
=& \frac{q^{-d}}{[1]_{q^{-1}}^2} \sum_{k=2}^\infty
\big(q^{-2k-1} -kq^{-k-2} +(k-1)q^{-k-1}\big) \frac{1}{[k]_{q^{-1}}!}.
\end{split}
\ee
Apply $\frac{d}{dz}$ to the identity \eqref{eq-pochhammerseries},
and we obtain:
\be
\label{eq-derivsumpoch}
\sum_{k\geq 1}\frac{kz^{k-1}}{[k]_{q^{-1}}!} =
\frac{1}{(z;q^{-1})_\infty} \cdot
\sum_{k=0}^\infty \frac{q^{-k}}{1-zq^{-k}}.
\ee
Then the following two identities hold:
\be
\label{eq-twoidgen}
\begin{split}
\sum_{k\geq 0}\frac{q^{-2k}}{[k]_{q^{-1}}!}
=\frac{[1]_{q^{-1}}}{[\infty]_{q^{-1}}!}; \qquad\qquad
\sum_{k\geq 1}\frac{kq^{-k}}{[k]_{q^{-1}}!} =
\frac{1}{[\infty]_{q^{-1}}!} \cdot
\sum_{k=1}^\infty \frac{q^{-k}}{1-q^{-k}}.
\end{split}
\ee
Plug these into \eqref{eq-lempfw3ooo} and we obtain:
\begin{equation*}
\begin{split}
\sum_{d_3=0}^\infty
\sum_{n>d_3} \sum_{k>n}
\frac{q^{-d + d_3 -n-k }}{[k]_{q^{-1}}!}
=& \frac{q^{-d}}{[1]_{q^{-1}}^2}
\Big(
\big(\frac{q^{-1}[1]_{q^{-1}}}{[\infty]_{q^{-1}}!}-q^{-1}-\frac{q^{-3}}{[1]_{q^{-1}}} \big)
- \big(\frac{q^{-2}}{[\infty]_{q^{-1}}!} \sum_{k=1}^\infty \frac{q^{-k}}{1-q^{-k}} -\frac{q^{-3}}{[1]_{q^{-1}}} \big) \\
& + \big(\frac{q^{-1}}{[\infty]_{q^{-1}}!} \sum_{k=1}^\infty \frac{q^{-k}}{1-q^{-k}} -\frac{q^{-2}}{[1]_{q^{-1}}} \big)
- \big( \frac{q^{-1}}{[\infty]_{q^{-1}}!} -q^{-1} -\frac{q^{-2}}{[1]_{q^{-1}}} \big)
\Big) \\
=& \frac{q^{-d-1}}{ [1]_{q^{-1}} \cdot [\infty]_{q^{-1}}!}\cdot
\sum_{k=2}^\infty \frac{q^{-k}}{1-q^{-k}}.
\end{split}
\end{equation*}
Plug this into \eqref{eq-lempfw3xyz} and we finish the proof.
\end{proof}

\begin{Lemma}
The following identity holds:
\begin{equation*}
\begin{split}
&\sum_{d_1+d_3 = d-1} \sum_{n>d_3} q^{-d_1-n-1} \big(\frac{1}{[d_1]_{q^{-1}}!} - \frac{1}{[\infty]_{q^{-1}}!}\big)\\
=& \frac{1}{[\infty]_{q^{-1}}!} \Big(
\frac{q^{-2d-2}}{[1]_{q^{-1}}^3 } - \frac{q^{-d-1}}{[1]_{q^{-1}} }\cdot\sum_{i\geq 1}\frac{q^{-i}}{1-q^{-i}}
\Big) \Big|_{q^{\geq -3d+1}}.
\end{split}
\end{equation*}
\end{Lemma}
\begin{proof}
The left-hand side (L.H.S.) equals to:
\begin{equation*}
\begin{split}
\text{L.H.S.} =& -\sum_{d_1+d_3 = d-1} \sum_{n>d_3} q^{-d_1-n-1}
\sum_{k>d_1} \frac{q^{-k}}{[k]_{q^{-1}}!} \\
=& -\sum_{d_1 = 0}^{d-1} \sum_{n>d-1-d_1} \sum_{k>d_1}
q^{-d_1-n-k-1} \big(  \frac{1}{[k]_{q^{-1}}!} - \frac{1}{[\infty]_{q^{-1}}!} \big)
- \frac{q^{-d-2}-q^{-2d-2}}{[1]_{q^{-1}}^3  [\infty]_{q^{-1}}!} \\
=& \Big( \sum_{d_1 = d}^{\infty} - \sum_{d_1=0}^\infty \Big)
\sum_{n>d-1-d_1} \sum_{k>d_1}
q^{-d_1-n-k-1} \big(  \frac{1}{[k]_{q^{-1}}!} - \frac{1}{[\infty]_{q^{-1}}!} \big)
- \frac{q^{-d-2}-q^{-2d-2}}{[1]_{q^{-1}}^3  [\infty]_{q^{-1}}!}.
\end{split}
\end{equation*}
When $d_1\geq d$ one has:
\begin{equation*}
\deg q^{-d_1-n-k-1} \big(  \frac{1}{[k]_{q^{-1}}!} - \frac{1}{[\infty]_{q^{-1}}!} \big)
<-3d,
\end{equation*}
and thus:
\begin{equation*}
\begin{split}
\text{L.H.S.}
=& -\sum_{d_1 = 0}^{\infty}
\sum_{n>d-1-d_1} \sum_{k>d_1}
q^{-d_1-n-k-1} \big(  \frac{1}{[k]_{q^{-1}}!} - \frac{1}{[\infty]_{q^{-1}}!} \big)
- \frac{q^{-d-2}-q^{-2d-2}}{[1]_{q^{-1}}^3  [\infty]_{q^{-1}}!} \\
=& \Big( -\sum_{d_1 = 0}^{\infty}
\sum_{n>d-1-d_1} \sum_{k>d_1}
\frac{q^{-d_1-n-k-1}}{[k]_{q^{-1}}!}
+ \frac{q^{-2d-2}}{[1]_{q^{-1}}^3  [\infty]_{q^{-1}}!}
\Big)\Big|_{q^{\geq -3d+1}}.
\end{split}
\end{equation*}
Here
\begin{equation*}
\begin{split}
\sum_{d_1 =0}^\infty\sum_{n>d-1-d_1} \sum_{k>d_1}
\frac{q^{-d_1-n-k-1}}{[k]_{q^{-1}}!}
= \sum_{k=1}^\infty \frac{1}{[k]_{q^{-1}}!}
\sum_{d_1=0}^{k-1}\sum_{n>d-1-d_1} q^{-d_1-n-k-1}
= \frac{q^{-d-1}}{[1]_{q^{-1}}}
\sum_{k=1}^\infty \frac{kq^{-k}}{[k]_{q^{-1}}!},
\end{split}
\end{equation*}
and then by \eqref{eq-twoidgen} one obtains:
\begin{equation*}
\begin{split}
\sum_{d_1 =0}^\infty\sum_{n>d-1-d_1} \sum_{k>d_1}
\frac{q^{-d_1-n-k-1}}{[k]_{q^{-1}}!}
= \frac{q^{-d-1}}{[1]_{q^{-1}} \cdot [\infty]_{q^{-1}}!} \cdot
\sum_{i\geq 1}\frac{q^{-i}}{1-q^{-i}}.
\end{split}
\end{equation*}
This completes the proof.
\end{proof}

Plug the above three lemmas into \eqref{eq-W3d-IIsum},
and we obtain the following expression for $\text{(II)}$:
\be
\label{eq-W3d-IIsum-2}
\begin{split}
\text{(II)}
=&\Big( -2\big(
\frac{q^{-2d-5}}{[1]_{q^{-1}}^3  [2]_{q^{-1}}^2}
+ \frac{q^{-2d-2}}{[1]_{q^{-1}}^4}
+ \frac{(d+1)q^{-2d-3}}{[1]_{q^{-1}}^3  [2]_{q^{-1}}}
+\frac{dq^{-d-1}}{[1]_{q^{-1}}^2 } -\frac{q^{-d-2}}{[1]_{q^{-1}}^3} \big) \\
& +4 \big( \frac{q^{-2d-3}}{[1]_{q^{-1}}^3  [2]_{q^{-1}}} + \frac{q^{-d-1}}{[1]_{q^{-1}}^2}
\big) \sum_{i\geq 2}\frac{q^{-i}}{1-q^{-i}}
\Big) \Big|_{q^{\geq -3d+1}}.
\end{split}
\ee

Finally we compute the third summation $\text{(III)}$.
From the definition \eqref{eq-W3d-IIIIII} we already see that:
\begin{equation*}
\begin{split}
\text{(III)} =&  \frac{q^{-d-1}\cdot ([\infty]_{q^{-1}}!)^2}{[1]_{q^{-1}}^2}
\sum_{d_1+d_3=d-1}
\frac{1}{[d_1]_{q^{-1}}!}\cdot \frac{1}{[d_3]_{q^{-1}}!} \\
=&  \frac{q^{-d-1}\cdot ([\infty]_{q^{-1}}!)^2}{[1]_{q^{-1}}^2}
\Big( \sum_{d_1+d_3=d-1}
\big( \frac{1}{[d_1]_{q^{-1}}!} - \frac{1}{[\infty]_{q^{-1}}!} \big)
\big( \frac{1}{[d_3]_{q^{-1}}!} - \frac{1}{[\infty]_{q^{-1}}!} \big) \\
&\qquad\qquad\qquad\qquad
+ \frac{2}{[\infty]_{q^{-1}}!}\sum_{d_1=0}^{d-1} \big( \frac{1}{[d_1]_{q^{-1}}!} - \frac{1}{[\infty]_{q^{-1}}!} \big)
+ \frac{d}{([\infty]_{q^{-1}}!)^2} \Big)
\end{split}
\end{equation*}
The first summation in the right-hand side is exactly $\text{(I)}/([\infty]_{q^{-1}})^2$
and the leading terms of $\text{(I)}$ has been given in \eqref{eq-W3d-Iresult}.
Here we only need the coefficients of $q^{\geq -2d+2}$
since there is a factor $q^{-d-1}$ in $\text{(III)}$.
The result is:
\be
\label{eq-W3d-Iresult2d}
\begin{split}
&\sum_{d_1+d_3=d-1}
\big( \frac{1}{[d_1]_{q^{-1}}!} - \frac{1}{[\infty]_{q^{-1}}!} \big)
\big( \frac{1}{[d_3]_{q^{-1}}!} - \frac{1}{[\infty]_{q^{-1}}!} \big) \\
=& \frac{1}{([\infty]_{q^{-1}})^2}
\Big(
\frac{d\cdot q^{-d-1}}{[1]_{q^{-1}}^2}
 - \frac{2q^{-d-1}}{[1]_{q^{-1}}^2}
\sum_{i\geq 2}\frac{q^{-i}}{1-q^{-i}}
\Big) \Big|_{q^{\geq -2d+2}}.
\end{split}
\ee

And the leading terms of the second summation in the right-hand side is given by:

\begin{Lemma}
\label{lem-sum1overd-inf}
The following identity holds:
\begin{equation*}
\begin{split}
\sum_{d_1=0}^{d-1} \big( \frac{1}{[d_1]_{q^{-1}}!} - \frac{1}{[\infty]_{q^{-1}}!} \big)
= \frac{1}{[\infty]_{q^{-1}}!} \Big( \frac{q^{-d-1}}{[1]_{q^{-1}}^2}
-\sum_{i\geq 1} \frac{q^{-i}}{1-q^{-i}}
\Big)\Big|_{q^{\geq -2d+2}}.
\end{split}
\end{equation*}
\end{Lemma}
\begin{proof}
By \eqref{eq-sum1overkfac-infty} one has:
\begin{equation*}
\begin{split}
\sum_{d_1=0}^{d-1} \big( \frac{1}{[d_1]_{q^{-1}}!} - \frac{1}{[\infty]_{q^{-1}}!} \big)
=& - \sum_{d_1=0}^{d-1} \sum_{k>d_1} \frac{q^{-k}}{[k]_{q^{-1}}!} \\
=& - \sum_{d_1=0}^{d-1} \sum_{k>d_1} q^{-k} \big(\frac{1}{[k]_{q^{-1}}!} - \frac{1}{[\infty]_{q^{-1}}!}\big)
- \frac{q^{-1} - q^{-d-1}}{[1]_{q^{-1}}^2  [\infty]_{q^{-1}}!}.
\end{split}
\end{equation*}
Here
\begin{equation*}
\begin{split}
& \sum_{d_1=0}^{d-1} \sum_{k>d_1} q^{-k} \big(\frac{1}{[k]_{q^{-1}}!} - \frac{1}{[\infty]_{q^{-1}}!}\big) \\
=& \sum_{d_1 =0}^\infty \sum_{k>d_1} q^{-k} \big(\frac{1}{[k]_{q^{-1}}!} - \frac{1}{[\infty]_{q^{-1}}!}\big)
-\sum_{d_1=d}^\infty \sum_{k>d_1} q^{-k} \big(\frac{1}{[k]_{q^{-1}}!} - \frac{1}{[\infty]_{q^{-1}}!}\big),
\end{split}
\end{equation*}
and for $d_1\geq d$ one always has:
\begin{equation*}
\deg q^{-k} \big(\frac{1}{[k]_{q^{-1}}!} - \frac{1}{[\infty]_{q^{-1}}!}\big) <-2d.
\end{equation*}
Therefore:
\begin{equation*}
\begin{split}
\sum_{d_1=0}^{d-1} \big( \frac{1}{[d_1]_{q^{-1}}!} - \frac{1}{[\infty]_{q^{-1}}!} \big)
=& - \sum_{d_1=0}^\infty \sum_{k>d_1} q^{-k} \big(\frac{1}{[k]_{q^{-1}}!} - \frac{1}{[\infty]_{q^{-1}}!}\big)
- \frac{q^{-1} - q^{-d-1}}{[1]_{q^{-1}}^2  [\infty]_{q^{-1}}!} \\
=& \Big( - \sum_{d_1=0}^\infty \sum_{k>d_1} \frac{q^{-k}}{[k]_{q^{-1}}!}
+ \frac{ q^{-d-1}}{[1]_{q^{-1}}^2  [\infty]_{q^{-1}}!}
\Big)\Big|_{q^{\geq -3d+1}}.
\end{split}
\end{equation*}
Moreover,
\begin{equation*}
\begin{split}
\sum_{d_1=0}^\infty \sum_{k>d_1} \frac{q^{-k}}{[k]_{q^{-1}}!}
= \sum_{k=1}^\infty  \frac{q^{-k}}{[k]_{q^{-1}}!} \sum_{d_1 = 0}^{k-1} 1
= \sum_{k=1}^\infty  \frac{kq^{-k}}{[k]_{q^{-1}}!},
\end{split}
\end{equation*}
and then the conclusion follows from \eqref{eq-twoidgen}.
\end{proof}

Now by this lemma and \eqref{eq-W3d-Iresult2d} we conclude that:
\be
\label{eq-W3d-IIsum-3}
\begin{split}
\text{(III)}
= \Big(
\frac{(d+2)q^{-2d-2}}{[1]_{q^{-1}}^4} + \frac{d  q^{-d-1}}{[1]_{q^{-1}}^2} -\frac{2q^{-d-2}}{[1]_{q^{-1}}^3}
-2\big(\frac{q^{-2d-2}}{[1]_{q^{-1}}^4}+ \frac{q^{-d-1}}{[1]_{q^{-1}}^2} \big) \sum_{i\geq 2} \frac{q^{-i}}{1-q^{-i}}
\Big)\Big|_{q^{\geq -3d+1}}.
\end{split}
\ee
And by combining \eqref{eq-W3d-Iresult},
\eqref{eq-W3d-IIsum-2},
\eqref{eq-W3d-IIsum-3},
we see that:
\begin{equation*}
\begin{split}
\text{(I)} + \text{(II)} + \text{(III)}
=& \Big( \frac{dq^{-2d-2}}{[1]_{q^{-1}}^4} -\frac{2d q^{-2d-3}}{[1]_{q^{-1}}^3 [2]_{q^{-1}}}
+ \frac{dq^{-2d-4}}{[1]_{q^{-1}}^2 [2]_{q^{-1}}^2}
-\frac{2q^{-2d-5}}{[1]_{q^{-1}}^3 [2]_{q^{-1}}^2} + \frac{2q^{-2d-6}}{[1]_{q^{-1}}^2 [2]_{q^{-1}}^3} \\
& +\big( -\frac{2q^{-2d-2}}{[1]_{q^{-1}}^4} +\frac{4q^{-2d-3}}{[1]_{q^{-1}}^3 [2]_{q^{-1}}}
 -\frac{2q^{-2d-4}}{[1]_{q^{-1}}^2 [2]_{q^{-1}}^2}\big) \sum_{i\geq 2}\frac{q^{-i}}{1-q^{-i}}
\Big) \Big|_{q^{\geq -3d+1}},
\end{split}
\end{equation*}
and after some simplification
we finally obtain the conclusion in the case $d_2 = 1$ in $W_d^3(q)$:

\begin{Proposition}
\label{prop-w3dcase2}
The following identity holds:
\begin{equation*}
\begin{split}
&\sum_{d_1+d_3=d-1}T_{1}^x\Big( \frac{(q^{-d_1 -1} ;q^{-1})_\infty}{(xq^{-d_1 -1} ;q^{-1})_\infty} -1 \Big)
 T_{d_3}^x\Big( \frac{(q^{ -2} ;q^{-1})_\infty}{(xq^{ -2} ;q^{-1})_\infty} -1 \Big) \\
=& \Big( \frac{dq^{-2d-2}}{[1]_{q^{-1}}^2 [2]_{q^{-1}}^2} -\frac{2q^{-2d-5}}{[1]_{q^{-1}}^2 [2]_{q^{-1}}^3}
-\frac{2q^{-2d-2}}{[1]_{q^{-1}}^2 [2]_{q^{-1}}^2 }\sum_{i\geq 2}\frac{q^{-i}}{1-q^{-i}}
\Big) \Big|_{q^{\geq -3d+1}}.
\end{split}
\end{equation*}
\end{Proposition}

\subsection{Leading terms of $W_d^3(q)$ -- Case 3) and summary}

Now we consider the case $d_2 = d$ in $W_d^3(q)$.
In this case one must have $d_1 = d_3 =0$,
and similar to \eqref{eq-comutedisumWd2} one has:
\begin{equation*}
\begin{split}
T_{d}^x\Big( \frac{(q^{ -1} ;q^{-1})_\infty}{(xq^{ -1} ;q^{-1})_\infty} -1 \Big)
 T_{0}^x\Big( \frac{(q^{-d -1} ;q^{-1})_\infty}{(xq^{-d -1} ;q^{-1})_\infty} -1 \Big)
= \Big( \frac{[\infty]_{q^{-1}}!}{[d]_{q^{-1}}!}-1 \Big)^2
=  \frac{q^{-2d-2}}{[1]_{q^{-1}}^2}
\Big|_{q^{\geq -3d+1}}.
\end{split}
\end{equation*}

Now we are able to consider the summation of the contributions of all the three cases
(Proposition \ref{prop-w3dcase1},
Proposition \ref{prop-w3dcase2},
and the above identity),
and in this way we obtain the following result
of leading terms of $W_d^3(q)$:
\begin{Proposition}
\label{prop-w3dleading}
One has:
\begin{equation*}
\begin{split}
W_d^3(q) = &
\frac{3}{[1]_{q^{-1}}^2 \cdot ([\infty]_{q^{-1}}!)^3} \bigg(
(d+1)q^{-d-2} +  q^{-2d-2}
+ \frac{d q^{-2d-2}}{ [2]_{q^{-1}}^2}
+ \frac{2q^{-2d-5}}{ [2]_{q^{-1}}!}  -\frac{2q^{-2d-5}}{ [2]_{q^{-1}}^3} \\
& + \frac{(d+1)q^{-2d-6}}{ [2]_{q^{-1}}^2}
+ \frac{2q^{-2d-8}}{ [2]_{q^{-1}}^3}
-2\Big( q^{-d-2} +\frac{q^{-2d-2}}{[2]_{q^{-1}}^2} + \frac{q^{-2d-6}}{[2]_{q^{-1}}^2} \Big)
\sum_{i\geq 2}\frac{q^{-i}}{1-q^{-i}}
\bigg) \bigg|_{q^{\geq -3d+1}}.
\end{split}
\end{equation*}
\end{Proposition}

\subsection{Leading terms of $W_d^4(q)$}

Now we compute the coefficients of $q^{-D}$ with $0\leq D\leq 3d-1$ in $W_d^4(q)$.
First we consider the degree of $W_d^4(q)$.
Assume $d_1+d_2+d_3 =d$,
and then by \eqref{eq-gzlem-5.7} one has:
\begin{equation*}
\begin{split}
&\deg\Big(
T_{d_1}^x\big( \frac{(q^{-d_3 -1} ;q^{-1})_\infty}{(xq^{-d_3 -1} ;q^{-1})_\infty} -1 \big)
T_{d_2}^x\big( \frac{(q^{-d_1 -1} ;q^{-1})_\infty}{(xq^{-d_1 -1} ;q^{-1})_\infty} -1 \big)
T_{d_3}^x\big( \frac{(q^{-d_2 -1} ;q^{-1})_\infty}{(xq^{-d_2 -1} ;q^{-1})_\infty} -1 \big)
\Big) \\
=& -(d_1+1)(d_3+1) -(d_2+1)(d_1+1) -(d_3+1)(d_2+1).
\end{split}
\end{equation*}
It is clear that if two of $d_1,d_2,d_3$ are positive,
one has:
\begin{equation*}
-(d_1+1)(d_3+1) -(d_2+1)(d_1+1) -(d_3+1)(d_2+1) \leq -(3d+2),
\end{equation*}
and thus we only need to consider the cases $(d_1,d_2,d_3) = (d,0,0),(0,d,0),(0,0,d)$.
Then by \eqref{eq-computed1=01} and \eqref{eq-computed2=01} one has:
\begin{equation*}
\begin{split}
W_d^4(q) =  &\frac{3}{([\infty]_{q^{-1}}!)^3}
\cdot
\Big( \frac{[\infty]_{q^{-1}}!}{[d]_{q^{-1}}!} -1 \Big)^2
\cdot \Big( [\infty]_{q^{-1}}! -1 \Big)
\Big|_{q^{\geq -3d+1}}.
\end{split}
\end{equation*}
And then similar to \eqref{eq-comutedisumWd2} one obtains:
\begin{equation*}
W_d^4(q) = \frac{3([\infty]_{q^{-1}}! -1)}{([\infty]_{q^{-1}}!)^3} \cdot
\Big( - \frac{q^{-d-1}}{1-q^{-1}} + \frac{q^{-2d-3}}{(1-q^{-1})(1-q^{-2})} \Big)^2
\bigg|_{q^{\geq -3d+1}}.
\end{equation*}
And by expanding the square and then restricting to $q^{\geq -3d+1}$,
one concludes that:

\begin{Proposition}
\label{prop-w4dleading}
We have:
\begin{equation*}
W_d^4(q) = \frac{3([\infty]_{q^{-1}}! -1)}{([\infty]_{q^{-1}}!)^3} \cdot
\frac{q^{-2d-2}}{[1]_{q^{-1}}^2}
\bigg|_{q^{\geq -3d+1}}.
\end{equation*}
\end{Proposition}

Now we have accomplished the computations leading terms of
$W_d^i(q)$ for $i=1,2,3,4$,
see equality \eqref{eq-defWd1Wd2Wd3},
Propositions
\ref{prop-w2dleading},
\ref{prop-w3dleading},
\ref{prop-w4dleading}.
In summary,
for $d\geq 2$ one has:
\begin{equation*}
\begin{split}
W_d(q) =& \frac{1}{([\infty]_{q^{-1}}!)^3} \bigg( \frac{(d+1)(d+2)}{2}
+  \frac{3(d+3)q^{-d-2}}{[1]_{q^{-1}}^2}
+ \frac{6q^{-2d-2}}{[1]_{q^{-1}}^2 [2]_{q^{-1}}}
+\frac{3d \cdot q^{-2d-2}}{[1]_{q^{-1}}^2 [2]_{q^{-1}}^2} \\
& -\frac{6q^{-2d-3}}{[1]_{q^{-1}} [2]_{q^{-1}}^3}
+\frac{6q^{-2d-5}}{[1]_{q^{-1}}^3 [2]_{q^{-1}}}
-\frac{6q^{-2d-5}}{[1]_{q^{-1}}^2 [2]_{q^{-1}}^3}
+\frac{3(d+1) q^{-2d-6}}{[1]_{q^{-1}}^2 [2]_{q^{-1}}^2}
+ \frac{6q^{-2d-7}}{[1]_{q^{-1}} [2]_{q^{-1}}^3} \\
& + \frac{6q^{-2d-8}}{[1]_{q^{-1}}^2 [2]_{q^{-1}}^3}
-6\Big( \frac{q^{-d-2}}{[1]_{q^{-1}}^2} +\frac{q^{-2d-2}}{[1]_{q^{-1}}^2[2]_{q^{-1}}^2}
+ \frac{q^{-2d-6}}{[1]_{q^{-1}}^2[2]_{q^{-1}}^2} \Big)
\Big( \sum_{i\geq 2}\frac{q^{-i}}{1-q^{-i}} \Big) \\
& -3 \Big(\sum_{i\geq 1} \frac{q^{-i}}{(1-q^{-i})^2} \Big)
+ \frac{3q^{-2d-2} [\infty]_{q^{-1}}!}{[1]_{q^{-1}}^2 }
\bigg)\bigg|_{q^{\geq -3d+1}}.
\end{split}
\end{equation*}
After some simplifications,
one finally obtain:
\begin{Theorem}
\label{thm-Wdqleadingterm}
For every $d\geq 2$,
the following identity holds:
\be
\begin{split}
W_d(q) =& \frac{3}{([\infty]_{q^{-1}}!)^3} \bigg( \frac{1}{3}\binom{d+2}{2}
+  \frac{q^{-d-2}}{[1]_{q^{-1}}^2} \Big( d+3-2\sum_{i\geq 2}\frac{q^{-i}}{1-q^{-i}} \Big)
 + \frac{q^{-2d-2}  [\infty]_{q^{-1}}!}{[1]_{q^{-1}}^2 } \\
&+ \frac{(2-2q^{-2}-2q^{-3}+3q^{-4}+3q^{-5}) \cdot q^{-2d-2}}{[1]_{q^{-1}}\cdot [2]_{q^{-1}}^3}
- \sum_{i\geq 1} \frac{q^{-i}}{(1-q^{-i})^2} \\
& +\frac{ (1+ q^{-4})\cdot  q^{-2d-2}}{[1]_{q^{-1}}^2 [2]_{q^{-1}}^2}
\Big( d- 2\sum_{i\geq 2}\frac{q^{-i}}{1-q^{-i}} \Big)
\bigg)\bigg|_{q^{\geq -3d+1}}.
\end{split}
\ee
\end{Theorem}

\section{Proof of Theorem \ref{thm-main1-leadI} --- Part II: Leading Terms of $I_d(q)$}
\label{app-secC-delta12}

In this section we finish the computation of the coefficients of $q^{\half d(d-3)-D}$ in $I_d(q)$
for $0\leq D\leq 3d-1$.
See \S \ref{sec-degreeest-Id} for the notations.
By Proposition \ref{prop-Ileqkdegree} we already know that
$\deg I(d) = \half d(d-3)$, and
\begin{equation*}
I(d) = I^{\leq 2} (d)\big|_{q^{\geq \half d(d-3) -3d+1}}.
\end{equation*}
In what follows,
we will denote:
\be
\label{eq-def-I=k}
I^{ = k}(d) =
\sum_{\sum_{i=1}^3 |\mu^i| =d, \text{ } \sum \delta_i = k}
(-1)^d q^{ \half \sum_{i=1}^3  \kappa_{\mu^i}}
W_{\mu^{3},\mu^1}(q) W_{\mu^{1},\mu^2}(q) W_{\mu^{2},\mu^3}(q)
\ee
for a nonnegative integer $k$,
where $\delta_i = |\mu^i| - \mu^i_1$ as before.
Then:
\begin{equation*}
I^{\leq 2}(d) = I^{=0}(d) +  I^{=1}(d) + I^{=2}(d).
\end{equation*}
The leading terms of
\be
\label{eq-relnI=0Wd}
I^{= 0}(d) = (-1)^d q^{\half d(d-3)} W_d(q)
\ee
have been computed in the previous section,
and in what follows we compute the coefficients of $I^{=1}(d)$ and $I^{=2}(d)$.
The main results of this section are
Theorem \ref{thm-I=1leading} and Theorem \ref{thm-I=2leading},
and as a corollary we prove Theorem \ref{thm-main1-leadI}.

\subsection{Leading terms of $I^{=1}(d)$ -- Case 1)}

First we compute the leading terms of $I^{= 1} (d)$.
By definition one has:
\begin{equation*}
\begin{split}
I^{=1}(d) = 3
\sum_{\substack{d_1 + d_2 +d_3 =d \\ d_2,d_3\geq 0; \text{ } d_1\geq 2}}
(-1)^d q^{ \half (d_1^2 +d_2^2 +d_3^2 -3d_1 -d_2-d_3)}
 W_{(d_3),(d_1-1,1)}(q) W_{(d_1-1,1),(d_2)}(q) W_{(d_2),(d_3)}(q),
\end{split}
\end{equation*}
and by plugging \eqref{eq-gz-183184} and \eqref{eq-gzlem6.7} into this equality one obtains:
\be
\begin{split}
I^{=1}(d) = & 3
\sum_{\substack{d_1 + d_2 +d_3 =d \\ d_2,d_3\geq 0; \text{ } d_1\geq 2}}
(-1)^d q^{ \half (d_1+d_2+d_3)^2 - \frac{5}{2} (d_1+d_2+d_3) -2} \cdot \tW_{(d_2),(d_3)}(q) \\
& \cdot
\Big( \frac{\tW_{(d_2),(d_1-1)}}{[1]_{q^{-1}}} -
\frac{q^{-(d_2+1)(d_1-1)}}{[d_2]_{q^{-1}}! [d_1]_{q^{-1}}!}
\Big)
\Big( \frac{\tW_{(d_3),(d_1-1)}}{[1]_{q^{-1}}} -
\frac{q^{-(d_3+1)(d_1-1)}}{[d_3]_{q^{-1}}! [d_1]_{q^{-1}}!}
\Big),
\end{split}
\ee
and then:
\be
\label{eq-appDelta=1-decomp}
\begin{split}
I^{=1}(d)
=& 3 \cdot (-1)^d q^{ \half d^2 - \frac{5}{2} d -2} \cdot
\Big(
\text{(i)} - 2\cdot \text{(ii)} + \text{(iii)}
\Big),
\end{split}
\ee
where:
\begin{equation*}
\begin{split}
& \text{(i)} =  \sum_{\substack{d_1 + d_2 +d_3 =d-1 \\ d_2,d_3\geq 0; \text{ } d_1\geq 1}}
\frac{1}{[1]_{q^{-1}}^2}\cdot
\tW_{(d_2),(d_3)}(q) \tW_{(d_1),(d_2)}(q) \tW_{(d_1),(d_3)}(q), \\
& \text{(ii)} = \sum_{\substack{d_1 + d_2 +d_3 =d \\ d_2,d_3\geq 0; \text{ } d_1\geq 2}}
\frac{q^{-(d_2+1)(d_1-1)}}{[1]_{q^{-1}}\cdot[d_2]_{q^{-1}}! [d_1]_{q^{-1}}!}  \cdot
\tW_{(d_2),(d_3)}(q) \tW_{(d_3),(d_1-1)}(q),  \\
& \text{(iii)} = \sum_{\substack{d_1 + d_2 +d_3 =d \\ d_2,d_3\geq 0; \text{ } d_1\geq 2}}
\frac{q^{-(d_2+d_3 +2)(d_1-1)}}{[d_2]_{q^{-1}}! [d_3]_{q^{-1}}! \cdot ([d_1]_{q^{-1}}!)^2} \cdot
\tW_{(d_2),(d_3)}(q).
\end{split}
\end{equation*}
Recall that our goal is to compute $I(d)|_{q^{\geq \half d(d-3) -3d+1}}$,
and thus we need to compute
the coefficients of $q^{-D}$ in $\text{(i)},\text{(ii)},\text{(iii)}$ for $0\leq D\leq 2d-3$.
We compute the leading terms of $\text{(i)}$ in the rest of this subsection,
and compute $\text{(ii)},\text{(iii)}$ in next two subsections.

By the definition of $W_d(q)$ one has:
\begin{equation*}
\begin{split}
\text{(i)} =& \frac{1}{[1]_{q^{-1}}^2} \Big(
 W_{d-1}(q) -\sum_{d_2+d_3 =d-1} \tW_{(d_2),(d_3)}(q) \tW_{(0),(d_2)}(q) \tW_{(0),(d_3)}(q) \Big).
\end{split}
\end{equation*}
By Theorem \ref{thm-Wdqleadingterm} we already know that:
\be
\label{eq-Wd-1-computei}
\begin{split}
W_{d-1}(q) = \frac{1}{([\infty]_{q^{-1}}!)^3} &\Big( \binom{d+1}{2}
+  \frac{3(d+2)q^{-d-1}}{[1]_{q^{-1}}^2} \\
& -3\sum_{i\geq 1} \frac{q^{-i}}{(1-q^{-i})^2}
 - \frac{6 q^{-d-1}}{[1]_{q^{-1}}^2}
\sum_{i\geq 2}\frac{q^{-i}}{1-q^{-i}}
\Big)\Big|_{q^{\geq -2d+3}}.
\end{split}
\ee
\begin{Lemma}
\label{lem-AppC-1}
For every $d\geq 2$ one has:
\begin{equation*}
\begin{split}
& \sum_{d_2+d_3 =d-1} \tW_{(d_2),(d_3)}(q) \tW_{(0),(d_2)}(q) \tW_{(0),(d_3)}(q) \\
=& \frac{1}{([\infty]_{q^{-1}})^3}
\Big( d-\frac{2q^{-1}}{[1]_{q^{-1}}}
+\frac{(d+2)q^{-d-1}}{[1]_{q^{-1}}^2}
-\frac{2q^{-d} [\infty]_{q^{-1}}!}{[1]_{q^{-1}}}
-2\big( 1+\frac{q^{-d-1}}{[1]_{q^{-1}}^2}\big) \sum_{i\geq 2}\frac{q^{-i}}{1-q^{-i}}
\Big)\Big|_{q^{\geq -2d+3}}.
\end{split}
\end{equation*}
\end{Lemma}
\begin{proof}
By \eqref{eq-gzlem-5.6} one has:
\begin{equation*}
\begin{split}
& \sum_{d_2+d_3 =d-1} \tW_{(d_2),(d_3)}(q) \tW_{(0),(d_2)}(q) \tW_{(0),(d_3)}(q) \\
=& \sum_{d_2+d_3 =d-1}
\Big( \frac{1}{[\infty]_{q^{-1}}!} \big( T_{d_2}^x
\frac{(q^{-d_3-1};q^{-1})_\infty}{(xq^{-d_3-1};q^{-1})_\infty} -1 \big)
+ \frac{1}{[\infty]_{q^{-1}}!}
\Big) \\
& \cdot \Big(
\big(\frac{1}{[d_2]_{q^{-1}}!} - \frac{1}{[\infty]_{q^{-1}}!}\big) + \frac{1}{[\infty]_{q^{-1}}!} \Big)
\Big(
\big(\frac{1}{[d_3]_{q^{-1}}!} -\frac{1}{[\infty]_{q^{-1}}!}\big) +\frac{1}{[\infty]_{q^{-1}}!}\Big).
\end{split}
\end{equation*}
Here
\begin{equation*}
\begin{split}
\deg
\big( T_{d_2}^x
\frac{(q^{-d_3-1};q^{-1})_\infty}{(xq^{-d_3-1};q^{-1})_\infty} -1 \big)
=-(d_2+1)(d_3+1) = -d_2d_3-d,
\end{split}
\end{equation*}
and
\begin{equation*}
\begin{split}
\deg
\big(\frac{1}{[d_2]_{q^{-1}}!} - \frac{1}{[\infty]_{q^{-1}}!}\big)
= -d_2-1,
\qquad
\deg
\big(\frac{1}{[d_3]_{q^{-1}}!} - \frac{1}{[\infty]_{q^{-1}}!}\big)
= -d_3-1.
\end{split}
\end{equation*}
Then by a simple degree-estimate and the symmetry \eqref{eq-symmTmx} one obtains:
\begin{equation*}
\begin{split}
& \sum_{d_2+d_3 =d-1} \tW_{(d_2),(d_3)}(q) \tW_{(0),(d_2)}(q) \tW_{(0),(d_3)}(q)
=\Big(\frac{d}{([\infty]_{q^{-1}}!)^3} +
\text{(a)}+\text{(b)}+\text{(c)}+\text{(d)}
 \Big)\Big|_{q^{\geq -2d+3}},
\end{split}
\end{equation*}
where:
\begin{equation*}
\begin{split}
&\text{(a)} =
\sum_{d_2+d_3 =d-1}
\frac{2}{([\infty]_{q^{-1}}!)^2} \big( T_{d_2}^x
\frac{(q^{-d_3-1};q^{-1})_\infty}{(xq^{-d_3-1};q^{-1})_\infty} -1 \big)
\big(\frac{1}{[d_2]_{q^{-1}}!} - \frac{1}{[\infty]_{q^{-1}}!}\big), \\
&\text{(b)} =\sum_{d_2+d_3 =d-1}
\frac{1}{([\infty]_{q^{-1}}!)^3} \big( T_{d_2}^x
\frac{(q^{-d_3-1};q^{-1})_\infty}{(xq^{-d_3-1};q^{-1})_\infty} -1 \big),\\
&\text{(c)} =\sum_{d_2+d_3 =d-1}
\frac{2}{([\infty]_{q^{-1}}!)^2} \big(\frac{1}{[d_2]_{q^{-1}}!} - \frac{1}{[\infty]_{q^{-1}}!}\big), \\
&\text{(d)} =\sum_{d_2+d_3 =d-1}
 \frac{1}{[\infty]_{q^{-1}}!}
\big(\frac{1}{[d_2]_{q^{-1}}!} - \frac{1}{[\infty]_{q^{-1}}!}\big)
\big(\frac{1}{[d_3]_{q^{-1}}!} - \frac{1}{[\infty]_{q^{-1}}!}\big).
\end{split}
\end{equation*}
By \eqref{eq-W3d-Iresult2d} we already know that:
\begin{equation*}
\begin{split}
\text{(d)}
=& \frac{1}{([\infty]_{q^{-1}}!)^3}
\Big(
\frac{d\cdot q^{-d-1}}{[1]_{q^{-1}}^2}
 - \frac{2q^{-d-1}}{[1]_{q^{-1}}^2}
\sum_{i\geq 2}\frac{q^{-i}}{1-q^{-i}}
\Big) \Big|_{q^{\geq -2d+3}}.
\end{split}
\end{equation*}
As for $\text{(c)}$,
by the method used in \S \ref{app-WdCompute} one has:
\begin{equation*}
\begin{split}
\text{(c)}
=& -\frac{2}{([\infty]_{q^{-1}}!)^2} \sum_{d_2 = 0}^{d-1}
\sum_{n>d_2} \frac{q^{-n}}{[n]_{q^{-1}}!}  \\
=& -\frac{2}{([\infty]_{q^{-1}}!)^2} \Big(
\frac{q^{-1} - q^{-d-1}}{[1]_{q^{-1}}^2 \cdot [\infty]_{q^{-1}}! } +
\sum_{d_2 = 0}^{d-1} \sum_{n>d_2}q^{-n}\big(\frac{1}{[n]_{q^{-1}}!} -\frac{1}{[\infty]_{q^{-1}}!} \big)
 \Big) \\
=& -\frac{2}{([\infty]_{q^{-1}}!)^2} \Big(
\frac{q^{-1} - q^{-d-1}}{[1]_{q^{-1}}^2 \cdot [\infty]_{q^{-1}}! } +
\sum_{d_2 = 0}^{\infty} \sum_{n>d_2}q^{-n}\big(\frac{1}{[n]_{q^{-1}}!} -\frac{1}{[\infty]_{q^{-1}}!} \big)
 \Big) \Big|_{q^{\geq -2d+3}},
\end{split}
\end{equation*}
and
\begin{equation*}
\begin{split}
\sum_{d_2 = 0}^{\infty} \sum_{n>d_2}q^{-n}\big(\frac{1}{[n]_{q^{-1}}!} -\frac{1}{[\infty]_{q^{-1}}! } \big)
= \Big(\sum_{n=1}^\infty \frac{n\cdot q^{-n}}{[n]_{q^{-1}}!} \Big)
- \frac{q^{-1}}{[1]_{q^{-1}}^2 \cdot [\infty]_{q^{-1}}!}.
\end{split}
\end{equation*}
Then by \eqref{eq-twoidgen} one concludes that:
\begin{equation*}
\begin{split}
\text{(c)}
=& -\frac{2}{([\infty]_{q^{-1}}!)^3} \Big(
-\frac{ q^{-d-1}}{[1]_{q^{-1}}^2 }  + \sum_{i\geq 1} \frac{q^{-i}}{1-q^{-i}}
 \Big) \Big|_{q^{\geq -2d+3}}.
\end{split}
\end{equation*}

Now consider $\text{(a)}$ and $\text{(b)}$.
One has:
\begin{equation*}
\begin{split}
\deg
\big( T_{d_2}^x
\frac{(q^{-d_3-1};q^{-1})_\infty}{(xq^{-d_3-1};q^{-1})_\infty} -1 \big)
 = -d_2d_3-d <-2d+3
\end{split}
\end{equation*}
provided $d_2\geq 1$ and $d_3\geq 1$,
and thus one only needs to consider the special cases $d_2 = 0$ or $d_3 =0$
in the computations of $\text{(a)}$ and $\text{(b)}$.
Therefore for $d\geq 2$ one has:
\begin{equation*}
\begin{split}
\text{(a)} =&
\frac{2}{([\infty]_{q^{-1}}!)^2} \Big(
\big( T_{0}^x
\frac{(q^{-(d-1)-1};q^{-1})_\infty}{(xq^{-(d-1)-1};q^{-1})_\infty} -1 \big)
\big(1  - \frac{1}{[\infty]_{q^{-1}}!}\big) \\
& +\big( T_{d-1}^x
\frac{(q^{-1};q^{-1})_\infty}{(xq^{-1};q^{-1})_\infty} -1 \big)
\big(\frac{1}{[d-1]_{q^{-1}}!} - \frac{1}{[\infty]_{q^{-1}}!}\big) \Big)
\Big|_{q^{\geq -2d+3}}, \\
\text{(b)} =&
\frac{2}{([\infty]_{q^{-1}}!)^3} \big( T_{0}^x
\frac{(q^{-(d-1)-1};q^{-1})_\infty}{(xq^{-(d-1)-1};q^{-1})_\infty} -1 \big)
\Big|_{q^{\geq -2d+3}},
\end{split}
\end{equation*}
and by \eqref{eq-computed1=01} one obtains:
\begin{equation*}
\begin{split}
\text{(a)} =&
\frac{2}{([\infty]_{q^{-1}}!)^3} \Big(
\big( [\infty]_{q^{-1}}! -1\big)
\big( \frac{[\infty]_{q^{-1}}!}{[d-1]_{q^{-1}}!} -1 \big)
+ \big( \frac{[\infty]_{q^{-1}}!}{[d-1]_{q^{-1}}!} -1 \big)^2
\Big) \Big|_{q^{\geq -2d+3}}, \\
\text{(b)} =&
\frac{2}{([\infty]_{q^{-1}}!)^3} \big( \frac{[\infty]_{q^{-1}}!}{[d-1]_{q^{-1}}!} -1 \big)
\Big|_{q^{\geq -2d+3}}.
\end{split}
\end{equation*}
Then by \eqref{eq-comutedisumWd2} one concludes that:
\begin{equation*}
\begin{split}
\text{(a)} =
-\frac{2([\infty]_{q^{-1}}! -1)}{([\infty]_{q^{-1}}!)^3} \cdot
\frac{q^{-d}}{[1]_{q^{-1}}}
\Big|_{q^{\geq -2d+3}},
\qquad
\text{(b)} =
-\frac{2}{([\infty]_{q^{-1}}!)^3} \cdot \frac{q^{-d}}{[1]_{q^{-1}}}
\Big|_{q^{\geq -2d+3}}.
\end{split}
\end{equation*}
This completes the proof.
\end{proof}

Now by this lemma and \eqref{eq-Wd-1-computei} we conclude that:
\begin{Proposition}
\label{prop-app-delta1i}
The leading terms of $\text{(i)}$ are given by:
\begin{equation*}
\begin{split}
\text{(i)} =& \frac{1}{[1]_{q^{-1}}^2 \cdot ([\infty]_{q^{-1}}!)^3} \Big(
\binom{d}{2} + \frac{2q^{-1}}{[1]_{q^{-1}}} + \frac{2(d+2)q^{-d-1}}{[1]_{q^{-1}}^2}
+\frac{2q^{-d} [\infty]_{q^{-1}}!}{[1]_{q^{-1}}} \\
&\qquad\qquad
 +(2- \frac{4 q^{-d-1}}{[1]_{q^{-1}}^2} )\sum_{i\geq 2}\frac{q^{-i}}{1-q^{-i}}
 -3\sum_{i\geq 1} \frac{q^{-i}}{(1-q^{-i})^2}
\Big) \Big|_{q^{\geq -2d+3}}.
\end{split}
\end{equation*}
\end{Proposition}

\subsection{Leading terms of $I^{=1}(d)$ -- Case 2)}

Now we compute the leading terms of $\text{(ii)}$.
We already know that:
\begin{equation*}
\begin{split}
\text{(ii)} =& \sum_{\substack{d_1 + d_2 +d_3 =d \\ d_2,d_3\geq 0; \text{ } d_1\geq 2}}
\frac{q^{-(d_2+1)(d_1-1)}}{[1]_{q^{-1}} [d_2]_{q^{-1}}! [d_1]_{q^{-1}}! ([\infty]_{q^{-1}})^2}
T_{d_2}^x \frac{(q^{-d_3-1};q^{-1})_\infty}{(xq^{-d_3-1};q^{-1})_\infty}
T_{d_3}^x \frac{(q^{-d_1};q^{-1})_\infty}{(xq^{-d_1};q^{-1})_\infty} \\
=& \sum_{\substack{d_1 + d_2 +d_3 =d \\ d_2,d_3\geq 0; \text{ } d_1\geq 2}}
\frac{q^{-(d_2+1)(d_1-1)}}{[1]_{q^{-1}} [d_2]_{q^{-1}}! [d_1]_{q^{-1}}! }
\Big( \frac{1}{[\infty]_{q^{-1}}!}
\big( T_{d_2}^x \frac{(q^{-d_3-1};q^{-1})_\infty}{(xq^{-d_3-1};q^{-1})_\infty} -1 \big)
+ \frac{1}{[\infty]_{q^{-1}}!} \Big) \\
&\qquad\qquad
\cdot \Big( \frac{1}{[\infty]_{q^{-1}}!}
\big( T_{d_3}^x \frac{(q^{-d_1};q^{-1})_\infty}{(xq^{-d_1};q^{-1})_\infty} -1 \big)
+ \frac{1}{[\infty]_{q^{-1}}!} \Big).
\end{split}
\end{equation*}
One may decompose $\text{(ii)}$ into:
\begin{equation*}
\begin{split}
\text{(ii)} =&  \text{(a')} + \text{(b')} + \text{(c')} + \text{(d')},
\end{split}
\end{equation*}
where:
\begin{equation*}
\begin{split}
\text{(a')}
=& \sum_{\substack{d_1 + d_2 +d_3 =d \\ d_2,d_3\geq 0; \text{ } d_1\geq 2}}
\frac{q^{-(d_2+1)(d_1-1)}}{[1]_{q^{-1}} [d_2]_{q^{-1}}! [d_1]_{q^{-1}}! \cdot ([\infty]_{q^{-1}}!)^2} \\
&\qquad\qquad \cdot
\Big(
T_{d_2}^x \frac{(q^{-d_3-1};q^{-1})_\infty}{(xq^{-d_3-1};q^{-1})_\infty} -1 \Big)
\cdot \Big(
T_{d_3}^x \frac{(q^{-d_1};q^{-1})_\infty}{(xq^{-d_1};q^{-1})_\infty} -1 \Big),
\end{split}
\end{equation*}
and
\begin{equation*}
\begin{split}
&\text{(b')}
= \sum_{\substack{d_1 + d_2 +d_3 =d \\ d_2,d_3\geq 0; \text{ } d_1\geq 2}}
\frac{q^{-(d_2+1)(d_1-1)}}{[1]_{q^{-1}} [d_2]_{q^{-1}}! [d_1]_{q^{-1}}! \cdot ([\infty]_{q^{-1}}!)^2}
\cdot \Big(
T_{d_2}^x \frac{(q^{-d_3-1};q^{-1})_\infty}{(xq^{-d_3-1};q^{-1})_\infty} -1 \Big),\\
&\text{(c')}
= \sum_{\substack{d_1 + d_2 +d_3 =d \\ d_2,d_3\geq 0; \text{ } d_1\geq 2}}
\frac{q^{-(d_2+1)(d_1-1)}}{[1]_{q^{-1}} [d_2]_{q^{-1}}! [d_1]_{q^{-1}}! \cdot ([\infty]_{q^{-1}}!)^2} \cdot
\Big(
T_{d_3}^x \frac{(q^{-d_1};q^{-1})_\infty}{(xq^{-d_1};q^{-1})_\infty} -1 \Big), \\
& \text{(d')}= \sum_{\substack{d_1 + d_2 +d_3 =d \\ d_2,d_3\geq 0; \text{ } d_1\geq 2}}
\frac{q^{-(d_2+1)(d_1-1)}}{[1]_{q^{-1}} [d_2]_{q^{-1}}! [d_1]_{q^{-1}}! \cdot ([\infty]_{q^{-1}}!)^2}.
\end{split}
\end{equation*}

First we compute the leading terms of $\text{(a')}$.
One has:
\begin{equation*}
\begin{split}
& \deg
q^{-(d_2+1)(d_1-1)}
\Big(
T_{d_2}^x \frac{(q^{-d_3-1};q^{-1})_\infty}{(xq^{-d_3-1};q^{-1})_\infty} -1 \Big)
\Big(
T_{d_3}^x \frac{(q^{-d_1};q^{-1})_\infty}{(xq^{-d_1};q^{-1})_\infty} -1 \Big)\\
=&  -(d_2+1)(d_1-1) - (d_2+1)(d_3+1) -(d_3+1)d_1.
\end{split}
\end{equation*}
It is easy to see that for $d_1+d_2+d_3 =d$ and $d_1\geq 2$,
one always has:
\begin{equation*}
-(d_2+1)(d_1-1) - (d_2+1)(d_3+1) -(d_3+1)d_1
 < -2d-3.
\end{equation*}
Thus in the computation of leading terms of $\text{(a')}$ we only need to consider the case $d_2 = d_3 =0$,
and then by \eqref{eq-computed2=01} one has:
\begin{equation*}
\begin{split}
\text{(a')} =& 0
\big|_{q^{\geq -2d+3}}.
\end{split}
\end{equation*}

Then we compute $\text{(b')}$.
Similar to the case $\text{(a')}$,
by a simple degree-estimate we can see that
we only need to consider the case $d_2 = 0$ in the computation of $\text{(b')}$.
One has:
\begin{equation*}
\begin{split}
\text{(b')}
= &\sum_{\substack{d_1  +d_3 =d \\ d_3\geq 0; \text{ } d_1\geq 2}}
\frac{q^{-d_1+1}}{[1]_{q^{-1}}  [d_1]_{q^{-1}}! \cdot ([\infty]_{q^{-1}}!)^2}
\cdot \Big(
T_{0}^x \frac{(q^{-d_3-1};q^{-1})_\infty}{(xq^{-d_3-1};q^{-1})_\infty} -1 \Big)\\
=&  \frac{1}{[1]_{q^{-1}}\cdot [\infty]_{q^{-1}}! }
\sum_{\substack{d_1  +d_3 =d \\ d_3\geq 0; \text{ } d_1\geq 2}}
\frac{q^{-d_1+1}}{  [d_1]_{q^{-1}}! }
\cdot \Big(
\frac{1}{[d_3]_{q^{-1}}!} - \frac{1}{[\infty]_{q^{-1}}!}
\Big)\Big|_{q^{\geq -2d+3}}.
\end{split}
\end{equation*}
The summation in the right-hand side has already been calculated in \S \ref{sec-app-w3dcase2}.
In fact, by \eqref{eq-W3d-IIIIII} and \eqref{eq-W3d-IIsum-2} we know that:
\begin{equation*}
\begin{split}
& \sum_{\substack{d_1  +d_3 =d \\ d_3\geq 0; \text{ } d_1\geq 2}}
 \frac{q^{-d_1+1}}{[d_1]_{q^{-1}}!}
\Big( \frac{1}{[d_3]_{q^{-1}}!} - \frac{1}{[\infty]_{q^{-1}}!}\Big)\\
=& \bigg(
\frac{q^2 \cdot [1]_{q^{-1}}}{([\infty]_{q^{-1}}!)^2}
\Big( -\frac{(d+1)q^{-d-2}}{[1]_{q^{-1}}^2} + \frac{q^{-d-3}}{[1]_{q^{-1}}^3}
+ \frac{2q^{-d-2}}{[1]_{q^{-1}}^2} \sum_{i\geq 2} \frac{q^{-i}}{1-q^{-i}} \Big) \\
& -q\big( \frac{1}{[d]_{q^{-1}}!} - \frac{1}{[\infty]_{q^{-1}}!} \big)
-\frac{1}{[1]_{q^{-1}}} \big( \frac{1}{[d-1]_{q^{-1}}!} - \frac{1}{[\infty]_{q^{-1}}!} \big)
\bigg) \bigg|_{q^{\geq -2d+3}}.
\end{split}
\end{equation*}
Finally use \eqref{eq-comutedisumWd2},
and we conclude that:
\be
\begin{split}
\text{(b')}
= &\bigg( \frac{1}{([\infty]_{q^{-1}}!)^3}
\Big( -\frac{(d+1)q^{-d}}{[1]_{q^{-1}}^2} + \frac{q^{-d-1}}{[1]_{q^{-1}}^3}
+ \frac{2q^{-d}}{[1]_{q^{-1}}^2} \sum_{i\geq 2} \frac{q^{-i}}{1-q^{-i}} \Big) \\
& + \frac{(2-q^{-1})q^{-d}}{[1]_{q^{-1}}^3 \cdot ([\infty]_{q^{-1}}!)^2}
\bigg)\bigg|_{q^{\geq -2d+3}}.
\end{split}
\ee

As for $\text{(c')}$,
we have:
\begin{equation*}
\deg q^{-(d_2+1)(d_1-1)}\Big(
T_{d_3}^x \frac{(q^{-d_1};q^{-1})_\infty}{(xq^{-d_1};q^{-1})_\infty} -1 \Big)
= -(d+2 -d_1)(d_1-1) - (d_3+1),
\end{equation*}
and thus one only needs to consider the case $d_1=2$.
Then by \eqref{eq-computed1=01} one has:
\begin{equation*}
\begin{split}
\text{(c')}
=& \sum_{d_2 +d_3 =d -2 }
\frac{q^{-d_2 -1}}{[1]_{q^{-1}}^2 [2]_{q^{-1}} [d_2]_{q^{-1}}! \cdot [\infty]_{q^{-1}}!}
\Big( \big( 1+ \frac{q^{-d_3-1}}{[1]_{q^{-1}}}\big)\frac{1}{[d_3]_{q^{-1}}!}
 - \frac{1}{[\infty]_{q^{-1}}!} \Big) \\
=& \frac{q^{-d}}{[1]_{q^{-1}}^3 [2]_{q^{-1}}  \cdot [\infty]_{q^{-1}}!}
\sum_{d_2 +d_3 =d -2 } \frac{1}{[d_2]_{q^{-1}}![d_3]_{q^{-1}}!} \\
&+ \frac{1}{[1]_{q^{-1}}^2 [2]_{q^{-1}}  \cdot [\infty]_{q^{-1}}!}
\sum_{d_2 +d_3 =d -2 } \frac{q^{-d_2-1}}{[d_2]_{q^{-1}}!}
\big(\frac{1}{[d_3]_{q^{-1}}!}-\frac{1}{[\infty]_{q^{-1}}!} \big).
\end{split}
\end{equation*}
Again by \eqref{eq-W3d-IIIIII} and \eqref{eq-W3d-IIsum-2} we know that:
\begin{equation*}
\begin{split}
& \sum_{d_2 +d_3 =d -2 } \frac{q^{-d_2-1}}{[d_2]_{q^{-1}}!}
\big(\frac{1}{[d_3]_{q^{-1}}!}-\frac{1}{[\infty]_{q^{-1}}!} \big) \\
=&
\frac{ [1]_{q^{-1}}}{([\infty]_{q^{-1}}!)^2}
\Big( -\frac{(d-1)q^{-d}}{[1]_{q^{-1}}^2} + \frac{q^{-d-1}}{[1]_{q^{-1}}^3}
+ \frac{2q^{-d}}{[1]_{q^{-1}}^2} \sum_{i\geq 2} \frac{q^{-i}}{1-q^{-i}}
\Big) \Big|_{q^{\geq -2d+3}}.
\end{split}
\end{equation*}
And one has:
\begin{equation*}
\begin{split}
\sum_{d_2 +d_3 =d -2 } \frac{1}{[d_2]_{q^{-1}}![d_3]_{q^{-1}}!}
=& \sum_{d_2 +d_3 =d -2 }
\big( \frac{1}{[d_2]_{q^{-1}}!} - \frac{1}{[\infty]_{q^{-1}}!} \big)
\big( \frac{1}{[d_3]_{q^{-1}}!} - \frac{1}{[\infty]_{q^{-1}}!} \big) \\
& + \frac{d-1}{([\infty]_{q^{-1}}!)^2}
+ \frac{2}{[\infty]_{q^{-1}}!} \sum_{d_2 +d_3 =d -2 }
\big( \frac{1}{[d_2]_{q^{-1}}!} - \frac{1}{[\infty]_{q^{-1}}!} \big).
\end{split}
\end{equation*}
The first summation in the right-hand side is of degree $-d$,
and thus:
\begin{equation*}
\begin{split}
\sum_{d_2 +d_3 =d -2 } \frac{1}{[d_2]_{q^{-1}}![d_3]_{q^{-1}}!}
=&  \Big( \frac{d-1}{([\infty]_{q^{-1}}!)^2}
+ \frac{2}{[\infty]_{q^{-1}}!} \sum_{d_2 =0 }^{d-2}
\big( \frac{1}{[d_2]_{q^{-1}}!} - \frac{1}{[\infty]_{q^{-1}}!} \big) \Big)
\Big|_{q^{\geq -d+1}}.
\end{split}
\end{equation*}
Then by applying Lemma \ref{lem-sum1overd-inf} to this identity one obtains:
\be
\label{eq-sum1overd2d3}
\begin{split}
\sum_{d_2 +d_3 =d -2 } \frac{1}{[d_2]_{q^{-1}}![d_3]_{q^{-1}}!}
=&  \Big( \frac{d-1}{([\infty]_{q^{-1}}!)^2}
- \frac{2}{([\infty]_{q^{-1}}!)^2} \sum_{i\geq 1} \frac{q^{-i}}{1-q^{-i}} \Big)
\Big|_{q^{\geq -d+1}}.
\end{split}
\ee
(Here we use the truncation $q^{\geq -d+1}$ instead of $q^{\geq -d+3}$ for the convenience of the computations
in next subsections.)
Then we may conclude that:
\begin{equation*}
\begin{split}
\text{(c')}
=& \frac{q^{-d}}{[1]_{q^{-1}}^3 [2]_{q^{-1}}  \cdot [\infty]_{q^{-1}}!}
\Big( \frac{d-1}{([\infty]_{q^{-1}}!)^2}
- \frac{2}{([\infty]_{q^{-1}}!)^2} \sum_{i\geq 1} \frac{q^{-i}}{1-q^{-i}} \Big) \\
&+ \frac{1}{[1]_{q^{-1}} [2]_{q^{-1}}  \cdot ([\infty]_{q^{-1}}!)^3}
\Big( -\frac{(d-1)q^{-d}}{[1]_{q^{-1}}^2} + \frac{q^{-d-1}}{[1]_{q^{-1}}^3}
+ \frac{2q^{-d}}{[1]_{q^{-1}}^2} \sum_{i\geq 2} \frac{q^{-i}}{1-q^{-i}} \Big) \\
=& \Big( -\frac{q^{-d-1}}{[1]_{q^{-1}}^4 [2]_{q^{-1}}  \cdot ([\infty]_{q^{-1}}!)^3}
\Big) \Big|_{q^{\geq -2d+3}}.
\end{split}
\end{equation*}

Finally we compute the leading terms of $\text{(d')}$.
One has:
\begin{equation*}
\begin{split}
\text{(d')}= & \frac{1}{[1]_{q^{-1}}  ([\infty]_{q^{-1}}!)^2}
\sum_{\substack{d_1 + d_2 \leq d \\ d_2 \geq 0; \text{ } d_1\geq 2}}
\frac{q^{-(d_2+1)(d_1-1)}}{ [d_2]_{q^{-1}}! [d_1]_{q^{-1}}! } \\
=&\frac{1}{[1]_{q^{-1}}  ([\infty]_{q^{-1}}!)^2} \Big(
\frac{1}{[2]_{q^{-1}}!}\sum_{d_2=0}^{d-2} \frac{q^{-d_2-1}}{[d_2]_{q^{-1}}!}
+ \sum_{d_1 = 3}^d \frac{q^{-d_1+1}}{[d_1]_{q^{-1}}!}
+ \sum_{\substack{d_1 + d_2 \leq d \\ d_2 \geq 1; \text{ } d_1\geq 3}}
\frac{q^{-(d_2+1)(d_1-1)}}{ [d_2]_{q^{-1}}! [d_1]_{q^{-1}}! } \Big).
\end{split}
\end{equation*}
When $d_1+d_2 \geq d+1$ and $d_2\geq 1$, $d_1\geq 3$ one always has
$\deg q^{-(d_2+1)(d_1-1)} < -2d+3$,
and thus:
\begin{equation*}
\begin{split}
\text{(d')}
=&\frac{1}{[1]_{q^{-1}}  ([\infty]_{q^{-1}}!)^2} \Big(
\frac{1}{[2]_{q^{-1}}!}\sum_{d_2=0}^{d-2} \frac{q^{-d_2-1}}{[d_2]_{q^{-1}}!}
+ \sum_{d_1 = 3}^d \frac{q^{-d_1+1}}{[d_1]_{q^{-1}}!}
 + \sum_{d_1 =3}^\infty \sum_{d_2 =1}^\infty
\frac{q^{-(d_2+1)(d_1-1)}}{ [d_2]_{q^{-1}}!  [d_1]_{q^{-1}}! } \Big)
\Big|_{q^{\geq -2d+3}}.
\end{split}
\end{equation*}
We need to compute the three summations in the right-hand side.
First one has:
\begin{equation*}
\begin{split}
\sum_{d_2=0}^{d-2} \frac{q^{-d_2-1}}{[d_2]_{q^{-1}}!}
=& \sum_{d_2=0}^{d-2} q^{-d_2-1} \big(\frac{1}{[d_2]_{q^{-1}}!}-\frac{1}{[\infty]_{q^{-1}}!} \big)
+ \frac{q^{-1} -q^{-d}}{[1]_{q^{-1}} \cdot [\infty]_{q^{-1}}!} \\
=& \Big( \sum_{d_2=0}^\infty q^{-d_2-1} \big(\frac{1}{[d_2]_{q^{-1}}!}-\frac{1}{[\infty]_{q^{-1}}!} \big)
+ \frac{q^{-1} -q^{-d}}{[1]_{q^{-1}} \cdot [\infty]_{q^{-1}}!}
\Big)\Big|_{q^{\geq -2d+3}},
\end{split}
\end{equation*}
and by \eqref{eq-pochhammerseries} one has:
\begin{equation*}
\begin{split}
\sum_{d_2=0}^\infty q^{-d_2-1} \big(\frac{1}{[d_2]_{q^{-1}}!}-\frac{1}{[\infty]_{q^{-1}}!} \big)
=& \frac{q^{-1}}{(q^{-1};q^{-1})_\infty} -\frac{1}{[\infty]_{q^{-1}}!}\cdot\frac{q^{-1}}{[1]_{q^{-1}}}
= -\frac{q^{-2}}{[1]_{q^{-1}}\cdot [\infty]_{q^{-1}}!},
\end{split}
\end{equation*}
and thus:
\be
\label{eq-sumqd2-11overd2}
\begin{split}
\sum_{d_2=0}^{d-2} \frac{q^{-d_2-1}}{[d_2]_{q^{-1}}!}
=& \frac{q^{-1} -q^{-2} -q^{-d}}{[1]_{q^{-1}} \cdot [\infty]_{q^{-1}}!}
\Big|_{q^{\geq -2d+3}}.
\end{split}
\ee
Similarly, by the same method one can prove that:
\begin{equation*}
\begin{split}
\sum_{d_1 = 3}^d \frac{q^{-d_1+1}}{[d_1]_{q^{-1}}!}= \Big(
\frac{-1 +q -q^{-d}}{[1]_{q^{-1}} \cdot [\infty]_{q^{-1}}!} -\frac{q}{[2]_{q^{-1}}!}
\Big)\Big|_{q^{\geq -2d+3}},
\end{split}
\end{equation*}
and here we omit the details.
Now we deal with the last summation.
By \eqref{eq-pochhammerseries} one has:
\begin{equation*}
\begin{split}
\sum_{d_1 =3}^\infty \sum_{d_2 =1}^\infty
\frac{q^{-(d_2+1)(d_1-1)}}{ [d_2]_{q^{-1}}! [d_1]_{q^{-1}}! }
=& \sum_{d_1 = 3}^\infty \frac{q^{-d_1+1}}{[d_1]_{q^{-1}}!}
\sum_{d_2 = 1}^\infty
\frac{(q^{-d_1+1})^{d_2}}{[d_2]_{q^{-1}}!} \\
=& \frac{1}{[\infty]_{q^{-1}}!}\sum_{d_1=3}^\infty \frac{q^{-d_1+1}}{[d_1-1]_{q^{-1}}[d_1]_{q^{-1}}}
- \sum_{d_1 = 3}^\infty\frac{q^{-d_1+1}}{[d_1]_{q^{-1}}!}.
\end{split}
\end{equation*}
Here
\begin{equation*}
\begin{split}
\sum_{d_1=3}^\infty \frac{q^{-d_1+1}}{[d_1-1]_{q^{-1}}[d_1]_{q^{-1}}}
= \frac{1}{[1]_{q^{-1}}}\sum_{d_1=3}^\infty \Big(  \frac{1}{[d_1-1]_{q^{-1}}}-\frac{1}{[d_1]_{q^{-1}}} \Big)
= \frac{1}{[2]_{q^{-1}}!} -\frac{1}{[1]_{q^{-1}}},
\end{split}
\end{equation*}
and then:
\begin{equation*}
\begin{split}
\sum_{d_1 =3}^\infty \sum_{d_2 =1}^\infty
\frac{q^{-(d_2+1)(d_1-1)}}{ [d_2]_{q^{-1}}! [d_1]_{q^{-1}}! }
=& \frac{1}{[\infty]_{q^{-1}}!}\big( \frac{1}{[2]_{q^{-1}}!} -\frac{1}{[1]_{q^{-1}}} \big)
- \big( \frac{q}{[\infty]_{q^{-1}}!} - q-\frac{1}{[1]_{q^{-1}}} -\frac{q^{-1}}{[2]_{q^{-1}}!}\big) \\
=& \frac{1-q+q^{-1}}{[2]_{q^{-1}}! \cdot [\infty]_{q^{-1}}!} + \frac{q}{[2]_{q^{-1}}!}.
\end{split}
\end{equation*}
Then we can conclude that:
\begin{equation*}
\begin{split}
\text{(d')}=\frac{1}{[1]_{q^{-1}}^2 ([\infty]_{q^{-1}}!)^3}  \Big(
 \frac{q^{-1}}{[1]_{q^{-1}}} - \big(1+\frac{1}{[2]_{q^{-1}}!} \big)q^{-d}
\Big) \Big|_{q^{\geq -2d+3}}.
\end{split}
\end{equation*}

Now summarize the above results for the leading terms of $\text{(a')}$,
$\text{(b')}$,
$\text{(c')}$,
$\text{(d')}$,
and we finally obtain the following:
\begin{Proposition}
\label{prop-app-delta1ii}
The leading terms of $\text{(ii)}$ are given by:
\begin{equation*}
\begin{split}
\text{(ii)} =& \frac{1}{[1]_{q^{-1}}^2 \cdot ([\infty]_{q^{-1}}!)^3} \Big(
\frac{q^{-1}}{[1]_{q^{-1}}} -(d+2+\frac{1}{[2]_{q^{-1}}!})q^{-d}
+ (1-\frac{1}{[2]_{q^{-1}}!}) \frac{q^{-d-1}}{[1]_{q^{-1}}} \\
& \qquad\qquad\qquad
 +\frac{[\infty]_{q^{-1}}!}{[1]_{q^{-1}}} (2-q^{-1})q^{-d}
+2q^{-d} \sum_{i\geq 2}\frac{q^{-i}}{1-q^{-i}}
\Big) \Big|_{q^{\geq -2d+3}}.
\end{split}
\end{equation*}
\end{Proposition}

\subsection{Leading terms of $I^{=1}(d)$ -- Case 3) and summary}

Now in this subsection we compute the leading terms of $(iii)$.
Recall that:
\begin{equation*}
\begin{split}
& \text{(iii)} = \sum_{\substack{d_1 + d_2 +d_3 =d \\ d_2,d_3\geq 0; \text{ } d_1\geq 2}}
\frac{q^{-(d_2+d_3 +2)(d_1-1)}}{[d_2]_{q^{-1}}! [d_3]_{q^{-1}}! \cdot ([d_1]_{q^{-1}}!)^2} \cdot
\tW_{(d_2),(d_3)}(q),
\end{split}
\end{equation*}
Here for $d_1\geq 3$ one always has
\begin{equation*}
\deg q^{-(d_2+d_3 +2)(d_1-1)} = -(d+2-d_1)(d_1-1)\leq -2d+2,
\end{equation*}
and thus one only need to consider the case $d_1=2$:
\begin{equation*}
\begin{split}
& \text{(iii)} = \frac{q^{-d}}{([2]_{q^{-1}}!)^2}
\sum_{d_2 +d_3 =d-2}
\frac{1}{[d_2]_{q^{-1}}!\cdot [d_3]_{q^{-1}}! } \cdot
\frac{1}{[\infty]_{q^{-1}}!}
T_{d_3}^x \frac{(q^{-d_2 -1} ;q^{-1})_\infty}{(xq^{-d_2 -1} ;q^{-1})_\infty}
\Big|_{q^{\geq -2d+3}}.
\end{split}
\end{equation*}
Then apply the decomposition
\begin{equation*}
\frac{1}{[\infty]_{q^{-1}}!}
T_{d_3}^x \frac{(q^{-d_2 -1} ;q^{-1})_\infty}{(xq^{-d_2 -1} ;q^{-1})_\infty}
=\frac{1}{[\infty]_{q^{-1}}!} \big(
T_{d_3}^x \frac{(q^{-d_2 -1} ;q^{-1})_\infty}{(xq^{-d_2 -1} ;q^{-1})_\infty} -1\big)
+ \frac{1}{[\infty]_{q^{-1}}!},
\end{equation*}
and we obtain:
\begin{equation*}
\begin{split}
& \text{(iii)} = \frac{q^{-d}}{([2]_{q^{-1}}!)^2\cdot [\infty]_{q^{-1}}!} \Big(
\sum_{d_2 +d_3 =d-2}
\frac{1}{[d_2]_{q^{-1}}! [d_3]_{q^{-1}}! } \big(
T_{d_3}^x \frac{(q^{-d_2 -1} ;q^{-1})_\infty}{(xq^{-d_2 -1} ;q^{-1})_\infty} -1\big) \\
&\qquad\qquad\qquad\qquad\qquad\qquad
+ \sum_{d_2+d_3 = d-2} \frac{1}{[d_2]_{q^{-1}}! [d_3]_{q^{-1}}! } \Big)
\Big|_{q^{\geq -2d+3}}.
\end{split}
\end{equation*}
Here for $d_2+d_3 =d-2$ one always has:
\begin{equation*}
\deg \big(
T_{d_3}^x \frac{(q^{-d_2 -1} ;q^{-1})_\infty}{(xq^{-d_2 -1} ;q^{-1})_\infty} -1\big)
=-(d_2 +1)(d_3+1) <-d+3,
\end{equation*}
and thus:
\begin{equation*}
\begin{split}
& \text{(iii)} = \frac{q^{-d}}{([2]_{q^{-1}}!)^2\cdot [\infty]_{q^{-1}}!}
\sum_{d_2+d_3 = d-2} \frac{1}{[d_2]_{q^{-1}}! [d_3]_{q^{-1}}! }
\Big|_{q^{\geq -2d+3}}.
\end{split}
\end{equation*}
By \eqref{eq-sum1overd2d3} one concludes that:
\begin{Proposition}
The leading terms of $\text{(iii)}$ are given by:
\begin{equation*}
\begin{split}
& \text{(iii)} = \frac{1}{([2]_{q^{-1}}!)^2\cdot ([\infty]_{q^{-1}}!)^3}
\Big( (d-1)q^{-d}
- 2q^{-d} \sum_{i\geq 1} \frac{q^{-i}}{1-q^{-i}} \Big)
\Big|_{q^{\geq -2d+3}}.
\end{split}
\end{equation*}
\end{Proposition}

Now we plug the results of this proposition and Propositions \ref{prop-app-delta1i}, \ref{prop-app-delta1ii}
into \eqref{eq-appDelta=1-decomp},
and in this way we obtain:
\begin{Theorem}
\label{thm-I=1leading}
For $d\geq 2$,
the leading terms for $I^{=1}(d)$ are given by:
\be
\label{eq-I=1leading}
\begin{split}
I^{=1}(d)
=& \frac{ 3 \cdot (-1)^d q^{ \half d^2 - \frac{5}{2} d -2} }{[1]_{q^{-1}}^2 \cdot ([\infty]_{q^{-1}}!)^3}
\Big( \binom{d}{2}
+\frac{q^{-d}\cdot (5-6q^{-1} +2q^{-2} +6q^{-3}-6q^{-4}) }{[1]_{q^{-1}}^2 [2]_{q^{-1}}} \\
& + \frac{q^{-d}\cdot (3 +2q^{-1}+2q^{-3}+2q^{-4})}{[2]_{q^{-1}}^2}
\big( d-2\sum_{i\geq 2}\frac{q^{-i}}{1-q^{-i}}\big)
-2q^{-d}\cdot [\infty]_{q^{-1}}! \\
&+2\sum_{i\geq 2}\frac{q^{-i}}{1-q^{-i}} -3\sum_{i\geq 1}\frac{q^{-i}}{(1-q^{-i})^2}
\Big)\Big|_{q^{\geq \half (d^2-3d) -3d+1}}.
\end{split}
\ee
\end{Theorem}

\subsection{Leading terms of $I^{=2}(d)$ -- Case 1)}

From now on we compute the leading terms of $I^{=2}(d)$.
By the definition \eqref{eq-def-I=k} one has:
\be
\label{eq-AppC-I=2decompJ}
I^{=2}(d) = 3\Big( (\text{j}_1) +(\text{j}_2)+ (\text{j}_3) \Big),
\ee
where:
\begin{equation*}
\begin{split}
& (\text{j}_1) =
\sum_{\substack{d_1+d_2+d_3=d\\ d_2,d_3\geq 0;\text{ }d_1\geq 4}}
(-1)^d q^{ \half (\kappa_{(d_1-2,2)}+\kappa_{(d_2)}+\kappa_{(d_3)})}
W_{(d_1-2,2),(d_2)}(q) W_{(d_1-2,2),(d_3)}(q) W_{(d_2),(d_3)}(q), \\
& (\text{j}_2) =
\sum_{\substack{d_1+d_2+d_3=d\\ d_2,d_3\geq 0;\text{ }d_1\geq 3}}
(-1)^d q^{ \half (\kappa_{(d_1-2,1^2)}+\kappa_{(d_2)}+\kappa_{(d_3)})}
W_{(d_1-2,1^2),(d_2)}(q) W_{(d_1-2,1^2),(d_3)}(q) W_{(d_2),(d_3)}(q), \\
\end{split}
\end{equation*}
and
\begin{equation*}
\begin{split}
(\text{j}_3) = &
\sum_{\substack{d_1+d_2+d_3=d\\ d_2,d_3\geq 2;\text{ }d_1\geq 0}}
(-1)^d q^{ \half (\kappa_{(d_1)}+\kappa_{(d_2-1,1)}+\kappa_{(d_3-1,1)})} \\
&\qquad\qquad
\cdot W_{(d_1),(d_2-1,1)}(q) W_{(d_1),(d_3-1,1)}(q) W_{(d_2-1,1),(d_3-1,1)}(q). \\
\end{split}
\end{equation*}

In the rest of this subsection we compute the leading terms $(\text{j}_1)|_{q^{\geq \half d(d-3) -3d+1}}$.
First by \eqref{eq-gz-183184} and Theorem \ref{thm-Delta2-expl} we have:
\begin{equation*}
\begin{split}
(\text{j}_1) =&
\sum_{\substack{d_1+d_2+d_3=d\\ d_2,d_3\geq 0;\text{ }d_1\geq 4}}
(-1)^d q^{ \half (d_1^2 -5d_1 +4 + d_2^2 -d_2 + d_3^2-d_3 )}
\cdot q^{d_2d_3-\half (d_2 + d_3)} \cdot \tW_{(d_2),(d_3)}(q) \\
&\cdot \Big(
\frac{q^{d_1d_2-\half d_1-\frac{5}{2} d_2 -2}}{[2]_{q^{-1}}!} \cdot \tW_{(d_1-1),(d_2)}(q)
- \frac{q^{-\frac{3}{2}d_1 +\half d_2 +1}}{[d_1-1]_{q^{-1}}!} \cdot \tW_{(2),(d_2)}(q)
\Big) \\
&\cdot \Big(
\frac{q^{d_1d_3-\half d_1-\frac{5}{2} d_3 -2}}{[2]_{q^{-1}}!} \cdot \tW_{(d_1-1),(d_3)}(q)
- \frac{q^{-\frac{3}{2}d_1 +\half d_3 +1}}{[d_1-1]_{q^{-1}}!} \cdot \tW_{(2),(d_3)}(q)
\Big).
\end{split}
\end{equation*}
Here $d_1+d_2+d_3 =d$, and thus:
\begin{equation*}
(\text{j}_1) = (-1)^d q^{\half d(d-7)-2} \cdot
\big( \text{(x)} - 2\cdot\text{(y)} + \text{(z)} \big),
\end{equation*}
where:
\begin{equation*}
\begin{split}
& \text{(x)} =
\sum_{\substack{d_1+d_2+d_3=d\\ d_2,d_3\geq 0;\text{ }d_1\geq 4}}
\frac{1}{([2]_{q^{-1}}!)^2}\cdot
\tW_{(d_2),(d_3)}(q)  \tW_{(d_1-1),(d_2)}(q) \tW_{(d_1-1),(d_3)}(q), \\
& \text{(y)} =
\sum_{\substack{d_1+d_2+d_3=d\\ d_2,d_3\geq 0;\text{ }d_1\geq 4}}
\frac{q^{-(d_1-3)(d_3+1)}}{[2]_{q^{-1}}! [d_1-1]_{q^{-1}}!} \cdot
\tW_{(d_2),(d_3)}(q)\tW_{(d_1-1),(d_2)}(q)\tW_{(2),(d_3)}(q), \\
&\text{(z)} =
\sum_{\substack{d_1+d_2+d_3=d\\ d_2,d_3\geq 0;\text{ }d_1\geq 4}}
\frac{q^{-(d_1-3)(d_2+d_3+2)}}{([d_1-1]_{q^{-1}}!)^2}\cdot
\tW_{(d_2),(d_3)} \tW_{(2),(d_2)}(q) \tW_{(2),(d_3)}(q).
\end{split}
\end{equation*}

First consider the leading terms of $\text{(z)}$.
For $d_1\geq 4$ one always has:
\begin{equation*}
-(d_1-3)(d_2+d_3+2) =-(d_1-3)(d+2-d_1) \leq -d+2,
\end{equation*}
and thus:
\be
\label{eq-appC-leadingz}
\text{(z)} = 0 \big|_{q^{\geq -d+3}}.
\ee

Now consider the leading terms of $\text{(x)}$.
One has:
\be
\label{eq-Idelta=2-xdecomp}
\begin{split}
\text{(x)} =&
\frac{1}{([2]_{q^{-1}}!)^2}\Big( W_{d-1}(q)
- \sum_{d_2+d_3 =d-1} \tW_{(d_2),(d_3)} \tW_{(0),(d_2)} \tW_{(0),(d_3)} \\
& - \sum_{d_2+d_3 =d-2} \tW_{(d_2),(d_3)} \tW_{(1),(d_2)} \tW_{(1),(d_3)}
- \sum_{d_2+d_3 =d-3} \tW_{(d_2),(d_3)} \tW_{(2),(d_2)} \tW_{(2),(d_3)}
\Big),
\end{split}
\ee
and by \eqref{eq-Wd-1-computei} and Lemma \ref{lem-AppC-1} we already know that:
\be
\label{eq-Idelta=2-xcomp1}
W_{d-1}(q) = \frac{1}{([\infty]_{q^{-1}}!)^3} \Big( \binom{d+1}{2}
 -3\sum_{i\geq 1} \frac{q^{-i}}{(1-q^{-i})^2}
\Big)\Big|_{q^{\geq -d+3}},
\ee
and
\be
\label{eq-Idelta=2-xcomp2}
\begin{split}
 \sum_{d_2+d_3 =d-1} \tW_{(d_2),(d_3)}(q) \tW_{(0),(d_2)}(q) \tW_{(0),(d_3)}(q)
= \frac{1}{([\infty]_{q^{-1}})^3}
\Big( d
-2 \sum_{i\geq 1}\frac{q^{-i}}{1-q^{-i}}
\Big)\Big|_{q^{\geq -d+3}}.
\end{split}
\ee
Moreover,
one has:
\begin{equation*}
\begin{split}
& \sum_{d_2+d_3 =d-2} \tW_{(d_2),(d_3)} \tW_{(1),(d_2)} \tW_{(1),(d_3)} \\
=&
\sum_{d_2+d_3 =d-2} \frac{1}{[d_2]_{q^{-1}}! [d_3]_{q^{-1}}!} \cdot \tW_{(d_2),(d_3)} \cdot
\big( 1+ \frac{q^{-d_2-1}}{[1]_{q^{-1}}} \big)\big( 1+ \frac{q^{-d_3-1}}{[1]_{q^{-1}}} \big) \\
=& \Big(
\sum_{d_2+d_3 =d-2} \frac{1}{[d_2]_{q^{-1}}! [d_3]_{q^{-1}}!} \cdot \tW_{(d_2),(d_3)} \cdot
\big( 1+ \frac{2 q^{-d_2-1}}{[1]_{q^{-1}}} \big)
\Big) \Big|_{q^{\geq -d+3}}.
\end{split}
\end{equation*}
Here
\be
\label{eq-tW23decomp}
\tW_{(d_2),(d_3)}(q) =
\frac{1}{[\infty]_{q^{-1}}!}\cdot \big(
T_{d_3}^x \frac{(q^{-d_2 -1} ;q^{-1})_\infty}{(xq^{-d_2 -1} ;q^{-1})_\infty} -1 \big)
+\frac{1}{[\infty]_{q^{-1}}!},
\ee
and
\begin{equation*}
\deg \big(
T_{d_3}^x \frac{(q^{-d_2 -1} ;q^{-1})_\infty}{(xq^{-d_2 -1} ;q^{-1})_\infty} -1 \big)
= -(d_2+1)(d_3+1) < -d+3,
\end{equation*}
and then we obtain:
\be
\label{eq-Idelta=2xx1sum}
\begin{split}
& \sum_{d_2+d_3 =d-2} \tW_{(d_2),(d_3)} \tW_{(1),(d_2)} \tW_{(1),(d_3)} \\
=& \Big(
\sum_{d_2+d_3 =d-2} \frac{1}{[d_2]_{q^{-1}}! [d_3]_{q^{-1}}!} \cdot
\frac{1}{[\infty]_{q^{-1}}!} \cdot
\big( 1+ \frac{2 q^{-d_2-1}}{[1]_{q^{-1}}} \big)
\Big) \Big|_{q^{\geq -d+3}}.
\end{split}
\ee
By \eqref{eq-sum1overd2d3} we already know:
\begin{equation*}
\begin{split}
\sum_{d_2 +d_3 =d -2 } \frac{1}{[d_2]_{q^{-1}}![d_3]_{q^{-1}}! \cdot [\infty]_{q^{-1}}! }
=&  \frac{1}{([\infty]_{q^{-1}}!)^3} \Big( d-1
- 2 \sum_{i\geq 1} \frac{q^{-i}}{1-q^{-i}} \Big)
\Big|_{q^{\geq -d+1}}.
\end{split}
\end{equation*}
Moreover,
\begin{equation*}
\begin{split}
\sum_{d_2 +d_3 =d -2 } \frac{2q^{-d_2-1}}{[d_2]_{q^{-1}}![d_3]_{q^{-1}}! }
=& \sum_{d_2 +d_3 =d -2 } \Big( \frac{2q^{-d_2-1}}{[d_2]_{q^{-1}}! }
\big( \frac{1}{[d_3]_{q^{-1}}!} - \frac{1}{[\infty]_{q^{-1}}!} \big)
+ \frac{2q^{-d_2-1}}{[d_2]_{q^{-1}}!  [\infty]_{q^{-1}}!} \Big)
\\
=& \Big( \sum_{d_2 +d_3 =d -2 }
\frac{2q^{-d_2-1}}{[d_2]_{q^{-1}}!  [\infty]_{q^{-1}}!}
\Big) \Big|_{q^{\geq -d+1}}.
\end{split}
\end{equation*}
and then by combining this with \eqref{eq-sumqd2-11overd2} we get:
\begin{equation*}
\begin{split}
\sum_{d_2 +d_3 =d -2 } \frac{2q^{-d_2-1}}{[d_2]_{q^{-1}}![d_3]_{q^{-1}}! }
=& \frac{2q^{-1}}{([\infty]_{q^{-1}}!)^2}
 \Big|_{q^{\geq -d+1}}.
\end{split}
\end{equation*}
Thus we have:
\be
\label{eq-Idelta=2-xcompxx}
\begin{split}
& \sum_{d_2+d_3 =d-2} \frac{1}{[d_2]_{q^{-1}}! [d_3]_{q^{-1}}!} \cdot
\frac{1}{[\infty]_{q^{-1}}!} \cdot
\big( 1+ \frac{2 q^{-d_2-1}}{[1]_{q^{-1}}} \big) \\
=&
\frac{1}{([\infty]_{q^{-1}}!)^3}\Big( d-1-2\sum_{i\geq 2}\frac{q^{-i}}{1-q^{-i}}
\Big) \Big|_{q^{\geq -d+1}},
\end{split}
\ee
and by plugging this into \eqref{eq-Idelta=2xx1sum} we obtain:
\be
\label{eq-Idelta=2-xcomp3}
\begin{split}
\sum_{d_2+d_3 =d-2} \tW_{(d_2),(d_3)} \tW_{(1),(d_2)} \tW_{(1),(d_3)}
= \frac{1}{([\infty]_{q^{-1}}!)^3}\Big( d-1-2\sum_{i\geq 2}\frac{q^{-i}}{1-q^{-i}}
\Big) \Big|_{q^{\geq -d+3}}.
\end{split}
\ee

Finally we need to compute the leading terms of:
\begin{equation*}
\begin{split}
& \sum_{d_2+d_3 =d-3} \tW_{(d_2),(d_3)} \tW_{(2),(d_2)} \tW_{(2),(d_3)} \\
=&
\sum_{d_2+d_3 =d-3} \frac{\tW_{(d_2),(d_3)} }{[d_2]_{q^{-1}}! [d_3]_{q^{-1}}!} \cdot
\big( 1+ \frac{q^{-d_2-1}}{[1]_{q^{-1}}} + \frac{q^{-2d_2-2}}{[2]_{q^{-1}}!}\big)
\big( 1+ \frac{q^{-d_3-1}}{[1]_{q^{-1}}} + \frac{q^{-2d_3-2}}{[2]_{q^{-1}}!}\big),
\end{split}
\end{equation*}
and again by degree reason and the decomposition \eqref{eq-tW23decomp} one obtains:
\begin{equation*}
\begin{split}
& \sum_{d_2+d_3 =d-3} \tW_{(d_2),(d_3)} \tW_{(2),(d_2)} \tW_{(2),(d_3)} \\
=&
\sum_{d_2+d_3 =d-3} \frac{\tW_{(d_2),(d_3)}}{[d_2]_{q^{-1}}! [d_3]_{q^{-1}}!} \cdot
\big( 1+ \frac{2 q^{-d_2-1}}{[1]_{q^{-1}}} + \frac{2 q^{-2d_2-2}}{[2]_{q^{-1}}!} \big)\\
=& \Big(
\sum_{d_2+d_3 =d-3} \frac{1 }{[d_2]_{q^{-1}}! [d_3]_{q^{-1}}!} \cdot \frac{1}{[\infty]_{q^{-1}}!}
\big( 1+ \frac{2 q^{-d_2-1}}{[1]_{q^{-1}}} + \frac{2 q^{-2d_2-2}}{[2]_{q^{-1}}!} \big)
\Big) \Big|_{q^{\geq -d+3}}.
\end{split}
\end{equation*}
From \eqref{eq-Idelta=2-xcompxx} we already know that:
\begin{equation*}
\begin{split}
\sum_{d_2+d_3 =d-3} \frac{1}{[d_2]_{q^{-1}}! [d_3]_{q^{-1}}! [\infty]_{q^{-1}}!}
\big( 1+ \frac{2 q^{-d_2-1}}{[1]_{q^{-1}}} \big)
=
\frac{1}{([\infty]_{q^{-1}}!)^3}\Big( d-2 -2\sum_{i\geq 2}\frac{q^{-i}}{1-q^{-i}}
\Big) \Big|_{q^{\geq -d+2}}.
\end{split}
\end{equation*}
Moreover,
one has:
\begin{equation*}
\begin{split}
\sum_{d_2+d_3 =d-3} \frac{q^{-2d_2-2}}{[d_2]_{q^{-1}}! [d_3]_{q^{-1}}!}
=&
\sum_{d_2+d_3 =d-3} \Big(
\frac{q^{-2d_2-2}}{[d_2]_{q^{-1}}!}
\big(  \frac{1}{[d_3]_{q^{-1}}!} - \frac{1}{[\infty]_{q^{-1}}!} \big)
+ \frac{q^{-2d_2-2}}{[d_2]_{q^{-1}}! [\infty]_{q^{-1}}!} \Big)\\
=&  \Big(\sum_{d_2+d_3 =d-3}
 \frac{q^{-2d_2-2}}{[d_2]_{q^{-1}}! [\infty]_{q^{-1}}!} \Big)
\Big|_{q^{\geq -d+3}},
\end{split}
\end{equation*}
and when restricting to $q^{\geq -d+3}$ one has
\begin{equation*}
\begin{split}
\sum_{d_2=0}^{d-3}
 \frac{q^{-2d_2-2}}{[d_2]_{q^{-1}}! }
=& \sum_{d_2=0}^{\infty}
 \frac{q^{-2d_2-2}}{[d_2]_{q^{-1}}! }
=\frac{q^{-2}}{(q^{-2};q^{-1})_\infty}
\Big|_{q^{\geq -d+3}},
\end{split}
\end{equation*}
therefore:
\begin{equation*}
\begin{split}
\sum_{d_2+d_3 =d-3} \frac{q^{-2d_2-2}}{[d_2]_{q^{-1}}! [d_3]_{q^{-1}}!}
=& \frac{[1]_{q^{-1}}\cdot q^{-2} }{([\infty]_{q^{-1}}!)^2}
\Big|_{q^{\geq -d+3}},
\end{split}
\end{equation*}
and then:
\be
\label{eq-Idelta=2-xcomp4}
\begin{split}
 \sum_{d_2+d_3 =d-3} \tW_{(d_2),(d_3)} \tW_{(2),(d_2)} \tW_{(2),(d_3)}
=&
\frac{1}{([\infty]_{q^{-1}}!)^3}\Big( d-2 -2\sum_{i\geq 3}\frac{q^{-i}}{1-q^{-i}}
\Big) \Big|_{q^{\geq -d+2}}.
\end{split}
\ee
Now we plug \eqref{eq-Idelta=2-xcomp1},
\eqref{eq-Idelta=2-xcomp2},
\eqref{eq-Idelta=2-xcomp3},
and \eqref{eq-Idelta=2-xcomp4}
into \eqref{eq-Idelta=2-xdecomp} and obtain:
\be
\label{eq-appC-leadingx}
\begin{split}
\text{(x)} =& \frac{1}{([2]_{q^{-1}}!)^2 \cdot ([\infty]_{q^{-1}}!)^3}\Big(
\binom{d+1}{2}-3d+3 + \frac{2q^{-1}}{[1]_{q^{-1}}} -\frac{2q^{-2}}{[2]_{q^{-1}}}\\
&\qquad\qquad\qquad\qquad
 + 6\sum_{i\geq 2} \frac{q^{-i}}{1-q^{-i}}
-3\sum_{i\geq 1} \frac{q^{-i}}{(1-q^{-i})^2}
\Big)\Big|_{q^{\geq -d+3}}.
\end{split}
\ee

Finally we compute the leading terms of $\text{(y)}$.
By decomposing the factor $\tW_{(d_2),(d_3)}(q)$ in the definition
of $\text{(y)}$ as \eqref{eq-tW23decomp},
we see that:
\begin{equation*}
\begin{split}
& \text{(y)} =\Big(
\sum_{\substack{d_1+d_2+d_3=d\\ d_2,d_3\geq 0;\text{ }d_1\geq 4}}
\frac{q^{-(d_1-3)(d_3+1)}}{[2]_{q^{-1}}! [d_1-1]_{q^{-1}}!} \cdot
\frac{1}{[\infty]_{q^{-1}}!}\cdot
\tW_{(d_1-1),(d_2)}(q)\tW_{(2),(d_3)}(q)
\Big) \Big|_{q^{\geq -d+3}}
\end{split}
\end{equation*}
by degree reason.
And similarly,
by the decomposition
\begin{equation*}
\tW_{(d_1-1),(d_2)}(q) =
\frac{1}{[\infty]_{q^{-1}}!}\cdot \big(
T_{d_1-1}^x \frac{(q^{-d_2 -1} ;q^{-1})_\infty}{(xq^{-d_2 -1} ;q^{-1})_\infty} -1 \big)
+\frac{1}{[\infty]_{q^{-1}}!},
\end{equation*}
and degree reason one obtains:
\begin{equation*}
\begin{split}
 \text{(y)} =&
\sum_{\substack{d_1+d_2+d_3=d\\ d_2,d_3\geq 0;\text{ }d_1\geq 4}}
\frac{q^{-(d_1-3)(d_3+1)}}{[2]_{q^{-1}}! [d_1-1]_{q^{-1}}!} \cdot
\frac{1}{([\infty]_{q^{-1}}!)^2}\cdot \tW_{(2),(d_3)}(q)\\
=& \frac{1}{[2]_{q^{-1}}!\cdot ([\infty]_{q^{-1}}!)^2}
\sum_{d_1 =4}^d \sum_{d_3 = 0}^{d-d_1} \frac{q^{-(d_1-3)(d_3+1)}}{[d_3]_{q^{-1}}! \cdot [d_1-1]_{q^{-1}}!}
\big( 1 +\frac{q^{-d_3-1}}{[1]_{q^{-1}}} + \frac{q^{-2d_3-2}}{[2]_{q^{-1}}!}
\big) \Big|_{q^{\geq -d+3}}.
\end{split}
\end{equation*}
When $d_1\geq 4$ and $d_1+d_3 \geq d+1$,
one always has:
\begin{equation*}
-(d_1-3)(d_3+1) \leq -d+2,
\end{equation*}
and thus:
\begin{equation*}
\begin{split}
 \text{(y)} =
 \frac{1}{[2]_{q^{-1}}!\cdot ([\infty]_{q^{-1}}!)^2}
\sum_{d_1 =4}^\infty \sum_{d_3 = 0}^{\infty} \frac{q^{-(d_1-3)(d_3+1)}}{[d_3]_{q^{-1}}! \cdot [d_1-1]_{q^{-1}}!}
\big( 1 +\frac{q^{-d_3-1}}{[1]_{q^{-1}}} + \frac{q^{-2d_3-2}}{[2]_{q^{-1}}!}
\big) \Big|_{q^{\geq -d+3}}.
\end{split}
\end{equation*}
And by \eqref{eq-pochhammerseries} one has:
\begin{equation*}
\begin{split}
& \sum_{d_1 =4}^\infty \sum_{d_3 = 0}^{\infty} \frac{q^{-(d_1-3)(d_3+1)}}{[d_3]_{q^{-1}}! \cdot [d_1-1]_{q^{-1}}!}
\big( 1 +\frac{q^{-d_3-1}}{[1]_{q^{-1}}} + \frac{q^{-2d_3-2}}{[2]_{q^{-1}}!}
\big) \\
=& \sum_{d_1=4}^\infty \frac{1}{[d_1-1]_{q^{-1}}!}
\Big( \frac{q^{-(d_1-3)}}{(q^{-(d_1-3)};q^{-1})_\infty}
+ \frac{q^{-(d_1-2)}}{[1]_{q^{-1}} (q^{-(d_1-2)};q^{-1})_\infty}
+ \frac{q^{-(d_1-1)}}{[2]_{q^{-1}}! \cdot(q^{-(d_1-1)};q^{-1})_\infty}
\Big) \\
=& \sum_{d_1=4}^\infty \frac{1}{ [\infty]_{q^{-1}}!}
\Big( \frac{q^{-(d_1-3)}}{[d_1-1]_{q^{-1}}[d_1-2]_{q^{-1}}[d_1-3]_{q^{-1}}}
 + \frac{q^{-(d_1-2)}}{[1]_{q^{-1}}\cdot [d_1-1]_{q^{-1}}[d_1-2]_{q^{-1}}} \\
&\qquad\qquad\qquad
+ \frac{q^{-(d_1-1)}}{[2]_{q^{-1}}! \cdot [d_1-1]_{q^{-1}}} \Big).
\end{split}
\end{equation*}
Here for $|q^{-1}|<1$ one has:
\begin{equation*}
\sum_{d_1=4}^\infty
\frac{q^{-(d_1-2)}}{ [d_1-1]_{q^{-1}}[d_1-2]_{q^{-1}}}
=-\sum_{d_1 =4}^\infty \frac{1}{[1]_{q^{-1}}}\big(
\frac{1}{[d_1-1]_{q^{-1}}} - \frac{1}{[d_1-2]_{q^{-1}}}
\big)
= \frac{q^{-2}}{[2]_{q^{-1}}!},
\end{equation*}
and similarly,
\be
\label{eq-1over3prod-split}
\begin{split}
& \sum_{d_1=4}^\infty
\frac{q^{-(d_1-3)}}{[d_1-1]_{q^{-1}}[d_1-2]_{q^{-1}}[d_1-3]_{q^{-1}}} \\
=& -\sum_{d_1=4}^\infty \frac{1}{[2]_{q^{-1}}}\big(
\frac{1}{[d_1-1]_{q^{-1}}[d_1-2]_{q^{-1}}} - \frac{1}{[d_1-2]_{q^{-1}}[d_1-3]_{q^{-1}}}
\big)
= \frac{q^{-1}+q^{-2}-q^{-3}}{[1]_{q^{-1}}\cdot [2]_{q^{-1}}^2}.
\end{split}
\ee
Therefore (after some simplifications) we obtain:
\begin{equation*}
\begin{split}
\text{(y)} =
\frac{1}{([2]_{q^{-1}}!)^2 ([\infty]_{q^{-1}}!)^3}
\Big( \frac{q^{-1}}{[1]_{q^{-1}}} + \sum_{i\geq 2} \frac{q^{-i}}{1-q^{-i}}
\Big)\Big|_{q^{\geq -d+3}}.
\end{split}
\end{equation*}

Now by this equality and \eqref{eq-appC-leadingz},
\eqref{eq-appC-leadingx}, we finally obtain:
\begin{Proposition}
\label{Prop-I=2-case1}
For $d\geq 4$,
the leading terms of $(\text{j}_1)$ are given by:
\be
\begin{split}
(\text{j}_1) =& \frac{(-1)^d q^{\half d(d-7)-2}}{([2]_{q^{-1}}!)^2 \cdot ([\infty]_{q^{-1}}!)^3}
\cdot \Big(
\binom{d+1}{2}-3d+3  -\frac{2q^{-2}}{[2]_{q^{-1}}} \\
&\qquad\qquad\qquad\qquad
 + 4\sum_{i\geq 2} \frac{q^{-i}}{1-q^{-i}}
-3 \sum_{i\geq 1} \frac{q^{-i}}{(1-q^{-i})^2}
\Big)\Big|_{q^{\geq \half d(d-3)-3d+1}}.
\end{split}
\ee
\end{Proposition}

\subsection{Leading terms of $I^{=2}(d)$ -- Case 2)}

In this subsection we compute the leading terms of $(\text{j}_2)$.
By \eqref{eq-gz-183184} and Theorem \ref{thm-Delta2-expl} one has:
\begin{equation*}
\begin{split}
(\text{j}_2) =&
\sum_{\substack{d_1+d_2+d_3=d\\ d_2,d_3\geq 0;\text{ }d_1\geq 3}}
(-1)^d q^{ \half (d_1^2-5d_1 + d_2^2-d_2+ d_3^2-d_3)}
\cdot q^{d_2 d_3-\half (d_2+d_3)} \cdot \tW_{(d_2),(d_3)}(q)
\\
&\cdot
\Big( \frac{q^{d_1d_2-\half d_1 - \frac{5}{2}d_2 -3}}{[2]_{q^{-1}}!} \tW_{(d_1-3),(d_2-1)}(q)
+ q^{-\half d_1 -\half d_2 -3}\cdot
\frac{[d_1-1]_{q^{-1}} [d_1-2]_{q^{-1}}}{[2]_{q^{-1}}! \cdot [d_1]_{q^{-1}}! \cdot [d_2]_{q^{-1}}!} \Big)\\
&\cdot
\Big( \frac{q^{d_1d_3-\half d_1 - \frac{5}{2}d_3 -3}}{[2]_{q^{-1}}!} \tW_{(d_1-3),(d_3-1)}(q)
+ q^{-\half d_1 -\half d_3 -3}\cdot
\frac{[d_1-1]_{q^{-1}} [d_1-2]_{q^{-1}}}{[2]_{q^{-1}}! \cdot [d_1]_{q^{-1}}! \cdot [d_3]_{q^{-1}}!} \Big),
\end{split}
\end{equation*}
where we use the notation $\tW_{(d_1-3),(-1)} = 0$.
Then:
\be
\label{eq-AppC5-J2decomp}
(\text{j}_2) = (-1)^d \cdot q^{\half d(d-7) -6} \cdot
\big( \text{(x')} +2\cdot \text{(y')} +\text{(z')} \big),
\ee
where:
\begin{equation*}
\begin{split}
&\text{(x')} = \frac{1}{([2]_{q^{-1}}! )^2}\sum_{\substack{d_1+d_2+d_3=d\\ d_2,d_3\geq 1;\text{ }d_1\geq 3}}
\tW_{(d_2),(d_3)} \tW_{(d_1-3),(d_2-1)} \tW_{(d_1-3),(d_3-1)}, \\
&\text{(y')} = \frac{1}{([2]_{q^{-1}}! )^2}
\sum_{\substack{d_1+d_2+d_3=d\\ d_2\geq 1;\text{ }d_3\geq 0;\text{ }d_1\geq 3}}
q^{-(d_1-2) d_3} \cdot \frac{[d_1-1]_{q^{-1}} [d_1-2]_{q^{-1}}}{ [d_1]_{q^{-1}}! \cdot [d_3]_{q^{-1}}!}
\tW_{(d_2),(d_3)} \tW_{(d_1-3),(d_2-1)}, \\
& \text{(z')} = \sum_{\substack{d_1+d_2+d_3=d\\ d_2,d_3\geq 0;\text{ }d_1\geq 3}}
q^{-(d_1-2)(d_2+d_3)} \cdot
\frac{[d_1-1]_{q^{-1}}^2 [d_1-2]_{q^{-1}}^2 }
{([2]_{q^{-1}}!)^2 \cdot ([d_1]_{q^{-1}}!)^2 \cdot [d_2]_{q^{-1}}!\cdot [d_3]_{q^{-1}}!}
\tW_{(d_2),(d_3)}.
\end{split}
\end{equation*}

First consider the leading terms of $\text{(z')}$.
One only needs to consider the special case $d_1=d$ and $d_2=d_3=0$, since
\begin{equation*}
\deg q^{-(d_1-2)(d_2+d_3)} = -(d_1-2)(d-d_1) \leq -d+3
\end{equation*}
if $d_1 \geq 1$.
Therefore:
\begin{equation*}
\begin{split}
\text{(z')} =
\frac{[d-1]_{q^{-1}}^2 [d-2]_{q^{-1}}^2 }
{([2]_{q^{-1}}!)^2 \cdot ([d]_{q^{-1}}!)^2 }
\Big|_{q^{\geq -d+7}}.
\end{split}
\end{equation*}
Moreover,
one has:
\begin{equation*}
\begin{split}
[d-1]_{q^{-1}}^2 [d-2]_{q^{-1}}^2 = 1 \big|_{q^{\geq -d+7}},\qquad\qquad
\frac{1}{([d]_{q^{-1}}!)^2}
= \frac{1}{([\infty]_{q^{-1}}!)^2} \Big|_{q^{\geq -d+7}},
\end{split}
\end{equation*}
and thus:
\be
\label{eq-zprime-leading}
\begin{split}
\text{(z')} =
\frac{1 }
{([2]_{q^{-1}}!)^2 \cdot ([\infty]_{q^{-1}}!)^2 }
\Big|_{q^{\geq -d+7}}.
\end{split}
\ee

Next we consider the leading terms of $\text{(y')}$.
Similar to the discussions in \S \ref{sec-appB1-gz},
we decompose each factor in the summand as follows:
\begin{equation*}
\begin{split}
\text{(x')} =\frac{1}{([2]_{q^{-1}}! )^2}
 \sum_{\substack{d_1+d_2+d_3=d\\ d_2,d_3\geq 1;\text{ }d_1\geq 3}}&
\Big(\frac{1}{[\infty]_{q^{-1}}!} \big(
T_{d_2}^x \frac{(q^{-d_3 -1} ;q^{-1})_\infty}{(xq^{-d_3 -1} ;q^{-1})_\infty} -1
\big) + \frac{1}{[\infty]_{q^{-1}}!} \Big) \\
&\cdot \Big(\frac{1}{[\infty]_{q^{-1}}!} \big(
T_{d_1 -3}^x \frac{(q^{-d_2} ;q^{-1})_\infty}{(xq^{-d_2} ;q^{-1})_\infty} -1
\big) + \frac{1}{[\infty]_{q^{-1}}!} \Big) \\
&\cdot \Big(\frac{1}{[\infty]_{q^{-1}}!} \big(
T_{d_1 -3}^x \frac{(q^{-d_3} ;q^{-1})_\infty}{(xq^{-d_3} ;q^{-1})_\infty} -1
\big) + \frac{1}{[\infty]_{q^{-1}}!} \Big),
\end{split}
\end{equation*}
and then by \eqref{eq-gzlem-5.7} and a simple degree estimate
we obtain (here we omit the details):
\be
\label{eq-appc-decompxprime}
\begin{split}
\text{(x')} = &
\frac{1}{( [2]_{q^{-1}}! )^2\cdot([\infty]_{q^{-1}}!)^3} \bigg( \binom{d-3}{2} +
\sum_{\substack{d_1+d_2+d_3=d\\ d_2,d_3\geq 1;\text{ }d_1\geq 3}} \Big(
\big(
T_{d_2}^x \frac{(q^{-d_3 -1} ;q^{-1})_\infty}{(xq^{-d_3 -1} ;q^{-1})_\infty} -1
\big) \\
& + \big(
T_{d_1 -3}^x \frac{(q^{-d_2} ;q^{-1})_\infty}{(xq^{-d_2} ;q^{-1})_\infty} -1
\big) + \big(
T_{d_1 -3}^x \frac{(q^{-d_3} ;q^{-1})_\infty}{(xq^{-d_3} ;q^{-1})_\infty} -1
\big)  \Big) \bigg) \bigg|_{q^{\geq -d+7}}.
\end{split}
\ee
Now consider the summation:
\begin{equation*}
\begin{split}
\sum_{\substack{d_1+d_2+d_3=d\\ d_2,d_3\geq 1;\text{ }d_1\geq 3}}
\big(
T_{d_2}^x \frac{(q^{-d_3 -1} ;q^{-1})_\infty}{(xq^{-d_3 -1} ;q^{-1})_\infty} -1
\big) =
\sum_{\substack{d_2+d_3 \leq d-3 \\ d_2,d_3\geq 1}}
\big(
T_{d_2}^x \frac{(q^{-d_3 -1} ;q^{-1})_\infty}{(xq^{-d_3 -1} ;q^{-1})_\infty} -1
\big).
\end{split}
\end{equation*}
If $d_2 + d_3\geq d-2$,
one has:
\begin{equation*}
\begin{split}
\deg
\big(
T_{d_2}^x \frac{(q^{-d_3 -1} ;q^{-1})_\infty}{(xq^{-d_3 -1} ;q^{-1})_\infty} -1
\big) =-(d_2+1)(d_3+1) \leq -2d+4,
\end{split}
\end{equation*}
and thus:
\begin{equation*}
\begin{split}
\sum_{\substack{d_1+d_2+d_3=d\\ d_2,d_3\geq 1;\text{ }d_1\geq 3}}
\big(
T_{d_2}^x \frac{(q^{-d_3 -1} ;q^{-1})_\infty}{(xq^{-d_3 -1} ;q^{-1})_\infty} -1
\big) =
\sum_{d_2 =1}^\infty \sum_{d_3=1}^\infty
\big(
T_{d_2}^x \frac{(q^{-d_3 -1} ;q^{-1})_\infty}{(xq^{-d_3 -1} ;q^{-1})_\infty} -1
\big)
\Big|_{q^{\geq -d+3}}.
\end{split}
\end{equation*}
(Here we use the truncation $q^{\geq -d+3}$ instead of $q^{\geq -d+7}$ for the convenience in the next subsection.)
By \eqref{eq-W2infty-def} and \eqref{eq-W2infty=} one already knows that:
\be
\label{eq-appC-doublesuminf}
\sum_{d_2 =0}^\infty \sum_{d_3=0}^\infty
\big(
T_{d_2}^x \frac{(q^{-d_3 -1} ;q^{-1})_\infty}{(xq^{-d_3 -1} ;q^{-1})_\infty} -1
\big) =
-\sum_{i\geq 1} \frac{q^{-i}}{(1-q^{-i})^2},
\ee
and then we have:
\begin{equation*}
\begin{split}
& \sum_{\substack{d_1+d_2+d_3=d\\ d_2,d_3\geq 1;\text{ }d_1\geq 3}}
\big(
T_{d_2}^x \frac{(q^{-d_3 -1} ;q^{-1})_\infty}{(xq^{-d_3 -1} ;q^{-1})_\infty} -1
\big)\\ = &
\Big( -\sum_{i\geq 1} \frac{q^{-i}}{(1-q^{-i})^2}
- \big(
T_{0}^x \frac{(q^{-1} ;q^{-1})_\infty}{(xq^{ -1} ;q^{-1})_\infty} -1
\big)
- 2\sum_{d_2 =1}^\infty
\big(
T_{d_2}^x \frac{(q^{-1} ;q^{-1})_\infty}{(xq^{ -1} ;q^{-1})_\infty} -1
\big) \Big)
\Big|_{q^{\geq -d+3}}.
\end{split}
\end{equation*}
Then apply \eqref{eq-computed1=01} and \eqref{eq-gz-prop5.8}, and one obtains:
\be
\label{eq-Compxprime-1}
\begin{split}
& \sum_{\substack{d_1+d_2+d_3=d\\ d_2,d_3\geq 1;\text{ }d_1\geq 3}}
\big(
T_{d_2}^x \frac{(q^{-d_3 -1} ;q^{-1})_\infty}{(xq^{-d_3 -1} ;q^{-1})_\infty} -1
\big)\\
=& \Big( [\infty]_{q^{-1}}! -1
-\sum_{i\geq 1} \frac{q^{-i}}{(1-q^{-i})^2}
+ 2\sum_{i\geq 1} \frac{q^{-i}}{1-q^{-i}}
\Big)
\Big|_{q^{\geq -d+3}}.
\end{split}
\ee

Next we compute the remaining two summations in \eqref{eq-appc-decompxprime}:
\begin{equation*}
\begin{split}
\sum_{\substack{d_1+d_2+d_3=d\\ d_2,d_3\geq 1;\text{ }d_1\geq 3}}  \big(
T_{d_1 -3}^x \frac{(q^{-d_2} ;q^{-1})_\infty}{(xq^{-d_2} ;q^{-1})_\infty} -1
\big) =
\sum_{\substack{d_1+d_2+d_3=d\\ d_2,d_3\geq 1;\text{ }d_1\geq 3}}  \big(
T_{d_1 -3}^x \frac{(q^{-d_2} ;q^{-1})_\infty}{(xq^{-d_3} ;q^{-1})_\infty} -1
\big).
\end{split}
\end{equation*}
Similarly by degree reason and \eqref{eq-appC-doublesuminf} one has:
\begin{equation*}
\begin{split}
\sum_{\substack{d_1+d_2+d_3=d\\ d_2,d_3\geq 1;\text{ }d_1\geq 3}}  \big(
T_{d_1 -3}^x \frac{(q^{-d_2} ;q^{-1})_\infty}{(xq^{-d_2} ;q^{-1})_\infty} -1
\big) =&
\sum_{d_1 = 3}^\infty \sum_{d_2 =1}^\infty \big(
T_{d_1 -3}^x \frac{(q^{-d_2} ;q^{-1})_\infty}{(xq^{-d_2} ;q^{-1})_\infty} -1
\big) \\
=& -\sum_{i\geq 1} \frac{q^{-i}}{(1-q^{-i})^2} \Big|_{q^{\geq -d+7}}.
\end{split}
\end{equation*}
Now plug these results into \eqref{eq-appc-decompxprime},
and then we obtain the leading terms of $\text{(x')}$:
\be
\label{eq-xprime-leading}
\begin{split}
\text{(x')} = &
\frac{1}{([2]_{q^{-1}}!)^2 \cdot([\infty]_{q^{-1}}!)^3} \Big( \binom{d-3}{2} + ([\infty]_{q^{-1}}! -1) \\
&\qquad\qquad\qquad
+2\sum_{i\geq 1}\frac{q^{-i}}{1-q^{-i}} -3\sum_{i\geq 1}\frac{q^{-i}}{(1-q^{-i})^2}
\Big)  \Big|_{q^{\geq -d+7}}.
\end{split}
\ee

Finally we compute the leading terms of $\text{(y')}$.
Decompose the factor $\tW_{(d_1-3),(d_2-1)}$ in $\text{(y')}$ as:
\begin{equation*}
\tW_{(d_1-3),(d_2-1)} = \frac{1}{[\infty]_{q^{-1}}!}
\big( T_{d_1 -3}^x \frac{(q^{-d_2} ;q^{-1})_\infty}{(xq^{-d_2} ;q^{-1})_\infty} -1 \big)
+ \frac{1}{[\infty]_{q^{-1}}!},
\end{equation*}
and then we have the degree estimate (Notice here $d_1\geq 3$ and $d_2\geq 1$):
\begin{equation*}
\deg q^{-(d_1-2)d_3}
\big( T_{d_1 -3}^x \frac{(q^{-d_2} ;q^{-1})_\infty}{(xq^{-d_2} ;q^{-1})_\infty} -1 \big)
= -(d_1-2)(d_2+d_3)
\leq -d+2
\end{equation*}
we obtain
\begin{equation*}
\begin{split}
&\text{(y')} = \frac{1}{([2]_{q^{-1}}! )^2}
\sum_{\substack{d_1+d_2+d_3=d\\ d_2\geq 1;\text{ }d_3\geq 0;\text{ }d_1\geq 3}}
q^{-(d_1-2) d_3} \cdot \frac{[d_1-1]_{q^{-1}} [d_1-2]_{q^{-1}}}{ [d_1]_{q^{-1}}! \cdot [d_3]_{q^{-1}}!}
\tW_{(d_2),(d_3)}\cdot \frac{1}{[\infty]_{q^{-1}}!}
\Big|_{q^{\geq -d+7}}. \\
\end{split}
\end{equation*}
Then we apply the decomposition:
\begin{equation*}
\tW_{(d_2),(d_3)} = \frac{1}{[\infty]_{q^{-1}}!}
\big( T_{d_2}^x \frac{(q^{-d_3-1} ;q^{-1})_\infty}{(xq^{-d_3-1} ;q^{-1})_\infty} -1 \big)
+ \frac{1}{[\infty]_{q^{-1}}!}.
\end{equation*}
In this case one has:
\begin{equation*}
\deg q^{-(d_1-2) d_3} \cdot
\big( T_{d_2}^x \frac{(q^{-d_3-1} ;q^{-1})_\infty}{(xq^{-d_3-1} ;q^{-1})_\infty} -1 \big)
= -d_3 (d-1-d_3) -(d_2 +1),
\end{equation*}
and thus by degree reason one only needs to consider the case $d_3 =0$.
In this way one obtains:
\be
\label{eq-AppC-decompyprime}
\begin{split}
\text{(y')} =& \frac{1}{([2]_{q^{-1}}!)^2  ([\infty]_{q^{-1}}!)^2} \Big(
\sum_{\substack{d_1+d_2 =d\\ d_2\geq 1;\text{ }d_1\geq 3}}
 \frac{[d_1-1]_{q^{-1}} [d_1-2]_{q^{-1}}}{ [d_1]_{q^{-1}}! }
\big( T_{d_2}^x \frac{(q^{-1} ;q^{-1})_\infty}{(xq^{-1} ;q^{-1})_\infty} -1 \big)  \\
& \qquad\qquad\qquad\qquad
+ \sum_{\substack{d_1+d_2+d_3=d\\ d_2\geq 1;\text{ }d_3\geq 0;\text{ }d_1\geq 3}}
q^{-(d_1-2) d_3} \cdot \frac{[d_1-1]_{q^{-1}} [d_1-2]_{q^{-1}}}{ [d_1]_{q^{-1}}! \cdot [d_3]_{q^{-1}}!}
\Big)
\Big|_{q^{\geq -d+7}}. \\
\end{split}
\ee
Consider the first summation in the right-hand side.
Notice that:
\begin{equation*}
[d_1-1]_{q^{-1}} [d_1-2]_{q^{-1}} =
1-q^{-d_1+1} -q^{-d_1+2} +q^{-2d_1+3},
\end{equation*}
and
\begin{equation*}
\begin{split}
\deg \Big( (-q^{-d_1+1} -q^{-d_1+2} +q^{-2d_1+3})
\big( T_{d_2}^x \frac{(q^{-1} ;q^{-1})_\infty}{(xq^{-1} ;q^{-1})_\infty} -1 \big)
\Big) <-d+2
\end{split}
\end{equation*}
for $d_1 +d_2 =d$,
and thus:
\be
\label{eq-AppC-yprime1stsum}
\begin{split}
&\sum_{\substack{d_1+d_2 =d\\ d_2\geq 1;\text{ }d_1\geq 3}}
 \frac{[d_1-1]_{q^{-1}} [d_1-2]_{q^{-1}}}{ [d_1]_{q^{-1}}! }
\big( T_{d_2}^x \frac{(q^{-1} ;q^{-1})_\infty}{(xq^{-1} ;q^{-1})_\infty} -1 \big) \\
=& \sum_{\substack{d_1+d_2 =d\\ d_2\geq 1;\text{ }d_1\geq 3}}
 \frac{1}{ [d_1]_{q^{-1}}! }
\big( T_{d_2}^x \frac{(q^{-1} ;q^{-1})_\infty}{(xq^{-1} ;q^{-1})_\infty} -1 \big)
\Big|_{q^{\geq -d+7}}.
\end{split}
\ee
Then by the decomposition
\begin{equation*}
\frac{1}{[d_1]_{q^{-1}}!} =
\big(\frac{1}{[d_1]_{q^{-1}}!} -\frac{1}{[\infty]_{q^{-1}}!} \big)
+ \frac{1}{[\infty]_{q^{-1}}!}
\end{equation*}
and by simple degree estimates one knows that
the right-hand side of \eqref{eq-AppC-yprime1stsum} equals to:
\begin{equation*}
\begin{split}
\frac{1}{[\infty]_{q^{-1}}!}
\sum_{d_2 = 1}^{d-3}
\big(\frac{[\infty]_{q^{-1}}!}{[d_2]_{q^{-1}}!}- 1 \big)
=  \sum_{d_2 = 1}^{\infty}
\big(\frac{1}{[d_2]_{q^{-1}}!} -\frac{1}{[\infty]_{q^{-1}}!} \big)
\Big|_{q^{\geq -d+7}},
\end{split}
\end{equation*}
where we have also used \eqref{eq-computed1=01}.
And then by \eqref{eq-gz-prop5.8} one gets:
\be
\label{eq-AppC-decompyprime-2}
\begin{split}
&\sum_{\substack{d_1+d_2 =d\\ d_2\geq 1;\text{ }d_1\geq 3}}
 \frac{[d_1-1]_{q^{-1}} [d_1-2]_{q^{-1}}}{ [d_1]_{q^{-1}}! }
\big( T_{d_2}^x \frac{(q^{-1} ;q^{-1})_\infty}{(xq^{-1} ;q^{-1})_\infty} -1 \big) \\
=& \frac{1}{[\infty]_{q^{-1}}!}
\Big( 1-[\infty]_{q^{-1}}! -\sum_{i\geq 1} \frac{q^{-i}}{1-q^{-i}}
\Big) \Big|_{q^{\geq -d+7}}.
\end{split}
\ee

Now consider the second summation in the right-hand side of \eqref{eq-AppC-decompyprime}.
One has:
\be
\label{eq-AppC5-2ndsum}
\begin{split}
&\sum_{\substack{d_1+d_2+d_3=d\\ d_2\geq 1;\text{ }d_3\geq 0;\text{ }d_1\geq 3}}
q^{-(d_1-2) d_3} \cdot \frac{[d_1-1]_{q^{-1}} [d_1-2]_{q^{-1}}}{ [d_1]_{q^{-1}}! \cdot [d_3]_{q^{-1}}!} \\
=& \sum_{d_1 = 3}^{d-1} \sum_{d_3 =0}^{d-1-d_1}
\Big( \frac{q^{-(d_1-2) d_3} }{ [d_1]_{q^{-1}}! \cdot [d_3]_{q^{-1}}!}
- \frac{ (1+q^{-1}) \cdot q^{-(d_1-2) (d_3+1)} }{ [d_1]_{q^{-1}}! \cdot [d_3]_{q^{-1}}!}
+ \frac{q^{-(d_1-2) (d_3+2)-1} }{ [d_1]_{q^{-1}}! \cdot [d_3]_{q^{-1}}!} \Big).
\end{split}
\ee
Decompose the right-hand side into $\text{(A)}+\text{(B)}+\text{(C)}$ where:
\begin{equation*}
\begin{split}
&\text{(A)} =\sum_{d_1 = 3}^{d-1} \sum_{d_3 =0}^{d-1-d_1}
\frac{q^{-(d_1-2) d_3} }{ [d_1]_{q^{-1}}! \cdot [d_3]_{q^{-1}}!},\\
&\text{(B)} = -(1+q^{-1}) \cdot \sum_{d_1 = 3}^{d-1} \sum_{d_3 =0}^{d-1-d_1}
\frac{  q^{-(d_1-2) (d_3+1)} }{ [d_1]_{q^{-1}}! \cdot [d_3]_{q^{-1}}!},\\
&\text{(C)} = q^{-1} \cdot \sum_{d_1 = 3}^{d-1} \sum_{d_3 =0}^{d-1-d_1}
\frac{q^{-(d_1-2) (d_3+2)} }{ [d_1]_{q^{-1}}! \cdot [d_3]_{q^{-1}}!}.
\end{split}
\end{equation*}
First we compute the leading terms of $\text{(A)}$.
One has:
\begin{equation*}
\begin{split}
\text{(A)} =&\sum_{d_1 = 3}^{d-1} \sum_{d_3 =1}^{d-1-d_1}
\frac{q^{-(d_1-2) d_3} }{ [d_1]_{q^{-1}}! \cdot [d_3]_{q^{-1}}!}
+ \sum_{d_1=3}^{d-1} \frac{1}{[d_1]_{q^{-1}}!} \\
=& \Big( \sum_{d_1 = 3}^{\infty} \sum_{d_3 =1}^{\infty}
\frac{q^{-(d_1-2) d_3} }{ [d_1]_{q^{-1}}! \cdot [d_3]_{q^{-1}}!}
+ \sum_{d_1=3}^{d-1} \frac{1}{[d_1]_{q^{-1}}!}
\Big) \Big|_{q^{\geq -d+7}},
\end{split}
\end{equation*}
since when $d_1+d_3\geq d$ and $d_1 \geq 3$,  $d_3\geq 1$ one always has
$\deg  q^{-(d_1-2) d_3} \leq -d+3$.
Here we have:
\begin{equation*}
\begin{split}
\sum_{d_1=3}^{d-1} \frac{1}{[d_1]_{q^{-1}}!}
=& \frac{d-3}{[\infty]_{q^{-1}}!} +
\sum_{d_1=3}^{d-1} \big( \frac{1}{[d_1]_{q^{-1}}!} - \frac{1}{[\infty]_{q^{-1}}!} \big) \\
=& \Big(\frac{d-3}{[\infty]_{q^{-1}}!} +
\sum_{d_1=3}^{\infty} \big( \frac{1}{[d_1]_{q^{-1}}!} - \frac{1}{[\infty]_{q^{-1}}!} \big)
\Big)\Big|_{q^\geq -d+7},
\end{split}
\end{equation*}
Moreover,
by \eqref{eq-pochhammerseries} one has:
\begin{equation*}
\begin{split}
& \sum_{d_1 = 3}^{\infty} \sum_{d_3 =1}^{\infty}
\frac{q^{-(d_1-2) d_3} }{ [d_1]_{q^{-1}}! [d_3]_{q^{-1}}!} \\
=& \frac{1}{[\infty]_{q^{-1}}!}
\sum_{d_1=3}^\infty \big(  \frac{1}{[d_1]_{q^{-1}}[d_1-1]_{q^{-1}}[d_1-2]_{q^{-1}}} -1\big)
- \sum_{d_1=3}^\infty \big(
\frac{1}{[d_1]_{q^{-1}}! } -\frac{1}{[\infty]_{q^{-1}}!}\big),
\end{split}
\end{equation*}
and thus:
\begin{equation*}
\begin{split}
\text{(A)} =&\sum_{d_1 = 3}^{d-1} \sum_{d_3 =1}^{d-1-d_1}
\frac{q^{-(d_1-2) d_3} }{ [d_1]_{q^{-1}}! \cdot [d_3]_{q^{-1}}!}
+ \sum_{d_1=3}^{d-1} \frac{1}{[d_1]_{q^{-1}}!} \\
=& \Big(  \frac{d-3}{[\infty]_{q^{-1}}!} +
\frac{1}{[\infty]_{q^{-1}}!}
\sum_{d_1=3}^\infty \big(  \frac{1}{[d_1]_{q^{-1}}[d_1-1]_{q^{-1}}[d_1-2]_{q^{-1}}} -1\big)
\Big) \Big|_{q^{\geq -d+7}},
\end{split}
\end{equation*}
The summand in the right-hand side admits the following partial fraction decomposition:
\begin{equation*}
\begin{split}
& \frac{1}{[d_1]_{q^{-1}}[d_1-1]_{q^{-1}}[d_1-2]_{q^{-1}}} -1 \\
=& \frac{1}{[2]_{q^{-1}}!}\cdot  \frac{1}{1-q^{-d_1+2}}
-\frac{q^{-1}}{[1]_{q^{-1}}^2}\cdot \frac{1}{1-q^{-d_1+1}}
+ \frac{q^{-3}}{[2]_{q^{-1}}!} \cdot \frac{1}{1-q^{-d_1}}-1\\
=& \frac{1}{[2]_{q^{-1}}!}\cdot  \frac{q^{-d_1+2}}{1-q^{-d_1+2}}
-\frac{q^{-1}}{[1]_{q^{-1}}^2}\cdot \frac{q^{-d_1+1}}{1-q^{-d_1+1}}
+ \frac{q^{-3}}{[2]_{q^{-1}}!} \cdot \frac{q^{-d_1}}{1-q^{-d_1}},
\end{split}
\end{equation*}
and then we may obtain:
\begin{equation*}
\begin{split}
\text{(A)}
=&  \frac{d-3}{[\infty]_{q^{-1}}!} +
\frac{1}{[\infty ]_{q^{-1}}!} \Big(
\frac{1}{[2]_{q^{-1}}!} \sum_{i\geq 1}\frac{q^{-i}}{1-q^{-i}}
-\frac{q^{-1}}{[1]_{q^{-1}}^2} \sum_{i\geq 2}\frac{q^{-i}}{1-q^{-i}}
+ \frac{q^{-3}}{[2]_{q^{-1}}!} \sum_{i\geq 3}\frac{q^{-i}}{1-q^{-i}}\Big) \\
=& \frac{1}{[\infty]_{q^{-1}}!} \Big(
d-3 +\frac{q^{-1}}{[1]_{q^{-1}}^2 [2]_{q^{-1}}} -\frac{q^{-5}}{[1]_{q^{-1}} [2]_{q^{-1}}^2}
+ \sum_{i\geq 2}\frac{q^{-i}}{1-q^{-i}}
\Big) \Big|_{q^{\geq -d+7}}.
\end{split}
\end{equation*}

Next we compute the leading terms of $\text{(B)}$.
Again by a degree estimate and \eqref{eq-pochhammerseries} one has:
\begin{equation*}
\begin{split}
\text{(B)} =& -(1+q^{-1}) \cdot \sum_{d_1 = 3}^\infty \sum_{d_3 =0}^\infty
\frac{  q^{-(d_1-2) (d_3+1)} }{ [d_1]_{q^{-1}}! \cdot [d_3]_{q^{-1}}!} \\
=&
- \frac{(1+q^{-1})}{[\infty]_{q^{-1}}!}  \sum_{d_1 = 3}^\infty
\frac{q^{-d_1 +2}}{[d_1]_{q^{-1}}[d_1-1]_{q^{-1}}[d_1-2]_{q^{-1}}}
\Big|_{q^{\geq -d+7}},
\end{split}
\end{equation*}
and then by the splitting \eqref{eq-1over3prod-split} one obtains:
\begin{equation*}
\begin{split}
\text{(B)}
=&
- \frac{q^{-1}+ 2q^{-2}-q^{-4}}{[1]_{q^{-1}} [2]_{q^{-1}}^2 \cdot [\infty]_{q^{-1}}!}
\Big|_{q^{\geq -d+7}} ,
\end{split}
\end{equation*}
As for $\text{(C)}$,
similarly one has:
\begin{equation*}
\begin{split}
\text{(C)} = q^{-1}  \sum_{d_1 = 3}^\infty \sum_{d_3 =0}^\infty
\frac{q^{-(d_1-2) (d_3+2)} }{ [d_1]_{q^{-1}}!  [d_3]_{q^{-1}}!}
= \frac{q^{-1}}{[\infty]_{q^{-1}}!} \sum_{d_1=3}^\infty
\frac{q^{-2d_1 +4}}{[d_1]_{q^{-1}}[d_1-1]_{q^{-1}}[d_1-2]_{q^{-1}}}
\Big|_{q^{\geq -d+7}},
\end{split}
\end{equation*}
and by the splitting
\begin{equation*}
\begin{split}
\frac{q^{-2d_1 +4}}{[d_1]_{q^{-1}}[d_1-1]_{q^{-1}}[d_1-2]_{q^{-1}}}
= -\frac{1}{[2]_{q^{-1}}}
\big(
\frac{q^{-2d_1+2}}{[d_1]_{q^{-1}}[d_1-1]_{q^{-1}}}
-\frac{q^{-2d_1+4}}{[d_1-1]_{q^{-1}}[d_1-2]_{q^{-1}}}
\big)
\end{split}
\end{equation*}
one obtains:
\begin{equation*}
\begin{split}
\text{(C)}
=& \frac{q^{-3}}{[1]_{q^{-1}} [2]_{q^{-1}}^2 \cdot [\infty]_{q^{-1}}!}
\Big|_{q^{\geq -d+7}}.
\end{split}
\end{equation*}

Now plug the above results for $\text{(A)}$, $\text{(B)}$, $\text{(C)}$ into \eqref{eq-AppC5-2ndsum},
and we obtain:
\begin{equation*}
\begin{split}
&\sum_{\substack{d_1+d_2+d_3=d\\ d_2\geq 1;\text{ }d_3\geq 0;\text{ }d_1\geq 3}}
q^{-(d_1-2) d_3} \cdot \frac{[d_1-1]_{q^{-1}} [d_1-2]_{q^{-1}}}{ [d_1]_{q^{-1}}! \cdot [d_3]_{q^{-1}}!} \\
=& \frac{1}{[\infty]_{q^{-1}}!} \Big(
d-3 -\frac{q^{-2}}{[2]_{q^{-1}}}
+ \sum_{i\geq 2} \frac{q^{-i}}{1-q^{-i}}
\Big) \Big|_{q^{\geq -d+7}},
\end{split}
\end{equation*}
and then by combining this with \eqref{eq-AppC-decompyprime}, \eqref{eq-AppC-decompyprime-2}
one obtains:
\begin{equation*}
\begin{split}
&\text{(y')} =
\frac{1}{([2]_{q^{-1}}!)^2 \cdot ([\infty]_{q^{-1}}!)^3} \Big(
-[\infty]_{q^{-1}}! +(d-2)
-\frac{q^{-1}}{[1]_{q^{-1}}} -\frac{q^{-2}}{[2]_{q^{-1}}}
\Big) \Big|_{q^{\geq -d+7}}.
\end{split}
\end{equation*}

Plug this and \eqref{eq-xprime-leading}, \eqref{eq-zprime-leading} into \eqref{eq-AppC5-J2decomp},
and we finally obtain:
\begin{Proposition}
\label{Prop-I=2-case2}
For $d\geq 3$,
the leading terms of $(\text{j}_2)$ are given by:
\begin{equation*}
\begin{split}
(\text{j}_2) =& \frac{(-1)^d \cdot q^{\half d(d-7) -6}}{([2]_{q^{-1}}!)^2 \cdot ([\infty]_{q^{-1}}!)^3} \Big(
\binom{d-1}{2}
-\frac{2(q^{-1} +2q^{-2})}{ [2]_{q^{-1}}}\\
&\qquad\qquad
+2\sum_{i\geq 1}\frac{q^{-i}}{1-q^{-i}} -3\sum_{i\geq 1}\frac{q^{-i}}{(1-q^{-i})^2}
\Big)  \Big|_{q^{\geq \half (d^2-3d) -3d+1}}.
\end{split}
\end{equation*}
\end{Proposition}

\subsection{Leading terms of $I^{=2}(d)$ -- Case 3) and summary}

Now we compute the leading terms of $(\text{j}_3)$.
By \eqref{eq-gzlem6.7} and Theorem \ref{thm-Delta2-expl} one has:
\begin{equation*}
\begin{split}
(\text{j}_3) = &
\sum_{\substack{d_1+d_2+d_3=d\\ d_2,d_3\geq 2;\text{ }d_1\geq 0}}
\frac{(-1)^d q^{ \half d(d-7) -2}}{[1]_{q^{-1}}^2} \cdot
\prod_{i=2}^3 \Big( \frac{\tW_{(d_1),(d_i-1)}}{[1]_{q^{-1}}} -
\frac{q^{-(d_1+1)(d_i-1)}}{[d_1]_{q^{-1}}! [d_i]_{q^{-1}}!} \Big) \\
&\cdot \Big( (1-q^{-1}+q^{-2}) \tW_{(d_2-2),(d_3-2)}
+ q^{-(d_2-1)(d_3-1)-1} \cdot \frac{[d_2-1]_{q^{-1}}  [d_3-1]_{q^{-1}}}{ [d_2]_{q^{-1}}! \cdot [d_3]_{q^{-1}}! }
\Big).
\end{split}
\end{equation*}
By a simple degree estimate (here we omit the details)
one may see that:
\be
\label{eq-Appc6leadingJ3}
\begin{split}
(\text{j}_3) = & \frac{(-1)^d q^{ \half d(d-7) -2}}{[1]_{q^{-1}}^2}
\Big( \text{(X)} + \text{(Y)} - 2\cdot \text{(Z)} \Big)
\Big|_{q^{\geq \half (d^2-3d) -3d+1}},
\end{split}
\ee
where:
\be
\label{eq-AppC-defXYZ}
\begin{split}
&\text{(X)} =
\sum_{\substack{d_1+d_2+d_3=d\\ d_2,d_3\geq 2;\text{ }d_1\geq 0}}
\frac{1-q^{-1}+q^{-2}}{({[1]_{q^{-1}}})^2}
\cdot \tW_{(d_1),(d_2-1)}
\tW_{(d_1),(d_3-1)}
\tW_{(d_2-2),(d_3-2)},\\
&\text{(Y)}= \sum_{\substack{d_1+d_2+d_3=d\\ d_2,d_3\geq 2;\text{ }d_1\geq 0}}
\frac{q^{-(d_2-1)(d_3-1)-1}}{({[1]_{q^{-1}}})^2}
\cdot \frac{[d_2-1]_{q^{-1}}  [d_3-1]_{q^{-1}}}{ [d_2]_{q^{-1}}! \cdot [d_3]_{q^{-1}}! }
\cdot \tW_{(d_1),(d_2-1)}
\tW_{(d_1),(d_3-1)},\\
&\text{(Z)} =
\sum_{\substack{d_1+d_2+d_3=d\\ d_2,d_3\geq 2;\text{ }d_1\geq 0}}
\frac{1-q^{-1}+q^{-2}}{{[1]_{q^{-1}}}}
\cdot \frac{q^{-(d_1+1)(d_2-1)}}{[d_1]_{q^{-1}}! [d_2]_{q^{-1}}!}
\cdot
\tW_{(d_1),(d_3-1)}
\tW_{(d_2-2),(d_3-2)}.
\end{split}
\ee

The computation of the leading terms of $\text{(X)}$ is similar to the case of $\text{(x')}$
in last subsection.
By decomposing each factor $\tW_{(m),(n)}$ into:
\be
\label{eq-decomptWmn}
\tW_{(m),(n)} = \frac{1}{[\infty]_{q^{-1}}!}
\big( T_{m}^x \frac{(q^{-n -1} ;q^{-1})_\infty}{(xq^{-n -1} ;q^{-1})_\infty} -1\big)
+ \frac{1}{[\infty]_{q^{-1}}!}
\ee
and a simple degree estimate
(see \eqref{eq-gzlem-5.7}, and here we omit the details),
we obtain:
\be
\label{eq-appC6-Xdecomp}
\begin{split}
&\text{(X)} =\frac{1-q^{-1}+q^{-2}}{({[1]_{q^{-1}}})^2 \cdot([\infty]_{q^{-1}}!)^3}
\bigg( \binom{d-2}{2}+
\sum_{\substack{d_1+d_2+d_3=d\\ d_2,d_3\geq 2;\text{ }d_1\geq 0}} \Big(
\big( T_{d_1}^x \frac{(q^{-d_2} ;q^{-1})_\infty}{(xq^{-d_2} ;q^{-1})_\infty} -1 \big) \\
&\qquad
+ \big( T_{d_1}^x \frac{(q^{-d_3} ;q^{-1})_\infty}{(xq^{-d_3} ;q^{-1})_\infty} -1 \big)
+ \big( T_{d_2-2}^x \frac{(q^{-d_3+1} ;q^{-1})_\infty}{(xq^{-d_3+1} ;q^{-1})_\infty} -1 \big)
\Big) \bigg) \bigg|_{q^{\geq -d+3}}.
\end{split}
\ee
First consider the summations:
\begin{equation*}
\begin{split}
\sum_{\substack{d_1+d_2+d_3=d\\ d_2,d_3\geq 2;\text{ }d_1\geq 0}}
\big( T_{d_1}^x \frac{(q^{-d_2} ;q^{-1})_\infty}{(xq^{-d_2} ;q^{-1})_\infty} -1 \big)
= \sum_{\substack{d_1+d_2+d_3=d\\ d_2,d_3\geq 2;\text{ }d_1\geq 0}}
\big( T_{d_1}^x \frac{(q^{-d_3} ;q^{-1})_\infty}{(xq^{-d_3} ;q^{-1})_\infty} -1 \big).
\end{split}
\end{equation*}
One has:
\begin{equation*}
\begin{split}
& \sum_{\substack{d_1+d_2+d_3=d\\ d_2,d_3\geq 2;\text{ }d_1\geq 0}}
\big( T_{d_1}^x \frac{(q^{-d_2} ;q^{-1})_\infty}{(xq^{-d_2} ;q^{-1})_\infty} -1 \big) \\
=& \sum_{d_2=2}^{d-2 } \big( T_{0}^x \frac{(q^{-d_2} ;q^{-1})_\infty}{(xq^{-d_2} ;q^{-1})_\infty} -1 \big)
+ \sum_{\substack{d_1+d_2+d_3=d\\ d_2,d_3\geq 2;\text{ }d_1\geq 1}}
\big( T_{d_1}^x \frac{(q^{-d_2} ;q^{-1})_\infty}{(xq^{-d_2} ;q^{-1})_\infty} -1 \big).
\end{split}
\end{equation*}
and the second summation in the right-hand side is exactly \eqref{eq-Compxprime-1}.
Moreover, by \eqref{eq-computed2=01} and degree reason we know that
the first summation in the right-hand side is:
\begin{equation*}
\begin{split}
\sum_{d_2=2}^{d-2 } \big( T_{0}^x \frac{(q^{-d_2} ;q^{-1})_\infty}{(xq^{-d_2} ;q^{-1})_\infty} -1 \big)
= \sum_{d_2=2}^{d-2 } \big( \frac{[\infty]_{q^{-1}}!}{[d_2-1]_{q^{-1}}!} -1\big)
= \sum_{d_2=2}^{\infty } \big( \frac{[\infty]_{q^{-1}}!}{[d_2-1]_{q^{-1}}!} -1\big)
\Big|_{q^{\geq -d+3}},
\end{split}
\end{equation*}
and then by \eqref{eq-gz-prop5.8} we obtain:
\begin{equation*}
\begin{split}
\sum_{d_2=2}^{d-2 } \big( T_{0}^x \frac{(q^{-d_2} ;q^{-1})_\infty}{(xq^{-d_2} ;q^{-1})_\infty} -1 \big)
= \Big( -[\infty]_{q^{-1}}! +1 -\sum_{i\geq 1}\frac{q^{-i}}{1-q^{-i}}
\Big)
\Big|_{q^{\geq -d+3}}.
\end{split}
\end{equation*}
Then by this equality and \eqref{eq-Compxprime-1} we get:
\begin{equation*}
\begin{split}
& \sum_{\substack{d_1+d_2+d_3=d\\ d_2,d_3\geq 2;\text{ }d_1\geq 0}}
\big( T_{d_1}^x \frac{(q^{-d_2} ;q^{-1})_\infty}{(xq^{-d_2} ;q^{-1})_\infty} -1 \big)
=\Big( \sum_{i\geq 1} \frac{q^{-i}}{1-q^{-i}}
-\sum_{i\geq 1} \frac{q^{-i}}{(1-q^{-i})^2}
\Big)
\Big|_{q^{\geq -d+3}}.
\end{split}
\end{equation*}

Now consider the second summation in the right-hand side of \eqref{eq-appC6-Xdecomp}.
Denote $k_i = d_i -2$ for $i=2,3$,
and then by a simple degree estimate one has:
\begin{equation*}
\begin{split}
\sum_{\substack{d_1+d_2+d_3=d\\ d_2,d_3\geq 2;\text{ }d_1\geq 0}}
\big( T_{d_2-2}^x \frac{(q^{-d_3+1} ;q^{-1})_\infty}{(xq^{-d_3+1} ;q^{-1})_\infty} -1 \big)
= &
\sum_{\substack{ k_2 + k_3 \leq d-4 \\ k_2,k_3\geq 0}}
\big( T_{k_2}^x \frac{(q^{-k_3-1} ;q^{-1})_\infty}{(xq^{-k_3-1} ;q^{-1})_\infty} -1 \big) \\
=& \sum_{k_2=0}^\infty \sum_{k_3=0}^\infty
\big( T_{k_2}^x \frac{(q^{-k_3-1} ;q^{-1})_\infty}{(xq^{-k_3-1} ;q^{-1})_\infty} -1 \big)
\Big|_{q^{\geq -d+3}}.
\end{split}
\end{equation*}
And then by \eqref{eq-W2infty-def}, \eqref{eq-W2infty=} we obtain:
\begin{equation*}
\begin{split}
\sum_{\substack{d_1+d_2+d_3=d\\ d_2,d_3\geq 2;\text{ }d_1\geq 0}}
\big( T_{d_2-2}^x \frac{(q^{-d_3+1} ;q^{-1})_\infty}{(xq^{-d_3+1} ;q^{-1})_\infty} -1 \big)
= -\sum_{i\geq 1} \frac{q^{-i}}{(1-q^{-i})^2}
\Big|_{q^{\geq -d+3}}.
\end{split}
\end{equation*}
And now we can conclude that:
\be
\label{eq-appC6-Xleading}
\begin{split}
&\text{(X)} =\frac{1-q^{-1}+q^{-2}}{({[1]_{q^{-1}}})^2 \cdot([\infty]_{q^{-1}}!)^3}
\Big( \binom{d-2}{2}+
2 \sum_{i\geq 1} \frac{q^{-i}}{1-q^{-i}}
-3\sum_{i\geq 1} \frac{q^{-i}}{(1-q^{-i})^2}
\Big)
\Big|_{q^{\geq -d+3}}.
\end{split}
\ee

Then we compute the leading terms of $\text{(Z)}$.
Again apply the decomposition \eqref{eq-decomptWmn} to all the factors
of the form $\tW_{(m),(n)}$ in the definition \eqref{eq-AppC-defXYZ} of $\text{(Y)}$,
and then by a simple degree estimate
(see \eqref{eq-gzlem-5.7}, and here we omit the details) one obtains:
\begin{equation*}
\begin{split}
&\text{(Z)} =
\sum_{\substack{d_1+d_2+d_3=d\\ d_2,d_3\geq 2;\text{ }d_1\geq 0}}
\frac{1-q^{-1}+q^{-2}}{{[1]_{q^{-1}}}}
\cdot \frac{q^{-(d_1+1)(d_2-1)}}{[d_1]_{q^{-1}}! [d_2]_{q^{-1}}!}
\cdot \frac{1}{([\infty]_{q^{-1}}!)^2}
\Big|_{q^{\geq -d+3}}.
\end{split}
\end{equation*}
Moreover,
when $d_1+d_2 \geq d-1$ and $d_1\geq 0$, $d_2\geq 2$ one has
$\deg q^{-(d_1+1)(d_2-1)} \leq -d+2$,
and thus:
\begin{equation*}
\begin{split}
&\text{(Z)} =
\frac{1-q^{-1}+q^{-2}}{{[1]_{q^{-1}}} \cdot ([\infty]_{q^{-1}}!)^2}
\sum_{d_2 = 2}^\infty \sum_{d_1=0}^\infty
 \frac{q^{-(d_1+1)(d_2-1)}}{[d_1]_{q^{-1}}! [d_2]_{q^{-1}}!}
\Big|_{q^{\geq -d+3}}.
\end{split}
\end{equation*}
And then by \eqref{eq-pochhammerseries} we get:
\begin{equation*}
\begin{split}
\text{(Z)} = &
\frac{1-q^{-1}+q^{-2}}{{[1]_{q^{-1}}} \cdot ([\infty]_{q^{-1}}!)^2}
\sum_{d_2 = 2}^\infty \frac{q^{-d_2 +1}}{[d_2]_{q^{-1}}!} \cdot
\frac{1}{(q^{-d_2+1};q^{-1})_\infty} \\
=& \frac{1-q^{-1}+q^{-2}}{{[1]_{q^{-1}}} \cdot ([\infty]_{q^{-1}}!)^3}
\sum_{d_2=2}^\infty \frac{q^{-d_2+1}}{[d_2]_{q^{-1}}[d_2-1]_{q^{-1}}}
\Big|_{q^{\geq -d+3}}.
\end{split}
\end{equation*}
Then by the splitting
\begin{equation*}
\frac{q^{-d_2+1}}{[d_2]_{q^{-1}}[d_2-1]_{q^{-1}}} =
-\frac{1}{[1]_{q^{-1}}} \big(
\frac{1}{[d_2]_{q^{-1}}} - \frac{1}{[d_2-1]_{q^{-1}}}
\big)
\end{equation*}
one concludes that:
\be
\label{eq-appC6-Zleading}
\begin{split}
\text{(Z)} = &
\frac{q^{-1}(1-q^{-1}+q^{-2})}{{([1]_{q^{-1}}})^3 \cdot ([\infty]_{q^{-1}}!)^3}
\Big|_{q^{\geq -d+3}}.
\end{split}
\ee

Finally we compute the leading terms of $\text{(Y)}$.
Similar to the discussions for $\text{(Z)}$,
apply the decomposition \eqref{eq-decomptWmn} to all the factors
of the form $\tW_{(m),(n)}$ in the definition \eqref{eq-AppC-defXYZ} of $\text{(Y)}$,
and then by a simple degree estimate we obtain:
\begin{equation*}
\begin{split}
\text{(Y)}= \sum_{\substack{d_1+d_2+d_3=d\\ d_2,d_3\geq 2;\text{ }d_1\geq 0}}
\frac{q^{-(d_2-1)(d_3-1)-1}}{({[1]_{q^{-1}}})^2}
\cdot \frac{[d_2-1]_{q^{-1}}  [d_3-1]_{q^{-1}}}{ [d_2]_{q^{-1}}! \cdot [d_3]_{q^{-1}}! }
\cdot \frac{1}{([\infty]_{q^{-1}}!)^2}
\Big|_{q^{\geq -d+3}}.
\end{split}
\end{equation*}
Moreover,
when $d_2+d_3\geq d+1$ and $d_2,d_3\geq 2$ one has
$\deg q^{-(d_2-1)(d_3-1)-1} \leq -d+1$,
and thus:
\be
\label{eq-AppC6Ysum}
\begin{split}
\text{(Y)}= \frac{1}{({[1]_{q^{-1}}})^2 ([\infty]_{q^{-1}}!)^2}
\sum_{d_2=2}^\infty \sum_{d_3=2}^\infty
q^{-(d_2-1)(d_3-1)-1 } \cdot
\frac{[d_2-1]_{q^{-1}}  [d_3-1]_{q^{-1}}}{ [d_2]_{q^{-1}}! \cdot [d_3]_{q^{-1}}! }
\Big|_{q^{\geq -d+3}}.
\end{split}
\ee
Now we compute the summation in the right-hand side.
By \eqref{eq-pochhammerseries} one knows that:
\begin{equation*}
\begin{split}
\sum_{n=1}^\infty \frac{[n]_{q^{-1}}}{[n+1]_{q^{-1}}!}z^n
=& \sum_{n=1}^\infty \frac{z^n }{[n+1]_{q^{-1}}!}
-\frac{(q^{-1}z)^n}{[n+1]_{q^{-1}}!} \\
=& \big(\frac{\frac{1}{(z;q^{-1})_\infty} -1 }{z} -\frac{1}{[1]_{q^{-1}}}\big)
- \big(\frac{\frac{1}{(q^{-1}z;q^{-1})_\infty} -1 }{q^{-1}z} -\frac{1}{[1]_{q^{-1}}}\big).
\end{split}
\end{equation*}
Notice that:
\begin{equation*}
\frac{1}{(q^{-1}z;q^{-1})_\infty} =
\frac{1-z}{(z;q^{-1})_\infty},
\end{equation*}
and then we may obtain:
\begin{equation*}
\begin{split}
\sum_{n=1}^\infty \frac{[n]_{q^{-1}}}{[n+1]_{q^{-1}}!} z^n
=& \frac{(z-1+q^{-1})\cdot \frac{1}{(z;q^{-1})_\infty} +1-q^{-1} }{q^{-1}z}.
\end{split}
\end{equation*}
Therefore one has:
\be
\label{eq-AppC6Ysum-1}
\begin{split}
& \sum_{d_2=2}^\infty \sum_{d_3=2}^\infty
q^{-(d_2-1)(d_3-1)-1} \cdot
\frac{[d_2-1]_{q^{-1}}  [d_3-1]_{q^{-1}}}{ [d_2]_{q^{-1}}! \cdot [d_3]_{q^{-1}}! } \\
=& q^{-1} \cdot \sum_{d_2=2}^\infty \frac{[d_2-1]_{q^{-1}}}{[d_2]_{q^{-1}}!}
\cdot \frac{(q^{-d_2+1}-1+q^{-1})\cdot \frac{1}{(q^{-d_2+1};q^{-1})_\infty} +1-q^{-1} }{q^{-1}\cdot q^{-d_2+1}} \\
=& \sum_{d_2=2}^\infty q^{d_2 -1} \Big(
\frac{q^{-d_2+1} - [1]_{q^{-1}}}{[d_2]_{q^{-1}} \cdot [\infty]_{q^{-1}}!}
+ \frac{[1]_{q^{-1}} \cdot [d_2-1]_{q^{-1}}}{ [d_2]_{q^{-1}}!}
\Big).
\end{split}
\ee
This summation can be decomposed in the following way:
\be
\label{eq-decompY-ab}
\begin{split}
& \sum_{d_2=2}^\infty q^{d_2-1} \Big(
\frac{q^{-d_2+1} - [1]_{q^{-1}}}{[d_2]_{q^{-1}} \cdot [\infty]_{q^{-1}}!}
+ \frac{[1]_{q^{-1}} \cdot [d_2-1]_{q^{-1}}}{ [d_2]_{q^{-1}}!}
\Big) = \alpha(q) + \beta(q),
\end{split}
\ee
where:
\begin{equation*}
\begin{split}
&\alpha(q) = \frac{[1]_{q^{-1}}}{[\infty]_{q^{-1}}!} \sum_{d_2=2}^\infty
(q^{d_2-1}-1)\big( \frac{[\infty]_{q^{-1}}!}{[d_2]_{q^{-1}}!} -1 + \frac{q^{-d_2-1}}{[1]_{q^{-1}}} \big),\\
&\beta(q) = \frac{1}{[\infty]_{q^{-1}}!} \sum_{d_2=2}^\infty
\big(\frac{q^{-d_2} -q^{-d_2 -1} +q^{-d_2 -2}}{[d_2]_{q^{-1}}} + q^{-d_2 -1} \big)
\end{split}
\end{equation*}
Here the summand in $\alpha(q)$ involves positive powers of $q$,
and thus we need first to check the well-definedness of $\alpha(q)$.
One has:
\begin{equation*}
\frac{[\infty]_{q^{-1}}!}{[d_2]_{q^{-1}}!} -1
=\prod_{j=d_2+1}^\infty (1-q^{-j}) -1
= -\sum_{i=d_2 +1}^\infty q^{-i} + O(q^{-2d_2-3})
= -\frac{q^{-d_2-1}}{[1]_{q^{-1}}} + O(q^{-2d_2-3}),
\end{equation*}
and thus:
\begin{equation*}
\deg (q^{d_2-1}-1)\big( \frac{[\infty]_{q^{-1}}!}{[d_2]_{q^{-1}}!} -1 + \frac{q^{-d_2-1}}{[1]_{q^{-1}}} \big)
\leq -d_2-4,
\end{equation*}
and thus $\alpha(q)$ is well-defined as a formal Laurent series in $q$.

Next we compute the leading terms of $\alpha(q)$ and $\beta(q)$.
It is clear that:
\begin{equation*}
\begin{split}
&\beta(q) = \frac{1}{[\infty]_{q^{-1}}!}
\Big( -q^{-1}+
(1-q^{-1}+q^{-2})\sum_{i\geq 1} \frac{q^{-i}}{1-q^{-i}}
\Big).
\end{split}
\end{equation*}
As for $\alpha(q)$,
one has:
\be
\label{eq-Appalpha-decomp}
\begin{split}
&\alpha(q) = \frac{[1]_{q^{-1}}}{[\infty]_{q^{-1}}!} \Big( \sum_{d_2=2}^\infty
q^{d_2-1}\big( \frac{[\infty]_{q^{-1}}!}{[d_2]_{q^{-1}}!} -1 + \frac{q^{-d_2-1}}{[1]_{q^{-1}}} \big)
-\sum_{d_2=2}^\infty
\big( \frac{[\infty]_{q^{-1}}!}{[d_2]_{q^{-1}}!} -1 + \frac{q^{-d_2-1}}{[1]_{q^{-1}}} \big)
\Big).
\end{split}
\ee
Let $z$ be a formal variable,
then by \eqref{eq-pochhammerseries} one has:
\begin{equation*}
\begin{split}
\sum_{d_2=0}^\infty
z^{d_2} \cdot
\big( \frac{[\infty]_{q^{-1}}!}{[d_2]_{q^{-1}}!} -1 + \frac{q^{-d_2-1}}{[1]_{q^{-1}}} \big)
=&\frac{[\infty]_{q^{-1}}!}{(z;q^{-1})_\infty} -\frac{1}{1-z}
+ \frac{q^{-1}}{[1]_{q^{-1}}} \cdot \frac{1}{1-zq^{-1}} \\
=& \frac{\frac{[\infty]_{q^{-1}}!}{(zq^{-1};q^{-1})_\infty} - 1}{1-z}
+ \frac{q^{-1}}{[1]_{q^{-1}}} \cdot \frac{1}{1-zq^{-1}}.
\end{split}
\end{equation*}
Let $z\to 1$,
and one obtains:
\be
\label{eq-Appalphaleading-1}
\begin{split}
\sum_{d_2=0}^\infty
\big( \frac{[\infty]_{q^{-1}}!}{[d_2]_{q^{-1}}!} -1 + \frac{q^{-d_2-1}}{[1]_{q^{-1}}} \big)
=& - \frac{d}{dz} \big( \frac{[\infty]_{q^{-1}}!}{(zq^{-1};q^{-1})_\infty} - 1 \big)\big|_{z=1}
+ \frac{q^{-1}}{[1]_{q^{-1}}} \cdot \frac{1}{1-q^{-1}} \\
=& -\sum_{i\geq 1}\frac{q^{-i}}{1-q^{-i}} + \frac{q^{-1}}{[1]_{q^{-1}}^2}.
\end{split}
\ee
Moreover,
denote:
\begin{equation*}
f(z) = \frac{ (z;q^{-1})_\infty }{1-zq^{-1}}
=(1-z)\cdot \prod_{j\geq 2} (1-zq^{-j}),
\end{equation*}
then one has:
\begin{equation*}
\begin{split}
\sum_{d_2=0}^\infty
z^{d_2} \cdot
\big( \frac{[\infty]_{q^{-1}}!}{[d_2]_{q^{-1}}!} -1 + \frac{q^{-d_2-1}}{[1]_{q^{-1}}} \big)
=&\frac{[\infty]_{q^{-1}}!}{ (1-zq^{-1}) f(z)} -\frac{1}{1-z}
+ \frac{q^{-1}}{[1]_{q^{-1}}} \cdot \frac{1}{1-zq^{-1}} \\
=& \frac{\frac{[\infty]_{q^{-1}}!}{f(z)} + \frac{q^{-1}}{[1]_{q^{-1}}} }{1-zq^{-1}}
-\frac{1}{1-z},
\end{split}
\end{equation*}
and then taking $z\to q$ gives:
\be
\label{eq-Appalphaleading-2}
\begin{split}
\sum_{d_2=0}^\infty
q^{d_2} \cdot
\big( \frac{[\infty]_{q^{-1}}!}{[d_2]_{q^{-1}}!} -1 + \frac{q^{-d_2-1}}{[1]_{q^{-1}}} \big)
=& \frac{ \frac{d}{dz}\big( \frac{[\infty]_{q^{-1}}!}{f(z)} + \frac{q^{-1}}{[1]_{q^{-1}}} \big) \big|_{z=q}}
{-q^{-1}}
-\frac{1}{1-q} \\
=& \frac{q^{-1}}{[1]_{q^{-1}}} \sum_{i\geq 1} \frac{q^{-i}}{1-q^{-i}}
-\frac{q^{-2}}{[1]_{q^{-1}}^2}.
\end{split}
\ee
Now plug \eqref{eq-Appalphaleading-1} and \eqref{eq-Appalphaleading-2} into \eqref{eq-Appalpha-decomp},
and we obtain:
\begin{equation*}
\begin{split}
&\alpha(q) = \frac{[1]_{q^{-1}}}{[\infty]_{q^{-1}}!} \Big(
[1]_{q^{-1}}\cdot [\infty]_{q^{-1}}! -\frac{2q^{-2}}{[1]_{q^{-1}}^2} -[1]_{q^{-1}}
+\frac{1-q^{-1}+q^{-2}}{[1]_{q^{-1}}} \sum_{i\geq 1} \frac{q^{-i}}{1-q^{-i}}
\Big).
\end{split}
\end{equation*}
And by combining the results for $\alpha(q)$, $\beta(q)$ with
\eqref{eq-AppC6Ysum}, \eqref{eq-AppC6Ysum-1}, \eqref{eq-decompY-ab},
we obtain:
\begin{equation*}
\begin{split}
\text{(Y)} = &\frac{1}{[1]_{q^{-1}}^2 ([\infty]_{q^{-1}}!)^3 } \Big(
[1]_{q^{-1}}^2 \cdot [\infty]_{q^{-1}}!
-\frac{1-2q^{-1} +4q^{-2} -q^{-3}}{[1]_{q^{-1}}} \\
& \qquad\qquad\qquad\qquad
+2(1-q^{-1}+q^{-2})\sum_{i\geq 1}\frac{q^{-i}}{1-q^{-i}}
\Big) \Big|_{q\geq -d+3}.
\end{split}
\end{equation*}

Now we can plug this equality and \eqref{eq-appC6-Xleading}, \eqref{eq-appC6-Zleading}
into \eqref{eq-Appc6leadingJ3},
and in this way we obtain:
\begin{Proposition}
\label{Prop-I=2-case3}
For $d\geq 4$,
the leading terms of $(\text{j}_3)$ are given by:
\begin{equation*}
\begin{split}
(\text{j}_3) = & \frac{(-1)^d q^{ \half d(d-7) -2}}{[1]_{q^{-1}}^4 ([\infty]_{q^{-1}}!)^3} \bigg(
[1]_{q^{-1}}^2 \cdot [\infty]_{q^{-1}}!
- \frac{1}{[1]_{q^{-1}}} \big(1+2q^{-2}+q^{-3} \big) \\
&+ (1-q^{-1}+q^{-2}) \Big( \binom{d-2}{2}+
4 \sum_{i\geq 1} \frac{q^{-i}}{1-q^{-i}}
-3\sum_{i\geq 1} \frac{q^{-i}}{(1-q^{-i})^2}
\Big)
\bigg)\bigg|_{q^{\geq \half (d^2-3d) -3d+1}}.
\end{split}
\end{equation*}
\end{Proposition}

Now plug the results of
Propositions \ref{Prop-I=2-case1}, \ref{Prop-I=2-case2}, \ref{Prop-I=2-case3} into \eqref{eq-AppC-I=2decompJ},
and we finally obtain the leading terms of $I^{=2}(d)$:
\begin{Theorem}
\label{thm-I=2leading}
For $d\geq 4$,
the leading terms of $I^{=2}(d)$ are given by:
\begin{equation*}
\begin{split}
I^{=2}(d) =&
\frac{3\cdot (-1)^d q^{\half d(d-7)-2}}{([2]_{q^{-1}}!)^2 \cdot ([\infty]_{q^{-1}}!)^3}
\Big( [2]_{q^{-1}}^2 \cdot [\infty ]_{q^{-1}}!
+ \frac{5+4q^{-1} -5q^{-2}-4q^{-3} -q^{-4}-3q^{-6}}{[2]_{q^{-1}}} \\
&+ (2+q^{-1}+q^{-3}+2q^{-4})\binom{d}{2}
-(4+2q^{-1} +2q^{-3}+3q^{-4}) \big( d-2\sum_{i\geq 2} \frac{q^{-i}}{1-q^{-i}} \big) \\
& -(6+3q^{-1}+3q^{-3}+6q^{-4}) \sum_{i\geq 1} \frac{q^{-i}}{(1-q^{-i})^2}
\Big) \Big|_{q^{\geq \half d(d-3) -3d+1}}.
\end{split}
\end{equation*}
\end{Theorem}
Then by this theorem, Theorem \ref{thm-I=1leading},
Theorem \ref{thm-Wdqleadingterm} and \eqref{eq-relnI=0Wd},
we finally obtain:
\begin{equation*}
\begin{split}
I(d)
=& \frac{ 3 \cdot (-1)^d q^{ \half d(d-3)} }{ ([\infty]_{q^{-1}}!)^3}
\bigg( \frac{1}{3}  \binom{d+2}{2}
+ \frac{q^{-d-2}}{[1]_{q^{-1}}^2}\cdot \big(\frac{ d^2}{2} + \frac{ d}{2} +3 \big) \\
&-\Big( 1 + \frac{3q^{-d-2}} {[1]_{q^{-1}}^2} +
\frac{ q^{-2d-2}(6+3q^{-1} +3q^{-3} +6q^{-4})}{([2]_{q^{-1}}!)^2}
\Big)\sum_{i\geq 1}\frac{q^{-i}}{(1-q^{-i})^2} \\
& + \frac{ q^{-2d-2} ( 2+q^{-1} +q^{-3}+2q^{-4} )}{([2]_{q^{-1}}!)^2}
 \big( \frac{ d^2}{2} - \frac{ d}{2} +6\big)
\bigg)\bigg|_{q^{\geq \half (d^2-3d) -3d+1}},
\end{split}
\end{equation*}
for $d\geq 4$,
and this is Theorem \ref{thm-main1-leadI}.

\vspace{.25in}

{\em Acknowledgements.}
The authors thank Professor Jian Zhou for the collaboration on some related works.
ZW is partly supported by the National Natural Science Foundation of China (No. 12601113), and the Fundamental Research Funds for the Central Universities (No. 2025BRSXB001).
CY is partly supported by the National Natural Science Foundation of China (No. 12401079) and Natural Science Foundation of Hubei (No. 2026AFA008).

\end{document}